\documentclass[11pt]{article}

\usepackage[english]{babel}
\usepackage{amssymb}
\usepackage{amsmath}
\usepackage{hyperref}
\usepackage{enumerate}
\usepackage{amsthm}
\usepackage{amsfonts}
\usepackage{listings}
\usepackage[curve]{xypic}

\usepackage{MnSymbol}

\DeclareMathSymbol{\mlq}{\mathord}{operators}{``}
\DeclareMathSymbol{\mrq}{\mathord}{operators}{`'}

\title{Derivator Six-Functor-Formalisms --- Construction II}
\date{April 6, 2022}
\author{Fritz H\"ormann\\ Mathematisches Institut, Albert-Ludwigs-Universit\"at Freiburg}

\usepackage[numbers,sort&compress]{natbib}

\usepackage{color}

\usepackage{tikz}
\newcommand*\numcirc[1]{\tikz[baseline=(char.base)]{
            \node[shape=circle,draw,inner sep=2pt] (char) {#1};}}

\newtheorem{SATZ}{Theorem}[section]
\newtheorem{HAUPTSATZ}[SATZ]{Main Theorem}
\newtheorem{LEMMA}[SATZ]{Lemma}
\newtheorem{DEF}[SATZ]{Definition}
\newtheorem{PROP}[SATZ]{Proposition}

\newtheorem{BEISPIEL}[SATZ]{Example}

\newtheorem{FUNDLEMMA}[SATZ]{Fundamental Lemma}

\newtheorem{KOR}[SATZ]{Corollary}

\newtheorem{BEM}[SATZ]{Remark}

\newtheoremstyle{bare}        
  {}            
  {}            
  {\normalfont}                 
  {}                            
  {\bfseries}                   
  {}                            
  {.0em}                           
  {\thmnumber{#2}#1{\thmnote{ \normalfont(#3)}}. } 

\theoremstyle{bare}
\newtheorem{PAR}[SATZ]{}

\newcommand{\comment}[1]{}
\newcommand{\commentempty}[1]{}

\newcommand{\iso}{\stackrel{\sim}{\longrightarrow}}
\newcommand{\isor}{\stackrel{\sim}{\longleftarrow}}

\newcommand{\N}{ \mathbb{N} }

\newcommand{\DD}{ \mathbb{D} }
\newcommand{\EE}{ \mathbb{E} }

\newcommand{\PP}{ \mathbb{P} }
\newcommand{\SSS}{ \mathbb{S} }
\newcommand{\OO}{ {\cal O} }

\DeclareMathOperator{\Cor}{Cor}

\DeclareMathOperator*{\hocolim}{hocolim}
\DeclareMathOperator*{\holim}{holim}

\DeclareMathOperator{\id}{id}

\DeclareMathOperator{\op}{op}
\DeclareMathOperator{\cart}{cart}
\DeclareMathOperator{\cocart}{cocart}
\DeclareMathOperator{\ws}{ws}

\DeclareMathOperator{\Hom}{Hom}
\DeclareMathOperator{\Fun}{Fun}

\DeclareMathOperator{\Mor}{Mor}

\DeclareMathOperator{\pr}{pr}

\DeclareMathOperator{\cor}{cor}
\DeclareMathOperator{\comp}{comp}

\DeclareMathOperator{\lax}{lax}
\DeclareMathOperator{\oplax}{oplax}
\DeclareMathOperator{\Cat}{Cat}
\DeclareMathOperator{\Dia}{Dia}
\DeclareMathOperator{\Dir}{Dir}

\DeclareMathOperator{\Catlf}{Catlf}

\DeclareMathOperator{\Posf}{Posf}

\DeclareMathOperator{\Invpos}{Invpos}
\DeclareMathOperator{\Dirlf}{Dirlf}
\DeclareMathOperator{\Invlf}{Invlf}
\DeclareMathOperator{\Cof}{Cof}
\DeclareMathOperator{\Fib}{Fib}

\DeclareMathOperator{\dia}{dia}

\newcommand{\tw}[1]{ {{}^{\downarrow \uparrow} #1 }} 
\newcommand{\tww}[1]{ {{}^{\downarrow\downarrow \uparrow} #1 }} 
\newcommand{\twwc}[1]{ {{}^{\downarrow \uparrow \uparrow \downarrow} #1 }} 
\newcommand{\twc}[1]{ {{}^{\downarrow \uparrow \downarrow} #1 }}

\begin{document}

\maketitle



\section*{Abstract}

Starting from simple and necessary axioms on a (derivator enhanced) four-functor-formalism, we construct 
derivator six-functor-formalisms using compactifications. This works, for instance, for the stable homotopy categories of Morel-Voevodsky-Ayoub, and also 
for the classical setting of unbounded complexes of sheaves of Abelian groups on `nice' topological spaces. 
The formalism of derivator six-functor-formalisms elegantly encodes
all isomorphisms between compositions of the six functors (and their compatibilities) and moreover it gives coherent enhancements over diagrams of correspondences. 
Such a formalism allows to extend six-functor-formalisms to stacks using (co)homological descent.

\tableofcontents

\section{Introduction}

This is the second article in a series of three \cite{Hor16, Hor22} concerning derivator six-functor-formalisms.

For a detailed introduction to classical six-functor-formalisms we refer to the previous article \cite{Hor16}. Recall that those formalisms are defined on a base category $\mathcal{S}$ and specify a (usually derived) category $\mathcal{D}_S$ for each object in $\mathcal{S}$, adjoint pairs of functors
\vspace{0.2cm}
\[
\begin{array}{lcrp{1cm}l}
f^* & &f_* &  & \text{\em for each $f$ in $\Mor(\mathcal{S})$}  \\
f_! & &f^! & & \text{\em for each $f$ in $\Mor(\mathcal{S})$} \\
\otimes & & \mathcal{HOM} & & \text{\em in each fiber $\mathcal{D}_S$}
\end{array}
\]
\vspace{0.1cm}

and isomorphisms between the left adjoints (corresponding isomorphisms between the right adjoints follow by adjunction):
\vspace{0.3cm}
\begin{center}
\begin{tabular}{r|l}
& isomorphisms \\ 
& between left adjoints \\
\hline
$(*,*)$ & $(fg)^* \iso g^* f^*$ \\
$(!,!)$ & $(fg)_! \iso f_! g_!$ \\ 
$(!,*)$ & $g^* f_! \iso F_! G^*$  \\
$(\otimes,*)$ & $f^*(- \otimes -) \iso f^*- \otimes f^* -$  \\
$(\otimes,!)$ & $f_!(- \otimes f^* -) \iso  (f_! -) \otimes -$  \\ 
$(\otimes, \otimes)$ &  $(- \otimes -) \otimes - \iso - \otimes (- \otimes -)$  
\end{tabular}
\end{center}
\vspace{0.3cm}
as well as isomorphisms $f_! \cong f_*$ for isomorphisms\footnote{Variant: for all proper morphisms, provided such a class has been specified.} $f$  in $\mathcal{S}$. 

Of course, these isomorphisms have to fulfill compatibilities as, for example, the pentagon axiom and many more. 

Let $\mathcal{S}^{\mathrm{cor}}$ be the symmetric 2-multicategory whose objects are the objects of $\mathcal{S}$ and in which a 1-morphism $\xi \in \Hom(S_1, \dots, S_n; T)$ is a multicorrespondence
\begin{equation}\label{excor}
 \vcenter{ \xymatrix{ 
 &&&  \ar[llld]_{g_1} A \ar[ld]^{g_n} \ar[rd]^{f} &\\
 S_1 & \cdots & S_n & ; &  T   } }
 \end{equation}
The composition of 1-morphisms is given by forming fiber products and the 2-morphisms are the isomorphisms of such multicorrespondences. 

In \cite{Hor15, Hor15b} it was explained that, using the language of (op)fibrations of 2-multicategories, one can package all isomorphisms and compatibilities in a six-functor-formalism into a neat definition: 
\vspace{0.2cm}

{\bf Definition. }{\em 
A (symmetric) {\bf six-functor-formalism} on $\mathcal{S}$ is a 1-bifibration and 2-bifibration of (symmetric) 2-multicategories with 1-categorical fibers}
\[ p: \mathcal{D} \rightarrow \mathcal{S}^{\mathrm{cor}}. \]
Such a bifibration can also be seen as a pseudo-functor of 2-multicategories
\[ \mathcal{S}^{\cor} \rightarrow \mathcal{CAT} \]
with the property that all multivalued functors in the image have right adjoints w.r.t.\@ all slots. 
Note that $\mathcal{CAT}$, the ``category''\footnote{having, of course, a {\em higher class} of objects} of categories, has naturally the structure of a symmetric 2-``multicategory'' where the 1-multimorphisms are functors of several variables. 
The pseudo-functor maps the correspondence (\ref{excor}) to a functor isomorphic to 
\[ f_! ((g_1^* -) \otimes_A \cdots \otimes_A (g_n^* -)) \]
where $\otimes_A$, $f_!$, and $g^*_i$, are the images of  the following correspondences
\[
 \vcenter{ \xymatrix{
&&&  A \ar@{=}[rd] \ar@{=}[ld] \ar@{=}[llld] \\
A &  & A  & &  A
} } \quad
 \vcenter{ \xymatrix{
& A \ar[rd]^f \ar@{=}[ld]  \\
A   & &  T
} } \quad
 \vcenter{ \xymatrix{
&  A \ar@{=}[rd] \ar[ld]_{g_i}  \\
S_i  & &  A
} }
\]

Six-functor-formalisms were first introduced by Grothendieck, Verdier and Deligne \cite{Ver77, SGAIV1, SGAIV2, SGAIV3} and there 
has been increasing interest in them in various contexts in the last two decades \cite{FHM, Ayo07I, Ayo07II, LO08I, LO08II, CD19, LH09, Ayo14, ZL14, ZL14b, Zh10, BS15, Sch15, GR16, Sch17, Hoy17, Dre18, AGV20, DG20, Cho21, KR21, GHW22}.

A {\bf derivator six-functor-formalism} $\DD \rightarrow \SSS^{\cor}$ (cf.\@ Definition~\ref{DEF6FUDER}) specifies not only a (derived) category for each object in $\mathcal{S}$ but
also a (derived) category for each diagram $I \rightarrow \mathcal{S}^{\cor}$ of correspondences in $\mathcal{S}$ for each small category $I$.
Over {\em constant diagrams} with value $S \in \mathcal{S}$, such datum gives back the derivator enhancement of the derived category $\mathcal{D}_S$.  
Objects in these categories $\DD(I)_{p^*S}$ should be seen as coherent versions of diagrams $I \rightarrow \mathcal{D}_S$ (for instance, they are
commutative diagrams of actual complexes up to point-wise weak-equivalences as opposed to diagrams commuting only up to homotopy).
The observation of Grothendieck and Heller was that this datum suffices to reconstruct the triangulated structure. Furthermore, all homotopy (co)limits (for example total complexes) exist, which cannot be constructed in the world of triangulated categories. 

What should the category $\DD(I)_{X}$ over an arbitrary diagram of correspondences  $X: I \rightarrow \mathcal{S}^{\cor}$ be?
A functor $\mathcal{E}: I \rightarrow \mathcal{D}$ lying over $X$ specifies for each $i \in I$ an object $\mathcal{E}(i)$ over $X(i)$ and for each morphism
$\alpha: i \rightarrow j$ a morphism
\[ f_! g^* \mathcal{E}(i) \rightarrow \mathcal{E}(j) \quad \text{or equivalently} \quad  \mathcal{E}(i) \rightarrow g_*  f^! \mathcal{E}(j) \]
in which $f$ and $g$ are the components of the correspondence $X(\alpha)$. Again, objects in $\DD(I)_{X}$ are certain ``coherent enhancements'' of such objects. 
Having those at our disposal allows for very general constructions, as for example the extension of six-functor-formalisms to higher stacks. Very roughly, a stack $S$ presented by
a simplicial object $(\Delta^{\op}, \widetilde{S})$ in $\mathcal{S}$ gives rise to two categories\footnote{Here $\widetilde{S}^{\op}: \Delta \rightarrow \mathcal{S}^{\cor}$ is the diagram obtained from $\widetilde{S}$ by flipping all correspondences.}
\begin{equation}\label{eqcats} \DD(\Delta)_{\widetilde{S}^{\op}}^{\cocart} \quad \text{and} \quad   \DD(\Delta^{\op})_{\widetilde{S}}^{\cart}.   \end{equation}
Cohomological (resp.\@ homological) descent as developed in \cite{Hor15, Hor21c} shows that these categories do not depend (up to equivalence) on the actual presentation of the stack, under mild assumptions on the derivator six-functor-formalism. The left hand side version of the category allows easily for the construction of $f^*, f_*$-functors for morphisms between stacks and the
right hand side version of the category allows easily for the construction of $f_!, f^!$-functors. However, one can show that for a higher geometric stack in the sense of \cite{TV08} we actually have a diagram of correspondences $\widetilde{X}$ of shape $\Delta^{\op} \times \Delta$ and morphisms
\[ \xymatrix{
& (\Delta^{\op} \times \Delta, \widetilde{X}) \ar[ld] \ar[rd]  \\
 (\Delta^{\op}, \widetilde{S}) & & (\Delta, \widetilde{S}^{\op})
} \]
in which the diagram $\widetilde{X}$ might still not have values in $\mathcal{S}$ but does have values that are ``less stacky'' than $\widetilde{S}$ which makes this construction amenable to induction.
Using the full derivator six-functor-formalism, this allows to show that the two categories (\ref{eqcats}) are equivalent, and also that base change and projection formula still hold for the combined operations. This will be explained in detail in a forthcoming article \cite{Hor22}. 

The purpose of this article is to {\em construct} such derivator six-functor-formalisms starting from a ``derivator four-functor-formalism'' by which we mean
a fibered multiderivator $\DD \rightarrow \SSS^{\op}$, where $\SSS^{\op}$ is the pre-multiderivator represented by $\mathcal{S}^{\op}$ considered as multicategory, setting:
\[ \Hom(S_1, \dots, S_n; T):= \Hom(T, S_1) \times \cdots \times \Hom(T, S_n). \]
This encodes a monoidal product $\otimes_S$ as push-forward along $(\id_S, \id_S)$ and pull-backs $f^*$ as pull-forwards along 1-ary morphisms $f^{\op}$ in $\mathcal{S}^{\op}$ together with their compatibilities. 

\begin{PAR}\label{PARFIBWE}
In contrast to six-functor-formalisms, a ``derivator four-functor-formalism'' is quite easy to construct. 
For the definition one needs merely a bifibration of multicategories
\[ p: \mathcal{D} \rightarrow \mathcal{S}^{\op} \]
equipped with a class of weak equivalences $\mathcal{W}_S \subset \Mor(\mathcal{D}_S)$ for each object $S$ of $\mathcal{S}$.
\end{PAR}

\begin{DEF}\label{DEFFIBDERMODEL}
We define a  pre-multiderivator as follows\footnote{cf.\@ \cite[Appendix~A.3]{Hor15} for localizations of multicategories}.  For any $I \in \Cat$ let
\[ \DD(I) := \Fun(I, \mathcal{D})[\mathcal{W}_I^{-1}]  \]
where $\mathcal{W}_I$ is the class of natural transformations which are element-wise in the union $\mathcal{W}:=\bigcup_S \mathcal{W}_S$.
The functor $p$ obviously induces a morphism of pre-multiderivators
\[ \widetilde{p}: \DD \rightarrow \SSS^{\op}. \]
\end{DEF}

\begin{BEISPIEL}\label{MAINEXAMPLE}
A basic example for the situation \ref{PARFIBWE} is the bifibration of symmetric multicategories
\[ p: \mathrm{Ch}(\mathrm{Mod}) \rightarrow  \mathrm{RSite}^{\op} \]
where $\mathrm{RSite}^{\op}$ is  the opposite of the category of {\em ringed sites}, considered as multicategory as above. 
The fiber $\mathrm{Ch}(\mathrm{Mod}(X, \OO_{X}))$ over a space $(X, \OO_{X})$ is the category of {\em unbounded} chain complexes of sheaves of $\OO_X$-modules on $X$, and the class $\mathcal{W}_{(X, \OO_{X})}$ is the class of quasi-isomorphisms. 
The push-forward along a multimorphism $g=(g_1, \dots, g_n)$ in $\mathrm{Ch}(\mathrm{Mod})^{\op}$ is given by 
\[ Lg_1^* \overset{L}{\otimes} \cdots \overset{L}{\otimes} L g_n^*. \]
Note that the multicategory structure is even the more natural structure because no particular tensor-product, resp.\@ pull-back, has to be chosen a priori to define it (even on complexes, cf.\@ \cite[1.5]{Rec19}). 
\end{BEISPIEL}

Of course, we would like the morphism of pre-multiderivators $\widetilde{p}: \DD \rightarrow \SSS^{\op}$ to be a left (resp.\@ right) fibered multiderivator. This is true, provided that the fibers are model categories whose structures are compatible with the structure of bifibration in the following sense:

\begin{DEF}[{\cite[Definition~5.1.3]{Hor15}}]\label{PARQUILLENN}
A {\bf bifibration of multi\nobreakdash-model categories}
is a bifibration of multicategories $p: \mathcal{D} \rightarrow \mathcal{S}^{\op}$ 
together with the collection of a closed model category structure on the fiber
\[ (\mathcal{D}_S, \Cof_S, \Fib_S, \mathcal{W}_S) \]
for any object $S$ in $\mathcal{S}$
 such that the following two properties hold:
\begin{enumerate}
\item For any $n\ge 1$ and for every multimorphism 
$f \in \Hom(S_1, \dots, S_n; T)$, the push-forward $f_\bullet$ and the various pull-backs $f^{\bullet,j}$ define a Quillen
adjunction in $n$-variables
\[ \xymatrix{ \prod_i (\mathcal{D}_{S_i}, \Cof_{S_i}, \Fib_{S_i}, \mathcal{W}_{S_i}) \ar[rr]^-{f_\bullet} && (\mathcal{D}_T, \Cof_T, \Fib_T, \mathcal{W}_T) } \]
\[ \xymatrix{(\mathcal{D}_T, \Cof_T, \Fib_T, \mathcal{W}_T) \times  \prod_{i \not= j} (\mathcal{D}_{S_i}, \Cof_{S_i}, \Fib_{S_i}, \mathcal{W}_{S_i}) \ar[rr]^-{f^{\bullet,j}} && (\mathcal{D}_{S_j}, \Cof_{S_j}, \Fib_{S_j}, \mathcal{W}_{S_j}) } \]

\item  a technical condition involving units (i.e.\@ push-forwards along 0-ary morphisms). 
\end{enumerate}
\end{DEF}

\begin{BEM}
If $\mathcal{S}=\{\cdot\}$ is the final multicategory, the above notion coincides with the notion of {\em closed monoidal model-category} in the sense of \cite[Definition 4.2.6]{Hov99}.
In this case it is enough to claim property 1. for $n=2$. 
\end{BEM}

It is quite easy to show (cf.\@ \cite[Theorem~5.1.5]{Hor15}) that under the conditions of Definition~\ref{PARQUILLENN} the morphism $\widetilde{p}: \DD \rightarrow \SSS^{\op}$
defined in \ref{DEFFIBDERMODEL} is a right fibered multiderivator on directed categories and a left fibered multiderivator on inverse categories. However, 
we have the following generalization of a theorem of Cisinski \cite{Cis03}:

\begin{SATZ}[{\cite[Theorem~5.8]{Hor17b}}]\label{SATZEXISTENCEFMULTIDER}
Under the conditions of Definition~\ref{PARQUILLENN} the morphism of pre-multiderivators 
\[ \widetilde{p}: \DD \rightarrow \SSS^{\op} \]
defined in \ref{DEFFIBDERMODEL} is a (left and right) fibered multiderivator with domain $\Cat$ (small categories).
\end{SATZ}

\begin{SATZ}[Hovey, Gillespie, Recktenwald, cf.\@ Theorem~\ref{SATZRECKTENWALD}]
\label{SATZEXAMPLETOP}
The bifibration of multicategories from Example~\ref{MAINEXAMPLE} 
\[ p: \mathcal{D} \rightarrow  \mathcal{S}^{\op} \]
can be equipped with the structure of  bifibration of symmetric multi-model categories, yielding thus by Theorem~\ref{SATZEXISTENCEFMULTIDER} a symmetric ``derivator four-functor-formalism'', i.e.\@ a symmetric fibered multiderivator
\[ \DD \rightarrow \SSS^{\op}. \]
\end{SATZ}

\begin{BEISPIEL}\label{MAINEXAMPLE2}
Let $\mathcal{S} := \mathcal{SCH}_S$ the category of quasi-projective schemes over a base scheme $S$ and
consider a triple $(\tau, \mathcal{M}, T)$ as in Ayoub \cite[Section~4.5]{Ayo07I} in which:
\begin{itemize}
\item $\tau$ is either the etale or Nisnevich topology on $\mathcal{SCH}_S$;
\item $\mathcal{M}$ is a category of coefficients in the sense of Definition~\ref{DEFCATCOEFF};
\item $T$ is a projectively cofibrant object of $\mathrm{PreShv}(\mathrm{Sm}/S, \mathcal{M})$ with the condition in \cite[4.5.18]{Ayo07I}.
\end{itemize}
Let $\SSS^{\op}$ be the symmetric pre-multiderivator represented by $\mathcal{S}^{\op} = \mathcal{SCH}_S^{\op}$ and the symmetric multicategory structure \ref{PAROPMULTCAT}.
It follows from work of Ayoub (see~Theorem~\ref{SATZAYOUB1}) that there is a symmetric fibered multiderivator with domain $\Cat$
\[ \mathbb{SH}^T_{\mathcal{M}} \rightarrow \SSS^{\op} \]
such that for a diagram $F: I \rightarrow \mathcal{S}^{\op}$ of quasi-projective schemes, we have  an equivalence of closed monoidal categories
\[ \mathbb{SH}^T_{\mathcal{M}}(I)_F = \mathbb{SH}^{T}_{\mathcal{M}}(F^{\op}, I^{\op}) \]
where the right hand side is the ``algebraic derivator'' defined by Ayoub \cite[D\'efinition 4.2.21]{Ayo07I} and such that the push-forward along a multimorphism $g=(g_1, \dots, g_n)$ in $\mathcal{S}^{\op}$ is given up to unique isomorphism by
\[ (L g^*_1 - ) \overset{L}{\otimes} \cdots \overset{L}{\otimes} (L g_n^*-) \]
with the functors $g_i^*$ as in \cite[Th\'eor\`eme~4.5.23]{Ayo07I}.
\end{BEISPIEL}

\begin{PAR}Coming back to the case of derivator six-functor-formalisms, 
the procedure carried out in this article is analogous to the construction of the push-forward with proper support $f_!$ in \cite[Expos\'e XVII]{SGAIV3} using compactifications. We thus assume that $\mathcal{S}$ is a category with compactifications (cf.\@ \ref{DEFCATCOMP}), i.e.\@ that we are given abstract classes of (dense) open embeddings and proper morphisms satisfying the usual properties, and such that
every morphism in $\mathcal{S}$ can be factored into a dense open embedding followed by a proper map. 
\end{PAR}

We prove the following:

\vspace{0.2cm}

{\bf Theorem }(cf.\@ Corollary~\ref{KORDER6FU}).  {\em 
Let $\mathcal{S}$ be a category with compactifications and $\SSS^{\op}$ the symmetric pre-multiderivator represented by $\mathcal{S}^{\op}$ with domain $\Cat$. 
Let $\DD \rightarrow \SSS^{\op}$ be a (symmetric)  fibered multiderivator with stable and perfectly generated fibers satisfying axioms (F1--F6) and (F4m--F5m) below. Assume that $\DD$ is infinite (i.e.\@ satisfies (Der1${}^\infty$)). 
Then there exists a natural (symmetric) derivator six-functor-formalism
\[ \EE \rightarrow \SSS^{\cor} \]
with domain $\Cat$ such that the pull-back of $\EE$ along the natural morphism 
\[ \SSS^{\op} \rightarrow \SSS^{\cor} \]
is equivalent to $\DD$ and such that there is a canonical isomorphism $f_! \cong f_*$ for proper morphisms $f$ and a canonical isomorphism $\iota^* \cong \iota^!$ for embeddings $\iota$. 
}

\vspace{0.2cm}

Actually, it suffices for the construction that $\DD$ is defined on $\Invlf$ (inverse locally finite small categories). Using the theory of enlargement  \cite{Hor17b} it even suffices
 to have $\DD$ defined on $\Invpos$ (inverse posets). However, in many cases of interest, Theorem~\ref{SATZEXISTENCEFMULTIDER}  already gives a fibered multiderivator with domain $\Cat$.

If $\DD \rightarrow \SSS^{\op}$ has well-generated fibers then also $\EE \rightarrow \SSS^{\cor}$ does, and the main theorems of cohomological and homological descent of \cite{Hor15, Hor21c} apply. 

The actual construction of $\EE \rightarrow \SSS^{\cor}$ is very formal. Stable and perfectly generated fibers are only needed to obtain the right-fiberedness of $\EE \rightarrow \SSS^{\cor}$ via Brown representability. The construction and left-fiberedness of $\EE \rightarrow \SSS^{\cor}$ can even be obtained without assuming stable fibers.  

The axioms (F1--F6) and (F4m--F5m) are very mild and obviously all necessary. Only two of them, (F1) and (F3) are actually concerned with the derivator enhancement. They 
require that the pull-back $\iota^*$ along a point-wise open embedding $\iota$ has a left adjoint $\iota_!$, that $\iota^*$, as morphism of derivators, commutes with homotopy limits\footnote{Equivalently: $\iota_!$ is computed point-wise for constant diagrams.}, 
and that $f_*$ for a proper morphism $f$, as morphism of derivators, commutes with homotopy colimits. 
The other axioms are the usual formulas (proper projection formula and base change, open embedding ``coprojection formula'' and base change, etc.\@) and it is sufficient to check them at the level of usual derived categories. Thus, in all cases of interest, the axioms are well-known statements. 
Only (F4m) and (F5m) are concerned with the multi- (i.e.\@ monoidal) aspect. 

The precise formulation of the axioms is as follows (the notions {\em proper} and {\em embedding} refer to the chosen compactification on $\mathcal{S}$):

\begin{itemize}
\item[(F1)] For each diagram $I \in \Invlf$ and point-wise embedding $\iota: S \hookrightarrow T$ in $\SSS(I^{\op})$, the functor $\iota^*$ (aka $(\iota^{\op})_\bullet$), as morphism of derivators, commutes with homotopy limits\footnote{By (Der2), for this statement one may assume $I=\cdot$.}, and has a left adjoint
\[ \iota_!: \DD(I)_{S^{\op}} \rightarrow \DD(I)_{T^{\op}}. \]
\item[(F2)] For each embedding $\iota: S \hookrightarrow T$ in $\mathcal{S}$ the corresponding functor
\[ \iota_!: \DD_S(\cdot) \rightarrow \DD_T(\cdot) \]
is fully faithful. 
\item[(F3)] For each proper morphism $f$ in $\mathcal{S}$ the functor $f_*$ (as morphism of derivators) commutes with homotopy colimits. 
\item[(F4)] For each proper morphism $f$ in $\mathcal{S}$ and any Cartesian square
\[ \xymatrix{
 \ar@{->>}[r]^F \ar[d]_G &  \ar[d]^g \\
 \ar@{->>}[r]_f & 
} \]
the natural exchange morphism (base change)
\[ G^* F_* \rightarrow f_* g^*   \]
is an isomorphism. 
\item[(F5)] For each embedding $\iota$ in $\mathcal{S}$ and for any Cartesian square
\[ \xymatrix{
 \ar[r]^{F} \ar@{^{(}->}[d]_{I} &  \ar@{^{(}->}[d]^\iota \\
 \ar[r]_{f} & 
} \]
the natural exchange morphism (base change)
\[ \iota^* f_* \to F_* I^* \]
is an isomorphism.
\item[(F6)] For each proper morphism $f$  and  embedding $\iota$ forming a Cartesian square
\[ \xymatrix{
 \ar@{->>}[r]^{F} \ar@{^{(}->}[d]_{I} &  \ar@{^{(}->}[d]^\iota \\
 \ar@{->>}[r]_{f} & 
} \]
the exchange of the base change isomorphism from (F4) (equivalently from (F5)) 
\[ \iota_! F_* \rightarrow f_* I_! \]
is an isomorphism.

\item[(F4m)] For each proper morphism $f$ in $\mathcal{S}$ we have projection formulas, i.e.\@ the natural exchange morphisms
\[  (f_* -) \otimes - \to f_*(- \otimes f^*-) \quad  - \otimes (f_* -) \to f_*((f^*-) \otimes -)  \]
are isomorphisms\footnote{In case $\DD$ is symmetric, these two assertions are equivalent.}. 
\item[(F5m)] For each embedding $\iota$ in $\mathcal{S}$ we have ``coprojection formulas'', i.e.\@ the natural exchange morphisms
\[ \iota^* \mathcal{HOM}_l(-, -) \to  \mathcal{HOM}_l(\iota^*-, \iota^*-) \quad \iota^* \mathcal{HOM}_r(-, -) \to  \mathcal{HOM}_r(\iota^*-, \iota^*-)  \]
are isomorphisms\footnote{In case $\DD$ is symmetric, we have $\mathcal{HOM}_l=\mathcal{HOM}_r$.}.
\end{itemize}

\begin{PAR}In Section~\ref{SECTIONTOP} it is discussed that the axioms are satisfied in the classical context of topological spaces, i.e.\@ for the restriction of the fibered multiderivator of Theorem~\ref{SATZEXAMPLETOP} to a nice class of ringed {\em spaces}. In particular, the
classical six-functor-formalism on `nice' topological spaces with values in unbounded chain complexes of sheaves of Abelian groups as e.g.\@ discussed in Spaltenstein~\cite{Spa88} has thus an enhancement as a derivator six-functor-formalism. 
\end{PAR}
\begin{PAR}
In Section~\ref{SECTIONAYOUB} it is discussed that the axioms are satisfied in Ayoub's setting (cf.\@ \ref{MAINEXAMPLE2}) which comprises the stable homotopy theory of Morel-Voevodsky and 
various kinds of Voevodsky motives. Thus also Ayoub's six-functor-formalism has an enhancement as a derivator six-functor-formalism.
\end{PAR}

\section{Compactification: The axioms}

\begin{PAR}\label{DEFCATCOMP}
Let $\mathcal{S}$ be a category with finite limits (equivalently: with final object and pull-backs). We say that $\mathcal{S}$ is a {\bf category with compactifications} if we are given 
subclasses $\mathcal{S}_i$, $i=0,1,2$ of morphisms of $\mathcal{S}$ satisfying properties (S0)--(S5) below. We call morphisms in $\mathcal{S}_2$ {\bf embeddings}, morphisms in $\mathcal{S}_1$ {\bf dense embeddings} and morphisms in $\mathcal{S}_0$ {\bf proper}.
\begin{itemize}
\item[(S0)] $\mathcal{S}_1 \subset \mathcal{S}_2$ and if $\gamma = \beta \alpha$ with $\alpha, \beta \in \mathcal{S}_2$ and $\gamma \in \mathcal{S}_1$ then $\alpha, \beta \in \mathcal{S}_1$.
\item[(S1)] $\mathcal{S}_0 \cap \mathcal{S}_1$ is the class of isomorphisms in $\mathcal{S}$. 
\item[(S2)] If $g \in \mathcal{S}_i$ then  $f \in \mathcal{S}_i \Leftrightarrow g \circ f \in \mathcal{S}_i$;
\item[(S3)] $\mathcal{S}_0$ and $\mathcal{S}_2$ are stable under pull-back;
\item[(S4)] For any object $S$ the diagonal $\Delta: S \rightarrow S \times S$ is in $\mathcal{S}_0$;
\item[(S5)] Any morphism $f: S \rightarrow T$ can be factored as $f = \overline{f} \circ \iota$ with $\overline{f} \in \mathcal{S}_0$ and $\iota \in \mathcal{S}_1$.
\end{itemize}
In diagrams we will denote embeddings by the symbol $\xymatrix{\ar@{^{(}->}[r] & }$ and proper morphisms by the symbol $\xymatrix{\ar@{->>}[r] & }$. A choice of factorization as in (S5) will be called a {\bf compactification} of $f$. 
 In the relevant examples morphisms in $\mathcal{S}_2$ will be something like open embeddings and morphisms in $\mathcal{S}_1$ dense open embeddings. 
\end{PAR}

\begin{BEISPIEL}
$\mathcal{S}$ is the category of quasi-compact, separated schemes, $\mathcal{S}_2$ (resp.\@ $\mathcal{S}_1$) the class of (dense) open immersions and 
$\mathcal{S}_0$ the class of proper morphisms. (S5) is Deligne's extension of Nagata's compactification Theorem\footnote{For (S5) it is actually sufficient that the destination of the morphism is quasi-separated but in the sequel we will need property (S4).}.
\end{BEISPIEL}
\begin{BEISPIEL}
$\mathcal{S}$ is the category of locally compact Hausdorff topological spaces, $\mathcal{S}_2$ (resp.\@ $\mathcal{S}_1$) the class of (dense) open immersions and $\mathcal{S}_0$ the class of proper morphisms. 
\end{BEISPIEL}

We say that a subclass  $\mathcal{S}_i$ of morphisms in $\mathcal{S}$ is {\bf stable under limits of shape $I$} if 
for each morphism $f$ in the category $\Fun(I, \mathcal{S})$, which is point-wise in $\mathcal{S}_i$, it follows that $\lim_I f$ is in $\mathcal{S}_i$.

\begin{LEMMA}\label{LEMMAPROPERTIESCOMP}
Let $\mathcal{S}$ be a category with compactifications.   
\begin{enumerate}
\item $\mathcal{S}_0$ and $\mathcal{S}_2$ are stable under finite products, i.e.\@ stable under limits of shape a finite set. 
\item $\mathcal{S}_0$ is stable under fiber products, i.e.\@ stable under limits of shape $\righthalfcup$. 
\item $\mathcal{S}_0$ is stable under arbitrary finite limits. 
\item For a 
finite diagram $F: I \rightarrow \mathcal{S}$ where $I$ has final object, and all morphisms in $F$ are in $\mathcal{S}_0$, the projections $\lim_I F \rightarrow F(i)$ are in $\mathcal{S}_0$ for all $i \in I$.
\end{enumerate}
\end{LEMMA}
\begin{proof}
1. Using (S3) the Cartesian diagram
\[ \xymatrix{
X \times Z \ar[r]^{\alpha \times \id} \ar[d] & Y \times Z \ar[d] \\
X  \ar[r]^\alpha & Y
} \]
shows that morphisms of the form $\alpha \times \id$ are in $\mathcal{S}_i$ (and similarly also of the form $\id \times \alpha$) hence also products of morphisms. 

2. Since diagonals are in $\mathcal{S}_0$ by (S4),  the Cartesian diagram
\[ \xymatrix{
X \times_Y Z \ar@{->>}[r] \ar[d] & X \times Z \ar[d] \\
Y \ar@{->>}[r]^-\Delta & Y \times Y
} \]
shows that the morphisms $X \times_Y Z \rightarrow X \times Z$ are in $\mathcal{S}_0$. 
We have a commutative diagram
\[ \xymatrix{
X \times_Y Z \ar@{->>}[r] \ar[d] & X \times Z \ar@{->>}[d] \\
X' \times_{Y'} Z'  \ar@{->>}[r] & X' \times Z'
} \]
All morphisms except possibly the left vertical one are in $\mathcal{S}_0$. Hence by (S2) also the left vertical one is. 

3. Let $f, f': S \rightarrow T$ be two morphisms. The equalizer of $f$ and $f'$ can be computed by the following Cartesian diagram
\[ \xymatrix{
\mathrm{Eq}(f, f') \ar[r] \ar[d] & T \ar[d]^{\Delta} \\
S \ar[r]_-{(f, f')} & T \times T
} \]
i.e.\@ as a fiber product. 

Since a finite limit can be computed by a finite product and an equalizer, $\mathcal{S}_0$ is stable under finite limits by 1.\@ and 2.

4. 
Let $I$ have a final object $j$ and let  $F: I \rightarrow \mathcal{S}$ be a functor. We have $F(j) = \lim_I F(j)$, where $F(j)$, by abuse of notation, also denotes the constant diagram with value $F(j)$. 
Then the morphism
\[ \lim_I F \rightarrow \lim_I F(j) \]
is in $\mathcal{S}_0$, as was just shown. For any object $i \in I$, it factors as follows (where $\lim_I F(j)$ denotes the limit over the {\em constant} diagram with value $F(j)$): 
\[ \lim_I F \rightarrow F(i) \rightarrow F(j)=\lim_I F(j) \]
and the rightmost morphism is in $\mathcal{S}_0$ by assumption, hence so is the projection $\lim_I F \rightarrow F(i)$ by (S2). 
\end{proof}

\begin{LEMMA}\label{LEMMACART1}
Any diagram of the form
\[ \xymatrix{
 \ar@{->>}[r] \ar@{^{(}->}[d] &  \ar@{^{(}->}[d] \\
 \ar@{->>}[r] & 
} \]
in which the left vertical embedding is dense, is Cartesian. 
\end{LEMMA}
\begin{proof}
If we form the pull-back 
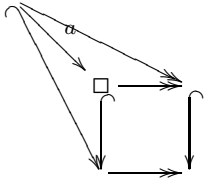
\[ \xymatrix{
\ar@{->>}[rrd] \ar@{_{(}->}[rdd] \ar[rd]^a  &  \\
 & \Box \ar@{->>}[r] \ar@{^{(}->}[d] & \ar@{^{(}->}[d] \\
& \ar@{->>}[r] & 
} \]
then the morphisms are embeddings (resp.\@ proper) as indicated using (S3). By (S2) the morphism $a$ is proper and an embedding, which is dense by (S0), hence by (S1) an isomorphism. 
\end{proof}

\begin{PAR}\label{DEFWEAKLYCART}
Let $\mathcal{S}$ be a category with compactifications. 
We say that a square
\[ \xymatrix{
W \ar[r] \ar[d] & Z \ar[d] \\
X \ar[r] & Y
} \]
in $\mathcal{S}$ is {\bf weakly Cartesian} if the induced morphism $W \rightarrow X \times_{Y} Z$ is proper. 
\end{PAR}

\begin{LEMMA}\label{LEMMACART2}
If in the diagram
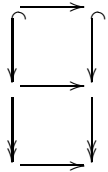
\[ \xymatrix{
 \ar[r] \ar@{^{(}->}[d] & \ar@{^{(}->}[d] \\
 \ar[r] \ar@{->>}[d] & \ar@{->>}[d] \\
 \ar[r] & 
} \]
the top left vertical embedding is dense and 
 the outer square is weakly Cartesian then the upper square is Cartesian. 
\end{LEMMA}
\begin{proof}
We may form a diagram in which the right squares are Cartesian:
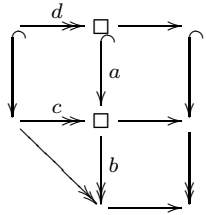
\[ \xymatrix{
\ar@{->>}[r]^d \ar@{^{(}->}[d]  & \Box \ar[r] \ar@{^{(}->}[d]^a & \ar@{^{(}->}[d] \\
\ar@{->>}[r]^c \ar@{->>}[rd] & \Box \ar[r] \ar@{->>}[d]^b & \ar@{->>}[d] \\
& \ar[r] & 
} \]
where the morphisms $a, b, c, d$ are an embedding (resp.\@ proper) by (S3) and (S2) and the definition of weakly Cartesian. Now the upper right square is Cartesian by construction and the upper left square is Cartesian because of Lemma~\ref{LEMMACART1}. Hence also the composite square, which is the upper square in the original diagram is Cartesian. 
\end{proof}

\section{Compactification of morphisms of inverse diagrams}

\begin{PROP}\label{PROPCOMPMOR}
Let $\mathcal{S}$ be a category with compactifications and let $I$ be an inverse diagram with finite matching diagrams.
\begin{enumerate}
\item Let $f: F \rightarrow G$ be a morphism in $\Fun(I, \mathcal{S})$. The morphism can be factored
\[ f = \overline{f} \circ \iota  \]
where $\overline{f}$ is point-wise in $\mathcal{S}_0$ and $\iota$ is point-wise in $\mathcal{S}_1$. 
\item Any two such factorizations are dominated by a third in the sense that for two factorizations  $f = \overline{f}_i \circ \iota_i, i=1,2$ we get a third factorization $f = \overline{f}_3 \circ \iota_3$ and a diagram
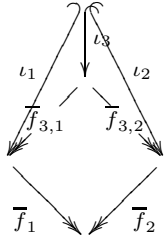
\[ \xymatrix{
& \ar@{_{(}->}[ldd]_{\iota_1} \ar@{^{(}->}[d]^{\iota_3} \ar@{^{(}->}[rdd]^{\iota_2} \\
& \ar@{->>}[ld]|{\overline{f}_{3,1}} \ar@{->>}[rd]|{\overline{f}_{3,2}} \\
 \ar@{->>}[rd]_{\overline{f}_1}  & &  \ar@{->>}[ld]^{\overline{f}_2} \\
& \\
} \]
such that $\overline{f}_3 = \overline{f}_1 \circ \overline{f}_{3,1} = \overline{f}_2 \circ \overline{f}_{3,2 }$. 
\item 
Any given compactification of $f$ restricted to a final subdiagram of $I$ can be extended to a compactification of the whole morphism $f$.
\end{enumerate}
\end{PROP}
\begin{proof}
1. We construct the compactification element-wise using induction on the degree of the element as usual. 
For degree 1 it follows directly from the compactification axiom. 
Let now $f$ be a morphism which for elements of degree $<n$ has been factored as required.
For each $i \in I$ of degree $n$ we get a diagram:
\[ \xymatrix{
F(i) \ar[rrrr]^{f(i)} \ar[d] & && & G(i) \ar[d] \\
\lim_{M_i} F \ar[rr]^{\lim_{M_i} \iota} && \lim_{M_i} \overline{F} \ar@{->>}[rr]^{\lim_{M_i} \overline{f}} & & \lim_{M_i}  G
} \]
in which $M_i$ is the matching diagram, i.e.\@ the full subcategory of the comma-category $i \times_{/I} I$ consisting of non-identities.  
By Lemma~\ref{LEMMAPROPERTIESCOMP}, 3.\@ and the assumption that the matching diagrams be finite,  $\lim_{M_i} \overline{f}$ is again proper. Hence by forming the pull-back and compactifying the induced morphism using (S5), we get:
\[ \xymatrix{
F(i) \ar@{^{(}->}[r]^{\iota(i)} \ar[d] & \overline{F}(i) \ar@{->>}[r] \ar@{->>}@/^10pt/[rrr]^{\overline{f}(i)} & \Box \ar[d] \ar@{->>}[rr] & & G(i) \ar[d] \\
\lim_{M_i} F \ar[rr] & & \lim_{M_i} \overline{F} \ar@{->>}[rr]^{\lim_{M_i} \overline{f}} & & \lim_{M_i}  G
} \]
and define $\iota(i), \overline{f}(i)$ and $\overline{F}(i)$ to be the so denoted objects in this diagram. Note that $\overline{f}(i)$ is proper by axioms (S2) and (S3). 

2. The statement is again proved by induction. For degree 1 elements let two compactifications $f = \overline{f}_i \circ \iota_i$ be given. We form the fiber product of the $\overline{f}_i$ and compactify the induced morphism using (S5):
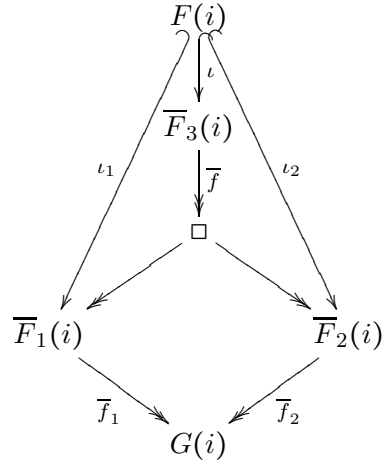
\[ \xymatrix{
& \ar@{_{(}->}[lddd]_{\iota_1} \ar@{^{(}->}[d]^{\iota} F(i) \ar@{^{(}->}[rddd]^{\iota_2} \\
& \overline{F}_3(i) \ar@{->>}[d]^{\overline{f}}  \\
& \Box \ar@{->>}[ld] \ar@{->>}[rd] \\
\overline{F}_1(i) \ar@{->>}[rd]_{\overline{f}_1}  & & \overline{F}_2(i)  \ar@{->>}[ld]^{\overline{f}_2} \\
& G(i) \\
} \]
It is clear how to extract a diagram as in the statement from this. 
Let now the diagram for elements of degree $<n$ be constructed. We get a diagram
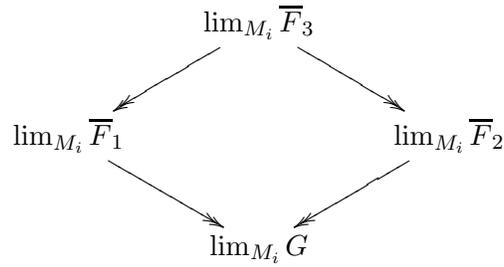
\[ \xymatrix{
& \lim_{M_i} \overline{F}_3 \ar@{->>}[ld] \ar@{->>}[rd] \\
\lim_{M_i} \overline{F}_1 \ar@{->>}[rd]_{}  & & \lim_{M_i} \overline{F}_2  \ar@{->>}[ld]^{} \\
& \lim_{M_i} G\\
} \]
in which all the morphisms are proper. Pulling it back along the morphism $G(i) \rightarrow  \lim_{M_i} G$ and inserting the given compactifications, we arrive at
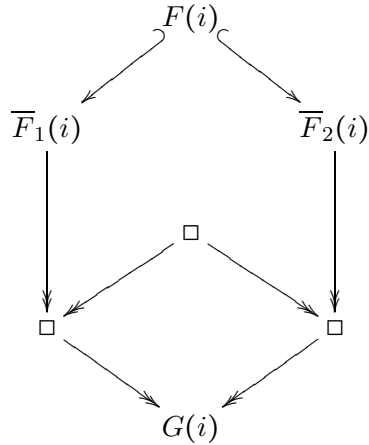
\[ \xymatrix{
& F(i) \ar@{_{(}->}[ld] \ar@{^{(}->}[rd] \\x
\overline{F}_1(i) \ar@{->>}[dd] & & \overline{F}_2(i) \ar@{->>}[dd] \\
& \Box \ar@{->>}[ld] \ar@{->>}[rd] \\
\Box \ar@{->>}[rd]_{}  & & \Box  \ar@{->>}[ld]^{} \\
& G(i)\\
} \]
Using Lemma~\ref{LEMMAPROPERTIESCOMP}, 4.\@ we see that the limit over the diagram consisting of the proper morphisms fits in a diagram
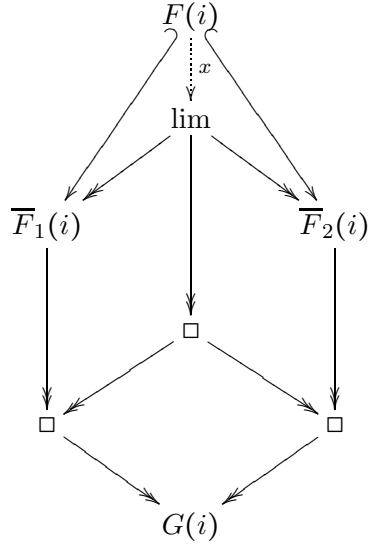
\[ \xymatrix{
& F(i) \ar@{_{(}->}[ldd] \ar@{.>}[d]^x \ar@{^{(}->}[rdd] \\
& \lim \ar@{->>}[ld] \ar@{->>}[rd] \ar@{->>}[dd]  \\
\overline{F}_1(i) \ar@{->>}[dd] & & \overline{F}_2(i) \ar@{->>}[dd] \\
& \Box \ar@{->>}[ld] \ar@{->>}[rd] \\
\Box \ar@{->>}[rd]_{}  & & \Box  \ar@{->>}[ld]^{} \\
& G(i)\\
} \]
in which all so indicated morphisms are proper. 
Now compactify the dotted morphism. It is clear that we may extract from this a diagram as claimed in the assumption. 

3. Clear. 
\end{proof}

\section{The induced compactification of diagrams}

Recall the following from \cite[7.3]{Hor16}:

\begin{PAR}\label{PARTW}
Let $I$ be a diagram, $n$ a natural number and $\Xi = (\Xi_1, \dots, \Xi_n) \in \{ \uparrow, \downarrow \}^n$ be a sequence of arrow directions. We define a diagram
\[ {}^\Xi I \]
whose objects are sequences of $n$ objects and $n-1$ morphisms in $I$
\[ \xymatrix{
i_1 \ar[r] & i_2 \ar[r] & \cdots \ar[r]  & i_n
} \]
and whose morphisms are commutative diagrams
\[ \xymatrix{
i_1 \ar[r] \ar@{<->}[d] &i_2 \ar[r] \ar@{<->}[d] & \cdots \ar[r]  & i_n \ar@{<->}[d] \\
i_1' \ar[r] &i_2' \ar[r] & \cdots \ar[r]  & i_n' \\
} \]
in which the $j$-th vertical arrow goes in the direction indicated by $\Xi_j$. 
We call a morphism {\bf of type $j$} if at most the morphism $i_j \rightarrow i_j'$ is {\em not} an identity. 

For a diagram $I$ and an object $i \in I$ we adopt the convention that $i$ denotes also the subcategory of $I$ consisting only of $i$ and its identity. 
In coherence with this convention ${}^\Xi i$ denotes the subcategory of ${}^\Xi I$ consisting of the sequence $i = \cdots = i$ and its identity. 

Examples: ${}^{\downarrow \downarrow} I = I \times_{/I} I$ is the comma category, ${}^{\uparrow} I = I^{\op}$, and $\tw I$ is the twisted arrow category. 
\end{PAR}

\begin{PAR}\label{PARTW2}
For any ordered subset $\{i_1, \dots, i_m\} \subseteq \{1, \dots, n\}$, denoting $\Xi'$ the restriction of $\Xi$ to the subset, we get an obvious restriction functor
\[ \pi_{i_1, \dots, i_m}:\  {}^\Xi I \rightarrow {}^{\Xi'} I. \]

If $\Xi = \Xi' \circ \Xi'' \circ \Xi'''$, where $\circ$ means concatenation, then the projection
\[ \pi_{1,\dots,n'}: \ {}^{\Xi}  I \rightarrow {}^{\Xi'} I   \]
is a {\em fibration} if the last arrow of $\Xi'$ is $\downarrow$ and an {\em opfibration} if the last arrow of $\Xi'$ is $\uparrow$ while the projection
\[ \pi_{n-n'''+1,\dots,n}: \  {}^{\Xi} I \rightarrow {}^{\Xi'''} I   \]
is an {\em opfibration} if the first arrow of $\Xi'''$ is $\downarrow$ and a {\em fibration} if the first arrow of $\Xi'''$ is $\uparrow$.
\end{PAR}

For the rest of the section, fix a category $\mathcal{S}$ with compactifications (Definition~\ref{DEFCATCOMP}). 

\begin{DEF}\label{DEFEXTCOMP}
 Let $I$ be a diagram and $S: I \rightarrow \mathcal{S}$ a functor.
 Any morphism $S \hookrightarrow \overline{S}$ in $\Fun(I, \mathcal{S})$ consisting point-wise of dense embeddings, and such that $\overline{S}$ is a diagram in which all morphisms are proper, is called an {\bf exterior compactification} of $S$. 
 A factorization $S \hookrightarrow \overline{S} \twoheadrightarrow \overline{S}'$, in which the composition is an exterior compactification again, is called a {\bf refinement}, 
\end{DEF}

We claim that exterior compactifications exist for $I$ an inverse diagram with finite matching diagrams: First compactify the morphism $F \rightarrow \cdot$ using Proposition~\ref{PROPCOMPMOR}, where $\cdot$ is the induced final object of $\mathrm{Fun}(I, \mathcal{S})$:
\[ \xymatrix{ F \ar[r]^{\iota} & \overline{F} \ar[r]^{\overline{f}} & \cdot   } \]
Then $\iota$ is an exterior compactification because all morphisms in the diagram $\overline{F}$ are automatically proper because of (S2). 

\begin{DEF}\label{DEFINTCOMP}
Let $I$ be a diagram. A functor $\widetilde{S}: {}^{\downarrow \downarrow}I \rightarrow \mathcal{S}$ (see \ref{PARTW} for the notation) together with an isomorphism $\Delta^* \widetilde{S} \cong S$ is called an {\bf interior compactification} of $S$ if every morphism of type 2 (cf.\@ \ref{PARTW}) is mapped to a dense embedding and every morphism of type 1 is mapped to a proper morphism. 
A morphism $\widetilde{S}_1 \rightarrow \widetilde{S}_2$ of compactifications (i.e.\@ a morphism compatible with the isomorphisms $\Delta^* \widetilde{S} \cong S$) is called a {\bf refinement} if it consists point-wise of proper morphisms.  
\end{DEF}

\begin{PROP}\label{PROPCOMPDIA}
Let $I$ be a diagram and $F \hookrightarrow \overline{F}$ be an exterior compactification in $\mathrm{Fun}(I, \mathcal{S})$. Then there is a canonical {\bf induced interior compactification} $\widetilde{F} \in \Hom({}^{\downarrow\downarrow} I, \mathcal{S})$. 
The association 
\[ (F \hookrightarrow \overline{F}) \mapsto \widetilde{F} \]
has the following properties
\begin{enumerate}
\item It is functorial in exterior compactifications, i.e.\@ if 
\[ \xymatrix{
F \ar[r] \ar@{^{(}->}[d] & G \ar@{^{(}->}[d] \\
\overline{F} \ar[r] & \overline{G}
} \]
is a commutative diagram in which the vertical morphisms are exterior compactifications then there is an induced morphism
\[ \widetilde{F} \rightarrow \widetilde{G}. \]
This association is functorial.
\item For a refinement of exterior compactifications consider the diagram
\[ \xymatrix{
F \ar@{=}[r] \ar@{^{(}->}[d] & F \ar@{^{(}->}[d] \\
\overline{F} \ar@{->>}[r] & \overline{F}'
} \]
Then the induced morphism $\widetilde{F} \rightarrow \widetilde{F}'$ is a refinement. 
\end{enumerate}
\end{PROP}
\begin{proof}
Note that the comma category ${}^{\downarrow \downarrow}I = I \times_{/I} I$ comes equipped with the following 2-commutative diagram
\[ \xymatrix{
{}^{\downarrow \downarrow} I \ar[r]^{\pi_2} \ar[d]_{\pi_1} \ar@{}[rd]|{\Nearrow^\mu} & I \ar@{=}[d] \\
I \ar@{=}[r] & I
} \]

Let $\iota: F \hookrightarrow \overline{F}$ be an exterior compactification. 
Define $\widetilde{F}: {}^{\downarrow \downarrow}I \rightarrow \mathcal{S}$ as the pull-back: 
\[ \xymatrix{
\widetilde{F} \ar[r] \ar[d] & \pi_2^* F \ar[d]^{\pi_2^* \iota} \\
\pi_1^* \overline{F} \ar[r] & \pi_2^* \overline{F}
} \]
where the bottom horizontal morphism is induced by the natural transformation $\mu: \pi_1 \Rightarrow \pi_2$. 
We have to see that $\widetilde{F}$ is an interior compactification. Taking $\Delta^*$ of the diagram (for $\Delta: I \rightarrow {}^{\downarrow \downarrow}I$
being the diagonal), we get
\[ \xymatrix{
\Delta^* \widetilde{F} \ar[r] \ar[d] &  F \ar[d]^{ \iota} \\
 \overline{F} \ar@{=}[r] &  \overline{F}
} \]
hence there is a canonical isomorphism $\Delta^* \widetilde{F} \cong F$. 
By definition a morphism 
\[ \xymatrix{
i \ar[r] \ar[d] & j \ar[d] \\
i' \ar[r] & j'
} \]
in ${}^{\downarrow \downarrow}I$ is mapped by $\widetilde{F}$ to the morphism
\[ \widetilde{F}(i \rightarrow j) = \overline{F}(i) \times_{\overline{F}(j)} F(j) \rightarrow \overline{F}(i') \times_{\overline{F}(j')} F(j') = \widetilde{F}(i' \rightarrow j') \]
Since the morphism $\overline{F}(i) \rightarrow \overline{F}(i')$ is proper this morphism is proper if $j=j'$ by Lemma~\ref{LEMMAPROPERTIESCOMP}, 2.
If $i=i'$ look at the following commuative diagram:
\[ \xymatrix{
F(i) \ar@{=}[d] \ar@{^{(}->}[r] & \widetilde{F}(i \rightarrow j) = \overline{F}(i) \times_{\overline{F}(j)} F(j) \ar@{^{(}->}[r] \ar[d] &  \overline{F}(i) \ar@{=}[d] \\
F(i) \ar@{^{(}->}[r] & \widetilde{F}(i \rightarrow j') = \overline{F}(i) \times_{\overline{F}(j')} F(j') \ar@{^{(}->}[r] & \overline{F}(i)
} \]
The horizontal morphisms are all embeddings by construction, by (S2), and by (S3). Since the composition $F(i) \hookrightarrow \overline{F}(i)$ is dense by construction, the horizontal embeddings are all dense by (S0) and hence so is the the middle vertical one by (S2). 
The observations together imply that $\widetilde{F}$ is an interior compactification of $F$. 
The claimed functoriality is clear. 
\end{proof}

\section{Fibered multiderivators over 2-categorical bases}\label{SECTFIBDER}

In this section, we recall from \cite{Hor16} the notion of 2-pre-multiderivator and fibered multiderivator (with 2-categorical bases).

\begin{DEF}[{\cite[Definition~2.1]{Hor16}}]\label{DEF2PREMULTIDER}
A {\bf 2-pre-multiderivator} is a functor $\SSS: \Dia^{1-\op} \rightarrow \text{2-$\mathcal{MCAT}$}$ which is strict in 1-morphisms (functors) and pseudo-functorial in 2-morphisms (natural transformations). 
More precisely, it associates with a diagram $I$ a 2-multicategory $\SSS(I)$, with a functor $\alpha: I \rightarrow J$ a strict functor
\[ \SSS(\alpha):  \SSS(J) \rightarrow \SSS(I) \] 
denoted also $\alpha^*$ if $\SSS$ is understood, and with a natural transformation $\mu: \alpha \Rightarrow \alpha'$ a pseudo-natural transformation
\[ \SSS(\eta): \alpha^* \Rightarrow (\alpha')^* \]
such that the following holds:
\begin{enumerate}
\item The association
\[ \Fun(I, J) \rightarrow \Fun^{\mathrm{strict}}(\SSS(J), \SSS(I)) \]
given by $\alpha \mapsto \alpha^*$, resp.\@ $\mu \mapsto \SSS(\mu)$, is
a pseudo-functor (this involves, of course, the choice of further data). Here $\Fun^{\mathrm{strict}}(\SSS(J), \SSS(I))$ is the 2-category of strict 2-functors, pseudo-natural transformations, and modifications. 
\item (Strict functoriality w.r.t.\@ compositons of 1-morphisms) For functors $\alpha: I \rightarrow J$ and $\beta: J \rightarrow K$, we have 
an {\em equality} of pseudo-functors $\Fun(I, J) \rightarrow \Fun^{\mathrm{strict}}(\SSS(I), \SSS(K))$
\[ \beta^* \circ \SSS(-) = \SSS(\beta \circ -).   \]
\end{enumerate}

A {\bf symmetric, resp.\@ braided 2-pre-multiderivator} is given by the structure of strictly symmetric (resp.\@ braided) 2-multicategory on $\SSS(I)$ such that
the strict functors $\alpha^*$ are equivariant w.r.t.\@ the action of the symmetric groups (resp.\@ braid groups). 

Similarly we define a {\bf lax, resp.\@ oplax, 2-pre-multiderivator} where the same as before holds but where the 
\[ \SSS(\eta): \alpha^* \Rightarrow (\alpha')^* \]
are lax (resp.\@ oplax) natural transformations and in 1.\@ ``pseudo-natural transformations'' is replaced by ``lax (resp.\@ oplax) natural transformations''.
\end{DEF}

\begin{DEF}[{\cite[Definition~2.2]{Hor16}}]\label{DEF2PREMULTIDERSTRICTMOR}
A strict morphism $p: \DD \rightarrow \SSS$ of 2-pre-multiderivators (resp.\@ lax/oplax 2-pre-multiderivators) is given by a collection of strict 2-functors
\[ p(I): \DD(I) \rightarrow \SSS(I) \]
for each $I \in \Dia$ such that we have $\SSS(\alpha) \circ p(J) = p(I) \circ \DD(\alpha)$ and $\SSS(\mu) \ast p(J) = p(I) \ast \DD(\mu)$ 
 for all functors $\alpha: I \rightarrow J$, $\alpha': I \rightarrow J$ and natural transformations $\mu: \alpha \Rightarrow \alpha'$ as illustrated by the following diagram:
\[ \xymatrix{
\DD(J) \ar[rr]^{p(J)} \ar@/_15pt/[dd]_{\DD(\alpha)}^{\phantom{x}\overset{\DD(\mu)}{\Rightarrow}} \ar@/^15pt/[dd]^{\DD(\alpha')} && \SSS(J) \ar@/_15pt/[dd]_{\SSS(\alpha)}^{\phantom{x}\overset{\SSS(\mu)}{\Rightarrow}} \ar@/^15pt/[dd]^{\SSS(\alpha')} \\
\\
\DD(I) \ar[rr]^{p(I)} && \SSS(I)
} \]
\end{DEF}

\begin{PAR}\label{PARDER12}
As with usual pre-multiderivators we consider the following axioms: 

\begin{itemize}
\item[(Der1)] For $I, J \in \Dia$, the natural functor $\DD(I \coprod J) \rightarrow \DD(I) \times \DD(J)$ is an equivalence of 2-multicategories. Moreover $\DD(\emptyset)$ is not empty.
\item[(Der2)]
For $I \in \Dia$ the `underlying diagram' functor
\[ \dia: \DD(I) \rightarrow \Fun(I, \DD(\cdot)) \quad \text{resp. } \Fun^{\lax}(I, \DD(\cdot)) \quad \text{resp. }  \Fun^{\oplax}(I, \DD(\cdot))\]
is 2-conservative (this means that it is conservative on 2-morphisms and that a 1-morphism $\alpha$ is an equivalence if $\dia(\alpha)$ is an equivalence).
\end{itemize}

A pre-2-multiderivator with domain $\Dia$ is called {\bf infinite} if $\Dia$ is infinite (i.e.\@ closed under arbitrary coproducts) and we have
\begin{itemize}
\item[(Der1${}^\infty$)] For $\{I_i\}_{i \in \mathcal{I}}$ a (possibly infinite) family with $I_i \in \Dia$, the natural functor $\DD(\coprod_{i} I_i) \rightarrow \prod_{i} \DD(I_i)$ is an equivalence of 2-multicategories. Moreover $\DD(\emptyset)$ is not empty.
\end{itemize}

In addition, consider the following axioms on a strict morphism $p: \DD \rightarrow \SSS$ of 2-pre-multiderivators (where (FDer0 left) is assumed for (FDer3--5 left) and similarly for the right case):
\begin{itemize}
\item[(FDer0 left)]
For each $I$ in $\Dia$ the morphism $p$ specializes to an 1-opfibered 2-multicategory with 1-categorical fibers. 
It is, in addition, 2-fibered in the lax case and 2-opfibered in the oplax case. 
Moreover any {\em fibration} $\alpha: I \rightarrow J$ in $\Dia$ induces a  diagram
\[ \xymatrix{
\DD(J) \ar[r]^{\alpha^*} \ar[d] & \DD(I) \ar[d]\\
\SSS(J) \ar[r]^{\alpha^*} & \SSS(I) 
}\]
of 1-opfibered and 2-(op)fibered 2-multicategories, i.e.\@ the top horizontal functor maps coCartesian 1-morphisms to coCartesian 1-morphisms and (co)Cartesian 2-morphisms to (co)Cartesian 2-morphisms.

We assume that corresponding push-forward functors between the fibers have been chosen and those will be denoted by $(-)_\bullet$.

\item[(FDer3 left)]
For each functor $\alpha: I \rightarrow J$ in $\Dia$ and $S \in \SSS(J)$ the functor
$\alpha^*$ between fibers (which are 1-categories by (FDer0 left))
\[ \DD(J)_{S} \rightarrow \DD(I)_{\alpha^*S} \]
has a left adjoint $\alpha_!^{(S)}$.

\item[(FDer4 left)]
For each functor $\alpha: I \rightarrow J$ in $\Dia$, and for any object $j \in J$, and for the 2-commutative square
\[ \xymatrix{  I \times_{/J} j \ar[r]^-\iota \ar[d]_{\alpha_j} \ar@{}[dr]|{\Swarrow^\mu} & I \ar[d]^\alpha \\
\{j\} \ar@{^{(}->}[r]^j & J \\
} \]
 the induced natural transformation of functors \[ \alpha_{j,!}^{(j^*S)} \SSS(\mu)(S)_\bullet \iota^* \rightarrow j^* {\alpha}_!^{(S)} \] is an isomorphism for all $S \in \SSS(J)$.
\item[(FDer5 left)] For any {\em opfibration} $\alpha: I \rightarrow J $ in $\Dia$, and for any 1-morphism $\xi \in \Hom(S_1, \dots, S_n; T)$ in $\SSS(J)$ for some $n\ge 1$, the natural transformations of functors
\[ \alpha_! (\alpha^*\xi)_\bullet (\alpha^*-, \cdots, \alpha^*-,\ \underbrace{-}_{\text{at }i}\ , \alpha^*-, \cdots, \alpha^*-) \cong  \xi_\bullet (-, \cdots, -,\ \underbrace{\alpha_!-}_{\text{at }i}\ , -, \cdots, -) \]
are isomorphisms for all $i=1, \dots,  n$.
\end{itemize}
\end{PAR}

Dually, we consider the following axioms: 

\begin{enumerate}
\item[(FDer0 right)]
For each $I$ in $\Dia$ the morphism $p$ specializes to a 1-fibered 2-multicategory with 1-categorical fibers.
It is, in addition, 2-opfibered in the lax case, and 2-fibered in the oplax case. 
Furthermore, any {\em opfibration} $\alpha: I \rightarrow J$
in $\Dia$  induces a diagram
\[ \xymatrix{
\DD(J) \ar[r]^{\alpha^*} \ar[d] & \DD(I) \ar[d]\\
\SSS(J) \ar[r]^{\alpha^*} & \SSS(I) 
}\]
of 1-fibered and 2-(op)fibered multicategories, i.e.\@ the top horizontal functor maps Cartesian 1-morphisms w.r.t.\@ the $i$-th slot to Cartesian 1-morphisms w.r.t.\@ the $i$-th slot for any $i$ and maps (co)Cartesian 2-morphisms to (co)Cartesian 2-morphisms.

We assume that corresponding pull-back functors between the fibers have been chosen and those will be denoted by $(-)^{\bullet,i}$.

\item[(FDer3 right)]
For each functor $\alpha: I \rightarrow J$ in $\Dia$ and $S \in \SSS(J)$ the functor
$\alpha^*$ between fibers (which are 1-categories by (FDer0 right))
\[ \DD(J)_{S} \rightarrow \DD(I)_{\alpha^*S} \]
has a right adjoint $\alpha_*^{(S)}$.
\item[(FDer4 right)]
For each morphism $\alpha: I \rightarrow J$ in $\Dia$, and for any object $j \in J$, and for the 2-commutative square
\[ \xymatrix{  j \times_{/J} I \ar[r]^-\iota \ar[d]_{\alpha_j} \ar@{}[dr]|{\Nearrow^\mu} & I \ar[d]^\alpha \\
\{j\} \ar@{^{(}->}[r]^j & J \\
} \]
 the induced natural transformation of functors \[  j^* \alpha_*^{(S)} \rightarrow \alpha_{j,*}^{(j^*S)} \SSS(\mu)(S)^\bullet \iota^* \] is an isomorphism for all $S \in \SSS(J)$.

\item[(FDer5 right)] For any {\em fibration} $\alpha: I \rightarrow J$ in $\Dia$, and for any 1-morphism $\xi \in \Hom(S_1, \dots, S_n; T)$ in $\SSS(J)$ for some $n\ge 1$, the natural transformations of functors
\[ \alpha_* (\alpha^*\xi)^{\bullet,i} (\alpha^*-, \overset{\widehat{i}}{\cdots}, \alpha^*-\ ;\ -) \cong  \xi^{\bullet,i} (-, \overset{\widehat{i}}{\cdots}, -\ ;\ \alpha_*-) \]
are isomorphisms for all $i= 1, \dots, n$.
\end{enumerate}

\begin{DEF}\label{DEFFIBDER}
A strict morphism $p: \DD \rightarrow \SSS$ of (op)lax 2-pre-multiderivators is called a  {\bf  (op)lax left (resp.\@ right) fibered multiderivator}
if $\DD$ and $\SSS$ both satisfy (Der1) and (Der2) and if (FDer0 left/right) and (FDer3--5 left/right) hold true. We just say ``fibered''  for ``left and right fibered''. 
\end{DEF}

\begin{BEM}
One can show that the axioms imply, in the plain case, that the second part of (FDer0 left) and (FDer5 right) --- which are then adjoint to each other ---  hold true for any functor $\alpha: I \rightarrow J$. 
Similarly {\em for 1-ary morphisms} the second part of (FDer0 right) and (FDer5 left) hold true for any functor $\alpha: I \rightarrow J$. 

In the oplax case the second part of (FDer0 left) and (FDer5 right) should be claimed to hold for any functor $\alpha: I \rightarrow J$, and in the lax case, and {\em for 1-ary morphisms}, the second part of  (FDer0 right) and (FDer5 left) should be claimed to hold for any functor $\alpha: I \rightarrow J$. It seems that this does not follow from the other axioms as stated. 
For the oplax left and lax right fibered multiderivators constructed in section~\ref{SECTCONSTPROPER} we will show explicitly that these stronger statements hold true. 
\end{BEM}

If in $\SSS$ all 2-morphisms are invertible then there is no difference between lax and oplax and we just say left (resp.\@ right) fibered multiderivator. 

\begin{DEF}\label{DEFSTABLE}For (op)lax fibered multiderivators over an (op)lax 2-pre-multiderivator $p: \DD \rightarrow \SSS$ and an object $S \in \SSS(I)$ we have that
\[ \DD_{I,S}: J \mapsto \DD(I \times J)_{\pr_2^*S}  \]
is a usual derivator. We say that $p$ has {\bf stable fibers} if $\DD_{I,S}$ is stable for all $S \in \SSS(I)$ and for all $I$. In fact, it suffices to require this for $I=\cdot$. 
\end{DEF}

We briefly mention that in \cite{Hor16} a fibered multiderivator was defined in a different, equivalent way, as follows. Above we gave the equivalent patchwork definition because the axioms are anyway the ones to be checked. 

A strict morphism $\DD \rightarrow \SSS$ of (op)lax 2-pre-multiderivators (Definition~\ref{DEF2PREMULTIDERSTRICTMOR}) such that $\DD$ and $\SSS$ each satisfy (Der1) and (Der2) (cf.\@ \ref{PARDER12}) is a
\begin{enumerate}
\item  lax left (resp.\@ oplax right) fibered multiderivator if and only if the corresponding strict functor of 2-multicategories
\[ \Dia^{\cor}(p): \Dia^{\cor}(\DD) \rightarrow \Dia^{\cor}(\SSS) \]
(cf.\@ \cite[Definition~3.6]{Hor16}) is a 1-opfibration (resp.\@ 1-fibration) and 2-fibration 
with 1-categorical fibers.
\item oplax left (resp.\@ lax right) fibered multiderivator if and only if the corresponding strict functor of 2-multicategories
\[ \Dia^{\cor}(p): \Dia^{\cor}(\DD^{2-\op}) \rightarrow \Dia^{\cor}(\SSS^{2-\op}) \]
(cf.\@ \cite[Definition~3.6]{Hor16}) is a 1-opfibration (resp.\@ 1-fibration) and 2-fibration 
with 1-categorical fibers.
\end{enumerate}

\section{Diagrams of correspondences and their compactifications}\label{SECTIONDIACORCOMP}

Let $\Xi \in \{\downarrow \uparrow \}^n$ be a sequence of arrow directions, and let $I$ be a diagram. 
Recall the diagram denoted ${}^\Xi I$ from \ref{PARTW}.

\begin{LEMMA}\label{LEMMACATLF}
\begin{enumerate}
\item
If $I$ is in $\Catlf$ (locally finite diagrams\footnote{A diagram (i.e.\@ a small category) is called {\bf locally finite} if any morphism can be factored only in a finite number of ways into non-identity morphisms.}) then $\tw I$ is inverse and locally finite and has finite matching diagrams. 
\item
If $I$ is in $\Dirlf$ then also ${}^{\uparrow \uparrow \downarrow}I$ and $\twwc{I}$ are in $\Dirlf$.
\end{enumerate}
\end{LEMMA}
\begin{proof}
1. We define a functor $\nu: \tw I \rightarrow (\N_0)^{\op}$ with maps a morphism $\alpha: i \rightarrow J$ to the maximum number of (non-identity) morphisms into which $\alpha$ can be factored or to 0 if $\alpha$ is an identity. The property of $\nu$ being a functor and that pre-images of identities are identities is clear. 
Any morphism from $\alpha$ to another $\beta$ in $\tw I$ corresponds to a factorization of $\alpha$. Since $I$ is locally finite the matching category of $\alpha$ is thus finite. 

2. If $I$ is in $\Catlf$ also ${}^{\downarrow\downarrow} I$ is in $\Catlf$. By 1.\@ thus ${}^{\uparrow\uparrow\downarrow} I$ (which is a subcategory of $(\tw({}^{\downarrow\downarrow} I))^{\op}$) is in $\Dirlf$. Therefore also
$I \times {}^{\uparrow\uparrow\downarrow} I$ and finally ${}^{\downarrow \uparrow\uparrow\downarrow} I$ are in $\Dirlf$.
\end{proof}

\begin{PAR}\label{PAROPMULTCAT}
For the rest of the section, let $\mathcal{S}$ be a category with finite limits. We consider $\mathcal{S}$ as a symmetric opmulticategory (and thus $\mathcal{S}^{\op}$ as symmetric multicategory) setting
\[ \Hom(S; T_1, \dots, T_n):= \Hom(S, T_1) \times \cdots \times \Hom(S, T_n) \]
with the obvious action of the symmetric group. 
\end{PAR}

\begin{PAR}
Let $\Xi \in \{\downarrow \uparrow \}^n$ be a sequence of arrow directions. 
Recall that a morphism $f: S \rightarrow T$ in $\Fun({}^\Xi I, \mathcal{S})$ is called {\bf type-$i$ admissible} if for every type-$i$ morphism $\alpha \rightarrow \beta$ the square
\[ \xymatrix{
S(\alpha) \ar[r]^{f(\alpha)} \ar[d] & T(\alpha) \ar[d] \\
S(\beta) \ar[r]_{f(\beta)} & T(\beta)
} \]
is Cartesian. Having chosen a class $\mathcal{S}_0$ of proper morphisms in $\mathcal{S}$, 
we will call a morphism $f: S \rightarrow T$ in $\Fun({}^\Xi I, \mathcal{S})$ {\bf weakly type-$i$ admissible} if the squares above are weakly Cartesian (cf.\@ \ref{DEFWEAKLYCART}). Of course, this definition depends on the chosen class of proper morphisms.
Note that the analog of \cite[Lemma~7.10]{Hor16} holds true for {\em weakly type-$i$ admissible}.

In the same way, if $f: S \rightarrow T_1, \dots, T_n$ is a multimorphism in $\Fun({}^\Xi I, \mathcal{S})$ then we say that $f$ is {\bf type-$i$ admissible} if for every type-$i$ morphism $\alpha \rightarrow \beta$ the square
\[ \xymatrix{
S(\alpha) \ar[r]^-{f(\alpha)} \ar[d] & T_1(\alpha), \dots, T_n(\alpha) \ar[d] \\
S(\beta) \ar[r]_-{f(\beta)} & T_1(\beta), \dots, T_n(\beta)
} \]
is Cartesian (i.e.\@ a multi-pullback). Since $\mathcal{S}$ carries the particular opmulticategory structure of (\ref{PAROPMULTCAT}) this means more concretely that 
\[ \xymatrix{
S(\alpha) \ar[r]^-{f(\alpha)} \ar[d] & \prod_i T_i(\alpha) \ar[d] \\
S(\beta) \ar[r]_-{f(\beta)} & \prod_i T_1(\beta)
} \]
is Cartesian. 
Similarly for weakly type-$i$ admissible. 
\end{PAR}

\begin{PAR}\label{PARSCOR}
Let $\mathcal{S}^{\cor}$ (resp.\@ $\mathcal{S}^{\cor, 0}$) be the symmetric 2-multicategory of multicorrespondences in $\mathcal{S}$ in which the 2-morphisms are given by
the isomorphisms (resp.\@ proper morphisms --- same class as in the definition of compactification on $\mathcal{S}$), cf.\@ \cite[3.6]{Hor15b}. 
The objects of both categories are the same as the objects of $\mathcal{S}$, and 1-morphisms $S_1, \dots, S_n \rightarrow T$ are multicorrespondences

\begin{equation}\label{excor2}
 \vcenter{ \xymatrix{ 
 &&&  \ar[llld]_{g_1} A \ar[ld]^{g_n} \ar[rd]^{f} &\\
 S_1 & \cdots & S_n & ; &  T   } }
 \end{equation}
The composition of 1-morphisms is given by forming fiber products and the 2-morphisms are the isomorphisms (in $\mathcal{S}^{\cor}$), or proper morphisms (in $\mathcal{S}^{\cor,0}$), of such multicorrespondences. 
The action of the symmetric groups is the obvious one. 
Strictly speaking, the above definitions give bimulticategories because the formation of fiber products is only associative up to isomorphism. One can, however, enlarge the class of objects adjoining strictly associative fiber products, cf.\@ \cite[Remark~5.3]{Hor16}. We will follow a different strictification strategy here (using Appendix~\ref{APPENDIX2MULTICAT}), cf.\@ Definition~\ref{DEFSCOR2} below.
\end{PAR}

\begin{PAR}\label{DEFTREE}
Recall that a {\bf tree} is a finite connected multicategory freely generated by a set of multimorphisms such that each
object occurs at most once as a source and at most once as a destination of one of these generating multimorphisms. 
The generating multimorphisms are allowed to be 0-ary. 

Examples:
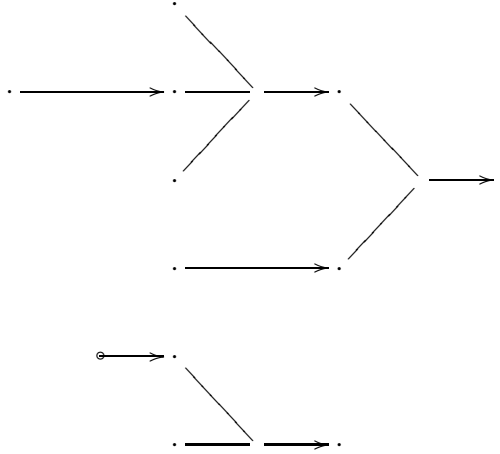
\[ \xymatrix{
&&\cdot \ar@{-}[dr] &   \\ 
\cdot \ar[rr] & & \cdot \ar@{-}[r] & \ar[r] & \cdot  \ar@{-}[rd] & &  \\
&&\cdot \ar@{-}[ur] & & & \ar@{->}[r] & \cdot \\
&&\cdot \ar[rr] & & \cdot \ar@{-}[ur] \\
 & \ar@{o->}[r]^{\text{0-ary}} & \cdot \ar@{-}[dr]&   \\
 && \cdot \ar@{-}[r]& \ar[r] & \cdot  \\
} \]

A {\bf symmetric tree} $\tau^S$ is obtained from a tree $\tau$ adding images (in the most free way possible)  of the multimorphisms (not only the generating ones) under the respective symmetric groups. Observe that there is an obvious composition turning a symmetric tree into a symmetric multicategory. Giving a functor (of multicategories) from a tree to a symmetric multicategory $\mathcal{S}$ is the same as giving a functor (of symmetric multicategories) from its symmetric variant to $\mathcal{S}$. 

The most basic trees are the $\Delta_{1,n}$, for $n \in \N_0$, consisting of $n+1$ objects and one $n$-ary morphism connecting them. 
Each tree has a well-defined destination object and a number (possibly zero) of source objects. Two trees $\tau_{1}$ and $\tau_{2}$ can be concatenated to a tree $\tau_{2} \circ_i \tau_{1}$ choosing any source object $i$ of the tree $\tau_{2}$. 
\end{PAR}

\begin{PAR}\label{MORLIST}
Let $M$ be a (small) multicategory. For each pair of ordered set of objects $\mathcal{E}:=(\mathcal{E}_1, \dots, \mathcal{E}_n)$, $\mathcal{F}:=(\mathcal{F}_1, \dots, \mathcal{F}_m)$ in $M$ we define the set of {\bf morphisms} from $\mathcal{E}$ to $\mathcal{F}$ to be
a sequence of integers $0 \le n_1 \le \dots \le n_{m-1} \le n$ and multimorphisms
\[ \mathcal{E}_1, \dots, \mathcal{E}_{n_1} \rightarrow \mathcal{F}_1; \ 
 \mathcal{E}_{n_1+1}, \dots, \mathcal{E}_{n_2} \rightarrow \mathcal{F}_2; \ 
 \dots; \ 
 \mathcal{E}_{n_{m-1}+1}, \dots , \mathcal{E}_{n} \rightarrow \mathcal{F}_m  \]
The integers $n_i$ may be equal and also $n=0$ is allowed.  If $n=m=0$ we understand there to be exactly one morphism. 
\end{PAR}

\begin{PAR}
Let $\Xi \in \{ \downarrow, \uparrow \}^l$ be sequence of arrow directions and let $M$ be a multidiagram (i.e.\@ a small multicategory like $\Delta_{1,n}$).
If $\Xi_l = \downarrow$, we define a small multicategory ${}^\Xi M$ and 
if $\Xi_l = \uparrow$, we define a small opmulticategory ${}^\Xi M$. We concentrate on the case $\Xi_l = \downarrow$ for definiteness. 

Objects are sequences
\[ \xymatrix{ [S_{1,1}, \dots, S_{1,n_1}] \ar[r] & [S_{2,1}, \dots, S_{2,n_2}] \ar[r] & \cdots \ar[r] &   [S_{l,1}] &   } \]
of $l$ lists of objects (can be empty) and morphisms in the sense of \ref{MORLIST} between them, where however the $l$-th list consist of exactly one object. 
Multimorphisms $S^{(1)}, \dots, S^{(n)} \rightarrow T$ are diagrams
\[ \xymatrix{ [S_{1,1}^{(1)}, \dots, S_{1,n_1}^{(1)}, \dots] \ar[r] \ar@{<->}[d] & [S_{2,1}^{(1)}, \dots, S_{2,n_2}^{(1)}, \dots] \ar[r]\ar@{<->}[d] & \cdots \ar[r] &   [S_{l,1}^{(1)}, \dots, S^{(n)}_{l,1}] \ar@{->}[d] &   \\
 [T_{1,1}, \dots, T_{1,n_1}] \ar[r] & [T_{2,1}, \dots, T_{2,n_2}] \ar[r] & \cdots \ar[r] &   [T_{l,1}] &   } \]
where the arrow direction in the $i$-th column is determined by $\Xi_i$. 
Such a morphism is called of {\bf type $i$} if all vertical morphisms except the $i$-th one are identities of lists. There are thus only $n$-ary morphisms for $n\not=1$ of type $l$ and not of any other type.

{\em Example:} For the tree $\Delta_{1,n}$ 
the multidiagram ${}^{\uparrow \downarrow} (\Delta_{1,n})$ is 
\[ \xymatrix{
&& [1,\dots,n] \rightarrow [n+1]  \\
& {\phantom{xxx}} \\ 
\id_{[1]} \ar[urur]^-{\text{type 2}}  & \cdots &  \ar@{-}[lu] \id_{[n]}  &  & \id_{[n+1]}  \ar[ulul]_{\text{type 1}}
} \]
\end{PAR}

\begin{PAR}\label{PARCIRC}
For a (small) opmulticategory $M$ define a usual category $M^\circ$ replacing all multimorphisms in $M$ in $\Hom(j; i_1, \dots, i_n)$ by a set of 1-ary morphisms $j \rightarrow i_1$, \dots,  $j \rightarrow i_n$ \footnote{that means, in particular, forgetting all $0$-ary morphisms}. Then, for the special opmulticategory structure (\ref{PAROPMULTCAT}) on $\mathcal{S}$, 
a functor of opmulticategories $M \rightarrow \mathcal{S}$ is the same as a functor between usual categories $M^\circ \rightarrow \mathcal{S}$.
\end{PAR}

\begin{PAR}\label{PARALTSCOR}
Let $I$ be a (multi)diagram. A diagram of correspondences, i.e.\@ a pseudo-functor $X: I \rightarrow \mathcal{S}^{\cor}$ (or equivalently $X: I \rightarrow \mathcal{S}^{\cor,0}$) will be encoded more conveniently as follows: 

A functor $X: \tw I \rightarrow \mathcal{S}$ is called {\bf admissible} or, by abuse of notation, {\bf diagrams of correspondences}, if 
\begin{itemize}

\item every square in the opmulticategory $\tw I$
\begin{equation} \label{eqsquare12} \vcenter{ \xymatrix{ i \ar[r] \ar[d] & j_1, \dots, j_n \ar[d] \\ i' \ar[r] & j_1', \dots, j_n'  } } \end{equation}
in which the horizontal morphisms are of type 2 and the vertical morphisms are of type 1 
is mapped by $X$ to a homotopy Cartesian diagram.
\end{itemize}
We denote the corresponding full subcategory of the functor category by $\Fun(\tw I, \mathcal{S})^{\mathrm{adm}}$.
Similarly, a functor $X: \tw I \rightarrow \mathcal{S}$ is called {\bf weakly admissible}, if all squares as above are mapped to weakly Cartesian squares (cf.\@ \ref{DEFWEAKLYCART}). 

A multimorphism $\xi: X_1, \dots, X_n \rightarrow Y$ in the multicategory $\Fun(I, \mathcal{S}^{\cor})$ then corresponds to a diagram in $\Fun(\tw I, \mathcal{S})^{\mathrm{adm}}$
\begin{equation}\label{eqcomp1} \vcenter{ \xymatrix{
& & & A \ar[rd]^f \ar[ld]^{g_n} \ar[llld]_{g_1} \\
X_1 & \dots & X_n & ; & Y
} }
\end{equation}
in which  $f$ is type-2 admissible, and $g$ {\em as a multimorphism in $\Fun(\tw I, \mathcal{S})$} is type-1 admissible.
A lax multimorphism $X_1, \dots, X_n \rightarrow Y$, i.e.\@ a multimorphism in $\Fun^{\lax}(I, \mathcal{S}^{\cor, 0})$ corresponds to a diagram of the same shape, in which, however the morphism $f$ is {\em weakly type-2 admissible}.  Similarly an oplax multimorphism $X_1, \dots, X_n \rightarrow Y$ corresponds to the same diagram, in which, however the multimorphism $g$ is only {\em weakly type-1 admissible}. 
A 2-morphism $\mu: \xi \rightarrow \xi'$ corresponds to by a morphism of multicorrespondences
\begin{equation}\label{eqcomp2} \vcenter{\xymatrix{
& & & A \ar[rd]^f \ar[ld]^{g_n} \ar[llld]_{g_1} \ar[dd]^{h} \\
X_1 & \dots & X_n & & Y \\
& & & A \ar[ru]_{f'} \ar[lu]_{{g_n'}} \ar[lllu]^{{g_1'}} 
} }
\end{equation}
in which the morphism $h$ is an isomorphism (or, in the lax and oplax case, any proper map). 
It is (similarly to the case of $\mathcal{S}^{\cor,G}$, cf.\@ \cite[Definition~3.2]{Hor15b}) automatically type-1 admissible and weakly type-2 admissible in the lax case and weakly type-1 admissible and type-2 admissible in the oplax case. 

The two viewpoint are essentially the same, cf.\@ Lemma~\ref{LEMMAPF} below. 
\end{PAR}

\begin{DEF}\label{DEFCOR}
Let $\tau$ be a tree with symmetrization $\tau^S$ and let $I$ be a (usual) diagram. 
We define a category with weak equivalences 
\[ \Cor_I(\tau^S) \]
whose objects are admissible objects $X \in \mathcal{S}^{\tw (\tau \times I)}$ in the sense of \ref{PARALTSCOR} (which are functors of opmulticategories, cf.\@ also \ref{PARCIRC}) and the morphisms are the point-wise isomorphisms. 
This defines a strict functor
\[ \Cor_I: \Delta_S \rightarrow \mathcal{CAT} \rightarrow \mathcal{CATW} \]
from symmetric trees to categories,  considering categories as categories with weak equivalences in which the latter consist of all isomorphisms\footnote{The machinery of Appendix~\ref{APPENDIX2MULTICAT} is more general allowing for non-trivial classes $\mathcal{W}$. This generality will be needed later. }.

Similarly, we define
\[ \Cor_I^{\lax}: \Delta_S \rightarrow \mathcal{CATW} \qquad \Cor_I^{\oplax}: \Delta_S \rightarrow \mathcal{CATW} \]
whose objects are objects $X \in \mathcal{S}^{\tw (\tau \times I)}$ which are {\em point-wise in $\tau$} admissible diagrams in $\mathcal{S}^{\tw I}$, and multimorphisms of type-2 in $\tw \tau$ are mapped to  type-1 admissible (resp.\@ weakly type-1 admissible) multimorphisms and morphisms of type-1 in $\tw \tau$ are mapped to weakly type-2 admissible (resp.\@  type-2 admissible) morphisms. Furthermore, they are admissible in the argument $I$, i.e.\@ every square (\ref{eqsquare12}) in $I$ (for $n=1$) in which the horizontal multimorphisms are of type 2 and the vertical morphisms are of type 1 are mapped to Cartesian squares. 
Morphisms are all point-wise proper morphisms. 
\end{DEF}
Note that $\Cor_I$, $\Cor^{\lax}_I$, and $\Cor^{\oplax}_I$, are indeed defined on {\em symmetric} trees.
For, observe that functors $\tw (\tau \times I) \rightarrow \mathcal{S}$ of opmulticategories are the same as functors of usual categories $(\tw (\tau \times I))^\circ \rightarrow \mathcal{S}$ (cf.\@ \ref{PARMULTICOMP}) and that 
every functor $\tau^S \rightarrow (\tau')^S$ induces an obvious functor $(\tw (\tau \times I))^\circ \rightarrow (\tw (\tau' \times I))^\circ$.

The following follows directly from the definition (admissibility in the $\tw \tau$ argument):
\begin{LEMMA}\label{LEMMAPROPCONSTRSYMMULTI1}
The strict functor $\tau^S \mapsto \Cor_I^{}(\tau^S)$ (resp.\@ $\tau^S \mapsto \Cor_I^{\lax}(\tau^S)$, $\tau^S \mapsto \Cor_I^{\oplax}(\tau^S)$)  satisfies the axioms 1 and 2 of Proposition~\ref{PROPCONSTRSYMMULTI}.
\end{LEMMA}

\begin{DEF}\label{DEFSCOR1}
Let $I$ be a diagram. We define $\SSS^{\cor}(I)$ (resp.\@ $\SSS^{\cor, 0, \lax}(I)$, resp.\@ $\SSS^{\cor, 0, \oplax}(I)$) to be the symmetric 2-multicategory of Proposition~\ref{PROPCONSTRSYMMULTI} constructed from the functor $C_I$ (resp.\@ $C_I^{\lax}$, resp.\@ $C_I^{\oplax})$.
\end{DEF}

\begin{LEMMA}\label{LEMMAPF}
Let $I$ and $J$ be (multi)diagrams. 
\begin{enumerate}
\item Any diagram $D \in \SSS^{\cor}(J \times I)$ gives rise to a canonical pseudo-functor of 2-multicategories
\[ \Dia(D): J \rightarrow \SSS^{\cor}(I) \]
defined on objects by $j \mapsto D|_{\tw I \times (\tw j)}$. 
\item The association $D \mapsto \Dia(D)$ yields an equivalence of 2-multicategories
\[ \SSS^{\cor}(J \times I) \cong \Fun(J, \SSS^{\cor}(I)) \]
and
\[   \SSS^{\cor, 0, \lax}(J \times I) \cong \Fun^{\lax}(J, \SSS^{\cor}(I)) \quad   \SSS^{\cor, 0, \oplax}(J \times I) \cong \Fun^{\oplax}(J, \SSS^{\cor}(I))  \] 
respectively.
\end{enumerate}
\end{LEMMA}
\begin{proof}
1.\@ We sketch the non-multi variant here, and leave the multi-case to the reader. 
Each morphism $\alpha: j \rightarrow j'$ in $J$ gives rise to a functor (pull-back of $D$)
\[ \alpha': \tw \Delta_{1} \rightarrow \mathcal{S}^{\tw I }\]
with values in admissible objects, more precisely, to a diagram of the form (\ref{eqcomp1}), i.e.\@ to a 1-morphism in $\SSS^{\cor}(I)$.
A composition $\beta \circ \alpha$ in $J$ gives  rise to  a functor
\[ \tw \Delta_2  \rightarrow\mathcal{S}^{\tw I } \]
and $(\beta \circ  \alpha)'$ is the pullback along 
\[ e_{02}: \tw \Delta_1 \rightarrow \tw \Delta_2.  \]
By construction, we get a 2-isomorphism 
\[ \beta' \circ \alpha' \Rightarrow (\beta \circ \alpha)'. \]
The analogous reasoning with a composition of three morphisms shows that this construction yields a pseudo-functor. 

2.\@ is left to the reader. 
\end{proof}

\begin{DEF}\label{DEFSCOR2}
We define a (lax, oplax) symmetric 2-pre-multiderivator $\SSS^{\cor}$ (resp.\@ $\SSS^{\cor,0,\lax}$, resp.\@ $\SSS^{\cor,0,\oplax}$). Let $I \in \Cat$ be a diagram. The 2-multicategory $\SSS^{\cor}(I)$ (resp.\@ $\SSS^{\cor,0,\lax}(I)$, resp.\@ $\SSS^{\cor,0,\oplax}(I)$) has been defined in Definition~\ref{DEFSCOR1}.
Those symmetric 2-multicategories are equipped with strict and symmetric pull-back functors
\[ \alpha^*: \SSS^{\cor}(J) \rightarrow \SSS^{\cor}(I). \]
for each functor $\alpha: I \rightarrow J$. For a natural transformation $\mu: \alpha \Rightarrow \beta$ we get a pseudo-natural transformation 
\[ \alpha^* \Rightarrow \beta^* \]
as follows: $\mu$ might be seen as a functor $I \times \Delta_1 \rightarrow J$ and so each admissible $X \in \mathcal{S}^{\tw I}$ gives rise to an admissible
 $\mu^* X \in \mathcal{S}^{\tw (\Delta_1 \times I)}$ which constitutes a 1-morphism from $\alpha^* X$ to $\beta^* X$. For each 1-morphism $\xi: X \rightarrow Y$, w.l.o.g.\@ in $\mathcal{S}^{\tw (\Delta_1 \times I)}$, the square
 \begin{equation}\label{eqsquare} \vcenter{ \xymatrix{ \alpha^*X \ar[r]^{\alpha^* \xi} \ar[d]_{\mu^* X} & \alpha^*Y \ar[d]^{\mu^*Y} \\
 \beta^*X \ar[r]_{\beta^* \xi} & \beta^*Y } 
 } \end{equation}
  commutes up to a uniquely determined 2-isomorphism. To see this one can use the pseudo-functor
  \[ \Delta_1^2 \rightarrow \SSS^{\cor}(I) \] 
  obtained from $\mu^* \xi \in \mathcal{M}^{\tw (\Delta_1^2 \times I)}$ via Lemma~\ref{LEMMAPF}, 1. 
One checks that this construction yields a pseudo-functor:
\[ \Fun(I, J) \rightarrow \Fun^{\mathrm{strict}}(\SSS^{\cor}(J), \SSS^{\cor}(I)). \]
In the lax case $\mu^* \xi \in \mathcal{M} ^{\tw (\Delta_1^2 \times I)}$ is not admissible but satisfies conditions on the images of morphisms of type-1 and type-2 in $\tw (\Delta_1^2)$. More precisely, we obtain a diagram
 \begin{equation}\label{eqsquare} \vcenter{ \xymatrix{ \alpha^*X   & \ar[l] \ar[r] \ar@{}[rd]|\blacksquare  & \alpha^*Y  \\
 \ar[u] \ar[d]  \ar@{}[rd]|\Box & \ar[r] \ar[d] W \ar[u] \ar[l] & \ar[u] \ar[d] \\
 \beta^*X  & \ar[l] \ar[r] & \beta^*Y  }  
 } \end{equation}
in which the morphisms going to the right are only weakly type-1 admissible and where the square denoted $\blacksquare$ is weakly Cartesian and the square denoted $\Box$ is Cartesian. 
Denote by $\eta$ the correspondence $\alpha^*X \leftarrow W \rightarrow \beta^* Y$ in $\mathcal{S}^{\tw (\Delta_1 \times I) }$. By definition of $\SSS^{\cor, 0, \lax}(I)$ this yields a 2-isomorphism
\[ \eta \overset{\sim}{\Longrightarrow} (\beta^* \xi )\circ (\mu^* X) \]
and a 2-morphism (not invertible in general)
\[ \eta \Longrightarrow (\mu^* Y) \circ (\alpha^* \xi ).  \]
Combining the two, we get a lax commutative square (\ref{eqsquare}).  One checks that this endows 
\[ \alpha^* \Rightarrow \beta^* \]
with the structure of pseudo-natural transformation and that one obtains a pseudo-functor
\[ \Fun(I, J) \rightarrow \Fun^{\lax, \mathrm{ strict}}(\SSS^{\cor,0,\lax}(J), \SSS^{\cor,0,\lax}(I)). \]
The oplax case is done similarly.  
\end{DEF}
Actually, by Lemma~\ref{LEMMAPF}, 2., the 2-pre-multiderivator $\SSS^{\cor}$ is equivalent to the 2-pre-multiderivator represented by the 2-multicategory $\mathcal{S}^{\cor}$ (e.g.\@ strictified as in \cite[Remark~5.3]{Hor15b} or taking $\mathcal{S}^{\cor} := \SSS^{\cor}(\cdot)$)

We will now refine these 2-pre-multiderivators taking compactifications into account. The resulting 2-pre-multiderivators will {\em still} be equivalent to the previous ones. This equivalence is the crucial step to show the independence of the main construction of the choice of compactifications later. 

\begin{DEF}\label{DEFCORCOMP}
Let $\tau$ be a tree with symmetrization $\tau^S$, $I$ a (usual) {\em locally finite} diagram. 
We define a category with weak equivalences 
\[ \Cor^{\comp}_I(\tau^S) \]
whose objects are exterior compactifications (cf.\@ Definition~\ref{DEFEXTCOMP}) $X \hookrightarrow \overline{X}$ in $\mathcal{S}^{\tw (\tau \times I)}$ (functors of  opmulticategories)
with $X$ admissible and morphisms are the morphisms of diagrams 
\begin{equation} \label{eqmorextcomp}
\vcenter{ \xymatrix{ X \ar@{^{(}->}[r] \ar[d] & \overline{X} \ar[d] \\
 X' \ar@{^{(}->}[r] & \overline{X}' 
} }
\end{equation}
in which  $X \rightarrow X'$ is an isomorphism. 
Note that $\overline{X}$ is {\em not} assumed to be admissible. These categories are equipped with the class of weak equivalences consisting of {\em all} morphisms. This defines a strict functor
\[ \Cor_I: \Delta_S \rightarrow \mathcal{CATW}  \]
from symmetric trees to categories with weak equivalences. 

Similarly, we define
\[ \Cor_I^{\comp, \lax}: \Delta_S \rightarrow \mathcal{CATW} \qquad \Cor_I^{\comp, \oplax}: \Delta_S \rightarrow \mathcal{CATW} \]
whose objects are exterior compactifications (cf.\@ Definition~\ref{DEFEXTCOMP}) $X \hookrightarrow \overline{X}$ in $\mathcal{S}^{\tw (\tau \times I)}$ such that $X$ is an object in $\Cor_I^{\lax}(\tau)$ (resp.\@ in $\Cor_I^{\oplax}(\tau)$).
Morphisms are morphisms of diagrams (\ref{eqmorextcomp})
 such that $X \rightarrow X'$ is point-wise proper, and weak equivalences are those in which $X \rightarrow X'$ is an isomorphism. 
\end{DEF}
For the functoriality in {\em symmetric} trees the same considerations as following Definition~\ref{DEFCOR} apply.

\begin{LEMMA}\label{LEMMAPROPCONSTRSYMMULTICOMP}
Let $I$ be a locally finite diagram. 
\begin{enumerate}
\item The forgetful functor
\[ \Cor_I^{\comp}(\tau^S)_{(X_o \hookrightarrow \overline{X}_o)_{o \in \tau}}[\mathcal{W}^{-1}_{(X_o \hookrightarrow \overline{X}_o)_{o \in \tau}}] \rightarrow \Cor_I(\tau^S)_{(X_o)_{o \in \tau}} \]
is an equivalence. Similarly for the lax and oplax case. 
\item The strict functor $\tau^S \mapsto \Cor_I^{\comp}(\tau^S)$ (resp.\@ $\tau^S \mapsto \Cor_I^{\comp, \lax}(\tau^S)$, resp.\@ $\tau^S \mapsto \Cor_I^{\comp, \oplax}(\tau^S)$)  satisfies the axioms 1 and 2 of Proposition~\ref{PROPCONSTRSYMMULTI}.
\end{enumerate}
\end{LEMMA}
\begin{proof}
1.\@
We will apply Lemma~\ref{LEMMALOC} below and have to show that the above functor is surjective on objects and morphisms and satisfies the axioms (i) and (ii). 

 The subdiagram \[ \bigcup_{o \in \tau} \tw  I \hookrightarrow \tw (\tau \times I)^\circ \] 
 is final (cf.\@ \ref{PARMULTICOMP} for the notation). 
Hence
 any partial exterior compactification on the left hand side can be extended by Proposition~\ref{PROPCOMPMOR}, 3. Note that by Lemma~\ref{LEMMACATLF} the diagram $\tw I$ and thus also $(\tw (\tau \times I))^\circ$
 is an inverse diagram with finite matching diagrams for any tree. Thus the functor is surjective on objects. The same argument for
 \[ \bigcup_{o \in \tau} (\tw  I) \times \Delta_1 \hookrightarrow (\tw (\tau \times I))^\circ \times \Delta_1 \]
shows that the functor is surjective on morphisms.  
The same argument for 
 \[ (\bigcup_{o \in \tau} (\tw  I) \times \Box) \cup ((\tw (\tau \times I))^\circ \times \righthalfcup) \hookrightarrow (\tw (\tau \times I))^\circ \times \Box \]
 shows axiom (i). 
The same argument for
 \[ (\bigcup_{o \in \tau}( \tw  I) \times \mathbin{\rotatebox[origin=c]{45}{$\boxtimes$}}) \cup ( (\tw (\tau \times I))^\circ \times \mathbin{\rotatebox[origin=c]{45}{$\Box$}}) \hookrightarrow (\tw (\tau \times I))^\circ \times \mathbin{\rotatebox[origin=c]{45}{$\boxtimes$}} \]
 shows axiom (ii). Here $\mathbin{\rotatebox[origin=c]{45}{$\Box$}}$ and $\mathbin{\rotatebox[origin=c]{45}{$\boxtimes$}}$ denote the diagrams
  \[ \xymatrix{ &  \ar[ld] \ar[rd] \\
   &  &  \\
  &  \ar[lu] \ar[ru] } \qquad 
 \xymatrix{ &  \ar[ld] \ar[rd] \\
   &  \ar[u] \ar[d] \ar[l] \ar[r] &  \\
  &  \ar[lu] \ar[ru] }.\]
  Note that $F(f) = F(f')$ implies that the underlying diagram (without exterior compactification) can be extended from $\mathbin{\rotatebox[origin=c]{45}{$\Box$}}$ to $\mathbin{\rotatebox[origin=c]{45}{$\boxtimes$}}$ in a trivial way. 

2.\@ follows from immediately from 1.\@ and Lemma~\ref{LEMMAPROPCONSTRSYMMULTI1}.
\end{proof}

\begin{DEF}\label{DEFSCOR3}
Let $I$ be a locally finite diagram. We define $\SSS^{\cor, \comp}(I)$ (resp.\@ $\SSS^{\cor, \comp, 0, \lax}(I)$, resp.\@ $\SSS^{\cor, \comp, 0, \oplax}(I)$) to be the symmetric 2-multicategory of Proposition~\ref{PROPCONSTRSYMMULTI} constructed from the functor $C_I^{\comp}$ (resp.\@ $C_I^{\comp, \lax}$, resp.\@ $C_I^{\comp, \oplax})$.
\end{DEF}

Note that by Lemma~\ref{LEMMAPROPCONSTRSYMMULTICOMP}, 1., the forgetful functors 
\begin{equation}\label{eqequi}
 \SSS^{\cor, \comp}(I) \to \SSS^{\cor}(I) \quad  \SSS^{\cor, 0, \comp, \lax}(I) \to \SSS^{\cor, 0, \lax}(I) \quad  \SSS^{\cor, 0, \comp, \oplax}(I) \to \SSS^{\cor, 0, \oplax}(I)  
 \end{equation}
induce equivalences on morphism categories.
They are, however, also surjective on objects, because by Proposition~\ref{PROPCOMPMOR}, 1.\@ every $X \in \mathcal{S}^{\tw I}$ has an exterior compactification.

\begin{DEF}\label{DEFSCOR4}
We define a (lax, oplax) symmetric 2-pre-multiderivator $\SSS^{\cor, \comp}$ (resp.\@ $\SSS^{\cor,\comp,0,\lax}$, $\SSS^{\cor,\comp,0,\oplax}$) with domain $\Catlf$. Let $I \in \Catlf$ be a diagram. The 2-multicategory $\SSS^{\cor, \comp}(I)$ (resp.\@ $\SSS^{\cor, \comp,0,\lax}(I)$, $\SSS^{\cor, \comp,0,\oplax}(I)$) has been defined in Definition~\ref{DEFSCOR3}.
To construct the functoriality in $I$ one proceeds as in Definition~\ref{DEFSCOR2}. 
To construct the 2-isomorphisms (resp.\@ 2-morphisms) for the pseudo- (resp.\@ lax, resp.\@ oplax) naturality of the morphism
\[ \alpha^* \Rightarrow \beta^* \]
and to construct the pseudo-functoriality constraints of 
\begin{eqnarray*} 
\Fun(I, J) &\rightarrow& \Fun^{\mathrm{strict}}(\SSS^{\cor, \comp}(J), \SSS^{\cor, \comp}(I)) \quad \text{resp.}  \\
 \Fun(I, J) &\rightarrow& \Fun^{\lax, \mathrm{strict}}(\SSS^{\cor,0,\comp,\lax}(J), \SSS^{\cor,0,\comp,\lax}(I)) \quad \text{resp.} \\
 \Fun(I, J) &\rightarrow& \Fun^{\oplax, \mathrm{strict}}(\SSS^{\cor,0,\comp,\oplax}(J), \SSS^{\cor,0,\comp,\oplax}(I)) 
\end{eqnarray*}  
one can use the equivalences (\ref{eqequi}). 
\end{DEF}

Finally, we put on record the important fact that (\ref{eqequi}) are equivalences of 2-multicategories:
\begin{PROP}\label{PROPEQUIVCOMP}
The forgetful morphisms
\[ \SSS^{\cor,\comp} \rightarrow  \SSS^{\cor} \quad (\text{resp.}\ \SSS^{\cor,\comp,0,\lax} \rightarrow  \SSS^{\cor,0,\lax}, \quad \text{resp.}\ \SSS^{\cor,\comp,0,\oplax} \rightarrow  \SSS^{\cor,0,\oplax})   \]
are equivalences of pre-2-multiderivators. 
\end{PROP}

Above the following Lemma was used:

\begin{LEMMA}\label{LEMMALOC}
Let $F: \mathcal{C} \rightarrow \mathcal{D}$ be a functor which is surjective on objects and morphisms. 
Denote by $\mathcal{W}$ the class of morphisms in $\mathcal{C}$ that are mapped to an identity in $\mathcal{D}$. 
Assume that 
\begin{enumerate}
\item[(i)] A solid diagram of the form 
\[ \xymatrix{  \ar@{.>}[r]^{w'} \ar@{.>}[d]_{f'} &  \ar[d]^f \\
  \ar[r]_{w} &  }\]
with $w \in \mathcal{W}$ can always be completed to a commutative diagram as indicated such that $w' \in \mathcal{W}$. 
\item[(ii)] A solid diagram of the form 
 \[ \xymatrix{ &  \ar[ld]_{w} \ar[rd]^f \\
   &  \ar@{.>}[u]|{w''} \ar@{.>}[d]|{w'''} \ar@{.>}[l] \ar@{.>}[r] &  \\
  &  \ar[lu]^{w'} \ar[ru]_{f'} } \]
 in which $w, w' \in \mathcal{W}$ and such that $F(f) = F(f')$ can always be completed to a commutative diagram as indicated with $w'', w''' \in \mathcal{W}$.
\end{enumerate}
Then the functor $F$ induces an equivalence $\mathcal{C}[\mathcal{W}^{-1}] \cong \mathcal{D}$. 
\end{LEMMA}

 \begin{proof}
 Assumption (i) implies that every morphism in $\mathcal{C}[\mathcal{W}^{-1}]$ can be represented by a roof
 \[ \xymatrix{ & \ar[ld]_{w} \ar[rd]^f \\
   & & } \]
  in which $w \in \mathcal{W}$.
  Now consider two parallel morphisms 
   \[ \xymatrix{ &  \ar[ld]_{w} \ar[rd]^f \\
   &&  \\
  &  \ar[lu]^{w'} \ar[ru]_{f'} } \]
in $\mathcal{C}[\mathcal{W}^{-1}]$ with $F(f) = F(f')$. Extending the diagram as in assumption (ii) we get
\[  f w^{-1} = f  (w'') (w'')^{-1} w^{-1} = f'  (w''') (w''')^{-1} (w')^{-1} = f' (w')^{-1}. \]
Therefore the induced functor $\mathcal{C}[\mathcal{W}^{-1}] \cong \mathcal{D}$ is faithful.  
 \end{proof}

\section{The input for the construction of derivator six-functor-formalisms}

\begin{PAR}\label{PARAXIOMS}

Let $\mathcal{S}$ be a category with compactifications, and $\SSS^{\op}$ the symmetric pre-multiderivator represented by $\mathcal{S}^{\op}$ with the symmetric multicategory structure \ref{PAROPMULTCAT}. 
We consider a (symmetric) fibered multiderivator $\DD \rightarrow \SSS^{\op}$ with domain $\Invlf$. 
This might be seen as a (symmetric) derivator four-functor-formalism encoding $f_*, f^*, \otimes, \mathcal{HOM}$ and their usual properties. 

More precisely, for a morphism $f$ in $\Fun(I, \mathcal{S}^{\op})$ we denote by $f^*$ a push-forward functor $(f^{\op})_\bullet$, by $f_*$ a pull-back functor $(f^{\op})^\bullet$, 
and for a diagram $S \in \Fun(I, \mathcal{S}^{\op})$ by $\otimes$ ($S$ being understood) a push-forward along the multimorphism $(\id_{S}, \id_{S})$, and by 
$\mathcal{HOM}_l$, resp.\@ $\mathcal{HOM}_r$, pull-back functors w.r.t.\@ the first, resp.\@ second slot along the multimorphism $(\id_{S}, \id_{S})$.
All pull-back and push-forward functors exist by (FDer0 left), resp.\@ (FDer0 right).

We consider the following axioms

\begin{itemize}
\item[(F1)] For each diagram $I \in \Invlf$ and point-wise embedding $\iota: S \hookrightarrow T$ in $\SSS(I^{\op})$, the functor $\iota^*$ (aka $(\iota^{\op})_\bullet$), as morphism of derivators, commutes with homotopy limits\footnote{by (Der2), for this statement one may assume $I=\cdot$.}, and has a left adjoint
\[ \iota_!: \DD(I)_{S} \rightarrow \DD(I)_{T}. \]
\item[(F2)] For each embedding $\iota: S \hookrightarrow T$ in $\mathcal{S}$ the corresponding functor
\[ \iota_!: \DD_S(\cdot) \rightarrow \DD_T(\cdot) \]
is fully faithful. 
\item[(F3)] For each proper morphism $f$ in $\mathcal{S}$ the functor $f_*$ (as morphism of derivators) commutes with homotopy colimits. 
\item[(F4)] For each proper morphism $f$ in $\mathcal{S}$ and any Cartesian square
\[ \xymatrix{
 \ar@{->>}[r]^F \ar[d]_G &  \ar[d]^g \\
 \ar@{->>}[r]_f & 
} \]
the natural exchange morphism (base change)
\[ G^* F_* \rightarrow f_* g^*   \]
is an isomorphism. 
\item[(F5)] For each embedding $\iota$ in $\mathcal{S}$ and for any Cartesian square
\[ \xymatrix{
 \ar[r]^{F} \ar@{^{(}->}[d]_{I} &  \ar@{^{(}->}[d]^\iota \\
 \ar[r]_{f} & 
} \]
 the natural exchange morphism (base change)
\[ \iota^* f_* \rightarrow F_* I^* \]
is an isomorphism.
\item[(F6)] For each proper morphism $f$  and  embedding $\iota$ forming a Cartesian square
\[ \xymatrix{
 \ar@{->>}[r]^{F} \ar@{^{(}->}[d]_{I} &  \ar@{^{(}->}[d]^\iota \\
 \ar@{->>}[r]_{f} & 
} \]
the exchange of the base change isomorphism from (F4) (equivalently from (F5)) 
\[ \iota_! F_* \rightarrow f_* I_! \]
is an isomorphism as well.

\item[(F4m)] For each proper morphism $f$ in $\mathcal{S}$ we have projection formulas, i.e.\@ the natural exchange morphisms
\[  (f_* -) \otimes - \rightarrow f_*(- \otimes f^*-) \quad  - \otimes (f_* -) \rightarrow f_*((f^*-) \otimes -)  \]
are isomorphisms\footnote{In case $\DD$ is symmetric, these two assertions are equivalent.}. 
\item[(F5m)] For each embedding $\iota$ in $\mathcal{S}$ we have ``coprojection formulas'', i.e.\@ the natural exchange morphisms
\[ \iota^* \mathcal{HOM}_l(-, -) \rightarrow  \mathcal{HOM}_l(\iota^*-, \iota^*-) \quad \iota^* \mathcal{HOM}_r(-, -) \rightarrow  \mathcal{HOM}_r(\iota^*-, \iota^*-)  \]
are isomorphisms\footnote{In case $\DD$ is symmetric, we have $\mathcal{HOM}_l=\mathcal{HOM}_r$.}.
\end{itemize}
\end{PAR}

\begin{BEM}Except for (F1) and (F3) these axioms only involve the underlying bifibration
\[ \DD(\cdot) \rightarrow \mathcal{S}^{\op} \]
and have thus nothing to do with the derivator enhancement.
If $\DD \rightarrow \SSS^{\op}$ is infinite and has stable fibers, then the commutation with arbitrary homotopy (co)limits can also be checked
on the underlying bifibration {\em together with} the triangulated structure on the fiber. Indeed a morphism of stable (infinite) derivators commutes with all homotopy (co)limits if and only if the functor on the underlying category is exact (i.e.\@ preserves distinguished triangles) and commutes with arbitrary (co)products. The infiniteness is needed to ensure that arbitrary homotopy (co)products are the same as (co)products in the underlying category. 
\end{BEM}
\begin{BEM}\label{BEMPROJFORMULAIOTA}
If $\iota^*$ has a left adjoint $\iota_!$ for any embedding $\iota$ in $\mathcal{S}$ (e.g.\@ if (F1) holds true) then (F5), resp.\@ (F5m), is equivalent to the condition that 
\[ I_! F^*  \to  f^* \iota_!   \qquad   \iota_! (- \otimes (\iota^* -)) \to (\iota_! -) \otimes - \qquad   \iota_! (( \iota^*-) \otimes -) \to -\otimes  (\iota_! -)   \]
are isomorphisms. 
\end{BEM}

\begin{PAR}Assume (F1) and (F2). Then
a morphism $\mathcal{E} \rightarrow \mathcal{F}$ in $\DD(\cdot)$ over an embedding is called {\bf strongly coCartesian}, if it is coCartesian, i.e. if it induces an isomorphism
\[ \iota^* \mathcal{E} \rightarrow \mathcal{F} \]
whose inverse
\[ \mathcal{F} \rightarrow \iota^* \mathcal{E} \]
induces an isomorphism
\[ \iota_! \mathcal{F} \overset{\sim}{\rightarrow}  \mathcal{E} \]

Let $\iota: U \hookrightarrow S$ be an embedding.
We say that an object $\mathcal{E}$ in $\DD(\cdot)_S$ has {\bf support in $U$} if it lies in the essential image of the fully-faithful functor $\iota_!$.
A coCartesian morphism 
\[ \mathcal{E} \rightarrow \mathcal{F} \]
over $\iota^{\op}$ is {\em strongly} coCartesian if and only if $\mathcal{E}$ has support in $U$. We use the notation $\cocart^*$ for strongly coCartesian. It will only be used over embeddings. 
\end{PAR}

\begin{LEMMA}\label{LEMMAF4}
Axioms (F4) and (F4m) are equivalent to the following statement:
For all Cartesian squares
\begin{equation}
 \xymatrix{
W  \ar[r]^-{G} \ar@{->>}[d]_{F} & Z_1, \dots, Z_n \ar@{->>}[d]^{f_1,\dots, f_n} \\ 
Y  \ar[r]_-{g}  & X_1, \dots, X_n  
}
\end{equation}
in which the $f_i$ and $F$ are proper,  for a commutative square
\[ \xymatrix{
\mathcal{H} \ar@{<-}[r]^-\delta \ar@{<-}[d]_\gamma & \mathcal{E}_1, \dots, \mathcal{E}_n  \ar@{<-}[d]^{\alpha_1, \dots, \alpha_n} \\
\mathcal{G} \ar@{<-}[r]_-\beta & \mathcal{F}_1, \dots, \mathcal{F}_n 
} \]
in $\DD(\cdot)$ above it, the following holds: If the $\alpha_i$ are Cartesian $(\mathcal{F}_i \cong f_{i,*}\mathcal{E}_i)$ and $\delta$ is coCartesian $(G^*(\mathcal{E}_1, \dots, \mathcal{E}_n) \cong \mathcal{H})$ then $\beta$ is coCartesian if and only if $\gamma$ is Cartesian or, in other words, the natural exchange 
\[ g^*( f_{1,*} -, \dots, f_{n,*} -)  \rightarrow F_* G^*(-, \dots, -)\]
is an isomorphism.
\end{LEMMA}
\begin{proof}
Exercise.
\end{proof}

\begin{LEMMA}\label{LEMMAF5}
Axioms (F5) and (F5m) are (in the presence of (F1--2)) equivalent to the following statement:  For all Cartesian squares
\begin{equation}
 \xymatrix{
W  \ar[r]^-{G} \ar@{^{(}->}[d]_{I} & Z_1, \dots, Z_n \ar@{^{(}->}[d]^{\iota_1. \dots, \iota_n} \\ 
Y  \ar[r]_-g  & X_1, \dots, X_n  
}
\end{equation}
in which the $\iota_i$, and $I$ are embeddings,  for a commutative square
\[ \xymatrix{
\mathcal{H} \ar@{<-}[r]^-\delta \ar@{<-}[d]_-\gamma & \mathcal{E}_1, \dots, \mathcal{E}_n \ar@{<-}[d]^{\alpha_1, \dots, \alpha_n} \\
\mathcal{G} \ar@{<-}[r]_-\beta & \mathcal{F}_1, \dots, \mathcal{F}_n
} \]
in $\DD(\cdot)$ above it, the following holds: If the $\alpha_i$ are strongly coCartesian $(\iota_i^*\mathcal{F}_i \cong \mathcal{E}_i$ inducing $\iota_{i,!}\mathcal{E}_i \cong \mathcal{F}_i)$ and $\delta$ is coCartesian $(G^*(\mathcal{E}_1, \dots, \mathcal{E}_n) \cong \mathcal{H})$ then $\beta$ is coCartesian if and only if $\gamma$ is strongly coCartesian or, in other words, the natural exchange 
\[ I_!  G^* (-, \dots, -)  \rightarrow g^* ( \iota_{1,!}-, \dots, \iota_{n,!}-) \] 
is an isomorphism. 
\end{LEMMA}
\begin{proof}
Exercise, cf.\@ also Remark~\ref{BEMPROJFORMULAIOTA}.
\end{proof}

\begin{PAR}
There are two possibilities of constructing the left adjoint $\iota_!$ required by (F1). 
One possibility is to use Brown representability for the dual:
\end{PAR}
\begin{PROP}
Assume that
$\DD \rightarrow \SSS^{\op}$ is infinite and has stable, {\em compactly generated} fibers, 
and $\iota^*$ for all embeddings $\iota: S \hookrightarrow T$ (as morphism of derivators) commutes  with homotopy limits as well.
Equivalently: $\iota^*$ (as functor on the underlying triangulated categories) is exact and commutes with infinite products. 
Then a left adjoint $\iota_!$ to $\iota^*$ exists for all morphisms $\iota$ as in (F1).
\end{PROP}
\begin{proof}
Cf.\@ \cite[Theorem 4.2.2]{Hor15}. Note that for a point-wise embedding $\iota: S \hookrightarrow T$ in $\Fun(I, \mathcal{S}^{\op})$ the associated morphism of derivators
(fibers) commutes with homotopy limits if this is point-wise the case.
\end{proof}

Another possibility by direct construction is available if $\DD \rightarrow \SSS^{\op}$ has been constructed from a bifibration of multi-model categories.  
Let $\mathcal{D} \rightarrow \mathcal{S}^{\op}$ be a bifibration of multi-model categories as in \cite[Definition 5.1.3]{Hor15}. 
In \cite{Hor17b} it was shown that the associated morphism of pre-multiderivators, i.e.\@
\[ \xymatrix{ \DD(I) := \Fun(I, \mathcal{D})[\mathcal{W}_I^{-1}] \ar[d] \\ \SSS^{\op}(I) := \Fun(I, \mathcal{S}^{\op}) }\]
 is a left and right fibered multiderivator {\em with domain $\Cat$}.  
\begin{PROP}\label{PROPIOTA}
Assume that for any embedding $\iota: S \hookrightarrow T$ the functor $\iota^*: \mathcal{D}_T \rightarrow \mathcal{D}_S$ has a left adjoint $\iota_!$ which is left Quillen as well, then $\DD \rightarrow \SSS^{\op}$ satisfies (F1). If $\iota_!$ is fully-faithful then also (F2) holds true. 
\end{PROP}
\begin{proof}We first let $I \in \Dir$ be a directed diagram.
Consider a point-wise embedding $\iota: S \hookrightarrow T$ in $\SSS(I^{\op})$. In \cite[5.1.18]{Hor15} model category structures have been constructed on the categories
\[ \Fun(I, \mathcal{D})_{S^{\op}} \quad \Fun(I, \mathcal{D})_{T^{\op}} \]
in which fibrations and weak equivalences are the point-wise ones, 
turning $\iota^*$ (aka $(\iota^{\op})_\bullet$) into a left Quillen functor. We show that $\iota^*$ has a left adjoint $\iota_!$ which preserves cofibrant objects and weak equivalences between cofibrant objects.  
$(\iota_! \mathcal{E})(i)$ is defined by induction on the degree $n$ of $i$ by the coCartesian square:
\[ \xymatrix{
\iota_! L_i \mathcal{E} \ar[r] \ar[d] & \iota_! (\mathcal{E}(i)) \ar[d] \\
 L_i \iota_! \mathcal{E} \ar[r] &  (\iota_! \mathcal{E})(i) 
} \]
where $L_i$ is the latching object functor (cf.\@ \cite[p.\@ 74]{Hor15}). $L_i$ respects cofibrations and trivial cofibrations because it is a composition of three functors, two of
which are left Quillen and one respects cofibrations and trivial cofibrations by \cite[Lemma 5.1.24]{Hor15}.

For a cofibrant object $\mathcal{E}$ the top horizontal morphism is a cofibration 
and therefore also the bottom horizontal morphism is a cofibration. 
By induction $\iota_! \mathcal{E}$ is cofibrant when restricted to objects of degree $< n$ and thus also $L_i \iota_! \mathcal{E}$ is cofibrant and hence the so extended $\iota_! \mathcal{E}$ is.
Summarizing, all entries in the above square are cofibrant and the top horizontal morphism is a cofibration. In particular, the diagram is also a homotopy push-out. 
For a weak equivalence $\mathcal{E} \rightarrow \mathcal{F}$ between cofibrant objects, we get a morphism of diagrams of shape $\lefthalfcap$ 
which consists point-wise of weak equivalences: For the lower left entry use induction and the fact that $\iota_! \mathcal{E}$ is cofibrant again, for the upper  entries by the assumption and Ken Browns Lemma $\iota_!$ maps weak equivalences between cofibrant objects to weak equivalences.
Therefore also the induced morphism $(\iota_! \mathcal{E})(i) \rightarrow (\iota_! \mathcal{F})(i)$ is a weak equivalence. 
We obtain functors $\iota_!$ and $\iota^*$
\[ \xymatrix{  \Fun(I,\mathcal{D})_{S^{\op}}^{\Cof} \ar@/^5pt/[rr]^{\iota_!} && \ar@/^5pt/[ll]^{\iota^*}  \Fun(I,\mathcal{D})_{T^{\op}}^{\Cof} } \]
which both preserve (point-wise) weak equivalences. By construction they are adjoint.
They thus induce morphisms between the respective localizations $\DD(I)_{S^{\op}}$ and $\DD(I)_{T^{\op}}$ which are adjoint again. 

Lastly, over a point $I= \cdot$, 
if the original $\iota_!$ is fully-faithful, then the unit of the adjunction is an isomorphism and thus also still when passing to the localizations, which is equivalent to the induced $\iota_!$ on the localization
being fully-faithful.

It remains to be shown that $\iota^*$ commutes with homotopy limits or, equivalently, that the functor $\iota_!$
constructed above is computed point-wise {\em on constant diagrams}. By definition of $\iota_!$ this is the case if  the morphism 
\[\iota_! L_i \mathcal{E} \rightarrow L_i \iota_! \mathcal{E} \]
is a weak equivalence. However, by induction, $\iota_!$ is computed point-wise when restricted to objects of degree $<n$ and hence
the statement follows from the fact that $\iota_!$, being a left adjoint, commutes with colimits. 

If $I \in \Cat$ is a general diagram and $\iota: S \rightarrow T$ be a point-wise embedding in $\SSS(I^{\op})$, by \cite[Proposition 3.7]{Hor17b} there is a diagram $N(I) \in \Dir$ with functor $\pi:  N(I) \rightarrow I$ such that $\pi^*$ induces an equivalence
\[ \DD(I)_{S^{\op}} \cong \DD(N(I))_{\pi^*S^{\op}}^{\cocart}  \]
for any $S \in \SSS(I^{\op})$. 
Therefore a left adjoint for $\iota^*$ is given by $\pi_!^{(T^{\op})} (\pi^*\iota)_! \pi^*$ where $(\pi^*\iota)_!$ is the adjoint constructed in the first part. 
\end{proof}

Later we will need more information about the functors $\iota_!$:

\begin{DEF}\label{DEFCOMPUTEDPOINTWISE}
Let $\mathcal{S}$ be a category and let $\SSS^{\op}$ be the pre-derivator represented by $\mathcal{S}^{\op}$.
Let $\DD \rightarrow \SSS^{\op}$ be a fibered derivator with domain $\Dia$ such that for all $I \in \Dia$ the functors
\[ \iota^* := (\iota^{\op})_\bullet:  \DD(I)_{T^{\op}} \rightarrow \DD(I)_{S^{\op}} \] 
have left adjoints $\iota_!$ for all morphisms $\iota: S \rightarrow T$ in $\SSS(I^{\op})$, which are point-wise is some subclass $\mathcal{S}_2$ of morphisms of $\mathcal{S}$.
We will say that $\iota_!$ {\bf commutes with $\alpha^*$} on $(J, S)$ for a functor $\alpha: I \rightarrow J$  if for all objects $\mathcal{E}$ in $\DD(I)_S$ the natural exchange morphism
\[ (\alpha^*\iota)_! \alpha^* \mathcal{E} \rightarrow  \alpha^* \iota_!  \mathcal{E}  \]
is an isomorphism. We will say that $\iota_!$ is {\bf computed point-wise} on $(J, S)$, if it commutes with $j^*$ for all $j \in J$. 
\end{DEF}

Note that it does {\em not} automatically follow from (F1), i.e.\@ the commutation of $\iota^*$ with homotopy colimits, that $\iota_!$ is computed point-wise. 
This is true only over constant diagrams in $\SSS(I)$ --- then it is a well-known and quite trivial statement about usual derivators. 

\begin{LEMMA}\label{LEMMAPOINTWISEEXBYZERO}
Let $\mathcal{S}$ be a category with compactifications, let $\SSS$ be the pre-derivator represented by $\mathcal{S}$, and 
let $\DD \rightarrow \SSS^{\op}$ be a fibered derivator on some diagram category $\Dia$ satisfying axioms (F1--F6).
\begin{enumerate}
\item 
Let $\alpha: I \rightarrow J$ be an opfibration in $\Dia$ and $\iota: S \rightarrow T$ a morphism in $\SSS(I^{\op})$.
If for all coCartesian morphisms $\mu: i \rightarrow i'$ in $I$ the square
\[ \xymatrix{
S_i \ar@{<-}[r]^{S(\mu)} \ar[d]_{\iota_i} & S_{i'} \ar[d]^{\iota_{i'}} \\
T_i \ar@{<-}[r]_{T(\mu)} & T_{i'}
} \]
is Cartesian, then $\iota_!$ commutes with $e_j^*$ for the inclusions $e_j: I_j \hookrightarrow I$ of the fibers. 
\item Consider $I \times J$ in $\Dia$ and a morphism $\iota: S \rightarrow T$ in $\SSS(I^{\op})$. Then $(\pr_1^*\iota)_!$ commutes with the inclusion
$I \times j \hookrightarrow I \times J$ for any $j \in J$ on $(I \times J, \pr_1^*S)$. In particular $\iota_!$ is computed point-wise over constant diagrams in $\SSS(J^{\op})$. 
\item Assume $\DD$ is infinite, or that $I$ is a diagram in $\Dia$ with finite $\Hom$-sets. Let $\iota: S \rightarrow T$ be a morphism in $\SSS(I^{\op})$. Then the functor
\[ \iota_!: \DD(I)_{S^{\op}} \rightarrow \DD(I)_{T^{\op}} \]
is computed point-wise on an object $\mathcal{E}$, if for any pair of morphisms $\alpha: i \rightarrow j$ and $\mu: j \rightarrow k$ in $I$ we have that
\begin{equation} \label{eqcommpointwise2}
 \iota_{k,!} S(\mu)^* S(\alpha)^* i^*\mathcal{E} \rightarrow  T(\mu)^* \iota_{j,!} S(\alpha)^* i^*\mathcal{E}   \end{equation}
is an isomorphism.
 
Hence $\iota_!$ is computed point-wise on an (absolutely) {\em coCartesian} object $\mathcal{E}$, if for any $\mu: j \rightarrow k$ in $I$ we have that
\begin{equation} \label{eqcommpointwise1}
 \iota_{k,!}  S(\mu)^* j^*\mathcal{E} \rightarrow  T(\mu)^* \iota_{j,!} j^*\mathcal{E}   \end{equation}
is an isomorphism. Note that (\ref{eqcommpointwise1}) implies (\ref{eqcommpointwise2}) for a coCartesian object. 
\end{enumerate}
\end{LEMMA}
\begin{proof}
1. The statement is equivalent to the natural exchange morphism
\[  \iota^* e_{j,*}  \mathcal{E}   \rightarrow e_{j, *} (e_j^*\iota)^* \mathcal{E}   \]
being an isomorphism. This can be checked point-wise at an object $i \in I$.
Consider the homotopy exact square:
\[ \xymatrix{
\alpha(i) \times_{/J} j \ar@{}[rd]|{\Nearrow_\mu} \ar[r]^-{\rho} \ar[d] & I_j \ar[d]^{e_j} \\
i \ar[r] & I
} \]
where the $\rho$ maps a morphism $\nu: \alpha(i) \rightarrow j$ to $\nu_\bullet(i)$ and $\mu(\nu)$ is given by the coCartesian morphism
$\widetilde{\nu}: i \rightarrow \nu_\bullet(i)$.
It shows that we have
\begin{eqnarray*}
 (\iota_i)^* i^* e_{j,*}  \mathcal{E} &\cong& (\iota_i)^* \holim_{\alpha(i) \times_{/J} j} S(\mu)_* \rho^* \mathcal{E}    \\
& \cong & \holim_{\alpha(i) \times_{/J} j}  S(\mu)_* \rho^* (e_j^*\iota)^* \mathcal{E}    \\
& \cong& i^* e_{j,*} (e_j^*\iota)^* \mathcal{E}  
\end{eqnarray*}
because $\iota^*$ commutes with homotopy limits by (F1), and with $S(\mu)_*$ by (F5) and the assumption. Note that the latter can be checked point-wise. 

2. A special case of 1.\@

3. Look at the following diagram
\[ \xymatrix{
 \twc I  \ar[r]^{\pi_1}  \ar[d]_\rho & I \ar@{=}[dd] \\
 \tw I \times I \ar[d]_{\pi_3} \ar@{}[rd]|{\Swarrow^\mu} &  \\
I \ar@{=}[r] & I
} \]
in which $\rho = (\pi_{12}, \pi_3)$ forgets the second morphism. The outer square is homotopy exact, hence 
\[  \mathcal{E} \cong \pi_{3,!} \rho_! S(\mu)^* \pi_1^* \mathcal{E}   \]
for any $\mathcal{E} \in \DD(I)_S$. Furthermore $\rho$ and $\pi_3$ are opfibrations, and $\pi_1$ is a fibration. $\rho$ has disrete fibers with fiber over $(\alpha: i \rightarrow j, k)$
equal to $\Hom(j, k)$. 

We will later show that $(\pi_3^*\iota)_!$ is computed point-wise on $(\tw I \times I, \pi_3^*S)$ on objects of the form
\[  \mathcal{F}:=  \rho_! S(\mu)^* \pi_1^* \mathcal{E}   \]
for any object $\mathcal{E}$ satisfying (\ref{eqcommpointwise2}).
We claim that the statement follows from this. 

First, by 2.\@, $(\pi_3^*\iota)_!$ commutes with $e_\alpha^*$ for the inclusions $e_\alpha: \alpha \times I \hookrightarrow \tw I \times I$ on $\pi_3^*S$,
and $(\iota_k)_!$ commutes with $(\alpha)^*$ for $(\alpha): \alpha \hookrightarrow \tw I$ on any constant object $S_k$ for any $k \in I$. 
Denote $e_k: \tw I \times k \rightarrow \tw I \times I$ the inclusion. 
Consider the following exchange morphisms 
\[  (\iota_k)_! (\alpha)^* e_k^* \mathcal{F} \rightarrow   (\alpha)^* (\iota_k)_! e_k^* \mathcal{F}  \rightarrow  (\alpha)^* e_k^*  (\pi_3^*\iota)_! \mathcal{F}.  \]
By assumption, the composition is an isomorphism because $\iota_!$ is computed point-wise on $\mathcal{F}$. Also the
left morphism is an isomorphism because $\iota_{k,!}$ commutes with $(\alpha)^*$. 
Hence by (Der2) $(\pi_3^*\iota)_!$ commutes also with $e_k^*$. 

Then 
\begin{eqnarray*}
 \iota_{k,!} k^* \pi_{3,!} \mathcal{F}   &\cong& \iota_{k,!}  \hocolim_{\tw I} e_k^* \mathcal{F}   \\
 &\cong&  \hocolim_{\tw I}  \iota_{k,!}  e_k^*  \mathcal{F} \\
 &\cong&    \hocolim_{\tw I} e_k^*  (\pi_3^* \iota)_!   \mathcal{F}  \\
 &\cong&  k^* \pi_{3,!}  (\pi_3^* \iota)_!  \mathcal{F}  \\
 &\cong&  k^*  \iota_! \pi_{3,!}  \mathcal{F}  
\end{eqnarray*}
using that $\iota_{k,!}$ commutes with arbitrary homotopy left Kan extensions. 
Hence $\iota_!$ is also computed point-wise on $(I, S)$ for $\mathcal{E} \cong \pi_{3,!} \mathcal{F}$.

Hence we are left to show that $(\pi_3^*\iota)_!$ is computed point-wise on $\mathcal{F}$ on $\tw I \times I$. 
For a morphism $\alpha: i \rightarrow j$ in $I$ we have using (Der1)
\begin{eqnarray*}
 \iota_{k,!} (\alpha, k)^* \mathcal{F}\ =\ \iota_{k,!} (\alpha, k)^* \rho_! S(\mu)^* \pi_1^* \mathcal{E} &\cong&  \iota_{k,!} k^* j_! S(\alpha)^* i^*\mathcal{E} \\
 &\cong&  \iota_{k,!} \bigoplus_{\beta \in \Hom(j, k)} S(\beta)^* S(\alpha)^* i^*\mathcal{E} \\
 &\cong& \bigoplus_{\beta \in \Hom(j, k)} T(\beta)^* \iota_{j,!} S(\alpha)^* i^* \mathcal{E} 
\end{eqnarray*}
because of the assumption (\ref{eqcommpointwise2}) and commutation of $(-)_!$ with homotopy colimits.
Furthermore, because of (Der1), $(-)_!$ is clearly computed point-wise on the discrete diagram $\Hom(j, k)$ over any object in $\SSS(\Hom(j, k))$. If $\Hom(j, k)$ is infinite
we need (Der1${}^\infty$), that is, $\DD$ has to be infinite.
Then this is isomorphic to
\begin{eqnarray*}
 &\cong& k^* j_!  \iota_{j,!} S(\alpha)^* i^*\mathcal{E} \\
 &\cong& k^* \iota_!  j_!   S(\alpha)^* i^*\mathcal{E} 
\end{eqnarray*}
using that $j_!$ commutes with $\iota_!$ because $j^*$ commutes with $\iota^*$. And finally to
\begin{eqnarray*}
 &\cong& k^* \iota_! e_{\alpha}^* \rho_! S(\mu)^* \pi_1^* \mathcal{E} \\
 &\cong& (\alpha,k)^* (\pi_3^* \iota)_! \rho_! S(\mu)^* \pi_1^* \mathcal{E}\ = \ (\alpha, k)^* (\pi_3^*\iota)_!  \mathcal{F}
\end{eqnarray*}
using that $(\pi_3^*\iota)_!$ commutes with $e_\alpha^*$. A tedious check shows that this composition of isomorphisms is the exchange morphism associated with the commutation of $(\pi_3^*\iota)^*$ and $(\alpha, k)^*$. 
\end{proof}

\section{Preliminaries for the construction of the derivator six-functor-formalism (non-multi-case)}\label{SECTPRELIM}

We will neglect the multi-aspect in this section and work with a fibered derivator (not multiderivator) $\DD \rightarrow \SSS^{\op}$ satisfying the  axioms (F1)--(F6).
In the next section the results are generalized to the multi-case. This is straightforward, but a bit more technical, hence it has been moved to the next section for the convenience of the reader. 

\begin{PAR}\label{INTERIORCOMPCOR}
Let $\mathcal{S}$ be a category with compactifications as in \ref{DEFCATCOMP}.
Let $I$ be a locally finite diagram and  $X: \tw I \rightarrow \mathcal{S}$ admissible in the sense of \ref{PARALTSCOR} (i.e.\@ equivalently: a diagram of correspondences $I \rightarrow \mathcal{S}^{\cor}$). 
For any exterior compactification $X \hookrightarrow \overline{X}$ we get an induced interior compactification ${}^{\downarrow\downarrow}(\tw I) \rightarrow \mathcal{S}$ (cf.\@ Proposition~\ref{PROPCOMPDIA}). We are rather interested in its pullback along the following functor
\[ \tww I \rightarrow {}^{\downarrow \downarrow} (\tw I)  \]
mapping $i \rightarrow j \rightarrow k$ to the diagram
\[ \xymatrix{
i \ar[r] \ar[d] & k \ar@{<-}[d]^{\id_k} \\
j \ar[r] & k 
} \]
The reason is that we need a compactification only for the morphism $f$ going to the right in a correspondence (\ref{eqcomp1}). The above functor forgets the interior compactification on the other morphism $g$ going to the left. 
We will therefore always denote by $\widetilde{X}$ the pull-back of the induced interior compactification to $\tww I$ and will call it an
interior compactification of $X$. 
It has the property that a type 1 morphism (\ref{PARTW}) is mapped to a proper morphism, and a type 2 morphism is mapped to a dense embedding. 

We will also need this w.r.t.\@ a morphism $\Fun(\Delta_n, \SSS^{\cor, 0, \lax}(I))$, resp.\@ $\Fun(\Delta_n, \SSS^{\cor, 0, \oplax}(I))$. We get, in each case, a diagram $X: \tw (\Delta_n \times I) \rightarrow \mathcal{S}$ which is, however, only weakly admissible in the sense of \ref{PARALTSCOR}. For  this diagram we can construct in the same way
exterior and interior compactifications.  
\end{PAR}

\begin{DEF}
A morphism of squares in $\mathcal{S}$
\begin{equation}\nonumber
 \xymatrix{
 & W \ar@{}[ddrr]|(.7){} \ar[rr]^{} \ar[dl]_{} \ar@{^{(}->}[dd]^(.7){\iota_W} && Z \ar@{^{(}->}[dd]^{\iota_Z} \ar[ld]_{} \\ 
Y \ar@{}[ddrr]|(.7){} \ar[rr]^(.4){} \ar@{^{(}->}[dd]^{\iota_Y} && X \ar@{^{(}->}[dd]_(.7){\iota_X} & \\
& \overline{W} \ar[dl]_{} \ar[rr]^{} && \overline{Z} \ar[dl]^{} \\
\overline{Y} \ar[rr]^{} && \overline{X} }
\end{equation}
such that all embeddings are dense,  the top, the front and the back are Cartesian is called a {\bf weak compactification} of the top Cartesian square. 
The bottom square does not need to be Cartesian, and neither the objects nor the morphisms in the bottom square are assumed to be proper.
\end{DEF}
This rather ad hoc definition will only be used in this section. Note that the orientation of the top square matters. To draw it in the plane we will always rotate the cube by $90^{\circ}$ in such a way
that it becomes the front face.  

\begin{LEMMA}\label{LEMMACARTDIA2}
Let $X: \tw I \rightarrow \mathcal{S}$ be weakly admissible and let 
$\widetilde{X}$ be any interior compactification. 
Any square of the form
\[ \xymatrix{
& \widetilde{X}(i \rightarrow j \rightarrow k)  \ar@{->>}[r] \ar@{^{(}->}[d] &  \widetilde{X}(i' \rightarrow i \rightarrow k)  \ar@{^{(}->}[d] \\
& \widetilde{X}(i \rightarrow j' \rightarrow k)  \ar@{->>}[r] & \widetilde{X}(i' \rightarrow j' \rightarrow k) 
} \]
is Cartesian. 
\end{LEMMA}

\begin{proof}
Cf.\@ Lemma~\ref{LEMMACART1}.
\end{proof}

\begin{LEMMA}\label{LEMMACARTDIA}
Let $X: \tw I \rightarrow \mathcal{S}$ be weakly admissible and let 
$\widetilde{X}$ be any interior compactification. 
Any square of the form
\[ \xymatrix{
& \widetilde{X}(i = i \rightarrow k)  \ar[r] \ar@{^{(}->}[d] &  \widetilde{X}(i  = i \rightarrow k')  \ar@{^{(}->}[d] \\
& \widetilde{X}(i \rightarrow j \rightarrow k)  \ar[r] & \widetilde{X}(i \rightarrow j \rightarrow k') 
} \]
is Cartesian. 
\end{LEMMA}
\begin{proof}
We may extend the diagram as follows
\[ \xymatrix{
& \widetilde{X}(i = i \rightarrow k)  \ar[r] \ar@{^{(}->}[d] &  \widetilde{X}(i  = i \rightarrow k')  \ar@{^{(}->}[d] \\
& \widetilde{X}(i \rightarrow j \rightarrow k)  \ar[r]  \ar@{->>}[d]& \widetilde{X}(i \rightarrow j \rightarrow k')  \ar@{->>}[d] \\
& \widetilde{X}(j = j \rightarrow k)  \ar[r] &  \widetilde{X}(j  = j \rightarrow k')  \\
} \]
in which the outer square is weakly Cartesian, because $X$ is weakly admissible. The statement follows therefore from Lemma~\ref{LEMMACART2}
\end{proof}

\begin{LEMMA}\label{LEMMAWEAKCOMP}
Let $X: \tw I \rightarrow \mathcal{S}$ be weakly admissible and let 
$\widetilde{X}$ be any interior compactification. 
The  cube
\begin{equation}\nonumber
 \xymatrix{
 & \widetilde{X}(i = i \rightarrow k) \ar@{}[ddrr]|(.7){} \ar[rr]^{} \ar@{=}[dl]_{} \ar@{^{(}->}[dd]^(.7){} && \widetilde{X}(i = i \rightarrow k') \ar@{^{(}->}[dd]^{} \ar@{=}[ld]_{} \\ 
\widetilde{X}(i = i \rightarrow k) \ar@{}[ddrr]|(.7){} \ar[rr]^(.4){} \ar@{^{(}->}[dd]^{} && \widetilde{X}(i = i \rightarrow k')  \ar@{^{(}->}[dd]_(.7){} & \\
& \widetilde{X}(i \rightarrow j \rightarrow k)  \ar@{^{(}->}[dl]_{} \ar[rr]^(.3){\overline{G}} && \widetilde{X}(i \rightarrow j \rightarrow k')  \ar@{^{(}->}[dl]^{} \\
\widetilde{X}(i \rightarrow j' \rightarrow k)  \ar[rr]_{\overline{g}} && \widetilde{X}(i \rightarrow j' \rightarrow k')  }
\end{equation}
is a weak compactification of the top (trivially Cartesian) square. If $X$ is admissible, the cube
\begin{equation}\nonumber
 \xymatrix{
 & \widetilde{X}(i = i \rightarrow k) \ar@{}[ddrr]|(.7){} \ar[rr]^{} \ar[dl]_{} \ar@{^{(}->}[dd]^(.7){} && \widetilde{X}(i = i \rightarrow k') \ar@{^{(}->}[dd]^{} \ar[ld]_{} \\ 
\widetilde{X}(i' = i' \rightarrow k) \ar@{}[ddrr]|(.7){} \ar[rr]^(.4){} \ar@{^{(}->}[dd]^{} && \widetilde{X}(i' = i' \rightarrow k')  \ar@{^{(}->}[dd]_(.7){} & \\
& \widetilde{X}(i \rightarrow j \rightarrow k)  \ar@{->>}[dl]_{} \ar[rr]^(.3){\overline{G}} && \widetilde{X}(i \rightarrow j \rightarrow k')  \ar@{->>}[dl]^{} \\
\widetilde{X}(i' \rightarrow j \rightarrow k)  \ar[rr]_{\overline{g}} && \widetilde{X}(i' \rightarrow j \rightarrow k')  }
\end{equation}
is a weak compactification of the top Cartesian square. 
\end{LEMMA}
\begin{proof}
We need to show that in each case the front and back squares are Cartesian. This is a consequence of Lemma~\ref{LEMMACARTDIA}.
\end{proof}

\begin{LEMMA}\label{LEMMACOCARTSQUARES1}
Consider a square in  $\mathcal{S}$
\[ \xymatrix{
W \ar@{->>}[r]^{\overline{F}} \ar@{^{(}->}[d]_I & X \ar@{^{(}->}[d]^{\iota} \\
Z \ar@{->>}[r]_{\overline{f}} & Y
} \]
in which $I$ is dense 
and a square
\[ \xymatrix{
\mathcal{H} \ar@{<-}[r]^\delta \ar@{<-}[d]_\gamma & \mathcal{E} \ar@{<-}[d]^\alpha \\
\mathcal{G} \ar@{<-}[r]_\beta & \mathcal{F}
} \]
above it. If $\gamma$ is strongly coCartesian $(I^*\mathcal{G} \cong \mathcal{H}$ inducing $I_!\mathcal{H} \cong \mathcal{G})$ and $\delta$ is Cartesian $(\mathcal{E} \cong \overline{F}_*\mathcal{H})$ then $\beta$ is Cartesian if and only if $\alpha$ is strongly coCartesian, or in other words the natural exchange
\[ \iota_! \overline{F}_* \rightarrow \overline{f}_* I_!    \]
is an isomorphism. 
\end{LEMMA}

\begin{proof}
By Lemma~\ref{LEMMACART1} the square is actually Cartesian and hence by (F6) the statement holds. 
\end{proof}

\begin{BEM}
A posterori the conclusion will hold regardless of $I$ being dense. 
\end{BEM}

\begin{LEMMA}\label{LEMMACOCARTSQUARES2}
Consider a weak compactification of a Cartesian square in  $\mathcal{S}$
\begin{equation}\nonumber
 \xymatrix{
 & W \ar@{}[ddrr]|(.7){} \ar[rr]^{} \ar[dl]_{} \ar@{^{(}->}[dd]^(.7){} && Z \ar@{^{(}->}[dd]^{} \ar[ld]_{} \\ 
Y \ar@{}[ddrr]|(.7){} \ar[rr]^(.4){} \ar@{^{(}->}[dd]^{} && X \ar@{^{(}->}[dd]_(.7){} & \\
& \overline{W} \ar@{->>}[dl]_{\overline{F}} \ar[rr]^(.3){\overline{G}} && \overline{Z} \ar@{->>}[dl]^{\overline{f}} \\
\overline{Y} \ar[rr]_{\overline{g}} && \overline{X} }
\end{equation}
in which $\overline{f}$ and $\overline{F}$ are proper. 
Then for a square
\[ \xymatrix{
\mathcal{H} \ar@{<-}[r]^\delta \ar@{<-}[d]_\gamma & \mathcal{E} \ar@{<-}[d]^\alpha \\
\mathcal{G} \ar@{<-}[r]_\beta & \mathcal{F}
} \]
in $\DD(\cdot)$ above the bottom square the following holds: If \underline{$\mathcal{E}$ has support in $Z$} and $\alpha$ is Cartesian $(\mathcal{F} \cong \overline{f}_*\mathcal{E})$ and $\delta$ is coCartesian $(\overline{G}^*\mathcal{E} \cong \mathcal{H})$ then $\beta$ is coCartesian if and only if $\gamma$ is Cartesian or, in other words, the natural exchange 
\[ \overline{g}^* \overline{f}_*  \rightarrow \overline{F}_* \overline{G}^*\]
is an isomorphism on objects with support in $Z$.
\end{LEMMA}
\begin{proof}
We look at the following diagram
\[ \xymatrix{
W \ar@{=}[r] \ar@{^{(}->}[d]_{\iota_W} \ar@{}[rd]|{\numcirc{1}} & W \ar[r]^{G} \ar@{^{(}->}[d]^\iota \ar@{}[rd]|{\numcirc{2}} & Z \ar@{^{(}->}[d]^{\iota_Z}  \\
\overline{W} \ar@{->>}[r]_{\overline{f}''}  \ar@{->>}[d]_{\overline{F}} 
& \Box \ar@{}[rd]|{\numcirc{3}}  \ar[r]^{G'}  \ar@{->>}[d]^{\overline{f}'} & \overline{Z} \ar@{->>}[d]^{\overline{f}}  \\
\overline{Y} \ar@{=}[r]   & \overline{Y}  \ar[r]_{\overline{g}}   & \overline{X} \\
} \]
in which squares $\numcirc{2}$ and $\numcirc{3}$ are Cartesian. 
The middle left horizontal morphism is proper because of (S2). Because of the support condition, we have to show that the natural exchange
\[ \overline{g}^*\overline{f}_* \iota_{Z,!} \rightarrow \overline{F}_* \overline{G}^* \iota_{Z,!} \]
is an isomorphism. Elementary properties of exchange morphisms imply that the morphism is the composition of the following (exchange) morphisms which are all isomorphisms because of the indicated reason:
\[
\begin{array}{rcll}
\overline{g}^* \overline{f}_* \iota_{Z,!} &\iso & \overline{f}'_*  (G')^* \iota_{Z,!} & \text{because $\numcirc{3}$ is Cartesian and (F4)} \\
&\isor & \overline{f}'_*   \iota_{!} G^*    & \text{because $\numcirc{2}$ is Cartesian and (F5)} \\
&\iso & \overline{f}'_*  \overline{f}''_* \iota_{W,!} G^* & \text{applying Lemma~\ref{LEMMACOCARTSQUARES1} for $\numcirc{1}$} \\
&\iso & \overline{F}_*  \iota_{W,!} G^* \\
&\iso & \overline{F}_* \overline{G}^*  \iota_{Z,!} & \text{because the composite of $\numcirc{1}$ and $\numcirc{2}$ is Cartesian and (F5).}
\end{array}
\]
\end{proof}

\begin{LEMMA}\label{LEMMACOCARTSQUARES3}
Consider a weak compactification of a Cartesian square in $\mathcal{S}$
\begin{equation}\nonumber
 \xymatrix{
 & W \ar@{}[ddrr]|(.7){} \ar[rr]^{} \ar@{^{(}->}[dl]_{I} \ar@{^{(}->}[dd]^(.7){} && Z \ar@{^{(}->}[dd]^{} \ar@{^{(}->}[ld]_{\iota} \\ 
Y \ar@{}[ddrr]|(.7){} \ar[rr]^(.4){} \ar@{^{(}->}[dd]^{} && X \ar@{^{(}->}[dd]_(.7){} & \\
& \overline{W} \ar@{^{(}->}[dl]_{\overline{I}} \ar[rr]^(.3){\overline{G}} && \overline{Z} \ar@{^{(}->}[dl]^{\overline{\iota}} \\
\overline{Y} \ar[rr]_{\overline{g}} && \overline{X} }
\end{equation}
in which $\overline{\iota}$, $\iota$, $I$ and $\overline{I}$ are embeddings. 
Then for a square
\[ \xymatrix{
\mathcal{H} \ar@{<-}[r]^\delta \ar@{<-}[d]_\gamma & \mathcal{E} \ar@{<-}[d]^\alpha \\
\mathcal{G} \ar@{<-}[r]_\beta & \mathcal{F}
} \]
in $\DD(\cdot)$ above the bottom square the following holds: If \underline{$\mathcal{E}$ has support in $Z$} and $\alpha$ is strongly coCartesian $(\overline{\iota}^*\mathcal{F} \cong \mathcal{E}$ inducing $\overline{\iota}_!\mathcal{E} \cong \mathcal{F})$ and $\delta$ is coCartesian $(\overline{G}^*\mathcal{E} \cong \mathcal{H})$ then $\beta$ is coCartesian if and only if $\gamma$ is strongly coCartesian or, in other words, the natural exchange 
\[ \overline{I}_!  \overline{G}^*  \rightarrow \overline{g}^* \overline{\iota}_!  \] 
is an isomorphism on objects with support in $Z$. 
\end{LEMMA}

\begin{proof}
Because of the support condition it suffices to see that the natural exchange
\[ \overline{I}_!  \overline{G}^* \iota_{Z,!}   \rightarrow \overline{g}^* \overline{\iota}_! \iota_{Z,!}  \] 
is an isomorphism. 
Elementary properties of exchange morphisms imply that the morphism is the composition of the following (exchange) morphisms which are all isomorphisms because of pseudofunctoriality and because
of the indicated reason:

\[ \begin{array}{rcll}
   \overline{I}_!  \overline{G}^*  \iota_{Z,!} 
   &\isor&     \overline{I}_!  \iota_{W,!}  G^* & \text{because the back square in Cartesian and (F5)} \\
&\iso&     \iota_{Y,!} I_! G^* & \text{} \\
&\iso&   \iota_{Y,!} g^* \iota_! & \text{because the top square is Cartesian and (F5)} \\
&\iso&   \overline{g}^*  \iota_{X,!} \iota_! & \text{because the front square is Cartesian and (F5)}  \\
&\iso&   \overline{g}^* \overline{\iota}_! \iota_{Z,!}  & \text{}
\end{array} \]
\end{proof}

We summarize the discussion in the following 
\begin{PROP}\label{PROPPROPERTIESCORCOMP}
Let $X: \tw I \rightarrow \mathcal{S}$ be weakly admissible and let 
$\widetilde{X}: \tww I \rightarrow \mathcal{S}$ be any interior compactification of it. 
Consider a diagram
\[ \xymatrix{
w \ar[r]^-d \ar[d]_c & z \ar[d]^a \\
y \ar[r]_-b & x
} \]
in $\tww I$.
Let 
\[ \xymatrix{
\mathcal{H} \ar@{<-}[r]^\delta \ar@{<-}[d]_\gamma & \mathcal{E} \ar@{<-}[d]^\alpha \\
\mathcal{G} \ar@{<-}[r]_\beta & \mathcal{F}
} \]
be a diagram in $\DD(\cdot)$ above {\em $\widetilde{X}^{\op}$ applied to the top square}. Then the following holds:

\begin{enumerate}
\item If $X: \tw I \rightarrow \mathcal{S}$ is admissible, let $a$ and $c$ be of type 1 and $b$ and $d$ of type 3.

If \underline{$\mathcal{E}$ has support in $X(\pi_{13}(z))$} and $\alpha$ is Cartesian $(\mathcal{F} \cong \widetilde{X}(a)_*\mathcal{E})$ and $\delta$ is coCartesian $(\widetilde{X}(d)^*\mathcal{E} \cong \mathcal{H})$ then $\beta$ is coCartesian if and only if $\gamma$ is Cartesian or, in other words, the natural exchange 
\[ \widetilde{X}(b)^* \widetilde{X}(a)_*  \rightarrow \widetilde{X}(c)_* \widetilde{X}(d)^*\]
is an isomorphism on objects with support in $X(\pi_{13}(z))$.

\item Let $a$ and $c$ be of type 2 and $b$ and $d$ of type 3.
If \underline{$\mathcal{E}$ has support in $X(\pi_{13}(z))$} and $\alpha$ is strongly coCartesian $(\widetilde{X}(a)^*\mathcal{F} \cong \mathcal{E}$ inducing $\widetilde{X}(a)_!\mathcal{E} \cong \mathcal{F})$ and $\delta$ is coCartesian $(\widetilde{X}(d)^*\mathcal{E} \cong \mathcal{H})$ then $\beta$ is coCartesian if and only if $\gamma$ is strongly coCartesian or, in other words, the natural exchange 
\[ \widetilde{X}(c)_!  \widetilde{X}(d)^*  \rightarrow \widetilde{X}(b)^* \widetilde{X}(a)_!  \] 
is an isomorphism on objects with support in $X(\pi_{13}(z))$. 

 \item 
 Let $a$ and $c$ be of type 2 and $b$ and $d$ of type 1.
 If $\gamma$ is strongly coCartesian $(\widetilde{X}(c)^*\mathcal{G} \cong \mathcal{H}$ inducing $\widetilde{X}(c)_!\mathcal{H} \cong \mathcal{G})$ and $\delta$ is Cartesian $(\mathcal{E} \cong \widetilde{X}(d)_*\mathcal{H})$ then $\beta$ is Cartesian if and only if $\alpha$ is strongly coCartesian, or in other words, the natural exchange
\[ \widetilde{X}(a)_! \widetilde{X}(d)_* \rightarrow \widetilde{X}(b)_* \widetilde{X}(c)_!    \]
is an isomorphism. 
\end{enumerate}
\end{PROP}
\begin{proof}
1.\@ follows from Lemma~\ref{LEMMACOCARTSQUARES2}, and 2.\@ from Lemma~\ref{LEMMACOCARTSQUARES3}, using Lemma~\ref{LEMMAWEAKCOMP} in each case.
3.\@ follows from Lemma~\ref{LEMMACOCARTSQUARES1}, using Lemma~\ref{LEMMACARTDIA2}. 
\end{proof}

\section{Preliminaries for the construction of the derivator six-functor-formalism (multi-case)}\label{SECTPRELIMM}

In this section, the discussion in the previous section will be repeated, making the necessary modifications to include the multi-case, needed later to include the
monoidal structure into the derivator six-functor-formalism. It should be skipped on a first reading. 

Let $\mathcal{S}$ be a category with compactifications, and $\SSS^{\op}$ the symmetric pre-multiderivator represented by $\mathcal{S}^{\op}$
with the symmetric multicategory structure \ref{PAROPMULTCAT}. 
Let $\DD \rightarrow \SSS^{\op}$ be a fibered multiderivator with domain $\Invlf$, satisfying axioms (F1)--(F6) and also (F4m)--(F5m).

\begin{PAR}\label{TWGENREL}
Consider a tree $\tau$. The category ${}^{\Xi} \tau$ can be generated by multimorphisms of type $l$
\[ \xymatrix{ [S_{1,1}^{(1)}, \dots, S_{1,n_1}^{(1)}, \dots] \ar[r] \ar@{=}[d] & \cdots \ar[r] &  [S_{i,1}^{(1)}, \dots, S_{i,n_i}^{(1)}, \dots] \ar[r]\ar@{=}[d] & \cdots \ar[r] &   [S_{l}^{(1)}, \dots, S^{(n)}_{l}] \ar@{->}[d] &   \\
 [S_{i,1}^{(1)}, \dots, S_{i,n_i}^{(1)}, \dots] \ar[r] & \cdots \ar[r] &  [S_{i,1}^{(1)}, \dots, S_{i,n_i}^{(1)}, \dots] \ar[r] & \cdots \ar[r] &   [T_{l}] &   } \]
where the morphism $S_{l}^{(1)}, \dots, S^{(n)}_{l} \rightarrow T_l$ is a generating morphism of $\tau$
and morphisms of type $i$
\[ \xymatrix{ [S_{1,1}, \dots, S_{1,n_1}] \ar[r] \ar@{=}[d] & \cdots \ar[r] &  [S_{i,1}, \dots, S_{i,n_i}] \ar[r]\ar@{<->}[d] & \cdots \ar[r] &   [S_{l}] \ar@{=}[d] &   \\
 [S_{1,1}, \dots, S_{1,n_1}] \ar[r] & \cdots \ar[r] &  [T_{i,1}, \dots, T_{i,n_i'}] \ar[r] & \cdots \ar[r] &   [S_{l}] &   } \]
in which the morphism of lists $[S_{i,1}, \dots, S_{i,n_i}] \leftrightarrow [T_{i,1}, \dots, T_{i,n_i'}]$ consists of {\em one} generating morphism of $\tau$ and identities otherwise. 
These generators are subject to the relations requiring that squares
\begin{equation} \label{squaretype} \vcenter{ \xymatrix{
w_1, \dots, w_n \ar[r] \ar[d] & z \ar[d] \\
x_1, \dots, x_n \ar[r] & y,
} } \end{equation}
in which the vertical and horizontal morphisms are generators as above, are commutative. Necessarily also only one of the left vertical morpisms is not an identity. In the non-multi-case $\tau = \Delta_n$ we do not have any non-trivial relation-squares in which the horizontal and vertical morphisms are of the same type. Otherwise this may happen for type $<l$. 
\end{PAR}

\begin{PAR} \label{WEAKLYADM}
Let $\mathcal{S}$ be a category equipped with the symmetric opmulticategory structure \ref{PAROPMULTCAT} and $M$ a multidiagram. A pseudo-functor of multicategories
\[ M \rightarrow \mathcal{S}^{\cor} \]
can be seen (cf.\@ Section~\ref{SECTIONDIACORCOMP}) as a functor of opmulticategories
\[ \tw M \rightarrow \mathcal{S} \]
which is {\bf admissible} in the sense that every square 
\begin{equation}\label{eqsquareinm}\vcenter{ \xymatrix{
i \ar[r] \ar[d] & j_1,\dots, j_n \ar[d] \\
i' \ar[r] & j_1',\dots,j_n'
} } \end{equation}
in which the horizontal morphisms are of type 2 and the vertical ones of type 1 is mapped to a Cartesian square. 
As is the non-multi-case we say that the functor $\tw M \rightarrow \mathcal{S}$ is {\bf weakly admissible} if the squares above
are instead mapped to weakly Cartesian squares. 

 Similarly: 
Consider a diagram $I$ (not multidiagram) and the 2-multicategory $\Fun(I, \mathcal{S}^{\cor})$. 
Let $M$ be a multicategory. 
A pseudo-functor 
\[ M \rightarrow \Fun(I, \mathcal{S}^{\cor}) \]
may be seen as a functor of usual 1-opmulticategories
\[ {}^{\downarrow \uparrow} M  \rightarrow  \Fun(\tw{I}, \mathcal{S})^{\mathrm{adm}}. \]
This functor has the property that each morphism of type 1 is mapped to a type-2-admissible morphism
and every (multi)morphism of type 2 is mapped to a type-1-admissible  (multi)morphism and
each diagram (\ref{eqsquareinm})
in which the horizontal morphisms are of type 2 and the vertical morphisms are of type 1 (necessarily 1-ary) is mapped
to a Cartesian square. 

Similarly, pseudo-functors of 2-multicategories
\[ M \rightarrow \Fun^{\mathrm{(op)lax}}(I, \mathcal{S}^{\cor}) \]
are the same as functors between 1-opmulticategories
\[ \tw{M}  \rightarrow \Fun(\tw{I}, \mathcal{S})^{\mathrm{adm}} \]
in which every morphism of type 1 is mapped to a (weakly in the oplax case) type-2-admissible morphism and every multimorphism
of type 2 is mapped to a (weakly in the lax case) type-1-admissible multimorphism and in which every diagram (\ref{eqsquareinm}) as above is mapped to a 
Cartesian square. It has obviously still the property, that the resulting functor
\[ \tw (I \times M)  \rightarrow \mathcal{S} \]
is {\em weakly} admissible. 
\end{PAR}

\begin{PAR}\label{PARMULTICOMP}
Let $\mathcal{S}$ be a category equipped with the symmetric opmulticategory structure \ref{PAROPMULTCAT}. 
Let $M$ be a locally finite multidiagram and let
 $X: \tw M \rightarrow \mathcal{S}$ be a functor of opmulticategories. Because of the particular opmulticategory structure on $\mathcal{S}$, cf.\@ \ref{PARCIRC}, $X$ can be compactified 
  yielding a point-wise dense embedding
 \[ X \hookrightarrow \overline{X}. \]
As for the plain case, $\overline{X}$ does not need to be admissible if $X$ is. 
This also shows that we get an interior compactification 
\[ \widetilde{X}: \tww M \rightarrow \mathcal{S} \]
for any (weakly) admissible 
\[ X: \tw M \rightarrow \mathcal{S} \]
(We leave it to the reader to construct a functor $(\tww M)^\circ \rightarrow {}^{\downarrow\downarrow} ((\tw M)^\circ)$ analogously to \ref{INTERIORCOMPCOR}.)
\end{PAR}

For an interior compactification \[ \widetilde{X}: \tww M \rightarrow \mathcal{S} \]
we denote by $\widetilde{X}(i \rightarrow j \rightarrow k)$
where $i, j$ and $k$ are {\em lists of objects} in $M$ the following. Let $k = [k_1, \dots, k_n]$.
Then $i$ and $j$ break up into sublists $i_1, \dots, i_n$ and $j_1, \dots, j_n$ with morphisms $i_1 \rightarrow j_1 \rightarrow [k_1]$, etc. 
Then define 
\[ \widetilde{X}(i \rightarrow j \rightarrow k) := \prod_\nu \widetilde{X}(i_\nu \rightarrow j_\nu \rightarrow [k_\nu]). \]
For any (multi)morphism $(i \rightarrow j \rightarrow k) \rightarrow (i' \rightarrow j' \rightarrow k')$ in the obvious sense, we get a 
corresponding morphism $\widetilde{X}(i \rightarrow j \rightarrow k) \rightarrow \widetilde{X}(i' \rightarrow j' \rightarrow k')$.

With this definition the same Lemmas as in the previous section hold mutatis mutandis, namely:

\begin{LEMMA}\label{LEMMACARTDIA2M}
Let $M$ be a multidiagram and $X: \tw M \rightarrow \mathcal{S}$ be weakly admissible and let 
$\widetilde{X}$ be any interior compactification of it. 
Any square of the form
\[ \xymatrix{
& \widetilde{X}(i \rightarrow j \rightarrow [k])  \ar@{->>}[r] \ar@{^{(}->}[d] &  \widetilde{X}(i' \rightarrow i \rightarrow [k])  \ar@{^{(}->}[d] \\
& \widetilde{X}(i \rightarrow j' \rightarrow [k])  \ar@{->>}[r] & \widetilde{X}(i' \rightarrow j' \rightarrow [k]) 
} \]
is Cartesian for all $i, i', j, j'$ lists of objects of $M$, and all objects $k \in M$.
\end{LEMMA}

\begin{proof}
Cf.\@ Lemma~\ref{LEMMACART1}.
\end{proof}

\begin{LEMMA}\label{LEMMACARTDIAM}
Let $M$ be a multidiagram and $X: \tw M \rightarrow \mathcal{S}$ be weakly admissible and let 
$\widetilde{X}$ be any interior compactification of it. 
Any square of the form
\[ \xymatrix{
& \widetilde{X}(i = i \rightarrow [k])  \ar[r] \ar@{^{(}->}[d] &  \widetilde{X}(i  = i \rightarrow [k_1]), \dots, \widetilde{X}(i  = i \rightarrow [k_n])  \ar@{^{(}->}[d] \\
& \widetilde{X}(i \rightarrow j \rightarrow [k])  \ar[r] & \widetilde{X}(i \rightarrow j \rightarrow [k_1]), \dots,  \widetilde{X}(i  \rightarrow j \rightarrow [k_n])
} \]
is Cartesian for all $i,j$ lists of objects of $M$, and all objects $k$ and $k_1, \dots, k_n$ of $M$. 
\end{LEMMA}
\begin{proof}
Writing $k':= [k_1, \dots, k_n]$ we have to check that the top square in the following diagram is Cartesian:
\[ \xymatrix{
& \widetilde{X}(i = i \rightarrow [k])  \ar[r] \ar@{^{(}->}[d] &  \widetilde{X}(i  = i \rightarrow k')  \ar@{^{(}->}[d] \\
& \widetilde{X}(i \rightarrow j \rightarrow [k])  \ar[r]  \ar@{->>}[d]& \widetilde{X}(i \rightarrow j \rightarrow k')  \ar@{->>}[d] \\
& \widetilde{X}(j = j \rightarrow [k])  \ar[r] &  \widetilde{X}(j  = j \rightarrow k')  \\
} \]
in which the outer square is weakly Cartesian, because $X$ is weakly admissible. The statement follows therefore from Lemma~\ref{LEMMACART2}.
\end{proof}

\begin{LEMMA}\label{LEMMAWEAKCOMPM}
Let $M$ be a multidiagram and $X: \tw M \rightarrow \mathcal{S}$ be weakly admissible and let 
$\widetilde{X}$ be any interior compactification. 
Then the cube
\begin{equation}\nonumber
\resizebox{\displaywidth}{!}{
 \xymatrix{
 & \widetilde{X}(i = i \rightarrow [k]) \ar@{}[ddrr]|(.7){} \ar[rr]^{} \ar@{=}[dl]_{} \ar@{^{(}->}[dd]^(.7){} && \widetilde{X}(i = i \rightarrow [k_1]), \dots,  \widetilde{X}(i = i \rightarrow [k_n]) \ar@{^{(}->}[dd]^{} \ar@{=}[ld]_{} \\ 
\widetilde{X}(i = i \rightarrow [k]) \ar@{}[ddrr]|(.7){} \ar[rr]^(.4){} \ar@{^{(}->}[dd]^{} && \widetilde{X}(i = i \rightarrow [k_1]), \dots, \widetilde{X}(i = i \rightarrow [k_n])  \ar@{^{(}->}[dd]_(.7){} & \\
& \widetilde{X}(i \rightarrow j \rightarrow [k])  \ar@{^{(}->}[dl]_{} \ar[rr]^(.3){\overline{G}} && \widetilde{X}(i \rightarrow j \rightarrow [k_1]), \dots, \widetilde{X}(i \rightarrow j \rightarrow [k_n])  \ar@{^{(}->}[dl]^{} \\
\widetilde{X}(i \rightarrow j' \rightarrow [k])  \ar[rr]_{\overline{g}} && \widetilde{X}(i \rightarrow j' \rightarrow [k_1]), \dots, \widetilde{X}(i \rightarrow j' \rightarrow [k_n])  }
}
\end{equation}
is a weak compactification of the top (trivially) Cartesian square, If $X$ is admissible, then the cube
\begin{equation}\nonumber
\resizebox{\displaywidth}{!}{
 \xymatrix{
 & \widetilde{X}(i = i \rightarrow [k]) \ar@{}[ddrr]|(.7){} \ar[rr]^{} \ar[dl]_{} \ar@{^{(}->}[dd]^(.7){} && \widetilde{X}(i = i \rightarrow [k_1]) , \dots,  \widetilde{X}(i = i \rightarrow [k_n]) \ar@{^{(}->}[dd]^{} \ar[ld]_{} \\ 
\widetilde{X}(i' = i' \rightarrow [k]) \ar@{}[ddrr]|(.7){} \ar[rr]^(.4){} \ar@{^{(}->}[dd]^{} && \widetilde{X}(i' = i' \rightarrow [k_1]), \dots,  \widetilde{X}(i' = i' \rightarrow [k_n])  \ar@{^{(}->}[dd]_(.7){} & \\
& \widetilde{X}(i \rightarrow j \rightarrow [k])  \ar@{->>}[dl]_{} \ar[rr]^(.3){\overline{G}} && \widetilde{X}(i \rightarrow j \rightarrow [k_1]) , \dots, \widetilde{X}(i \rightarrow j \rightarrow [k_n])  \ar@{->>}[dl]^{} \\
\widetilde{X}(i' \rightarrow j \rightarrow [k])  \ar[rr]_{\overline{g}} && \widetilde{X}(i' \rightarrow j \rightarrow [k_1]) , \dots, \widetilde{X}(i' \rightarrow j \rightarrow [k_n])  }
}
\end{equation}
is a weak compactification of the top Cartesian square. 
\end{LEMMA}
\begin{proof}
We need to show that in each case the front and back squares are Cartesian squares. This is a consequence of Lemma~\ref{LEMMACARTDIAM}.
\end{proof}

\begin{LEMMA}\label{LEMMACOCARTSQUARES2M}
Consider a weak compactification of a Cartesian square. 
\begin{equation}\nonumber
 \xymatrix{
 & W \ar@{}[ddrr]|(.7){} \ar[rr]^{} \ar[dl]_{} \ar@{^{(}->}[dd]^(.7){} && Z_1, \dots, Z_n \ar@{^{(}->}[dd]^{} \ar[ld]_{} \\ 
Y \ar@{}[ddrr]|(.7){} \ar[rr]^(.4){} \ar@{^{(}->}[dd]^{} && X_1, \dots, X_n \ar@{^{(}->}[dd]_(.7){} & \\
& \overline{W} \ar@{->>}[dl]_{\overline{F}} \ar[rr]^(.3){\overline{G}} && \overline{Z}_1, \dots, \overline{Z}_n \ar@{->>}[dl]^{\overline{f}_1, \dots, \overline{f}_n} \\
\overline{Y} \ar[rr]_{\overline{g}} && \overline{X}_1, \dots, \overline{X}_n }
\end{equation}
in which the $\overline{f}_i$ and $\overline{F}$ are proper. 
Then for a square
\[ \xymatrix{
\mathcal{H} \ar@{<-}[r]^-\delta \ar@{<-}[d]_\gamma & \mathcal{E}_1, \dots, \mathcal{E}_n  \ar@{<-}[d]^{\alpha_1, \dots, \alpha_n} \\
\mathcal{G} \ar@{<-}[r]_-\beta & \mathcal{F}_1, \dots, \mathcal{F}_n 
} \]
in $\DD(\cdot)$ above the bottom square the following holds: If \underline{$\mathcal{E}_i$ has support in $Z_i$} for all $i$, and the $\alpha_i$ are Cartesian ($\mathcal{F}_i \cong \overline{f}_{i,*}\mathcal{E}_i$) and $\delta$ is coCartesian ($\overline{G}^*(\mathcal{E}_1, \dots, \mathcal{E}_n) \cong \mathcal{H}$) then $\beta$ is coCartesian if and only if $\gamma$ is Cartesian or, in other words, the natural exchange 
\[ \overline{g}^*( \overline{f}_{1,*} -, \dots, \overline{f}_{n,*} -)  \rightarrow \overline{F}_* \overline{G}^*(-, \dots, -)\]
is an isomorphism on tupels of objects with support in $Z_1, \dots, Z_n$.
\end{LEMMA}
\begin{proof} 
We look at the following diagram
\[ \xymatrix{
W \ar@{=}[r] \ar@{^{(}->}[d]_{\iota_W} \ar@{}[rd]|{\numcirc{1}} & W \ar[r]^-{G} \ar@{^{(}->}[d]^\iota \ar@{}[rd]|{\numcirc{2}} & Z_1, \dots, Z_n \ar@{^{(}->}[d]^{\iota_Z}  \\
\overline{W} \ar@{->>}[r]_{\overline{f}''}  \ar@{->>}[d]_{\overline{F}} 
& \Box \ar@{}[rd]|{\numcirc{3}}  \ar[r]^-{G'}  \ar@{->>}[d]^{\overline{f}'} & \overline{Z}_1, \dots, \overline{Z}_n  \ar@{->>}[d]^{\overline{f}}  \\
\overline{Y} \ar@{=}[r]   & \overline{Y}  \ar[r]_-{\overline{g}}   & \overline{X}_1, \dots, \overline{X}_n \\
} \]
in which squares $\numcirc{2}$ and $\numcirc{3}$ are Cartesian. 
The middle left horizontal morphism is proper because of (S2). Because of the support condition, we have to show that the natural exchange
\[ \overline{g}^*(\overline{f}_{1,*} \iota_{Z_1,!}-, \dots, \overline{f}_{n,*} \iota_{Z_n,!}-) \rightarrow \overline{F}_* \overline{G}^* (\iota_{Z_1,!}-, \dots, \iota_{Z_n,!}-)  \]
is an isomorphism. Elementary properties of exchange morphisms imply that the morphism is the composition of the following exchange morphisms which are all isomorphisms because of the indicated reason.
\[
\begin{array}{rcll}
\overline{g}^* (\overline{f}_{1,*} \iota_{Z_1,!}-,\dots, \overline{f}_{n,*} \iota_{Z_n,!}-)  &\iso & \overline{f}'_*  (G')^* (\iota_{Z_1,!}-, \dots, \iota_{Z_n,!}-)  & \parbox{15em}{because $\numcirc{3}$ is Cartesian and (F4) in the form of Lemma~\ref{LEMMAF4}} \vspace{.5em} \\
&\isor & \overline{f}'_*   \iota_{!} G^*(-, \dots, -)    & \parbox{15em}{because $\numcirc{2}$ is Cartesian and (F5m) in the form of Lemma~\ref{LEMMAF5}}  \vspace{.5em}  \\
&\iso & \overline{f}'_*  \overline{f}''_* \iota_{W,!} G^*(-,\dots,-) & \parbox{15em}{applying Lemma~\ref{LEMMACOCARTSQUARES1} for $\numcirc{1}$}  \vspace{.5em}  \\
&\iso & \overline{F}_*  \iota_{W,!} G^*(-,\dots,-) \\
&\iso & \overline{F}_* \overline{G}^* ( \iota_{Z_1,!}-, \dots,  \iota_{Z_n,!}-)  & \parbox{15em}{because the composite of $\numcirc{1}$ and $\numcirc{2}$ is Cartesian and (F5m) in the form of Lemma~\ref{LEMMAF5}.}
\end{array}
\]
\end{proof}

\begin{LEMMA}\label{LEMMACOCARTSQUARES3M}
Consider a weak compactification of a Cartesian square. 
\begin{equation}\nonumber
 \xymatrix{
 & W \ar@{}[ddrr]|(.7){} \ar[rr]^{} \ar@{^{(}->}[dl]_{I} \ar@{^{(}->}[dd]^(.7){} && Z_1, \dots, Z_n \ar@{^{(}->}[dd]^{} \ar@{^{(}->}[ld]_{\iota_1, \dots, \iota_n} \\ 
Y \ar@{}[ddrr]|(.7){} \ar[rr]^(.4){} \ar@{^{(}->}[dd]^{} && X_1, \dots, X_n \ar@{^{(}->}[dd]_(.7){} & \\
& \overline{W} \ar@{^{(}->}[dl]_{\overline{I}} \ar[rr]^(.3){\overline{G}} && \overline{Z}_1, \dots, \overline{Z}_n \ar@{^{(}->}[dl]^{\overline{\iota}_1, \dots, \overline{\iota}_n} \\
\overline{Y} \ar[rr]_{\overline{g}} && \overline{X}_1, \dots,  \overline{X}_n }
\end{equation}
in which all $\overline{\iota}_i$, all $\iota_i$, $I$, and $\overline{I}$ are embeddings. 
Then for a commutative square
\[ \xymatrix{
\mathcal{H} \ar@{<-}[r]^-\delta \ar@{<-}[d]_-\gamma & \mathcal{E}_1, \dots, \mathcal{E}_n \ar@{<-}[d]^{\alpha_1, \dots, \alpha_n} \\
\mathcal{G} \ar@{<-}[r]_-\beta & \mathcal{F}_1, \dots, \mathcal{F}_n
} \]
in $\DD(\cdot)$ above the bottom square the following holds: If \underline{$\mathcal{E}_i$ has support in $Z_i$} for all $I$, and the $\alpha_i$ are strongly coCartesian $(\overline{\iota}_i^*\mathcal{F}_i \cong \mathcal{E}_i$ inducing $\overline{\iota}_{i,!}\mathcal{E}_i \cong \mathcal{F}_i)$ and $\delta$ is coCartesian $(\overline{G}^*(\mathcal{E}_1, \dots, \mathcal{E}_n) \cong \mathcal{H})$ then $\beta$ is coCartesian if and only if $\gamma$ is strongly coCartesian or, in other words, the natural exchange 
\[ \overline{I}_!  \overline{G}^* (-, \dots, -)  \rightarrow \overline{g}^* ( \overline{\iota}_{1,!}-, \dots, \overline{\iota}_{n,!}-) \] 
is an isomorphism on tupels of objects with support in $Z_1, \dots, Z_n$. 
\end{LEMMA}

\begin{proof}
Because of the support condition it suffices to see that the natural exchange
\[ \overline{I}_!  \overline{G}^* (\iota_{Z_1,!}-,\dots,\iota_{Z_n,!}-)   \rightarrow \overline{g}^* (\overline{\iota}_{1,!} \iota_{Z_1,!}-,\dots,\overline{\iota}_{n,!} \iota_{Z_n,!}-)   \] 
is an isomorphism. 
Elementary properties of exchange morphisms imply that the morphism is the composition of the following exchange morphisms which are all isomorphisms because of pseudofunctoriality and because
of the indicated reason:

\[ \begin{array}{rcll}
   \overline{I}_!  \overline{G}^*  (\iota_{Z_1,!}-, \dots, \iota_{Z_n,!}-) 
   &\isor&     \overline{I}_!  \iota_{W,!}  G^*(-, \dots, -)  & \parbox{16em}{because the back square in Cartesian and (F5m) in the form of Lemma~\ref{LEMMAF5}} \\
&\iso&     \iota_{Y,!} I_! G^*(-, \dots, -) & \text{} \\
&\iso&   \iota_{Y,!} g^*( \iota_{1,!}-,\dots,\iota_{n,!}-)  & \parbox{16em}{because the top square is Cartesian and and (F5m) in the form of Lemma~\ref{LEMMAF5}} \vspace{.6em} \\
&\iso&   \overline{g}^*  (\iota_{X_1,!} \iota_{1,!}-, \dots, \iota_{X_n,!} \iota_{n,!}-) & \parbox{16em}{because the front square is Cartesian and (F5m) in the form of Lemma~\ref{LEMMAF5}}  \\
&\iso&   \overline{g}^*( \overline{\iota}_{1,!} \iota_{Z_1,{!}}-, \dots,  \overline{\iota}_{n,!} \iota_{Z_n,{!}}-)  &
\end{array} \]
\end{proof}

We summarize the discussion in the following Proposition:
\begin{PROP}\label{PROPPROPERTIESCORCOMPM}
Let $M$ be a multidiagram, $X: \tw M \rightarrow \mathcal{S}$ be weakly admissible, and let 
$\widetilde{X}: \tww M \rightarrow \mathcal{S}$ be any interior compactification of it. 
Consider a diagram of the form

\[ \xymatrix{
w \ar[r]^-d \ar[d]_c & z_1, \dots, z_n \ar[d]^{a_1, \dots, a_n} \\
y \ar[r]_-b & x_1, \dots, x_n
} \]
in $\tww M$ where the $b$ and $d$ are multimorphisms of type 3. 
Let 
\[ \xymatrix{
\mathcal{H} \ar@{<-}[r]^-\delta \ar@{<-}[d]_\gamma & \mathcal{E}_1, \dots, \mathcal{E}_n \ar@{<-}[d]^{\alpha_1, \dots, \alpha_n} \\
\mathcal{G} \ar@{<-}[r]_-\beta & \mathcal{F}_1, \dots, \mathcal{F}_n
} \]
be a diagram in $\DD(\cdot)$ above {\em $\widetilde{X}^{\op}$ applied to the top square}. Then the following holds:

\begin{enumerate}
\item If $X$ is admissible, let the $a_i$ and $c$ be of type 1.
If \underline{$\mathcal{E}_i$ has support in $X(\pi_{13}(z_i))$} for all $i$, and all $\alpha_i$ are Cartesian $(\mathcal{F}_i \cong \widetilde{X}(a_i)_*\mathcal{E}_i)$ and $\delta$ is coCartesian $(\widetilde{X}(d)^*(\mathcal{E}_1, \dots, \mathcal{E}_n) \cong \mathcal{H})$ then $\beta$ is coCartesian if and only if $\gamma$ is Cartesian or, in other words, the natural exchange 
\[ \widetilde{X}(b)^* (\widetilde{X}(a_1)_{*}-, \dots, \widetilde{X}(a_n)_{*}-)  \rightarrow \widetilde{X}(c)_* \widetilde{X}(d)^*(-, \dots, -) \]
is an isomorphism on $n$-tupels of objects with support in $X(\pi_{13}(z_1)), \dots, X(\pi_{13}(z_n))$.

\item Let the $a_i$ and $c$ be of type 2.
If \underline{$\mathcal{E}_i$ has support in $X(\pi_{13}(z_i))$} for all $i$, and all $\alpha_i$ are strongly coCartesian $(\widetilde{X}(a_i)^*\mathcal{F}_i \cong \mathcal{E}_i$ inducing $\widetilde{X}(a_i)_{!}\mathcal{E}_i \cong \mathcal{F}_i)$ and $\delta$ is coCartesian $(\widetilde{X}(d)^*(\mathcal{E}_1, \dots, \mathcal{E}_n) \cong \mathcal{H})$ then $\beta$ is coCartesian if and only if $\gamma$ is strongly coCartesian or, in other words, the natural exchange 
\[ \widetilde{X}(c)_!  \widetilde{X}(d)^*(-, \dots, -)  \rightarrow \widetilde{X}(b)^* (\widetilde{X}(a_1)_{!}-, \dots,  \widetilde{X}(a_n)_{!}-)  \] 
is an isomorphism on $n$-tupels of objects with support in $X(\pi_{13}(z_1)), \dots, X(\pi_{13}(z_n))$. 

 \item 
 Let $a$ and $c$ be of type 2 and $b$ and $d$ of type 1.
 If $\gamma$ is strongly coCartesian $(\widetilde{X}(c)^*\mathcal{G} \cong \mathcal{H}$ inducing $\widetilde{X}(c)_!\mathcal{H} \cong \mathcal{G})$ and $\delta$ is Cartesian $(\mathcal{E} \cong \widetilde{X}(d)_*\mathcal{H})$ then $\beta$ is Cartesian if and only if $\alpha$ is strongly coCartesian, or in other words, the natural exchange
\[ \widetilde{X}(a)_! \widetilde{X}(d)_* \rightarrow \widetilde{X}(b)_* \widetilde{X}(c)_!    \]
is an isomorphism.

\end{enumerate}
\end{PROP}
\begin{proof}
1.\@ follows from Lemma~\ref{LEMMACOCARTSQUARES2M}, and 2.\@ from Lemma~\ref{LEMMACOCARTSQUARES3M}, using Lemma~\ref{LEMMAWEAKCOMPM} in each case.
3.\@ is proven exactly as in Proposition~\ref{PROPPROPERTIESCORCOMP}.
\end{proof}

\section{The construction of derivator six-functor-formalisms}\label{SECTCONST}

\begin{PAR}\label{BEGINSECTIONDER6FUPRE}
Let $\mathcal{S}$ be a category with compactifications, and $\SSS^{\op}$ the symmetric pre-multiderivator represented by $\mathcal{S}^{\op}$ with domain $\Cat$, where $\mathcal{S}^{\op}$ is equipped with the symmetric multicategory structure \ref{PAROPMULTCAT}. Recall from Definition~\ref{DEFSCOR2} the definition of the symmetric 
2-pre-multiderivator  $\SSS^{\cor}$, $\SSS^{\cor, 0, \lax}$, and \@ $\SSS^{\cor, 0, \oplax}$ (the latter formed w.r.t.\@ the given class of proper morphisms), respectively, with domain $\Cat$.
\end{PAR}

\begin{DEF}[{\cite[Definition~6.1]{Hor16}}]\label{DEF6FUDER}
\begin{enumerate}

\item A {\bf (symmetric) derivator six-functor-formalism} is a left and right fibered (symmetric) multiderivator with domain $\Cat$
\[ \DD \rightarrow \SSS^{\cor}. \] 

\item A {\bf (symmetric) proper derivator six-functor-formalism} is as before with an extension as oplax left fibered (symmetric) multiderivator with domain $\Cat$
\[ \DD' \rightarrow \SSS^{\cor, 0, \oplax},  \] 
and an extension as lax right fibered (symmetric) multiderivator with domain $\Cat$
\[ \DD'' \rightarrow \SSS^{\cor, 0, \lax}.  \] 
\end{enumerate}
\end{DEF}
There is a dual notion of etale derivator six-functor-formalism which will not play any role in this article.
The word ``symmetric'' in brackets indicates that there is a symmetric and a non-symmetric variant of the definition. 
In the symmetric variant all functors of multicategories occurring the various definitions 
have to be compatible with the actions of the symmetric groups.

\begin{PAR}\label{BEGINSECTIONDER6FU}
Let $\SSS^{\op}$ be the symmetric pre-multiderivator as in \ref{BEGINSECTIONDER6FUPRE}. 
Let $\DD \rightarrow \SSS^{\op}$ be a (symmetric) fibered multiderivator with domain $\Dirlf$ satisfying axioms (F1)--(F6) and (F4m)--(F5m) of  \ref{PARAXIOMS}. Assume that $\DD$ is infinite (i.e.\@ satisfies (Der1${}^{\infty}$)). 
The goal is to construct a natural (symmetric) derivator six-functor-formalism $\EE \rightarrow \SSS^{\cor}$ (and finally a  (symmetric) {\em proper} derivator six-functor-formalism) whose restriction to $\SSS^{\op}$ is equivalent to $\DD$.
We will first construct a left fibered multiderivator $\EE \rightarrow \SSS^{\cor,\comp}$ with domain $\Dirlf$ such that
\[ \EE(I)_{X \hookrightarrow \overline{X}} = \DD({}^{\downarrow \uparrow \uparrow \downarrow} I)^{4-\cocart, 3-\cocart^*, 2-{\cart}}_{\pi_{234}^* \widetilde{X}^{\op} } \]
for all compactified correspondences $X\hookrightarrow \overline{X}$ in $\SSS^{\cor,\comp}(I)$, and where $\widetilde{X}$ denotes the corresponding interior compactification. 
The superscripts mean that we consider the full subcategory in which the underlying morphism in $\DD(\cdot)$ is coCartesian for any morphism of type 4 in $\twwc I$, is strongly coCartesian for any morphism of type 3, and 
is Cartesian for any morphism of type 2. We will also say that the objects are 4-coCartesian, strongly 3-coCartesian, and 2-Cartesian, respectively. 
\end{PAR}

\begin{BEISPIEL}
If $I = \Delta_1$, and $X: I \rightarrow \mathcal{S}^{\cor}$ is a correspondence
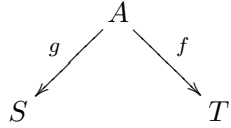
\begin{equation}\nonumber
 \xymatrix{
 & A \ar[rd]^f \ar[ld]_{g}  \\
S &  & T
} 
\end{equation}
with compactification $X \hookrightarrow \overline{X}$ in $\mathcal{S}^{\cor,\comp}$, inducing the interior compactification $\widetilde{X}$
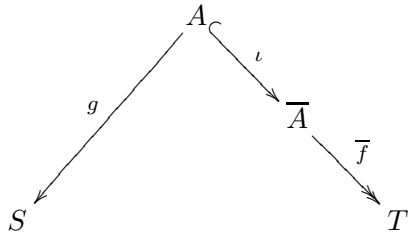
\begin{equation}\nonumber
 \xymatrix{
 & & A \ar@{^{(}->}[rd]^\iota \ar[ldld]_{g}  \\
 & & & \overline{A} \ar@{->>}[rd]^{\overline{f}}  \\
S & & & & T
} 
\end{equation}
then the {\em underlying diagram} of an object in $\mathcal{E} \in \EE(I)_{X\hookrightarrow \overline{X}}$ is determined (up to isomorphism) by an object $\mathcal{E}_0$ in $\DD(\cdot)_S$ and  $\mathcal{E}_1$ in $\DD(\cdot)_T$
together with a morphism
\[ \overline{f}_* \iota_! g^* \mathcal{E}_0  \rightarrow \mathcal{E}_1. \]
This is, of course, what is intended. It is therefore reasonable to believe that the above definition gives the correct derivator enhancement of this situation. 
\end{BEISPIEL}

We begin with a couple of Lemmas that are used to construct a 1-opfibration and 2-fibration $\EE(I) \rightarrow \SSS^{\cor, \comp}(I)$. The first crucial step is to understand
how the given pull-back functors, (even multivalued)  push-forward functors, and the additional $\iota_!$, for $\DD({}^{\downarrow \uparrow \uparrow \downarrow} I) \rightarrow \SSS^{\op}({}^{\downarrow \uparrow \uparrow \downarrow} I)$ behave w.r.t.\@ to the conditions of being 4-coCartesian, strongly 3-coCartesian, and 2-Cartesian, respectively.

\begin{PAR}\label{COMPONENTS}
Let $I$ be in $\Dirlf$, let $\tau$ be a tree, and let
\[ X \hookrightarrow \overline{X} \]
be an exterior compactification in $\mathcal{S}^{\tw(\tau \times I)}$ with $X$ admissible, i.e.\@ an object in $\Cor^{\comp}_I(\tau)$ (cf.\@ Definition~\ref{DEFCORCOMP}).
By the construction in \ref{PARMULTICOMP} it has an interior compactification
\[ \widetilde{X}: \tww (\tau \times I) \rightarrow \mathcal{S} \]
to which we may apply the results of the previous section.

Let $o$ be a multimorphism in $\tww \tau$. If $o$ is of type 3 (resp.\@ type 2, resp.\@ type 1) denote by
\[ \begin{array}{rrcl} 
 \widetilde{g}: &\widetilde{A} &\rightarrow& \widetilde{S}_1, \dots, \widetilde{S}_n \\
 \widetilde{\iota}: &\widetilde{A} &\rightarrow& \widetilde{A}' \\
 \widetilde{f}:  &\widetilde{A}' &\rightarrow& \widetilde{T} 
\end{array} \]
their images in $\mathcal{S}^{\tww I}$.
The notation is borrowed from the following example. Keep in mind, however, that not every multimorphism $o$ occurs as the ones considered there. In particular, the
$\widetilde{A}, \widetilde{S}_i, \widetilde{T}$ do not need to be --- unlike in the example --- compactifications of admissible diagrams themselves. 

For each object $x$ in $\tww \tau$ there is a canonical object $y$ in the image of $\tw \tau$ with a morphism
\[ o: y \rightarrow x \]
of type 2. Let $\widetilde{\iota}: \widetilde{S}' \rightarrow \widetilde{S}$ be the corresponding morphism.
We say that an object in $\DD(\twwc I)_{\pi_{234}^* \widetilde{S}^{\op}}$ is {\bf well-supported}, if it is strongly 3-coCartesian and the underlying diagram lies {\em point-wise} at all $\mu \in \twwc I$ in the essential image of
$\widetilde{\iota}(\pi_{234}(\mu))_!$. The corresponding full subcategory will be denoted by $\DD(\twwc I)_{\pi_{234}^* \widetilde{S}^{\op}}^{\ws}$.
\end{PAR}

\begin{BEISPIEL}[$\tau = \Delta_{1,n}$] \label{EXCOMPONENTS}
In this case $X$  induces the diagram in $\Fun(\tw  I, \mathcal{S})$
\[ \xymatrix{
& A \ar[ld]^{g_1} \ar[rd]_{g_n}  \ar[rrrd]^{f} \\
S_1  & \dots &   S_n &  ; &   T
} \]
with $g=(g_1, \dots, g_n)$ type-1 admissible as multimorphism, and $f$ type-2 admissible, 
with exterior compactification
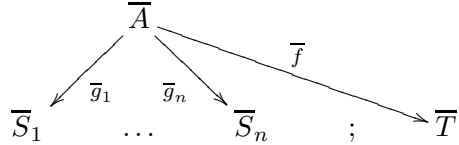
\[ \xymatrix{
& \overline{A} \ar[ld]^{\overline{g}_1} \ar[rd]_{\overline{g}_n}  \ar[rrrd]^{\overline{f}} \\
\overline{S}_1  & \dots &   \overline{S}_n &  ; &   \overline{T}
} \]
(Note that the diagrams $\overline{S}_1, \dots, \overline{S}_n, \overline{A}, \overline{T}$ have no admissibility properties and neither do the morphisms.)
 
It induces an interior compactification, i.e.\@ a diagram of shape $(\tww \Delta_{1,n})^\circ$ in $\Fun(\tww  I  
\rightarrow \mathcal{S})$:
\[ \xymatrix{
& & \widetilde{A} \ar[ldld]^{\widetilde{g}_1} \ar[rdrd]_{\widetilde{g}_n}  \ar@{^{(}->}[rrrd]^{\widetilde{\iota}} \\
& & & & & \widetilde{A}'  \ar@{->>}[rrrd]^{\widetilde{f}}\\
\widetilde{S}_1 & & \dots & &  \widetilde{S}_n & & ; & &  \widetilde{T}
} \]
\end{BEISPIEL}

\begin{LEMMA}\label{LEMMAEXISTENCE3FUNCTORS}
Consider an object
$X \hookrightarrow \overline{X}$
 in $\Cor^{\comp}_I(\tau)$. With the notation as in \ref{BEGINSECTIONDER6FU} and \ref{COMPONENTS} we have
\begin{enumerate}
\item For any $o \in \tww \tau$ (numbered with indices 2--4) of type 4,
the multivalued functor
\[ \widetilde{X}(o)_{\#} := (\pi_{234}^*\widetilde{g})^*: \DD(\twwc I)_{\pi_{234}^* \widetilde{S}_1^{\op}} \times \cdots \times  \DD(\twwc I)_{\pi_{234}^* \widetilde{S}_n^{\op}} \rightarrow  \DD(\twwc I)_{\pi_{234}^* \widetilde{A}^{\op}}.    \]
is computed point-wise (in $\twwc I$) and on well-supported objects preserves the condition of being 4-coCartesian, well-supported, and 2-Cartesian. 

\item For any $o \in \tww \tau$ of type 3, the functor 
\[  \widetilde{X}(o)_{\#} :=  (\pi_{234}^*\widetilde{\iota})_!: \DD(\twwc I)_{\pi_{234}^* \widetilde{A}^{\op}}  \rightarrow  \DD(\twwc I)_{\pi_{234}^* (\widetilde{A}')^{\op}}    \]
i.e.\@ the left adjoint of $(\pi_{234}^*\widetilde{\iota})^*$, which exists by (F1), 
is computed point-wise (in $\twwc I$)   on 4-coCartesian and well-supported objects, 
and on such it preserves the conditions of being 4-coCartesian, well-supported, and 2-Cartesian. 

\item For any $o \in \tww \tau$ of type 2, 
the functor
\[  \widetilde{X}(o)_{\#} :=  (\pi_{234}^*\widetilde{f})_*: \DD(\twwc I)_{\pi_{234}^* (\widetilde{A}')^{\op}} \rightarrow  \DD(\twwc I)_{\pi_{234}^* \widetilde{T}^{\op}}.    \]
is computed point-wise (in $\twwc I$)  (cf.\@ Definition~\ref{DEFCOMPUTEDPOINTWISE})  and on well-supported objects it preserves the condition of being 4-coCartesian, well-supported, and 2-Cartesian. 
\end{enumerate} 
\end{LEMMA}
\begin{proof}
1.\@ $(\pi_{234}^*\widetilde{g})^*$ is computed point-wise by axiom (FDer0 left) even for $n$-ary $\widetilde{g}$. 
Therefore the preservation of the coCartesianity is clear, and the preservation of the conditions of being well-supported and 2-Cartesian follows from Proposition~\ref{PROPPROPERTIESCORCOMPM}, 1--2.
The fact that the inputs are well-supported is needed for the support condition of the Proposition. 

2.\@ The functor $(\pi_{234}^*\widetilde{\iota})_!$, which exists by Axiom (F1), it is not automatically computed point-wise. We will show below in two steps that it is actually computed point-wise {\em on the specified full subcategory}. Therefore again, the preservation of good support is clear, and the preservation of Cartesianity and coCartesianity conditions follows from Proposition~\ref{PROPPROPERTIESCORCOMPM}, 2--3.

STEP 1: Note that $\pi_{12}: \twwc I \rightarrow \tw I$ is an opfibration. For $\alpha \in \tw I$ denote $e_\alpha: (\twwc I)_\alpha \hookrightarrow \twwc I$ the inclusion 
of the fiber. 
We claim that the functor $\widetilde{\iota}_!$ commutes with $e_{\alpha}^*$ for $e_\alpha: (\twwc I)_\alpha \hookrightarrow \twwc I$ denoting the 
inclusion of the fiber. By Lemma~\ref{LEMMAPOINTWISEEXBYZERO}, 1.\@ we only have to show that for each coCartesian $\rho: \mu \rightarrow \tau_\bullet \mu$ the square
\[ \xymatrix{
(\widetilde{A}')^{\op}(\pi_{234}(\tau_\bullet\mu)) \ar@{->>}[rr]^-{\widetilde{S}^{\op}(\pi_{234}\rho)} \ar@{^{(}->}[d] && (\widetilde{A}')^{\op}(\pi_{234}( \mu)) \ar@{^{(}->}[d] \\
\widetilde{T}^{\op}(\pi_{234}(\tau_\bullet\mu)) \ar@{->>}[rr]_-{\widetilde{T}^{\op}(\pi_{234}\rho)} && \widetilde{T}^{\op}(\pi_{234}(\mu))
} \]
is Cartesian. However, $\pi_{234}(\rho)$ is of type 1 and hence the horizontal morphisms are proper and the vertical morphisms come from morphisms of type 2 and are hence dense embeddings. The square is therefore Cartesian by Lemma~\ref{LEMMACARTDIA2}. 

STEP 2: We are thus reduced to show that on $\DD((\twwc I)_\alpha)_{e_\alpha^* \pi_{234}^* \widetilde{S}^{\op}}$ 
the functor $\iota_!$ is computed point-wise when restricted to 4-coCartesian, and well-supported objects. Note however, that all morphisms in the fiber are compositions of ones of type 3 and type 4 and hence the objects are thus (absolutely) coCartesian. 
Using Lemma~\ref{LEMMAPOINTWISEEXBYZERO}, 3.\@, we are reduced to show that 
for each morphism $\mu: x \rightarrow z$ in $(\twwc I)_\alpha$ denoting $H:=e_\alpha^* \pi_{234}^* (\widetilde{A}')^{\op}(\mu)$ and $h:=e_\alpha^* \pi_{234}^* \widetilde{T}^{\op}(\mu)$
the morphism
\[  \widetilde{\iota}(y)_!  H^* \mathcal{E}_x \rightarrow h^* \widetilde{\iota}(x)_! \mathcal{E}_x   \]
is an isomorphism for 4-coCartesian, and strongly 3-coCartesian objects $\mathcal{E}$.
Here, we wrote $\mathcal{E}_x$ for $x^* \mathcal{E}$. 
Writing $\mu = \mu_3 \circ \mu_4: x \rightarrow y \rightarrow z$ where $\mu_3$ is of type 3, and $\mu_4$ is of type 4, and denoting $h = h_4 \circ h_3$, resp.\@ $H= H_4 \circ H_3$ the corresponding factorization of $h$, resp.\@ $H$, 
we have thus
\begin{equation*}
\begin{array}{rcll}
 \widetilde{\iota}(z)_!  H^*  \mathcal{E}_x &\cong &  \widetilde{\iota}(z)_! H_3^* H_4^* \mathcal{E}_x   \\
 & \cong & \widetilde{\iota}(z)_! H_3^*  \mathcal{E}_y & \text{coCartesianity} \\
 & \cong & h_3^*  \widetilde{\iota}(y)_!  \mathcal{E}_y & \text{because $\mathcal{E}_y$ has support in $(\widetilde{A}')^{\op}(\pi_{234} z)$} \\
 & \cong &  h_3^* \widetilde{\iota}(y)_!  H_4^* \mathcal{E}_x & \text{coCartesianity} \\
 & \cong &  h_3^* h_4^*  \widetilde{\iota}(x)_! \mathcal{E}_x & \text{Proposition~\ref{PROPPROPERTIESCORCOMPM}, 2.} \\
 & \cong & h^* \widetilde{\iota}(x)_! \mathcal{E}_x
 \end{array}
\end{equation*}
It is not true that, in general, $\widetilde{\iota}_!$ commutes with $h^*$ here! The support conditions are essential.

3.\@ The functors $(\pi_{234}^*\widetilde{f})_*$, which exist by (FDer0 right) are also computed point-wise {\em for 1-ary $\widetilde{f}$}. Therefore the preservation of Cartesianity is clear, and the preservation of good support and coCartesianity conditions follows from from Proposition~\ref{PROPPROPERTIESCORCOMPM}, 1.\@ and 3.
The fact that the input is well-supported is needed for the support condition in 1.\@ of the Proposition. 
\end{proof}

\begin{PROP}\label{PROPPROPERTIESCORCOMPMDIA}
Let $\tau$ be a tree, let $I$ be in $\Dirlf$, let
$X \hookrightarrow \overline{X}$ be an object
 in $\Cor^{\comp}_I(\tau)$, and let 
$\widetilde{X}: \tww (\tau \times I) \rightarrow \mathcal{S}$ be the associated interior compactification. 
Then the functors constructed in \ref{LEMMAEXISTENCE3FUNCTORS} commute. More precisely: 
Consider a diagram of the form
\[ \xymatrix{
w_1, \dots, w_n \ar[r]^-d \ar[d]_c & z_1,\dots,z_n \ar[d]^{a} \\
y \ar[r]_-b & x
} \]
in $^{\downarrow\downarrow\downarrow} \tau$  {\em[sic!]} (numbered with indices 2--4) 
where $a$ and $c$ are of some type 2,3 or 4 (1-ary in case of type $\not= 4$) and $b$ and $d$ are of some type 2,3, or 4, respectively.  

Then we have a canonical isomorphism\footnote{In each case an exchange (or its inverse) of a natural commutation given by functoriality of $\DD$ we we leave to the reader to construct}
\begin{equation} \label{eqexchange} \widetilde{X}(a')_{\#} \widetilde{X}(d')_{\#} \rightarrow \widetilde{X}(b')_{\#} \widetilde{X}(c')_{\#}   \end{equation}
where $a' \in \tww \tau$ denotes the morphism corresponding to $a \in {}^{\downarrow\downarrow\downarrow} \tau$, etc.
In case $a$ and $c$ are of type 4, and $n$-ary, $\widetilde{X}(a')_{\#}$ and $\widetilde{X}(c')_{\#}$ are multivalued functors and $\widetilde{X}(d')_{\#}$ denotes the corresponding $n$-tupel of functors (of which only one is not an identity). 
\end{PROP}

\begin{proof}
This follows from Proposition~\ref{PROPPROPERTIESCORCOMPM} taking into account (cf.\@ Lemma~\ref{LEMMAEXISTENCE3FUNCTORS})
that strongly (co)Cartesian can be checked point-wise.
\end{proof}

Note that any multimorphism $\alpha$ in  ${}^{\downarrow\downarrow\downarrow} \tau$ can be decomposed {\em canonically} into a sequence 
of functors $\alpha = abc$ where $a$ is of type 2, $b$ of type 3 etc., and $a$ and $b$ are 1-ary. Taking the cubical structure of ${}^{\downarrow\downarrow\downarrow} \tau$ and the  discussion in \ref{TWGENREL}
into account Proposition~\ref{PROPPROPERTIESCORCOMPMDIA} may be restated as follows
\begin{PROP}\label{PROPPF}
The association 
\[ \alpha = abc \mapsto \widetilde{X}(\alpha)_{\#} := \widetilde{X}(a')_{\#} \widetilde{X}(b')_{\#} \widetilde{X}(c')_{\#}, \] 
where $a'$ is the (multi)morphism of $\tww \tau$ corresponding to $a$ in ${}^{\downarrow\downarrow\downarrow} \tau$, etc.\@,  
 defines a well-defined pseudo-functor on ${}^{\downarrow\downarrow\downarrow} \tau$ such that the isomorphism (\ref{eqexchange}) becomes the one
induced by the pseudo-functoriality.
\end{PROP}

The following works in the non-multi-case, multi-case, and symmetric multi-case alike. 
If the reader is interested in a non-multi-situation, consider $\SSS^{\op}$ as a usual pre-derivator, i.e.\@ forget its multistructure, and
let all trees $\tau$ be of the form $\Delta_n$. 
\begin{DEF}\label{DEFDER6FU1}
Assume that $\DD \rightarrow \SSS^{\op}$ is a (symmetric) fibered multiderivator with domain $\Dirlf$. 
Let $\tau$ be a tree, and let $I \in \Dirlf$. We define a category 
\[ E_I(\tau) \quad (\text{resp. } E_I(\tau^S))   \]
with objects pairs $(X \hookrightarrow \overline{X}, \mathcal{F})$ of an object $(X \hookrightarrow \overline{X}) \in \Cor^{\comp}_I(\tau)$ (cf.\@ Definition~\ref{DEFCORCOMP}), and an object
\[  \mathcal{F} \in  \Fun(\twwc \tau, \DD(\twwc I))^{4-\cocart, 3-\cocart^*, 2-{\cart}}_{\pi_{234}^* \widetilde{X}^{\op}} \]
(functors of multicategories) where $\widetilde{X}$ is the interior compactification associated with  $(X \rightarrow \overline{X})$ (cf.\@ \ref{DEFEXTCOMP}). Note that $\twwc I \in \Dirlf$ by Lemma~\ref{LEMMACATLF}. 

Morphisms $\mathcal{F} \rightarrow \mathcal{F}'$ are the morphisms $\xi: (X_1 \hookrightarrow \overline{X}_1) \rightarrow (X_2 \hookrightarrow \overline{X}_2)$ in $\Cor^{\comp}_I(\tau)$ together with an isomorphism
\[ ( \pi_{234}^*\widetilde{f})_{*} \mathcal{F} \Rightarrow \mathcal{F}'  \]
where $\widetilde{f}: \widetilde{X}_1 \rightarrow \widetilde{X}_2$ is the corresponding morphism between interior compactifications. 

Note that, in the symmetric case, {\em all} functors $\tau^S \rightarrow (\tau')^S$ induce functors $E_I(\tau^S) \rightarrow E_I((\tau')^S)$ using the symmetry of $\DD$.

Finally, we say that a morphism is a {\em weak equivalence}, if the underlying morphism in $\Cor^{\comp}_I(\tau)$ is a weak equivalence, i.e.\@ if the morphism $X_1 \rightarrow X_2$ is  an isomorphism. 
\end{DEF}

\begin{FUNDLEMMA}
\label{LEMMACOMPSIXFU}With the notation as in Definition~\ref{DEFDER6FU1}. 
\begin{enumerate}
\item Let $\tau$ be a tree, and consider an object $(X \hookrightarrow \overline{X})$ in $\Cor_I^{\comp}(\tau^S)$.
Let $\widetilde{X}: \tww ( \tau \times I) \rightarrow \mathcal{S}$ be the associated interior compactification (\ref{PARMULTICOMP}). 

Let $(\mathcal{E}_o)_{o \in \tau}$ be a collection of objects with $\mathcal{E}_o \in \DD(\twwc I)^{4-\cocart, 3-\cocart^*, 2-{\cart}}_{\pi_{234}^* \widetilde{X}_o^{\op}}$, where
$\widetilde{X}_o$ is the value of $\widetilde{X}$ at $\tww o$. 
Then the fiber\footnote{the fiber over $(X \hookrightarrow \overline{X}, \{\mathcal{E}_o\}_{o \in \tau})$ for the functor $E_I(\tau) \rightarrow \Cor^{\comp}_I(\tau) \times \prod_o E_I(\cdot)$  }
\[ E_I(\tau)_{(X \hookrightarrow \overline{X})}(\{\mathcal{E}_o\}_{o \in \tau})  \] 
 is equivalent to a set.

\item For $\tau = \tau_1 \circ_i \tau_2$, where $i$ is a source object of $\tau_1$ (which we identify with the
final object of $\tau_2$),
the square
\[ \xymatrix{
E_I(\tau)_{(X \hookrightarrow \overline{X})}((\mathcal{E}_o)_{o \in \tau}) \ar[r] \ar[d] & E_I(\tau_1)_{(X_1 \hookrightarrow \overline{X}_1)}((\mathcal{E}_o)_{o \in \tau_1}) \ar[d] \\
E_I(\tau_2)_{(X_2 \hookrightarrow \overline{X}_2)}((\mathcal{E}_o)_{o \in \tau_2}) \ar[r] & \cdot 
} \]
is 2-Cartesian. Hence if we consider the $E_I(\tau)_{(\cdots)}(\cdots)$ as sets, we have
\[ E_I(\tau)_{(X \hookrightarrow \overline{X})}((\mathcal{E}_o)_{o \in \tau}) \cong E_I(\tau_1)_{(X_1 \hookrightarrow \overline{X}_1)}((\mathcal{E}_o)_{o \in \tau_1}) \times E_I(\tau_2)_{(X_2 \hookrightarrow \overline{X}_2)}((\mathcal{E}_o)_{o \in \tau_2}). \]

\item For $\tau = \Delta_{1,n}$, we have canonically an isomorphism of sets
\[ E_I(\Delta_{1,n})_{(X \hookrightarrow \overline{X})}(\mathcal{E}_1, \dots, \mathcal{E}_n, \mathcal{E}_{n+1}) \cong \Hom_{\DD(\twwc I)_{\pi_{234}^* \widetilde{T}^{\op}}}( \widetilde{f}_* \widetilde{\iota}_! \widetilde{g}^*(\mathcal{E}_1, \dots, \mathcal{E}_n); \mathcal{E}_{n+1})  \]
for any choice of pull-back $\widetilde{g}^*$, push-forward $\widetilde{f}_*$, and adjoint $\widetilde{\iota}_!$ to the pull-back $\widetilde{\iota}^*$.
Here $\widetilde{T}$ is the restriction of $\widetilde{X}$ to $\tww (n+1)$, and $\widetilde{g}$, $\widetilde{\iota}$, and $\widetilde{f}$ are the components of $\widetilde{X}$ as in \ref{EXCOMPONENTS}. 

An object $\mathcal{F}$ on the left hand side is mapped to an isomorphism if and only if it is also Cartesian (or equivalently coCartesian) w.r.t.\@ the projection 
\[ \pi_{234} \times \id: \twwc (\Delta_{1,n} \times I) \rightarrow {}^{\uparrow \uparrow \downarrow}{\Delta_{1,n}} \times \twwc I. \]

\item Let $f: (X_1 \hookrightarrow \overline{X}_1) \rightarrow (X_2 \hookrightarrow \overline{X}_2)$ be a morphism in $\Cor^{\comp}_I(\tau)((X_o \hookrightarrow \overline{X}_o)_{o \in \tau})$. We obtain an associated refinement $\widetilde{f}: \widetilde{X}_1 \rightarrow \widetilde{X}_2$ of interior compactifications. 
Then the functors
$(\pi_{234}^*\widetilde{f})^*, (\pi_{234}^*\widetilde{f})_*$ are mutually inverse bijections 
\[ \xymatrix{ E_I(\tau)_{(X_1 \hookrightarrow \overline{X}_1)}((\mathcal{E}_o)_{o \in \tau}) \ar@/^10pt/[r]^-{(\pi_{234}\widetilde{f})_*} & \ar@/^10pt/[l]^-{(\pi_{234}^*\widetilde{f})^*} E_I(\tau)_{(X_2 \hookrightarrow \overline{X}_2)}((\mathcal{E}_o)_{o \in \tau}) }  \]
\end{enumerate}
\end{FUNDLEMMA}

\begin{proof}
Given an object $i = (i_1 \to i_2 \to i_3 \to i_4)$ in $\twwc \tau$
we have a sequence of morphisms in $\twwc \tau$ writing $i_1 = [i_1^{(1)}, \dots, i_1^{(n)}]$: 

\[\xymatrix{
(i_1^{(1)} \to i_1^{(1)} \to i_1^{(1)} \to i_1^{(1)}) \ar@{-}[rd]& \cdots & (i_1^{(n)} \to i_1^{(n)} \to i_1^{(n)} \to i_1^{(n)}) \ar@{-}[ld]  \\
& \ar[d]^{\alpha_{i,4}} \\
& (i_1 \to i_1 \to i_1 \to i_4) \ar@{<-}[d]^{\alpha_{i,3}} \\
& (i_1 \to i_1 \to i_3 \to i_4) \ar@{<-}[d]^{\alpha_{i,2}} \\
& (i_1 \to i_2 \to i_3 \to i_4) 
}   \]
in which $\alpha_{i,m}$ is of type $m$.  
Note that any of the objects $\twwc i_1^{(k)}$ lies either in $\twwc \tau_1$ or $\twwc \tau_2$.

1.\@ 
In view of the observation above, the statement follows from the uniqueness (up to unique isomorphism) of
the source (resp.\@ target) of a (strongly) Cartesian (resp.\@ coCartesian) morphism. 

2.\@ Given restrictions of $\mathcal{F}$ to the union of $\twwc \tau_1$ and $\twwc \tau_2$, we construct
an extension to $\twwc (\tau_2 \circ  \tau_1)$ as follows: 
Every morphism $\alpha_{i,m}$ as above has a corresponding morphism $\alpha_{i,m}' \in  {}^{\downarrow\downarrow\downarrow\downarrow} \tau$ and
we define an extension on objects by setting for $i \not\in \twwc \tau_1 \cup \twwc \tau_2$
\[ \mathcal{F}(i) := \widetilde{X}(\pi_{234}(\alpha_{i,2}'\alpha_{i,3}'\alpha_{i,4}'))_\#(\mathcal{F}(\twwc i_1^{(1)}), \cdots, \mathcal{F}(\twwc i_1^{(n)}))  \]

Observe that for $i \in \twwc \tau_1$ and $i \in \twwc \tau_2$ there is already a {\em canonical} isomorphism
\begin{equation} \label{eqrewrite} \mathcal{F}(i) \cong \widetilde{X}(\pi_{234}(\alpha_{i,2}'\alpha_{i,3}'\alpha_{i,4}'))_\#(\mathcal{F}(\twwc i_1^{(1)}), \cdots, \mathcal{F}(\twwc i_1^{(n)}))  \end{equation}
because of the (strong) (co)Cartesianity conditions on $\mathcal{F}|_{\twwc \tau_1}$ and $\mathcal{F}|_{\twwc \tau_2}$. 

For a generating morphism $\beta: i \rightarrow j$ of type 2, 3, or 4 there is an isomorphism coming from the pseudo-functoriality of $\widetilde{X}(\cdot)_\#$ (cf.\@ Proposition~\ref{PROPPF})
\[ \widetilde{X}(\pi_{234}(\beta'))_\# \widetilde{X}(\alpha_{i,2}'\alpha_{i,3}'\alpha_{i,4}')_\# \cong \widetilde{X}(\alpha_{j,2}'\alpha_{j,3}'\alpha_{j,4}')_\# \]
and thus  a coCartesian morphism (for type 4),  resp.\@ strongly coCartesian morphism (for type 3), resp.\@ a Cartesian morphism (for type 2)
\[ \widetilde{X}(\pi_{234}(\alpha_{i,2}'\alpha_{i,3}'\alpha_{i,4}'))_\# (\mathcal{F}(\twwc i_{1}^{(1)}), \dots, \mathcal{F}(\twwc i_{1}^{(n)})) \rightarrow \widetilde{X}(\pi_{234}(\alpha_{j,2}'\alpha_{j,3}'\alpha_{j,4}'))_\# (\mathcal{F}(\twwc i_{1}^{(1)}), \dots, \mathcal{F}(\twwc i_{1}^{(n)}))   \]
over $\widetilde{X}(\pi_{234}(\beta))$ which we define to be $\mathcal{F}(\beta)$. For a {\em generating} morphism $\beta: i \rightarrow j$ of type 1 we get the following diagram writing $j_1 = [(j_1)^{(1)}, \dots, (j_1)^{(n)}]$ and in which $i_1^{(1)}, \dots, i_1^{(n')}$ are possibly still  lists of objects: 
\[  \xymatrix{
{ \substack{ (i_1^{(1)} \to j_1^{(1)} \to j_1^{(1)} \to j_1^{(1)})  \\ \cdots \\  (i_1^{(n')} \to j_1^{(n')} \to j_1^{(n')} \to j_1^{(n')}) } } \ar[d]^{\gamma_{4}}  \ar[r]^{\substack{ \beta_1^{(1)}\\  \dots\\  \beta_1^{(n')}}}  & 
{ \substack{ (j_1^{(1)} \to j_1^{(1)} \to j_1^{(1)} \to j_1^{(1)}) \\ \cdots \\ (j_1^{(n')} \to j_1^{(n')} \to j_1^{(n')} \to j_1^{(n')})  }}  \ar[d]^{\alpha_{j,4}}  \\
 (i_1 \to j_1 \to j_1 \to i_4) \ar@{<-}[d]^{\gamma_3}&  (j_1 \to j_1 \to j_1 \to i_4) \ar@{<-}[d]^{\alpha_{j,3}} \\
 (i_1 \to j_1 \to i_3 \to i_4) \ar@{<-}[d]^{\gamma_2}&  (j_1 \to j_1 \to i_3 \to i_4) \ar@{<-}[d]^{\alpha_{j,2}} \\
 (i_1 \to i_2 \to i_3 \to i_4)  \ar[r]^\beta & (j_1 \to i_2 \to i_3 \to i_4)  
}   \]
Because $\beta$ is generating, only one of the $\beta_1^{(k)}$ is not an identity and lies entirely in $\twwc \tau_1$ or $\twwc \tau_2$. 
Consider the morphism 
\[ \mathcal{F}(\beta_1^{(k)}): \mathcal{F} (i_1^{(k)} \to j_1^{(k)} \to j_1^{(k)} \to j_1^{(k)})  \rightarrow  \mathcal{F}( \twwc j_1^{(k)} ) . \]
Applying $\widetilde{X}(\pi_{234}(\alpha_{j,2}'\alpha_{j,3}'\alpha_{j,4}'))_{\#}$ to it (and the various identities), we get
\begin{gather*}
 \widetilde{X}(\pi_{234}(\alpha_{j,2}'\alpha_{j,3}'\alpha_{j,4}'))_{\#}  \mathcal{F} (i_1^{(k)} \to j_1^{(k)} \to j_1^{(k)} \to j_1^{(k)})    \\ \rightarrow
  \widetilde{X}(\pi_{234}(\alpha_{j,2}'\alpha_{j,3}'\alpha_{j,4}'))_{\#} \mathcal{F}( \twwc j_1^{(k)}).
\end{gather*}
Using  (\ref{eqrewrite}), and the pseudo-functoriality of $\widetilde{X}(\cdot)_\#$ (cf.\@ Proposition~\ref{PROPPF}), the left hand side might be rewritten as 
\[ \widetilde{X}(\alpha_{i,3}'\alpha_{i,2}'\alpha_{i,1}')_{\#}  \mathcal{F} (\twwc i_1^{(k)})  \rightarrow  \widetilde{X}(\alpha_{j,2}'\alpha_{j,3}'\alpha_{j,4}')_{\#} \mathcal{F}(\twwc j_1^{(k)}) . \]
This yields a morphism
\[ \mathcal{F}(i) \rightarrow \mathcal{F}(j) \]
 which we define to be $\mathcal{F}(\beta)$. 
To get a valid diagram $\mathcal{F}$ one has to check that the relation squares are mapped to commutative diagrams. This follows from the pseudo-functoriality of $\widetilde{X}(\cdot)_\#$ (cf.\@ Proposition~\ref{PROPPF}). 

3.\@ is clear. 

4. Note that the functors $(\pi_{234}^*\widetilde{f})^*$ and $(\pi_{234}^*\widetilde{f})_*$ are both computed point-wise. By Lemma~\ref{LEMMAEXISTENCE3FUNCTORS}, 3.\@ the
functor $(\pi_{234}^*\widetilde{f})_*$ preserves the conditions of being $4$-coCartesian, strongly $3$-coCartesian, and $2$-Cartesian. 
For each element $i_1 \rightarrow i_2 \rightarrow i_3 \rightarrow i_4 \in \twwc \tau$ where $i_1, \dots, i_4$ are lists of objects in $\tau$, with $i_4$ containing one object, 
consider
\[ \xymatrix{
  \widetilde{X}_1( i_2 = i_2 \rightarrow i_4) \ar@{=}[d] \ar@{^{(}->}[r] & \widetilde{X}_1( i_2 \rightarrow i_3 \rightarrow i_4) \ar@{->>}[d]^{\widetilde{f}(i_2 \rightarrow i_3 \rightarrow i_4)} \ar@{->>}[r] & \widetilde{X}_1( i_3 \rightarrow i_3 \rightarrow i_4) \ar@{=}[d] \\
  \widetilde{X}_2( i_2 = i_2 \rightarrow i_4)  \ar@{^{(}->}[r] & \widetilde{X}_2( i_2 \rightarrow i_3 \rightarrow i_4)  \ar@{->>}[r] & \widetilde{X}_2( i_3 \rightarrow i_3 \rightarrow i_4) 
} \]

By Lemma~\ref{LEMMACART1} the left square is (point-wise) Cartesian. Therefore by (F5) and (F6) $\widetilde{f}_*$ and $\widetilde{f}^*$ are inverse (up to isomorphism) to each other on strongly $3$-coCartesian objects. 
By Lemma~\ref{LEMMAEXISTENCE3FUNCTORS}, $\widetilde{f}_*$ preserves the conditions of being 2-Cartesian and 4-coCartesian, and $\widetilde{f}^*$ preserves the condition of being 4-coCartesian. We claim that the latter also preserves the condition of being 2-Cartesian. Indeed for strongly 3-coCartesian objects it suffices to check 2-Cartesianity over morphisms of the form
\[ \xymatrix{ (i_1 \rightarrow i_2 \rightarrow i_3 \rightarrow i_4) \ar@{->>}[r] & (i_1 \rightarrow i_3 = i_3 \rightarrow i_4)  } \]
where, however, $\widetilde{f}_*$ reflects 2-Cartesianity by the diagram above.

Therefore both functors preserve all conditions of being (strongly) (co)Cartesian. 
Since they are adjoints (between groupoids) they become mutually inverse when we pass to the sets of isomorphism classes.  
\end{proof}

\begin{LEMMA}\label{LEMMADER6FU3}
\begin{enumerate}
\item The forgetful functors
\[ E_I(\tau) \rightarrow \Cor_I^{\comp}(\tau) \qquad   E_I(\tau)_{(\mathcal{E}_o)}[\mathcal{W}_{(\mathcal{E}_o)}^{-1}] \rightarrow \Cor_I^{\comp}(\tau)_{(X_o)}[\mathcal{W}^{-1}_{(X_o)}] \]
(cf.\@ Definition~\ref{DEFDER6FU1}) are bifibrations with discrete fibers. 
\item The strict 2-functor 
\[ \Delta_n \mapsto E_I(\Delta_n) \quad (\text{resp. } \tau \mapsto E_I(\tau), \quad \text{resp. } \tau^S \mapsto E_I(\tau^S))   \]
satisfies the properties of Proposition~\ref{PROPCONSTR2CAT}, resp.\@ of Proposition~\ref{PROPCONSTRSYMMULTI}.
\end{enumerate}
\end{LEMMA}
\begin{proof}
1.\@ The statement regarding the first functor follows directly from Lemma~\ref{LEMMACOMPSIXFU}, 1.\@ and 4. In a bifibration with discrete fibers all
push-forward and pull-back functors are isomorphisms (regarding the fibers as sets), hence also the statement for the second functor \footnote{Up to equivalence one can also see the opfibration $E_I(\tau)_{(\mathcal{E}_o)} \rightarrow \Cor_I^{\comp}(\tau)_{(X_o)}$ as the Grothendieck construction of a functor $\Cor_I^{\comp}(\tau)_{(X_o)} \rightarrow \mathcal{SET}$ mapping all morphisms to isomorphisms. Thus it induces a functor $\Cor_I^{\comp}(\tau)_{(X_o)}[\mathcal{W}^{-1}_{(X_o)}] \rightarrow \mathcal{SET}$. Its associated Grothendieck construction is precisely  $E_I(\tau)_{(\mathcal{E}_o)}[\mathcal{W}^{-1}_{(\mathcal{E}_o)}]$. }. 

2.\@ Axiom 1 (the surjectivity on objects) is clear and the equivalence in Axiom 2 follows from 1., Lemma~\ref{LEMMACOMPSIXFU}, 2., and the validity of the axioms for $\Cor_I^{\comp}$ (Lemma~\ref{LEMMAPROPCONSTRSYMMULTICOMP}, 2.). Indeed, we get  a strictly commutative diagram (notation for the non-symmetric case)
\begin{equation*} \xymatrix{
  E_I(\tau)_{(\mathcal{E}_o)}[\mathcal{W}^{-1}_{(\mathcal{E}_o)}]  \ar[r] \ar[d] &  \prod_{m} E_I(\Delta_{1,k_m})_{(\mathcal{E}_o)}[\mathcal{W}^{-1}_{(\mathcal{E}_o)}]  \ar[d] \\
  \Cor_I^{\comp}(\tau)_{(X_o)}[\mathcal{W}^{-1}_{(X_o)}]  \ar[r] &  \prod_{m} \Cor_I^{\comp}(\Delta_{1,k_m})_{(X_o)}[\mathcal{W}^{-1}_{(X_o)}] 
 }
\end{equation*}
in which the vertical functors are bifibrations with discrete fibers and the lower vertical functor is an equivalence. 
Induction on Lemma~\ref{LEMMACOMPSIXFU}, 2. shows that the upper horizontal functor induces an equivalence on fibers. The statement follows. 
\end{proof}

\begin{DEF}\label{DEFDER6FU2}With the notation as in \ref{BEGINSECTIONDER6FU},
we construct a morphism of (symmetric) pre-2-(multi)derivators with domain $\Dirlf$
\[
 \EE \rightarrow \SSS^{\cor, \comp} .
\]
For a diagram $I \in \Dirlf$, let $\EE(I)$ be the (symmetric) 2-(multi)category obtained by applying Proposition~\ref{PROPCONSTR2CAT} (resp.\@ Proposition~\ref{PROPCONSTRSYMMULTI}) to the strict functor of Definition~\ref{DEFDER6FU1}
\[ \Delta_n \mapsto E_I(\Delta_n) \quad (\text{resp. } \tau \mapsto E_I(\tau), \quad \text{resp. } \tau^S \mapsto E_I(\tau^S)).   \]
It comes equipped with an obvious morphism to $\SSS^{\cor, \comp}(I)$ which was constructed applying the same Proposition to 
$\tau^S \mapsto \Cor_I^{\comp}(\tau^S)$. 
For a functor $\alpha: I \rightarrow J$ in $\Catlf$ we define the pullback $\alpha^*=\EE(\alpha)$
to be $\DD(\twwc \alpha \times \id \times \id)$. Note that $\alpha$ induces a functor $\twwc \alpha: \twwc I \rightarrow \twwc J$ and that $\DD(\twwc \alpha \times \id \times \id)$ preserves the relevant conditions of being (co)Cartesian. 
The pre-2-multiderivator $\EE$ is defined on natural transformations as follows. A natural transformation $\mu: \alpha \Rightarrow \beta$ can be seen as a functor $\mu: \Delta_1 \times I \rightarrow J$.
The pullback of a diagram in $\mathcal{E} \in \EE(\twwc J \times \Delta^{\op} \times \Delta)$ and taking partial underlying diagram gives a functor in   $\Fun(\twwc \Delta_1, \EE(\twwc I \times \Delta^{\op} \times \Delta))$ which has the correct
(strong) (co)Cartesianity conditions. It is, by definition, a morphism \[ \alpha^* \mathcal{E} = e_0^* \mu^* \mathcal{E} \rightarrow  \beta^* \mathcal{E} = e_1^*  \mu^* \mathcal{E} \]  in $\EE(I)$ which we 
define to be the pseudo-natural transformation $\EE(\mu)$ at $\mathcal{E}$. 
We proceed to describe the pseudo-naturality constraint of $\EE(\mu)$ and restrict to 1-ary morphisms for simplicity. 
 A 1-ary 1-morphism given by $\xi \in E_I(\Delta_1)$ (note that a general 1-ary 1-morphism in $\EE(I)$ is freely generated by those) consists of a compactification $X \hookrightarrow \overline{X}$ in $\mathcal{S}^{\tw (\Delta_1 \times I)}$ plus an object
  \[ \mathcal{F} \in \Fun(\twwc \Delta_1,  \DD( \twwc   I))^{4-\cocart, 3-\cocart^*, 2-{\cart}}_{\pi_{234}^*\widetilde{X}^{\op}}.  \]
  It is a morphism from $\mathcal{E}_1 := e_0^*\mathcal{F}$ to $\mathcal{E}_2 := e_1^*\mathcal{F}$.
We get an object $\mu^*\mathcal{F}$ and after taking partial underlying diagrams an object 
 \[ (\mu^* \mathcal{F})' \in  \Fun(\twwc \Delta_1^2,  \DD( \twwc   I))^{4-\cocart, 3-\cocart^*, 2-{\cart}}_{ \pi_{234}^*(\mu^* \widetilde{X}^{\op})}.  \]

The two (non-degenerate) embeddings $\Delta_2 \rightarrow \Delta_1^2$ yield two 2-isomorphisms
\[ (\mu^* \mathcal{E}_1) \circ  (\alpha^* \mathcal{F}) \Rightarrow ( \twwc \delta )^* (\mu^* \mathcal{F})' \Leftarrow  (\beta^* \mathcal{F})  \circ (\mu^* \mathcal{E}_2).      \]
where $\delta: \Delta_1 \rightarrow \Delta_1^2$ denotes the diagonal. This endows $\EE(\mu): \alpha^* \Rightarrow \beta^*$ with the structure of pseudo-natural transformation and one checks that this construction yields a pseudo-functor:
\[ \Fun(I, J) \rightarrow \Fun^{\mathrm{strict}}(\EE(J), \EE(I)). \]
\end{DEF}

The next goal is to establish that the 
morphism of (symmetric) pre-2-(multi)derivators
\[ \EE \rightarrow \SSS^{\cor, \comp}  \]
is a left fibered (symmetric) (multi)derivator with domain $\Dirlf$.
If the input fibered multiderivator has stable, well-generated fibers, it will be right fibered as well (using Brown representability).

Using the equivalence of 2-pre-multiderivators $\SSS^{\cor, \comp}  \cong \SSS^{\cor}$, this will finally allow to construct
the desired derivator six-functor-formalism, i.e.\@ fibered multiderivator
\[ \EE' \rightarrow \SSS^{\cor}  \]
using the techniques of \cite{Hor17b} even with domain $\Cat$. 

Axiom (FDer0 left) will follow from the following Lemma whose proof is an immediate consequence of the discussion above.  

\begin{LEMMA}\label{LEMMA6FUOPFIB}
The functor of (symmetric) 2-(multi)categories
\begin{eqnarray*}
 \EE(I) &\rightarrow& \SSS^{\cor, \comp}(I) 
 \end{eqnarray*}
constructed in Definition~\ref{DEFDER6FU2} is a 1-opfibration and 2-bifibration with 1-categorical fibers. 
\end{LEMMA}
\begin{proof}
The functors have 1-categorical fibers and are 2-bifibered by Definition~\ref{DEFDER6FU1}, 3.
By Lemma~\ref{LEMMACOMPSIXFU}, 3.\@, weakly coCartesian morphisms which consist of a single element $\mathcal{F} \in E_I(\Delta_{1,n})_{(X \hookrightarrow \overline{X})}$
(cf.\@ Proposition~\ref{PROPCONSTRSYMMULTI}) are exactly those such that the corresponding object $\mathcal{F}$ is also $\pi_{234} \times \id$-(co)Cartesian. For an object 
\[ \mathcal{F} \in E_I(\Delta_{1,n} \circ^i \Delta_{1,m})_{(X \hookrightarrow \overline{X})}  \] 
One verifies that, if the pull-backs of $\mathcal{F}$ to $E_I(\Delta_{1,n})$ and $E_I(\Delta_{1,m})$ are $\pi_{234} \times \id$-(co)Cartesian (i.e.\@ morphisms of type 1 in $\tau$ are mapped to isomorphisms) then also the whole object is $\pi_{234} \times \id$-(co)Cartesian and hence so is its pullback to 
the composition $\Delta_{1,n+m-1}$. It follows that the composition of weakly coCartesian morphisms are weakly coCartesian and hence the functor is a 1-opfibration by \cite[Proposition~2.7]{Hor15b}. Note that in the statement that the composition of weakly coCartesian morphisms is weakly coCartesian all three involved morphisms may be replaced by an isomorphic one (hence here by one consisting of a single element) because we have a 2-bifibration. 
\end{proof}

\section{Relative Kan extensions}

\begin{LEMMA}\label{LEMMAOPFIB}
Let $\alpha: I \rightarrow J$ be an opfibration in $\Cat$, and consider the sequence of functors:
\[ \xymatrix{ \twwc{I} \ar[rrr]^-{q_1=(\twwc{\alpha},\pi_{123})} &&& \twwc{J} \times_{({}^{\downarrow \uparrow \uparrow}{J})} {}^{\downarrow \uparrow \uparrow}{I} \ar[rr]^-{q_2=\id \times \pi_1} && \twwc{J} \times_J I. } \]
\begin{enumerate}
\item The functor $q_1$ is an opfibration. The fiber of $q_1$ over a pair $j_1 \rightarrow j_2 \rightarrow j_3 \rightarrow j_4$ and $i_1 \rightarrow i_2 \rightarrow i_3$ is
\[ i_4 \times_{/I_{j_4}} I_{j_4}\]
where $i_4$ is the target of a coCartesian arrow over $j_3 \rightarrow j_4$ with source $i_3$. 
\item The functor $q_2$ is a fibration. The fiber of $q_2$ over a pair $j_1 \rightarrow j_2 \rightarrow j_3 \rightarrow j_4$ and $i_1$ (lying over $j_1$) is 
\[ (i_2 \times_{/I_{j_2}} I_{j_2} \times_{/I_{j_3}} I_{j_3})^{\op} \]
where $i_2$ is the target of a coCartesian arrow over $j_1 \rightarrow j_2$ with source $i_1$ and the second comma category is 
constructed via the functor $I_{j_2} \rightarrow I_{j_3}$ being the coCartesian push-forward along $j_2 \rightarrow j_3$.  
\end{enumerate}
\end{LEMMA}

\begin{proof}Straightforward. \end{proof}

\begin{LEMMA}\label{LEMMAKAN2}
Under the assumptions of \ref{BEGINSECTIONDER6FU},
if $\alpha: I \rightarrow J$ is an opfibration in $\Dirlf$ then
the functors
\[ \xymatrix{ \DD(\twwc{J} \times_J I)_{\pi_{234}^*(\widetilde{S}^{\op})}^{4-\cocart, 3-\cocart^*, 2-{\cart}} \ar[d]^-{q_2^*} \\
 \DD(\twwc{J} \times_{({}^{\downarrow \uparrow \uparrow}{J})} {}^{\downarrow \uparrow \uparrow}{I})^{4-\cocart, 3-\cocart^*, 2-{\cart}}_{\pi_{234}^* (\widetilde{S}^{\op})}  \ar[d]^-{q_1^*} \\
   \DD(\twwc{I})^{4-\cocart, 3-\cocart^*, 2-{\cart}}_{\pi_{234}^*(\widetilde{S}^{\op})}  } \]
are equivalences.
In particular (applying this to $J=\cdot$ and variable $I$) we have an equivalence of fibers:
\[  \EE_S \cong \DD_S. \]
\end{LEMMA}
Here $\EE$ is the pre-2-multiderivator constructed in Definition~\ref{DEFDER6FU2}. Note that  $\EE_S$ is a usual pre-derivator, though.
\begin{proof}
We first treat the case of $q_1^*$.
We know by Lemma~\ref{LEMMAOPFIB} that $q_1$ is an opfibration with fibers of the form $i_4 \times_{/I_{j_4}} I_{j_4}$.
Neglecting the conditions of being (co)Cartesian, we know that $q_1^*$ has a left adjoint:
\[ q_{1,!}: \DD(\twwc{I})_{\pi_{234}^*(\widetilde{S}^{\op})}
\rightarrow \DD(\twwc{J} \times_{( {}^{\uparrow \uparrow \downarrow}{J})} {}^{\uparrow \uparrow \downarrow}{I})_{\pi_{234}^* (\widetilde{S}^{\op})} \]
We will show that the unit and counit
\[ \id \Rightarrow q_1^* q_{1,!}  \qquad  q_{1,!} q_1^*  \Rightarrow \id  \]
are isomorphisms {\em when restricted to the subcategory of $4$-coCartesian objects}. Since the conditions of being $2$-Cartesian and strongly $3$-coCartesian objects match under $q_1^*$ this shows the first assertion. 
Since $q_1$ is an opfibration this is the same as to show that for any object in $\twwc{J} \times_{({}^{\uparrow \uparrow \downarrow}{J})} {}^{\uparrow \uparrow \downarrow}{I}$ with fiber $F = i_4 \times_{/I_{j_4}} I_{j_4}$ the unit and counit
\begin{equation}\label{eq7} \id \Rightarrow p_F^* p_{F,!}  \qquad  p_{F,!} p_F^*  \Rightarrow \id  \end{equation}
are isomorphisms when restricted to the subcategory of $4$-coCartesian objects. Since all morphisms in the fiber $F$ are of type 4, we have to show that the morphisms in (\ref{eq7}) are isomorphisms when restricted to (absolutely) (co)Cartesian objects. This follows from the fact that  $F$ has an initial object \cite[Lemma 6.21 and Corollary~6.22]{Hor16}.

We now treat the case of $q_2^*$.
We know by Lemma~\ref{LEMMAOPFIB} that $q_2$ is a fibration with fibers of the form $(i_2 \times_{/I_{j_2}} I_{j_2} \times_{/I_{j_3}} I_{j_3})^{\op}$.
Neglecting the conditions of being (co)Cartesian, we know that $q_1^*$ has a right adjoint:
\[ q_{2,*}: \DD(\twwc{J} \times_{({}^{\uparrow \uparrow \downarrow}J)} {}^{\uparrow \uparrow \downarrow}{I})_{\pi_{234}^* (\widetilde{S}^{\op})}
\rightarrow \DD(\twwc{J} \times_{J} I)_{\pi_{234}^*(\widetilde{S}^{\op})} \]
We will show that the unit and counit
\[ \id \Rightarrow q_{2,*} q_2^*   \qquad  q_2^* q_{2,*}  \Rightarrow \id  \]
are isomorphisms {\em when restricted to the subcategory of $2$-Cartesian and strongly $3$-coCartesian objects}. Since the conditions of being $4$-coCartesian match under $q_2^*$ this shows the second assertion. 
Since $q_2$ is a fibration this is the same as to show that for any object in $\twwc{J} \times_J I$ with fiber $F=(i_2 \times_{/I_{j_2}} I_{j_2} \times_{/I_{j_3}} I_{j_3})^{\op}$ the the unit and counit
\begin{equation}\label{eq8} \id \Rightarrow  p_{F,*} p_F^*  \qquad  p_F^* p_{F,*}  \Rightarrow \id  \end{equation}
are isomorphisms when restricted to the subcategory of $2$-Cartesian and strongly $3$-coCartesian objects. 

Since every morphism in the fiber $(i_2 \times_{/I_{j_2}} I_{j_2} \times_{/I_{j_3}} I_{j_3})^{\op}$ is a composition of morphisms of type 2 and 3, this means that it suffices to show that (\ref{eq8}) are isomorphisms when restricted to (absolutely) (co)Cartesian objects. This follows from the fact that  $(i_2 \times_{/I_{j_2}} I_{j_2} \times_{/I_{j_3}} I_{j_3})^{\op}$ has a final object \cite[Lemma 7.21 and Corollary~7.22]{Hor16}. 
\end{proof}

\begin{LEMMA}~\label{LEMMAKAN4}
Let the situation be as in \ref{BEGINSECTIONDER6FU} and let $p': \EE \rightarrow \SSS^{\cor, \comp}$ be the morphism of 2-pre-multiderivators defined in Definition~\ref{DEFDER6FU2}. 
Let $\alpha: I \rightarrow J$ be an opfibration in $\Dirlf$ and $X \hookrightarrow \overline{X}$ an element of $\SSS^{\cor, \comp}(J)$. Then $\alpha^*: \EE(J)_{X \hookrightarrow \overline{X}} \rightarrow \EE(I)_{\alpha^*X \hookrightarrow  \alpha^*\overline{X}} $ has a left adjoint $\alpha_!^{(X\hookrightarrow \overline{X})}$.
\end{LEMMA}
\begin{proof}Let $\widetilde{X} \in \Fun(\tww J, \mathcal{S})$ be the corresponding interior compactification. 
We have to show that
\[ (\twwc{\alpha})^*: \DD(\twwc{J})^{4-\cocart, 3-\cocart^*, 2-{\cart}}_{\pi_{234}^*(\widetilde{X}^{\op}) } \rightarrow \DD(\twwc{I})^{4-\cocart, 3-\cocart^*, 2-{\cart}}_{\pi_{234}^* (\tww\alpha^*\widetilde{X})^{\op}} \]
has a left adjoint. The right hand side category is by Lemma~\ref{LEMMAKAN2} equivalent to  
\[ \DD((\twwc{J}) \times_J I)^{4-\cocart, 3-\cocart^*, 2-{\cart}}_{\pi_{234}^* (\widetilde{X}^{\op})}, \]
hence we have to show that
\[ \pr_1^*: \DD(\twwc{J})^{4-\cocart, 3-\cocart^*, 2-{\cart}}_{\pi_{234}^* (\widetilde{X}^{\op})} \rightarrow \DD((\twwc{J}) \times_J I)^{4-\cocart, 3-\cocart^*, 2-{\cart}}_{\pi_{234}^* (\widetilde{X}^{\op})} \]
has a left adjoint. By assumption, the functor has a left adjoint $\pr_{1,!}$ forgetting the coCartesianity conditions. It suffices to show that $\pr_{1,!}$ preserves the conditions of being $4$-coCartesian, $2$-Cartesian, and strongly $3$-coCartesian, respectively. 

{\em Strongly $3$-coCartesian:} Let $\kappa: j \rightarrow j'$
\[ \xymatrix{
j=(j_1 \ar[r]^{} \ar@<10pt>@{=}[d] &j_2 \ar[r]^{} \ar@{=}[d] & j_3 \ar[r]^{} & \ar@{=}[d]  j_4)  \\
j'=(j_1 \ar[r]_{} &j_2 \ar[r]_{} & j_3' \ar[u]_{} \ar[r]_{} & j_4)
} \]
be a morphism of type 3 in $\twwc{J}$. Denote \[ \iota:=\widetilde{X}(\pi_{234}(\kappa)): \widetilde{X}(\pi_{234}(j)) \rightarrow \widetilde{X}(\pi_{234}(j')) \] the corresponding morphism in $\mathcal{S}$. Note that $\iota$ is an open embedding by the properties of induced interior compactifications. 
 
We have to show that the induced map
\[ \iota^* j^* \pr_{1,!} \rightarrow (j')^* \pr_{1,!} \]
is an isomorphism on strongly $3$-coCartesian objects and its inverse induces an isomorphism 
\[ \iota_!  (j')^* \pr_{1,!}  \rightarrow j^* \pr_{1,!}. \]
 Since $\pr_1$ is an opfibration, the first step is the same as to show that the natural morphism
\[  \iota^* \hocolim_{I_{j_1}} e_{j}^* \rightarrow \hocolim_{I_{j_1}} e_{j'}^*  \]
is a  isomorphism on $2$-Cartesian objects where $e_j$, resp.\@ $e_{j'}$ denotes the inclusion of the respective fiber. Since $\iota^*$ commutes with homotopy colimits, this is to say that
\[ \hocolim_{I_{j_1}}  \iota^* e_{j}^* \rightarrow \hocolim_{I_{j_1}} e_{j'}^*  \]
is an isomorphism. However the fibers over $j$ and $j'$ in $(\twwc{J}) \times_J I$ are both isomorphic to $I_{j_1}$ and the natural morphism
\[ \iota^*  e_{j}^* \rightarrow e_{j'}^*  \]
is already an isomorphism on $3$-coCartesian objects by definition. Similarly its inverse induces an isomorphism
\[ \iota_! e_{j'}^*  \rightarrow e_{j}^* \]
and the same reasoning using 1.\@ that $\iota_!$, being a left adjoint, commutes with homotopy colimits, and 2.\@ that it is computed point-wise on constant diagrams (cf.\@ Lemma~\ref{LEMMAPOINTWISEEXBYZERO}, 2.)  allows to conclude.  

{\em $2$-Cartesian:} Let $\kappa: j \rightarrow j'$
\[ \xymatrix{
j=(j_1 \ar[r]^{} \ar@<10pt>@{=}[d] &j_2 \ar[r]^{}  & j_3 \ar@{=}[d] \ar[r]^{} & \ar@{=}[d]  j_4)  \\
j'=(j_1 \ar[r]_{} &j_2' \ar[r]_{}  \ar[u]_{}  & j_3 \ar[r]_{} & j_4)
} \]
be a morphism of type 2 in $\twwc{J}$. Denote \[ \overline{f}:=\widetilde{X}(\pi_{234}(\kappa)): \widetilde{X}(\pi_{234}(j)) \rightarrow \widetilde{X}(\pi_{234}(j')) \] the corresponding morphism in $\mathcal{S}$. Note that $\overline{f}$ is proper by the properties of induced interior compactifications. 
We have to show that the induced map
\[ j^* \pr_{1,!} \rightarrow \overline{f}_* (j')^* \pr_{1,!} \]
is an isomorphism on $2$-Cartesian objects. 
Since $\pr_1$ is an opfibration, this is the same as to show that the natural morphism
\[  \hocolim_{I_{j_1}} e_{j}^* \rightarrow \overline{f}_* \hocolim_{I_{j_1}} e_{j'}^*  \]
is a  isomorphism on $2$-Cartesian objects. Since $\overline{f}_*$ commutes with homotopy colimits by (F3), this is to say that
\[  \hocolim_{I_{j_1}} e_{j}^* \rightarrow  \hocolim_{I_{j_1}} \overline{f}_* e_{j'}^*  \]
is an isomorphism. However the fibers over $j$ and $j'$ in $(\twwc{J}) \times_J I$ are both isomorphic to $I_{j_1}$ and the natural morphism
\[ e_{j}^* \rightarrow \overline{f}_* e_{j'}^*  \]
is already an isomorphism on $2$-Cartesian objects by definition. 

{\em $4$-coCartesian:} Let $\kappa: j \rightarrow j'$
\[ \xymatrix{
	j=(j_1 \ar[r]^{} \ar@<10pt>@{=}[d] &j_2 \ar[r]^{} \ar@{=}[d] & j_3  \ar[r]^{} & j_4) \ar[d] \\
j'=(j_1 \ar[r]_{} &j_2 \ar[r]_{} & j_3 \ar@{=}[u] \ar[r]_{} & j_4')
} \]

be a morphism of type 4 in $\twwc{J}$. Denote 
\[ g:=\widetilde{X}(\pi_{234}(\kappa)): \widetilde{X}(\pi_{234}(j)) \rightarrow \widetilde{X}(\pi_{234}(j'))\]
 the corresponding morphism in $\mathcal{S}$.

We have to show that the induced map
\[ g^* j^* \pr_{1,!} \rightarrow  (j')^* \pr_{1,!} \]
is an isomorphism on $4$-coCartesian objects. This is the same as to show that the natural morphism
\[ g^* \hocolim_{I_{j_1}} e_{j}^* \rightarrow  \hocolim_{I_{j_1}} e_{j'}^*  \]
is an isomorphism on $4$-coCartesian objects. Since $g^*$ commutes with homotopy colimits, this is to say that
\[ \hocolim_{I_{j_1}} g^* e_{j}^* \rightarrow  \hocolim_{I_{j_1}}   e_{j'}^*  \]
is an isomorphism. However, the fibers over $j$ and $j'$ in $(\twwc{J}) \times_J I$ are both isomorphic to $I_{j_1}$ and the natural morphism
\[ g^*  e_{j}^* \rightarrow e_{j'}^*  \]
is already an isomorphism on 4-coCartesian objects by definition. 
\end{proof}

\begin{BEM}\label{REMKANEXT}
For an opfibration $\alpha$ the proof shows that 
we have actually just
\[ \alpha_!^{(X \hookrightarrow \overline{X})} = (\twwc \alpha)_!^{(\pi_{234}^*(\widetilde{X}^{\op}))}\]
where the right hand side is the relative left Kan extension in $\DD$. 
Indeed by construction (we omit the bases of the rel.\@ Kan extensions for $\DD$)
\[ \alpha_!^{(X \hookrightarrow \overline{X})} = \pr_{1,!} q_{2,*} q_{1,!}. \]
However, applying this to an object in the category
\[  \DD(\twwc{I})^{4-\cocart, 3-\cocart^*, 2-{\cart}}_{\pi_{234}^*((\tww\alpha)^* \widetilde{X})^{\op}}  \]
 $q_{2,*}$ receives an object which lies in the essential image of $q_{2}^*$ and we have proven
\[ \id \cong q_{2,*} q_{2}^*. \]
Hence, taking adjoints, we also have
\[ q_{2,!} q_{2}^* \cong \id . \]
Hence on the essential image of $q_{2}^*$, we have an isomorphism $q_{2,!} \cong q_{2,*}$.
\end{BEM}

\begin{BEISPIEL}\label{EXRELKAN}
We need to understand precisely how relative left Kan extensions along an inclusion of an object $i \hookrightarrow I$ looks like.
Note that this is not an obfibration, but a relative left Kan extension exists by the arguments in \cite[Theorem 4.2]{Hor16}.
It can be computed using the homotopy exact square
\[ \xymatrix{
i \times_{/I} I \ar[r]^-p \ar[d]_\pi  \ar@{}[dr]|{\Swarrow^\mu} & i \ar@{^{(}->}[d] \\
I \ar@{=}[r] & I
} \]
(in which $\pi$ is an opfibration) as
\[ \pi_!^{(S)} \SSS(\mu)(S)_\bullet p^*\]
Here $\pi_!^{(S)}$ and $(-)_\bullet$ are the relative left Kan extension and push-forward associated with the left fibered multiderivator $\EE \rightarrow \SSS^{\cor, \comp}$ of Main Theorem~\ref{HAUPTSATZ} below.

In this case the functors
\[ \xymatrix{ \twwc(i \times_{/I} I) \ar[rrr]^-{q_1=(\twwc{\alpha},\pi_{123})} &&& \twwc{I} \times_{({}^{\downarrow \uparrow \uparrow}{I})} {}^{\downarrow \uparrow \uparrow}(i \times_{/I} I) \ar[rr]^-{q_2=\id \times \pi_1} && \twwc{I} \times_I (i \times_{/I} I). } \]
are both isomorphisms of diagrams. In fact all diagrams are isomorphic to $i \times_{/I,\pi_1} \twwc{I}$.
Let $\widetilde{X}$ be an interior compactification of $X: \tw I \rightarrow \mathcal{S}$ on $I$. 
Consider the obvious morphisms (where a zero in the index of $\pi$ signifies that the object $i$ appers at the corresponding position): 
\[ \xymatrix{ X_{i} = \pi_{000}^*\widetilde{X}  & \ar[l]_-g \pi_{004} \widetilde{X} \ar@{^{(}->}[r]^\iota & \pi_{034}^*\widetilde{X} \ar@{->>}[r]^f & \pi_{234}^*\widetilde{X}.  } \]
We have then by construction
\[ i_!^{(X \hookrightarrow \overline{X})}  \cong (\twwc \pi)_!^{(\pi_{234}^*(\widetilde{X}^{\op}))} (\pi_{234}^*f)_* (\pi_{234}^*\iota)_! (\pi_{234}^*g)^* (\twwc p)^*. \]
\end{BEISPIEL}

\section{Conclusion}

\begin{HAUPTSATZ}\label{HAUPTSATZ}Let $\mathcal{S}$ be a category with compactifications, and let $\SSS^{\op}$ be the symmetric pre-multiderivator represented by $\mathcal{S}^{\op}$ with the symmetric multicategory structure \ref{PAROPMULTCAT}.
Let $\DD \rightarrow \SSS^{\op}$ be a  (symmetric) fibered (multi)derivator with domain $\Dirlf$ satisfying axioms (F1--F6) and (F4m--F5m) of \ref{PARAXIOMS}. Assume that $\DD$ is infinite (i.e.\@ satisfies (Der1${}^\infty$)). 

The morphism of (symmetric) pre-2-(multi)derivators
\[ \EE \rightarrow \SSS^{\cor, \comp}  \]
constructed in Definition~\ref{DEFDER6FU2} is a (symmetric) left fibered (multi)derivator with domain $\Dirlf$.
\end{HAUPTSATZ}
\begin{proof}The 2-pre-multiderivator $\EE$, as defined in Definition~\ref{DEFDER6FU2}, satisfies axioms (Der1) and (Der2) because $\DD$ satisfies them. The first part of axiom (FDer0 left) was shown in Lemma~\ref{LEMMA6FUOPFIB} and the second part follows from Lemma~\ref{LEMMAEXISTENCE3FUNCTORS}.
Instead of Axioms (FDer3--4 left) it is sufficient to show Axioms (FDer3--4 left') (cf.\@ \cite[Theorem 4.2]{Hor16}). (FDer3 left') is Lemma~\ref{LEMMAKAN4}, and
axiom (FDer4 left') follows from the proof of Lemma~\ref{LEMMAKAN4}.

(FDer5 left):  Since every push-forward functor along a 1-multimorphism in $\SSS^{\cor,\comp}$ is of the form $\overline{f}_* \iota_! g^*(- , \dots, -)$ this
follows from (FDer5 left) for $\DD \rightarrow \SSS$, the fact that $\iota_!$ commutes with homotopy colimits (because it is the left adjoint of a morphism of pre-derivators), and that $\overline{f}_*$ commutes with homotopy colimits (F3).
\end{proof}

\begin{BEM}\label{REMSCOR}
From the (symmetric) left fibered (multi)derivator of Main Theorem~\ref{HAUPTSATZ}, we may construct an {\em equivalent} (symmetric) left fibered multiderivator with domain $\Dirlf$
\[ \EE' \rightarrow \SSS^{\cor}. \]
This uses that $\SSS^{\cor, \comp}$ is equivalent as a symmetric pre-2-multiderivator to $\SSS^{\cor}$ by Proposition~\ref{PROPEQUIVCOMP}.
The construction is best seen via the alternative description of a (symmetric) left fibered (multi)derivator
as a pseudo-functor of (symmetric) 2-(multi)categories (cf.\@ \cite[Theorem~4.2]{Hor16})
\[ \Dirlf^{\cor}(\SSS^{\cor, \comp}) \rightarrow \mathcal{CAT}. \]
The equivalence of symmetric pre-2-multiderivators induces an equivalence of symmetric 2-multicategories 
$\Dirlf^{\cor}(\SSS^{\cor, \comp}) \cong \Dirlf^{\cor}(\SSS^{\cor})$ and by composing with a quasi-inverse functor we get a pseudo-functor
\[ \Dirlf^{\cor}(\SSS^{\cor}) \rightarrow \mathcal{CAT} \]
which may be strictified (replacing its values by equivalent categories) to get a strict 2-functor. From that one, a strict morphism (i.e.\@ a morphism of (symmetric) 2-pre-(multi)derivators)
\[ \EE' \rightarrow \SSS^{\cor}.  \]
may be reconstructed. We will keep the notation $\EE \rightarrow \SSS^{\cor}$ for this equivalent left fibered (multi)derivator. 
\end{BEM}

Recall the notions of perfectly generated, well-generated, and compactly generated fibers for a fibered (multi)derivator \cite[Definition 4.8]{Hor15}.

\begin{LEMMA}\label{LEMMAGENERATION}
If $\DD \rightarrow \SSS^{\op}$ has stable (perfectly generated, resp.\@ well-generated, resp.\@ compactly generated) fibers then 
the fibers of $\EE \rightarrow \SSS^{\cor}$ are right derivators with domain (at least) $\Posf$, stable (and perfectly generated, resp.\@ well-generated, resp.\@ compactly generated). 
\end{LEMMA}
\begin{proof}
By Lemma~\ref{LEMMAKAN2} there is an equivalence (compatible with pull-backs in $J$):
\[ \EE(I \times J)_{\pr_1^* X \hookrightarrow \pr_1^*\overline{X}} \cong \DD((\twwc I) \times J)_{\pr_1^*\pr_{234}^*\widetilde{X}}^{4-\cocart, 3-\cocart^*, 2-{\cart}}.  \]
Hence the statement follows if we can show that for $\alpha: J_1 \rightarrow J_2$ in $\Posf$, the right Kan extension functor
\[ (\id \times \alpha)_*: \DD((\twwc I) \times J_1) \rightarrow \DD((\twwc I) \times J_2) \]
respects the conditions of being 4-coCartesian, strongly 3-coCartesian, and 2-Cartesian.
This follows because the commutation with homotopy colimits implies that all functors $g^*$, $f_*$ and $\iota_!$ involved in the 
definitions of (strongly) (co)Cartesian are {\em exact} and hence commute also with {\em homotopy limits of shape $\Posf$} (actually {\em homotopy finite} is sufficient, cf.\@ \cite[Theorem 7.1]{PS14}).
By \cite[Lemma 4.7]{Hor15}, the properties of being perfectly, compactly or well-generated can be checked over fibers above actual objects of $\SSS^{\cor}(\cdot) = \mathcal{S}^{\cor}$ where
the fibers are actually equivalent to those of $\DD \rightarrow \SSS^{\op}$.
\end{proof}

\begin{KOR}\label{KORDER6FU}
With the assumptions of Main Theorem~\ref{HAUPTSATZ}, if $\DD \rightarrow \SSS^{\op}$
 is infinite (i.e.\@ satisfies (Der1${}^\infty$)) and has stable and perfectly generated fibers then there is
 a (unique up to equivalence) left and right fibered (symmetric) (multi)derivator $\EE \rightarrow \SSS^{\cor}$ with domain $\Cat$, hence --- in the multi-case --- a (symmetric) {\em derivator six-functor-formalism}, whose restriction to $\Dirlf$ is equivalent to the (symmetric) left fibered (multi)derivator of Main~Theorem~\ref{HAUPTSATZ} (cf.\@ also \ref{REMSCOR}).
\end{KOR}
\begin{proof}
By \cite[Corollary 1.3]{Hor17b}, $\EE \rightarrow \SSS^{\cor}$ extends to a (symmetric) left fibered multiderivator on all of $\Cat$. Note that 
$\SSS^{\cor}$ has domain $\Cat$ already. (The reader may check that the techniques of \cite{Hor17b} go through for the case of pre-2-multiderivators instead of pre-multiderivators, because the formal properties of $\Dia^{\cor}(\SSS)$ used in \cite{Hor17b} for a pre-2-multiderivator $\SSS$ remain exactly the same by \cite[Section~3]{Hor16}.) 
The same proof as the one of \cite[Theorem 6.3]{Hor16} thus shows that $\EE \rightarrow \SSS^{\cor}$ is also a right fibered (symmetric) multiderivator. 
\end{proof}

\section{The construction of proper derivator six-functor-formalisms}\label{SECTCONSTPROPER}

\begin{PAR}
This section is concerned with the construction of a {\em proper} derivator six-functor-formalism which, in addition, encodes a natural morphism
$f_! \rightarrow f_*$ (which is an isomorphism for proper morphisms $f$). In the classical (i.e.\@ non derivator enhanced) case it might be included (as explained in \cite[Section~8]{Hor15b}) 
by enlarging the 2-multicategory $\mathcal{S}^{\cor}$ to $\mathcal{S}^{\cor, 0}$. The latter 2-multicategory contains in addition non-invertible 2-morphisms defined by arbitrary proper morphisms between (multi)correspondences. In this 2-category the correspondences 
\[ \xymatrix{ & S \ar@{=}[ld] \ar[rd]^f \\ S & & T } \qquad  \xymatrix{ & S \ar[ld]_f \ar@{=}[rd] \\ T & & S }  \]
 become formally adjoint in the 2-category $\mathcal{S}^{\cor, 0}$ for a proper morphism $f: S \rightarrow T$. Hence $f_!$ becomes right adjoint to $f^*$, whence a canonical isomorphism $f_! \cong f_*$. On the derivator side this opens the possibility to include lax, resp.\@ oplax, morphisms of diagrams of correspondences which are important, for instance, to encode the classical exact triangles related to a pair of complementary open and closed embeddings, cf.\@ \cite[Section~8]{Hor16}. 

We will need some rather technical facts about the existence of Cartesian and coCartesian projectors whose discussion we postpone to Appendix~\ref{COCARTPROJ}.
\end{PAR}

\begin{PAR}\label{ILLUSTRATIONOPLAX}We illustrate the push-forward and pull-back along (op)lax morphisms in the simplest case based on the underlying diagrams. 
A 2-commutative square in $\mathcal{S}^{\cor, 0}$
\[ \vcenter{ \xymatrix{
S \ar[r] \ar[d] \ar@{}[rd]|{\Downarrow^\mu} & T \ar[d] \\
S' \ar[r] & T'
} } \quad \text{resp.} \quad \vcenter{ \xymatrix{
S \ar[r] \ar[d] \ar@{}[rd]|{\Uparrow_\mu} & T \ar[d] \\
S' \ar[r] & T'
} } \]
in which $\mu$ is not invertible, that is, an oplax natural transformation (l.h.s.), or a lax natural transformation (r.h.s.), from the top diagram of shape $\Delta_1$ to the bottom diagram of shape $\Delta_1$, can be encoded (as in \ref{PARALTSCOR}) as a commutative diagram in $\mathcal{S}$
\[ \xymatrix{
S \ar@{<-}[r] \ar@{<-}[d] & A \ar[r] \ar@{}[rd]|{\numcirc{1}} \ar@{<-}[d] & T \ar@{<-}[d] \\
B \ar@{<-}[r] \ar[d]  \ar@{}[rd]|{\numcirc{2}} & C \ar[r] \ar[d] & D \ar[d]   \\
S' \ar[r] & A' \ar[r] & T'
} \]
in which, in the oplax case, the square $\numcirc{1}$ is Cartesian and the square $\numcirc{2}$ is weakly Cartesian or,
in the lax case, the square $\numcirc{1}$ is weakly Cartesian and the square $\numcirc{2}$ is Cartesian.

Consider a weakly Cartesian square (with the corresponding Cartesian square inserted): 
\[ \xymatrix{
S' \ar@{->>}[rd]^H \ar@/^10pt/[rrrd]^F \ar@/_10pt/[rddd]_G \\
&\Box \ar[rr]^{F'} \ar[dd]_{G'} && T' \ar[dd]^g \\
\\
&S \ar[rr]^f & & T
} \]
A proper six-functor-formalism allows for the following two operations: 
\begin{enumerate}
\item[(O)] From a morphism
\[ G^* \mathcal{E} \rightarrow \mathcal{F} \]
applying $F_!$ one gets a morphism
\[  (F'H)_!(G'H)^* \mathcal{E} = F_! G^* \mathcal{E}  \rightarrow F_! \mathcal{F}. \]
Composing it with the natural transformation 
\[ (F'H)_!(G'H)^*  \cong F'_! H_! H^* (G')^* \cong F'_! H_* H^* (G')^* \leftarrow F'_! (G')^* \cong g^* f_! \] 
(using $H'_! \cong H_*$ and the unit for the adjunction $H^*$, $H_*$) applied to $\mathcal{E}$, we get 
\[ g^* \mathcal{E}'  \rightarrow \mathcal{F}' \]
for $\mathcal{E}' := f_! \mathcal{E}$ and $\mathcal{F}' := F_! \mathcal{F}$.

\item[(L)] From a morphism 
\[  \mathcal{E} \rightarrow F^! \mathcal{F} \]
applying $G_*$ one gets a morphism
\[ G_* \mathcal{E} \rightarrow (G'H)_*(F'H)^! \mathcal{F}. \]
Composing it with the morphism 
\[ (G'H)_*(F'H)^!  \cong G'_* H_* H^! (F')^! \cong G'_* H_! H^! (F')^! \rightarrow G'_* (F')^! \cong f^! g_* \] (using $H_! \cong H_*$ and the counit for the adjunction $H_!$, $H^!$) applied to $\mathcal{E}$, 
we get 
\[ \mathcal{E}'  \rightarrow f^! \mathcal{F}' \]
for $\mathcal{E}' := G_* \mathcal{E}$ and $\mathcal{F}' := g_* \mathcal{F}$.
\end{enumerate}

In other words, the operation (O) allows for the construction of a push-forward along the oplax morphism, and the operation (L) allows for the construction of a pull-back along the lax morphism, both being computed point-wise.  
In our approach it is essential to construct the left fibered version (with push-forwards) first. Hence the construction of the lax pull-back has to be a bit indirect. 
It turns out that the lax pull-back $\xi^\bullet$ {\em does have} a left adjoint $\xi_\bullet$ which is, however, not computed point-wise anymore (similar to the existence of internal Homs of diagrams which are also not computed point-wise). 
It is this left adjoint that will be constructed first. The right adjoint (a posteriori constructed via Brown representability) is then indeed computed point-wise as expected which, however, has to be proven by establishing the adjoint formula
\[ \alpha_! (\alpha^*\xi)_\bullet \cong \xi_\bullet \alpha_!   \]
involving the left adjoint. In the multi-case, the lax pull-back and oplax push-forward exist for $n$-ary morphisms as well. The oplax push-forward (involving construction (O) above) involves essentially only
a $1$-ary construction, whereas the lax pull-back involves a multi-version of construction (L) above. 

This section is concerned with the derivator analogue of these constructions for arbitrary diagrams in $\Catlf$. They are rather technical and will sometimes only be sketched. 
\end{PAR}

Recall from Definition~\ref{DEFSCOR4} the (op)lax 2-pre-multiderivators $\SSS^{\cor, 0, \comp, \oplax}(I)$, resp.\@ $\SSS^{\cor, 0, \comp, \lax}(I)$.
We proceed as in section \ref{SECTCONST} and begin with a couple of Lemmas that will be used to construct a 1-opfibration and 2-opfibration
$\EE^{\oplax}(I) \rightarrow \SSS^{\cor, 0, \comp, \oplax}(I)$, resp.\@ $\EE^{\lax}(I) \rightarrow \SSS^{\cor, 0, \comp, \lax}(I)$.

\begin{PAR}\label{COMPONENTSOPLAX}
Recall the notation from \ref{COMPONENTS}. In the (op)lax case we have the following modifications.
Let $I$ be in $\Catlf$, let $\tau$ be a tree, and let
\[ X \hookrightarrow \overline{X} \]
be an exterior compactification in $\mathcal{S}^{\tw(\tau \times I)}$ forming an object in $\Cor^{\comp, \lax}_I(\tau)$ (resp.\@ in $\Cor^{\comp, \oplax}_I(\tau)$), i.e.\@ with $X$ satisfying the conditions stated in Definition~\ref{DEFCORCOMP}.
By the construction in \ref{PARMULTICOMP} it has an interior compactification
\[ \widetilde{X}: \tww (\tau \times I) \rightarrow \mathcal{S} \]
to which we may apply the results of sections \ref{SECTPRELIM}--\ref{SECTPRELIMM}.

Let $o$ be a multimorphism in $\tww \tau$. If $o$ is of type 3 (resp.\@ type 2, resp.\@ type 1) again denote by
\[ \begin{array}{rrcl} 
 \widetilde{g}: &\widetilde{A} &\rightarrow& \widetilde{S}_1, \dots, \widetilde{S}_n \\
 \widetilde{\iota}: &\widetilde{A} &\rightarrow& \widetilde{A}' \\
 \widetilde{f}:  &\widetilde{A}' &\rightarrow& \widetilde{T} 
\end{array} \]
their images in
\[ \Fun(\tww I, \mathcal{S}). \]

Example~\ref{EXCOMPONENTS} remains valid, however, with $g=(g_1, \dots, g_n)$ only weakly type-1 admissible as multimorphism, resp.\@ with 
$f$ only weakly type-2 admissible.
\end{PAR}

\begin{LEMMA}\label{LEMMAEXISTENCE3FUNCTORSOPLAX}
With the notation as in \ref{COMPONENTSOPLAX}.
\begin{enumerate}
\item If $X \hookrightarrow \overline{X}$ is an object in $\Cor^{\comp, \oplax}_I(\tau)$, and $o$ is any morphism of type 4 in $\tww \tau$ (numbered with indices 2-4), 
the multivalued functor
\[  (\pi_{234}^*\widetilde{g})^*: \DD(\twwc I)_{\pi_{234}^* \widetilde{S}_1^{\op}} \times \cdots \times  \DD(\twwc I)_{\pi_{234}^* \widetilde{S}_n^{\op}} \rightarrow  \DD(\twwc I)_{\pi_{234}^* \widetilde{A}^{\op}}.    \]
is computed point-wise (in $\twwc I$) and on well-supported objects preserves the condition of being $4$-coCartesian, well-supported, and $2$-Cartesian. 

If $X \hookrightarrow \overline{X}$ is an object in $\Cor^{\comp, \lax}_I(\tau)$, and for any morphism $o$ of type 4 in $\tww \tau$, the functor $(\pi_{234}^*\widetilde{g})^*$,
 on well-supported objects, preserve the condition of being well-supported, and $4$-coCartesian, but not necessarily the condition of being $2$-Cartesian. 

\item For any morphism $o$ of type 3  in $\tww \tau$, in both cases, the functor 
\[ (\pi_{234}^*\widetilde{\iota})_!: \DD(\twwc I)_{\pi_{234}^* \widetilde{A}^{\op}}  \rightarrow  \DD(\twwc I)_{\pi_{234}^* (\widetilde{A}')^{\op}}    \]
i.e.\@ the left adjoint of $(\pi_{234}^*\widetilde{\iota})^*$, which exists by (F1),
is computed point-wise (in $\twwc I$) on $4$-coCartesian and well-supported objects, 
and on such it preserves the conditions of being $4$-coCartesian, well-supported, and $2$-Cartesian. 

\item If $X \hookrightarrow \overline{X}$ is an object in $\Cor^{\comp, \lax}_I(\tau)$, and $o$ is any morphism of type 2 in $\tww  \tau$, 
the functor
\[  (\pi_{234}^*\widetilde{f})_*: \DD(\twwc I)_{\pi_{234}^* (\widetilde{A}')^{\op}} \rightarrow  \DD(\twwc I)_{\pi_{234}^* \widetilde{T}^{\op}}.    \]
is computed point-wise (in $\twwc I$)   and on well-supported objects it preserves the condition of being $4$-coCartesian, well-supported, and $2$-Cartesian. 

 If $X \hookrightarrow \overline{X}$ is an object in $\Cor^{\comp, \oplax}_I(\tau)$, the functor $(\pi_{234}^*\widetilde{f})_*$, on well-supported objects, preserves the condition of being well-supported, and $2$-Cartesian, but not necessarily the condition of $4$-coCartesian. 

\item If $f: (X_1 \hookrightarrow \overline{X}_1)  \rightarrow (X_2 \hookrightarrow \overline{X}_2)$ is a morphism in $\Cor^{\comp, \lax}_I(\tau)$, denote by $\widetilde{f}: \widetilde{X}_1 \rightarrow \widetilde{X}_2$ the corresponding morphism of interior compactifications, and let $\mu$ be any object in $\tww  \tau$.
The functor
\[ (\pi_{234}^*\widetilde{f}_\mu)_*: \DD(\twwc I)_{\pi_{234}^* \widetilde{X}_{1,\mu}^{\op}} \rightarrow  \DD(\twwc I)_{\pi_{234}^* \widetilde{X}_{2, \mu}^{\op}}.    \]
is computed point-wise (in $\twwc I$)   and on well-supported objects it preserves the condition of being $4$-coCartesian, well-supported, and $2$-Cartesian. 

The functor
\[ (\pi_{234}^*\widetilde{f}_\mu)^*: \DD(\twwc I)_{\pi_{234}^* \widetilde{X}_{2,\mu}^{\op}} \rightarrow  \DD(\twwc I)_{\pi_{234}^* \widetilde{X}_{1, \mu}^{\op}}.    \]
is computed point-wise (in $\twwc I$)   and on well-supported objects it preserves the condition of being $4$-coCartesian, well-supported, but not necessarily the condition of being $2$-Cartesian. 

If $f: (X_1 \hookrightarrow \overline{X}_1)  \rightarrow (X_2 \hookrightarrow \overline{X}_2)$ is a morphism in $\Cor^{\comp, \oplax}_I(\tau)$, denote by $\widetilde{f}: \widetilde{X}_1 \rightarrow \widetilde{X}_2$ the corresponding morphism of interior compactifications, and let $\mu$ be any object in $\tww  \tau$.
The functor
\[ (\pi_{234}^*\widetilde{f}_\mu)_*: \DD(\twwc I)_{\pi_{234}^* \widetilde{X}_{1,\mu}^{\op}} \rightarrow  \DD(\twwc I)_{\pi_{234}^* \widetilde{X}_{2, \mu}^{\op}}.    \]
is computed point-wise (in $\twwc I$)  and on well-supported objects it preserves the condition of being $2$-Cartesian, well-supported, but not necessarily the condition of $4$-coCartesian. 

The functor
\[ (\pi_{234}^*\widetilde{f}_\mu)^*: \DD(\twwc I)_{\pi_{234}^* \widetilde{X}_{2,\mu}^{\op}} \rightarrow  \DD(\twwc I)_{\pi_{234}^* \widetilde{X}_{1, \mu}^{\op}}.    \]
is computed point-wise (in $\twwc I$) and on well-supported objects it preserves the condition of being $4$-coCartesian, well-supported, and $2$-Cartesian. 

\end{enumerate} 
\end{LEMMA}
\begin{proof}
1.\@ As in Lemma~\ref{LEMMAEXISTENCE3FUNCTORS}, 1.\@
In the lax case preservation of coCartesianity is still clear, and the preservation of the condition of being well-supported still follows from Proposition~\ref{PROPPROPERTIESCORCOMPM}, 2.
In the oplax case, the relevant top squares in the second part of Lemma~\ref{LEMMAWEAKCOMPM} are still Cartesian. 

2.\@ As in Lemma~\ref{LEMMAEXISTENCE3FUNCTORS}, 2.\@ --- note that Proposition~\ref{PROPPROPERTIESCORCOMPM}, 2. holds true for weakly admissible diagrams. 

3.\@ As in Lemma~\ref{LEMMAEXISTENCE3FUNCTORS}, 3. In the lax case, the relevant top squares in the second part of Lemma~\ref{LEMMAWEAKCOMP} are still Cartesian. 
In the oplax case 
preservation of Cartesianity is clear as well and the preservation of strong coCartesianity follows still from
Proposition~\ref{PROPPROPERTIESCORCOMPM}, 3. 

4.\@ is exactly shown as 1.\@ and 3.\@
\end{proof}

\begin{PAR}\label{CONDPOINTWISE}
Over the morphism 
$(\pi_{234}^*\widetilde{f})^{\op}$, resp.\@ $(\pi_{234}^*\widetilde{g})^{\op}$, coming from a $X \hookrightarrow \overline{X}$ in $\Cor^{\comp, \oplax}_I(\tau)$  (resp.\@ in $\Cor^{\comp, \lax}_I(\tau)$) and a morphism in $\tww \tau$ of type 1 (resp.\@ of type 3), there will be in general no (co)Cartesian morphism\footnote{that is, (co)Cartesian w.r.t.\@ $\DD(\twwc I) \rightarrow \SSS(\twwc I)$.} in 
\[  \DD(\twwc I)^{4-\cocart, \ws, 2-{\cart}}. \] 

However, we can use the (co)Cartesian projectors from Propositions~\ref{PROPCARTPROJ} and \ref{PROPCOCARTPROJ}, and say that $\mathcal{E} \rightarrow \mathcal{F}$ is 
\begin{enumerate}
\item {\bf oplax Cartesian}, if
it induces an isomorphism 
\[ \mathcal{E} \rightarrow  \Box_* (\pi_{234}^*\widetilde{f})_* \mathcal{F} \]
in the fiber,
\item {\bf lax coCartesian}, if
it induces an isomorphism 
\[  \Box_! (\pi_{234}^*\widetilde{g})^* \mathcal{E} \rightarrow  \mathcal{F} \]
in the fiber. 
\end{enumerate}
\end{PAR}

\begin{LEMMA}\label{LEMMACOMPOPLAX}
Compositions of lax coCartesian morphisms are lax coCartesian and of oplax Cartesian morphisms are oplax Cartesian, if the sources (resp.\@ destinations) are 4-coCartesian, well-supported and 2-Cartesian. 
(Of course, we only consider morphisms over (multi)morphisms of the form $\pi_{234}^*\widetilde{g}$, resp.\@ $\pi_{234}^*\widetilde{f}$, considered above.)
\end{LEMMA}
\begin{proof}
For lax coCartesian, we have to show:
\[ \Box_! (\pi_{234}^*\widetilde{g}_1)^* \Box_! (\pi_{234}^*\widetilde{g}_2)^* \cong \Box_! (\pi_{234}^*(\widetilde{g}_2 \circ \widetilde{g}_1))^*.  \]
For a morphism $g$ denote by $g_*^{\cart}$ be the restriction of $g_*$ to the full subcategory of $2$-Cartesian objects (note that any $g_*$ preserves the condition of being $2$-Cartesian). 
The left adjoint of $g_*^{\cart}$ is the functor $\Box_! g^*$ restricted to the full subcategory of $2$-Cartesian objects.
The morphism is therefore the adjoint of the isomorphism $(\pi_{234}^*\widetilde{g}_1)_*^{\cart} \circ (\pi_{234}^*\widetilde{g}_2)_*^{\cart} \cong (\pi_{234}^*\widetilde{g}_1\circ\widetilde{g}_2)_*^{\cart}$ and thus an isomorphism as well.

For oplax Cartesian we have to show:
\[ \Box_* (\pi_{234}^*\widetilde{f}_1)_* \Box_* (\pi_{234}^*\widetilde{f}_2)_* \cong \Box_* (\pi_{234}^*\widetilde{f}_1\circ \widetilde{f}_2)_*.  \]
On the full subcategories of 2-Cartesian, well-supported and 4-coCartesian objects, this formula can be shown point-wise on objects of the form $\twwc i$. However, by Proposition~\ref{PROPCOCARTPROJ}, $\Box_*$ does nothing over objects of the form $\twwc i$.
\end{proof}
\begin{PAR}\label{PAROPLAXFUNCTORS}Analogously to the plain case (cf.\@ Lemma~\ref{LEMMAEXISTENCE3FUNCTORS}) we denote  in the oplax case
\[ \begin{array}{rcl|l}
\widetilde{X}(o)_{\#} &:=& (\pi_{234}^*\widetilde{g})^* & \text{for $o$ of type 4} \\
\widetilde{X}(o)_{\#} &:=& (\pi_{234}^*\widetilde{\iota})_! & \text{for $o$ of type 3} \\
\widetilde{X}(o)_{\#} &:=& \Box_* (\pi_{234}^*\widetilde{f})_* & \text{for $o$ of type 2}
\end{array} \]
and in the lax case
\[ \begin{array}{rcl|l}
\widetilde{X}(o)_{\#} &:=& \Box_!  (\pi_{234}^*\widetilde{g})^* & \text{for $o$ of type 4} \\
\widetilde{X}(o)_{\#} &:=& (\pi_{234}^*\widetilde{\iota})_! & \text{for $o$ of type 3} \\
\widetilde{X}(o)_{\#} &:=& (\pi_{234}^*\widetilde{f})_* & \text{for $o$ of type 2}
\end{array} \]
\end{PAR}

\begin{PROP}[(Op)lax version of Proposition~\ref{PROPPROPERTIESCORCOMPMDIA}]\label{PROPPROPERTIESCORCOMPMDIAOPLAX}
Let $\tau$ be a tree, $X \hookrightarrow \overline{X}$ be an object in $\Cor^{\comp, \oplax}_I(\tau)$ (resp.\@ in $\Cor^{\comp, \lax}_I(\tau)$)  be a functor, and let 
$\widetilde{X}: \tww (\tau \times I) \rightarrow \mathcal{S}$ be the associated interior compactification. 

Then the functors constructed in \ref{PAROPLAXFUNCTORS} commute. More precisely: 

Consider a diagram of the form
\[ \xymatrix{
w_1, \dots, w_n \ar[r]^-d \ar[d]_c & z_1,\dots,z_n \ar[d]^{a} \\
y \ar[r]_-b & x
} \]
in $^{\downarrow\downarrow\downarrow} \tau$  [sic!] (numbered with indices 2--4) 
where $a$ and $c$ are of some type 2,3 or 4 (1-ary in case of type $\not= 4$) and $b$ and $d$ are of some type 2,3, or 4, respectively.  

Then we have that the natural morphism\footnote{In each case an exchange of a natural commutation given by functoriality of $\DD$ we we leave to the reader to construct}
\begin{equation} \label{eqexchangeoplax} \widetilde{X}(a')_{\#} \widetilde{X}(d')_{\#} \rightarrow \widetilde{X}(b')_{\#} \widetilde{X}(c')_{\#}   \end{equation}
is an isomorphism, where $a'$, etc.\@, denote the corresponding morphism in $\tww \tau$.
In case $a$ and $c$ are of type 4 and $n$-ary $\widetilde{X}(a')_{\#}$ and $\widetilde{X}(c')_{\#}$ are multivalued functors and $\widetilde{X}(d')_{\#}$ denotes the corresponding $n$-tupel of functors (of which only one is not an identity). 
\end{PROP}
\begin{proof}
Oplax cases: It suffices to see that
\begin{equation}
 \widetilde{X}(b)^* (\Box_*\widetilde{X}(a_1)_{*}-, \dots, \Box_*\widetilde{X}(a_n)_{*}-)  \rightarrow \Box_*\widetilde{X}(c)_* \widetilde{X}(d)^*(-, \dots, -) 
 \end{equation}
 and
\begin{equation}
 \widetilde{X}(a)_!  \Box_* \widetilde{X}(d)_* \rightarrow  \Box_* \widetilde{X}(b)_* \widetilde{X}(c)_!    
 \end{equation}
 are isomorphisms on objects of the form $\twwc i$, where $i$ is an object of $I$, because of the (strong) (co)Cartesianity conditions. However, by Proposition~\ref{PROPCOCARTPROJ}, $\Box_*$ does nothing over objects of the form $\twwc i$.

Lax cases: By Lemmas~\ref{COMMCARTPROJIOTA}--\ref{COMMCARTPROJF}, $\Box_!$ commutes with $\widetilde{X}(c)_*$ and $\widetilde{X}(c)_!$. Therefore
\begin{equation}
 \Box_! \widetilde{X}(b)^* (\widetilde{X}(a_1)_{*}-, \dots, \widetilde{X}(a_n)_{*}-)  \rightarrow \widetilde{X}(c)_* \Box_!  \widetilde{X}(d)^*(-, \dots, -) 
 \end{equation}
and
\begin{equation}
 \widetilde{X}(c)_! \Box_! \widetilde{X}(d)^*(-, \dots, -)  \rightarrow  \Box_! \widetilde{X}(b)^* (\widetilde{X}(a_1)_{!}-, \dots,  \widetilde{X}(a_n)_{!}-)  \
 \end{equation}
are isomorphisms --- 
note that $\widetilde{X}(c)_!$ is also computed point-wise, if the argument is not assumed to be $2$-Cartesian, cf.\@ Lemma~\ref{LEMMAEXISTENCE3FUNCTORS}, 2. 
\end{proof}

Analogously to the plain case, we have the following reformulation: 

\begin{PROP}[(Op)lax version of Proposition~\ref{PROPPF}]\label{PROPPFOPLAX}
The association 
\[ \alpha = abc \mapsto \widetilde{X}(\alpha)_{\#} := \widetilde{X}(a')_{\#} \widetilde{X}(b')_{\#} \widetilde{X}(c')_{\#}, \] 
where $a'$ is the morphism of $\tww \tau$ corresponding to $a$ in ${}^{\downarrow\downarrow\downarrow} \tau$, etc.\@,  
 defines a well-defined pseudo-functor on ${}^{\downarrow\downarrow\downarrow} \tau$ such that the isomorphism (\ref{eqexchangeoplax}) becomes the one
induced by the pseudo-functoriality.
\end{PROP}

\begin{DEF}\label{DEFDER6FU1OPLAX}
Assume that $\DD \rightarrow \SSS^{\op}$ is a (symmetric) fibered multiderivator. 
Let $\tau$ be a tree. We define a category 
\[ (E_I')^{\oplax}(\tau) \quad (\text{resp. } (E_I')^{\oplax}(\tau^S))   \]
with objects pairs $(X \hookrightarrow \overline{X}, \mathcal{F})$ of an object $(X \hookrightarrow \overline{X}) \in \Cor^{\comp, \oplax}_I(\tau^S)$ (cf.\@ Definition~\ref{DEFCORCOMP}), and an object
\[  \mathcal{F} \in  \Fun(\twwc \tau, \DD(\twwc I))^{4-\cocart, 3-\cocart^*, 2-\oplax-{\cart}}_{\pi_{234}^* \widetilde{X}^{\op}} \]
(functors of multicategories) where $\widetilde{X}$ is the interior compactification associated with  $(X \rightarrow \overline{X})$ (cf.\@ \ref{DEFEXTCOMP}).
Similarly define 
\[ (E_I')^{\lax}(\tau) \quad (\text{resp. } (E_I')^{\lax}(\tau^S))   \]
with objects pairs $(X \hookrightarrow \overline{X}, \mathcal{F})$ of an object $(X \hookrightarrow \overline{X}) \in \Cor^{\comp, \lax}_I(\tau^S)$, and an object
\[  \mathcal{F} \in  \Fun(\twwc \tau, \DD(\twwc I))^{4-\lax-\cocart, 3-\cocart^*, 2-{\cart}}_{\pi_{234}^* \widetilde{X}^{\op}}. \]
The three superscripts have to be interpreted in the following way. 
a) for each $\mu \in \twwc \tau$ the value lies in $\DD(\twwc I)^{4-\cocart, 3-\cocart^*, 2-\cart}$ and b) for the functor
\[ \twwc \tau \rightarrow \DD(\twwc I) \]
multimorphisms of type 4 are mapped to (lax) coCartesian multimorphisms, morphisms of type 3 are mapped
to strongly coCartesian morphisms, and morphisms of type 2 are mapped to (oplax) Cartesian morphisms.

Morphisms $\mathcal{F}_1 \rightarrow \mathcal{F}_2$ over $\xi: (X_1 \hookrightarrow \overline{X}_1) \rightarrow (X_2 \hookrightarrow \overline{X}_2)$ in $\Cor^{\comp,\mathrm{(op)lax}}_I(\tau)$ are morphisms $\mathcal{F}_2 \rightarrow \mathcal{F}_1$ [sic.] in 
\[  \Fun(\twwc \tau, \DD(\twwc I))^{4-\cocart, 3-\cocart^*, 2-\oplax-{\cart}}  \quad (\text{resp.\@ in}\  \Fun(\twwc \tau, \DD(\twwc I))^{4-\lax-\cocart, 3-\cocart^*, 2-{\cart}}  )\]
over $\pi_{234}^*$ applied to $\widetilde{f}^{\op}: \widetilde{X}_2^{\op} \rightarrow \widetilde{X}_1^{\op}$ which are Cartesian (=coCartesian) when restricted to objects of the form 
$\twwc o \times (0 \to 1)$ for $o \in I$.  Warning: The morphisms will in general not be Cartesian w.r.t.\@ {\em all} morphisms of the form $ \mu \times (0 \to 1)$ for $\mu \in \twwc \tau$.

Note that, in the symmetric case, {\em all} functors $\tau^S \rightarrow (\tau')^S$ induce functors $(E_I')^{\mathrm{(op)lax}}(\tau^S) \rightarrow (E_I')^{\mathrm{(op)lax}}((\tau')^S)$ using the symmetry of $\DD$.

Finally, we say that a morphism is a {\em weak equivalence}, if the underlying morphism in $\Cor^{\comp,\mathrm{(op)lax}}_I(\tau)$ is a weak equivalence, i.e.\@ if the morphism $X_1 \rightarrow X_2$ is  an isomorphism. 
\end{DEF}

\begin{FUNDLEMMA}[(Op)lax version of Fundamental Lemma~\ref{LEMMACOMPSIXFU}]\label{LEMMACOMPSIXFUOPLAX}With the notation as in Definition~\ref{DEFDER6FU1OPLAX}.
\begin{enumerate}
\item Let $\tau$ be a tree, and consider an object $X \hookrightarrow \overline{X}$ in $\Cor^{\comp, \oplax}_I(\tau^S)$ $($resp.\@ in $\Cor^{\comp, \lax}_I(\tau^S))$. Let  $\widetilde{X}: \tww ( \tau \times I) \rightarrow \mathcal{S}$ be the associated interior compactification (\ref{PARMULTICOMP}). 
Let $(\mathcal{E}_o)_{o \in \tau}$ be a collection of objects with $\mathcal{E}_o \in \DD(\twwc I)^{4-\cocart, 3-\cocart^*, 2-{\cart}}_{\pi_{234}^* \widetilde{X}_o^{\op}}$, where
$\widetilde{X}_o$ is the value of $\widetilde{X}$ at $\tww o$. 

 The functor $(E_I')^{\lax, \mathrm{(op)lax}}(\tau)_{(\mathcal{E}_o)} \rightarrow (\Cor^{\comp,\mathrm{(op)lax}}_I(\tau)_{(X_o \hookrightarrow \overline{X}_o)})^{\op}$ is a fibration with discrete fibers. 

\item For $\tau = \tau_1 \circ_i \tau_2$, where $i$ is a source object of $\tau_1$ (which we identify with the
final object of $\tau_2$),
the square
\[ \xymatrix{
 (E_I')^{\mathrm{(op)lax}}(\tau)_{(X \hookrightarrow \overline{X})}(\{ \mathcal{E}_o\}_{o \in \tau}) \ar[r] \ar[d] &  (E_I')^{\mathrm{(op)lax}}(\tau)_{(X_1 \hookrightarrow \overline{X}_1)}(\{ \mathcal{E}_o\}_{o \in \tau_1}) \ar[d] \\
 (E_I')^{\mathrm{(op)lax}}(\tau)_{(X_2 \hookrightarrow \overline{X}_2)}(\{ \mathcal{E}_o\}_{o \in \tau_2}) \ar[r] & \cdot
} \]
is 2-Cartesian. Hence if we consider the $E_I^{\mathrm{(op)lax}}(\cdots)_{(\cdots)}$ as sets, we have
\[ (E_I')^{\mathrm{(op)lax}}(\tau)_{(X \hookrightarrow \overline{X})}(\{ \mathcal{E}_o\}_{o \in \tau}) \cong  (E_I')^{\mathrm{(op)lax}}(\tau)_{(X_1 \hookrightarrow \overline{X}_1)}(\{ \mathcal{E}_o\}_{o \in \tau_1}) \times  (E_I')^{\mathrm{(op)lax}}(\tau)_{(X_2 \hookrightarrow \overline{X}_2)}(\{ \mathcal{E}_o\}_{o \in \tau_2}). \]

\item[3.] (oplax case) For $\tau = \Delta_{1,n}$, we have canonically an isomorphism of sets
\begin{equation}\label{eqoplax} (E_I')^{\oplax}(\Delta_{1,n})_{(X \hookrightarrow \overline{X})}( \mathcal{E}_1, \dots, \mathcal{E}_n; \mathcal{E}_{n+1}) \cong \Hom_{\DD(\twwc I)_{\pi_{234}^* \widetilde{T}^{\op}}}(\Box_* \widetilde{f}_* \widetilde{\iota}_! \widetilde{g}^*(\mathcal{E}_1, \dots, \mathcal{E}_n); \mathcal{E}_{n+1})  \end{equation}
for any choice of pull-back $\widetilde{g}^*$, push-forward $\widetilde{f}_*$, and adjoint $\widetilde{\iota}_!$ to the pull-back $\widetilde{\iota}^*$.
Here $\widetilde{T}$ is the restriction of $\widetilde{X}$ to $\tww (n+1)$, and $\widetilde{g}$, $\widetilde{\iota}$, and $\widetilde{f}$ are the components of $\widetilde{X}$ as in \ref{EXCOMPONENTS}, and $\Box_*$ is the right coCartesian projector of Proposition~\ref{PROPCOCARTPROJ}. 

An object $\mathcal{F}$ on the left hand side is mapped to an isomorphism if and only if it is also Cartesian (or equivalently coCartesian) w.r.t.\@ the projection 
\[ \pi_{234} \times \id: \twwc (\Delta_{1,n} \times I) \rightarrow {}^{\uparrow \uparrow \downarrow}{\Delta_{1,n}} \times \twwc I. \] 

\item[3.] (lax case) For $\tau = \Delta_{1,n}$, we have canonically an isomorphism of sets
\begin{equation}\label{eqlax} (E_I')^{\lax}(\Delta_{1,n})_{(X \hookrightarrow \overline{X})}( \mathcal{E}_1, \dots, \mathcal{E}_n; \mathcal{E}_{n+1}) \cong \Hom_{\DD(\twwc I)_{\pi_{234}^* \widetilde{T}^{\op}}}(\widetilde{f}_* \widetilde{\iota}_! \Box_! \widetilde{g}^*(\mathcal{E}_1, \dots, \mathcal{E}_n); \mathcal{E}_{n+1})  \end{equation}
for any choice of pull-back $\widetilde{g}^*$, push-forward $\widetilde{f}_*$, and adjoint $\widetilde{\iota}_!$ to the pull-back $\widetilde{\iota}^*$.
Here $\widetilde{T}$ is the restriction of $\widetilde{X}$ to $\tww (n+1)$, and $\widetilde{g}$, $\widetilde{\iota}$, and $\widetilde{f}$ are the components of $\widetilde{X}$ as in \ref{EXCOMPONENTS}, and $\Box_!$ is the left Cartesian projector of Proposition~\ref{PROPCARTPROJ}. 

An object $\mathcal{F}$ on the left hand side is mapped to an isomorphism if and only if it is also Cartesian (or equivalently coCartesian) w.r.t.\@ the projection 
\[ \pi_{234} \times \id: \twwc (\Delta_{1,n} \times I) \rightarrow {}^{\uparrow \uparrow \downarrow}{\Delta_{1,n}} \times \twwc I. \] 

\item[4.] (oplax case) If $h: (X_1 \hookrightarrow \overline{X}_1)  \rightarrow (X_2 \hookrightarrow \overline{X}_2)$ is a morphism in $\Cor^{\comp, \oplax}_I(\Delta_{1,n})_{(S_1 \hookrightarrow \overline{S}_1, \dots, S_1 \hookrightarrow \overline{S}_n; T \hookrightarrow \overline{T})}$, denote as follows the components of the morphism between interior compactifications: 
\begin{equation}\label{eqint} \vcenter{ \xymatrix{ 
\widetilde{S}_1, \dots, \widetilde{S}_n  \ar@{=}[d] & \ar[l]_-{\widetilde{g}_1} \widetilde{A}_1 \ar@{^{(}->}[r]^{\widetilde{\iota}_1} \ar[d]^{\widetilde{h}} &   \widetilde{A}'_1 \ar@{->>}[r]^{\widetilde{f}_1} \ar[d]^{\widetilde{h}'} & \widetilde{T} \ar@{=}[d] \\
\widetilde{S}_1, \dots, \widetilde{S}_n & \ar[l]^-{\widetilde{g}_2} \widetilde{A}_2 \ar@{^{(}->}[r]_{\widetilde{\iota}_2} &   \widetilde{A}_2' \ar@{->>}[r]_{\widetilde{f}_2} &  \widetilde{T}   
} } \end{equation} the push-forward along $h$  (i.e.\@ pull-back along $h^{\op}$) is given on the r.h.s.\@ of (\ref{eqoplax}) by the composition
\[ \xymatrix{  \Box_* \widetilde{f}_{2,*} \widetilde{\iota}_{2,!}  \widetilde{g}_2^* \mathcal{E}  \ar[r]^-{\mathrm{unit}} &   \Box_* \widetilde{f}_{2,*} \widetilde{\iota}_{2,!}  \Box_* \widetilde{h}_* \widetilde{h}^* \widetilde{g}_2^* \mathcal{E} \ar[r]^-\sim &  \Box_* \widetilde{f}_{1,*} \widetilde{\iota}_{1,!} \widetilde{g}_1^* \mathcal{E} \ar[r] & \mathcal{F}  }  \]
involving the unit of the adjunction $\widetilde{h}^*, \Box_* \widetilde{h}_*$, 
using the (exchange) isomorphism $\widetilde{\iota}_{2,!} \Box_*  \widetilde{h}_{*} \cong \Box_*  \widetilde{h}_{*}' \widetilde{\iota}_{1,!}$ from analogues of Proposition~\ref{PROPPROPERTIESCORCOMPMDIAOPLAX}, 3.\@ and Lemma~\ref{LEMMACOMPOPLAX}.

\item[4.] (lax case) If $h: (X_1 \hookrightarrow \overline{X}_1)  \rightarrow (X_2 \hookrightarrow \overline{X}_2)$ is a morphism in $\Cor^{\comp, \lax}_I(\Delta_{1,n})_{(S_1 \hookrightarrow \overline{S}_1, \dots, S_1 \hookrightarrow \overline{S}_n; T \hookrightarrow \overline{T})}$, denoting as in (\ref{eqint}) the components of the morphism between interior compactifications
 the push-forward along $h$  (i.e.\@ pull-back along $h^{\op}$) is given on the r.h.s.\@ of (\ref{eqlax}) by the composition
\[ \xymatrix{  \widetilde{f}_{2,*} \widetilde{\iota}_{2,!}   \Box_! \widetilde{g}_2^* \mathcal{E}  \ar[r]^-{\mathrm{unit}} &   \widetilde{f}_{2,*} \widetilde{\iota}_{2,!}  \widetilde{h}_* \Box_!  \widetilde{h}^*  \Box_! \widetilde{g}_2^* \mathcal{E} \ar[r]^-\sim &  \widetilde{f}_{1,*} \widetilde{\iota}_{1,!}  \Box_! \widetilde{g}_1^* \mathcal{E} \ar[r] & \mathcal{F}  }  \]
involving the unit of the adjunction $\Box_! \widetilde{h}^*, \widetilde{h}_*$
using the (exchange) isomorphism $\widetilde{\iota}_{2,!}  \widetilde{h}_{*} \cong  \widetilde{h}_{*}' \widetilde{\iota}_{1,!}$ and an analogue of  Lemma~\ref{LEMMACOMPOPLAX}.
\end{enumerate}
\end{FUNDLEMMA}

\begin{proof}
1. As in the proof of Lemma~\ref{LEMMACOMPSIXFU}, 1., using Proposition~\ref{PROPPROPERTIESCORCOMPMDIAOPLAX} instead of Proposition~\ref{PROPPROPERTIESCORCOMPMDIA}.  The proof shows that one can actually construct a diagram, unique up to isomorphism in
\[ \Fun((\twwc \tau) \times \Delta_1, \DD(\twwc I)_{(\pi_{234} \times \id)^* \widetilde{F}^{\op} }) \]
where $\widetilde{F} \in \mathcal{S}^{(\tww I) \times \Delta_1}$ is the interior compactification of a morphism $\xi: (X_1 \hookrightarrow \overline{X}_1) \rightarrow (X_2 \hookrightarrow \overline{X}_2)$ 
with given restriction 
\[ \mathcal{E}_1 \in \Fun(\twwc \tau, \DD(\twwc I)_{(\pi_{234})^* \widetilde{X}_1^{\op} }) \]
and with the additional property that morphisms of the form $(\twwc o) \times (0 \to 1)$ go to the identity of $\mathcal{E}_o$. 
This shows that 
\[ (E_I')^{\mathrm{(op)lax}}(\tau)_{(\mathcal{E}_o)} \rightarrow (\Cor^{\comp, \mathrm{(op)lax}}_I(\tau)_{(X_o \hookrightarrow \overline{X}_o)}))^{\op} \]
is a fibration with discrete fibers. 

2.\@ is shown as Lemma~\ref{LEMMACOMPSIXFU}, 2.\@

3.\@ and 
4.\@ are omitted. 
\end{proof}

\begin{DEF}\label{DEFDER6FU1OPLAX2}
Let $\tau$ be a tree. We define opfibrations
\[ E_I^{\mathrm{(op)lax}}(\tau) \rightarrow \Cor^{\comp, \mathrm{(op)lax}}_I(\tau)  \]
as the opfibration with the same discrete fibers corresponding to the fibrations
\[ (E_I')^{\mathrm{(op)lax}}(\tau) \rightarrow (\Cor^{\comp, \mathrm{(op)lax}}_I(\tau))^{\op}.   \]
\end{DEF}

\begin{LEMMA}\label{LEMMADER6FUOPLAX3}
The strict 2-functor 
\[ \Delta_n \mapsto E_I^{\mathrm{(op)lax}}(\Delta_n) \quad (\text{resp. } \tau \mapsto E_I^{\mathrm{(op)lax}}(\tau), \quad \text{resp. } \tau^S \mapsto E_I^{\mathrm{(op)lax}}(\tau^S))   \]
satisfies the properties of Proposition~\ref{PROPCONSTR2CAT}, resp.\@ of Proposition~\ref{PROPCONSTRSYMMULTI}. 
\end{LEMMA}
\begin{proof}
Axiom 1 (the surjectivity on objects) is clear and the equivalence in Axiom 2 follows from Lemma~\ref{LEMMACOMPSIXFUOPLAX}, 1--2., and the validity of the axioms for $\Cor_I^{\comp, \mathrm{(op)lax}}$ (Lemma~\ref{LEMMAPROPCONSTRSYMMULTICOMP}, 2.). Indeed, we get  a strictly commutative diagram (notation for the non-symmetric case)
\begin{equation*} \xymatrix{
  E_I^{\mathrm{(op)lax}}(\tau)_{(\mathcal{E}_o)}[\mathcal{W}^{-1}_{(\mathcal{E}_o)}]  \ar[r] \ar[d] &  \prod_{m} E_I^{\mathrm{(op)lax}}(\Delta_{1,k_m})_{(\mathcal{E}_o)}[\mathcal{W}^{-1}_{(\mathcal{E}_o)}]  \ar[d] \\
  \Cor_I^{\comp, \mathrm{(op)lax}}(\tau)_{(X_o)}[\mathcal{W}^{-1}_{(X_o)}]  \ar[r] &  \prod_{m} \Cor_I^{\comp, \mathrm{(op)lax}}(\Delta_{1,k_m})_{(X_o)}[\mathcal{W}^{-1}_{(X_o)}] 
 }
\end{equation*}
in which the vertical functors are opfibrations with discrete fibers and the lower vertical functor is an equivalence. 
Induction on Lemma~\ref{LEMMACOMPSIXFUOPLAX}, 2.\@ shows that the upper horizontal functor induces an equivalence on fibers. The statement follows. 
\end{proof}

\begin{DEF}\label{DEFDER6FUOPLAX1}
For a diagram $I \in \Catlf$, let $\EE^{\mathrm{(op)lax}}(I)$ be the (symmetric) 2-(multi)category obtained by applying Proposition~\ref{PROPCONSTR2CAT} (resp.\@ Proposition~\ref{PROPCONSTRSYMMULTI}) to the strict functor of Definition~\ref{DEFDER6FU1OPLAX2}
\[ \Delta_n \mapsto E_I^{\mathrm{(op)lax}}(\Delta_n) \quad (\text{resp. } \tau \mapsto E_I^{\mathrm{(op)lax}}(\tau), \quad \text{resp. } \tau^S \mapsto E_I^{\mathrm{(op)lax}}(\tau^S)).   \]
It comes equipped with an obvious morphism to $\SSS^{\cor, \comp,0,\mathrm{(op)lax}}(I)$ which was constructed applying the same Proposition to 
$\tau^S \mapsto \Cor_I^{\comp, \mathrm{(op)lax}}(\tau^S)$. 
\end{DEF}

\begin{PROP}\label{PROPFIBRATIONLAX}
The functors
\begin{eqnarray*}
 \EE^{\mathrm{(op)lax}}(I) &\rightarrow& \SSS^{\cor, 0, \comp, \mathrm{(op)lax}}(I)   
 \end{eqnarray*}
constructed in Definition~\ref{DEFDER6FUOPLAX1} are 1-opfibrations and 2-opfibrations with 1-categorical fibers. 
If $\DD \rightarrow \SSS^{\op}$ is infinite, and has stable, perfectly generated fibers, then the functor in the lax case is  also a 1-fibration. 
\end{PROP}
\begin{proof}
The first statement follows as in the plain case from Lemma~\ref{LEMMACOMPSIXFUOPLAX}, 3.\@
For the additional statement, we have to show that all push-forward functor have right adjoints w.r.t.\@ all slots. 
According to Lemma~\ref{LEMMACOMPSIXFUOPLAX}, 3.\@ the push-forward functors are given by $ \widetilde{f}_* \widetilde{\iota}_! \Box_! \widetilde{g}^*$ (Notation as in the Lemma).
We claim that all functors commute with homotopy colimits (as morphisms of derivators). For this it is irrelevant whether they are considered on the subderivator of (co)Cartesian objects since the latter are closed under homotopy colimits. 
The functors $\widetilde{\iota}_!$ and $\widetilde{g}^*$ have a right adjoint as morphism between the full derivators and thus commutes with homotopy colimits. 
The functor $\widetilde{f}_*$ commutes with homotopy colimits by (F3). 
The left Cartesian projector $\Box_!$ is defined (cf.\@ Proposition~\ref{PROPCARTPROJ}) as a composition of functors which commute with homotopy colimits.  
Thus the push-forward has a right adjoint by Brown representability. Note that $\EE^{\mathrm{(op)lax}}(I)_{X \hookrightarrow \overline{X}}$ are the same as $\EE(I)_{X \hookrightarrow \overline{X}}$ (plain case) and thus perfectly generated by Lemma~\ref{LEMMAGENERATION}.
\end{proof}

\begin{PROP}\label{PROPCARTPROJCOLIMITS}
Let $\tau = \Delta_{1,n}$ (cf.\@ also Example~\ref{EXCOMPONENTS}). 
The functor $\Box_! (\pi_{234}^*\widetilde{g})^*$ from \ref{CONDPOINTWISE} is a functor $\EE^{\lax}(J)_{S_1 \hookrightarrow \overline{S}_1} \times \dots \times \EE^{\lax}(J)_{S_n \hookrightarrow \overline{S}_n} \rightarrow \EE^{\lax}(J)_{A \hookrightarrow \overline{A}}$. 
It commutes with relative left Kan extensions for $\alpha: I \rightarrow J$ an opfibration in the following sense:
For all $i$, there are natural isomorphisms
\[ \alpha_!^{(A \hookrightarrow \overline{A})} \Box_! (\alpha^*\pi_{234}^*\widetilde{g})^* (\alpha^*-, \dots, -, \dots, \alpha^*-)  \rightarrow \Box_! (\pi_{234}\widetilde{g})^* (-, \dots, \alpha_!^{(S_i \hookrightarrow \overline{S}_i)}, \dots, -).  \]
If $\widetilde{g}$ is 1-ary then the above holds for $\alpha: i \hookrightarrow J$ the inclusion of an object.
\end{PROP}

This Proposition implies that, if $\DD \rightarrow \SSS^{\op}$ is infinite, and has stable, perfectly generated fibers then the (1-ary) pull-back functors in the fibration $\EE^{\lax}(I) \rightarrow \SSS^{\cor, 0, \comp, \lax}(I)$ are computed point-wise. 

If $\widetilde{g}$ is 1-ary and the functors $\Box_! (\pi_{234}^*\widetilde{g})^*$ have right adjoints, then Axiom (Der2) implies that the functor commutes with arbitrary $\alpha_!$ in the sense above.
This can probably be shown directly, but we will not need to do this, because we are interested in the right adjoints only, anyway.

\begin{proof}
Using the exchange morphism
\begin{equation*} \Box_! ((\twwc{\alpha})^* \pi_{234}^*\widetilde{g})^*((\twwc\alpha)^*-, \dots, (\twwc\alpha)^*-)  \rightarrow (\twwc\alpha)^* \Box_! (\pi_{234}^*\widetilde{g})^*(-, \dots, -)   \end{equation*}
define a natural transformation
\begin{eqnarray*} \alpha_!^{(A \hookrightarrow \overline{A})} \Box_! ((\twwc{\alpha})^*\pi_{234}^*\widetilde{g})^*(\alpha^*-, \dots, -, \dots, \alpha^*-) &\rightarrow& \alpha_!^{(T \hookrightarrow \overline{T})} \Box_! ((\twwc{\alpha})^*\pi_{234}^*\widetilde{g})^* (\alpha^*-, \dots, \alpha^* \alpha_!^{(S_i \hookrightarrow \overline{S}_i)} -, \dots, \alpha^* -)   \\
&\rightarrow& \alpha_!^{(A \hookrightarrow \overline{A})} \alpha^* \Box_! (\pi_{234}^*\widetilde{g})^*(-, \dots, \alpha_!^{(S_i \hookrightarrow \overline{S}_i)}-, \dots, -) \\
 &\rightarrow& \Box_! (\pi_{234}^*\widetilde{g})^* (-, \dots, \alpha_!^{(S_i \hookrightarrow \overline{S}_i)}-, \dots, -).
 \end{eqnarray*}
 Hence we have natural transformations even for any $\alpha$.  We first show that it is an isomorphism if $\alpha$ is an opfibration.
We have the following commutative diagram
 \[ \xymatrix{
\DD({}^{\downarrow \uparrow \uparrow \downarrow} J)^{2-{\cart}}_{\pi_{234}^* \widetilde{S}^{\op}} \ar@{^{(}->}[r] \ar[d]^{(\twwc \alpha)^*|_{2-{\cart}}} & \DD({}^{\downarrow \uparrow \uparrow \downarrow} J)^{}_{\pi_{234}^* \widetilde{S}^{\op}} \ar[d]^{(\twwc \alpha)^*} \\
\DD({}^{\downarrow \uparrow \uparrow \downarrow} I)^{2-{\cart}}_{\pi_{234}^* \alpha^*\widetilde{S}^{\op}} \ar@{^{(}->}[r] & \DD({}^{\downarrow \uparrow \uparrow \downarrow} J)^{}_{\pi_{234}^* \widetilde{S}^{\op}}
} \]
Obviously a left adjoint of the restriction $(\twwc \alpha)^*|_{2-{\cart}}$ of $(\twwc \alpha)^*$ to the 2-Cartesian subcategory is given by $\Box_! (\twwc \alpha)_!$. 
We have therefore as adjoint to the above commutative diagram:
\[ \Box_! (\twwc \alpha)_! \Box_! \cong \Box_! (\twwc \alpha)_!. \]
If we evaluate this natural transformation on the category 
\begin{equation}\label{eqsub1}
 \DD({}^{\downarrow \uparrow \uparrow \downarrow} J)^{4-{\cocart},3-{\cocart^*}}_{\pi_{234}^* \widetilde{S}^{\op}}  
 \end{equation}
then, since $\Box_!$ preserves the conditions of being strongly 3-coCartesian and 4-coCartesian (cf.\@ Proposition~\ref{PROPCARTPROJ}), the $(\twwc \alpha)_!$ on the left hand side receives an object in 
\begin{equation}\label{eqsub2}
 \DD({}^{\downarrow \uparrow \uparrow \downarrow} J)^{4-{\cocart},3-{\cocart^*},2-{\cart}}_{\pi_{234}^* \widetilde{S}^{\op}}  
 \end{equation}
where it is simply $\alpha_!^{(A \hookrightarrow \overline{A})}$ (cf.\@ Remark~\ref{REMKANEXT}), hence the leftmost $\Box_!$ on the left hand side is superfluous. We arrive at
\[ \alpha_!^{(A \hookrightarrow \overline{A})} \Box_! \cong \Box_! (\twwc \alpha)_! \]
(still on the subcategory (\ref{eqsub1})) and therefore
\[ \alpha_!^{(A \hookrightarrow \overline{A})} \Box_! ((\twwc{\alpha})^* \pi_{234}^*\widetilde{g})^*(\alpha^*-, \dots, -, \dots, \alpha^*-) \cong \Box_! (\twwc \alpha)_!  (\pi_{234}^*\widetilde{g})^*(\alpha^*-, \dots, -, \dots, \alpha^*-) \]
on the subcategory (\ref{eqsub2}), where the $\alpha^*$ denotes $\EE(\alpha)$, i.e.\@ $(\twwc \alpha)^*$, hence by (FDer5 left) for $\DD$: 
\[ \alpha_!^{(A \hookrightarrow \overline{A})} \Box_! ((\twwc{\alpha})^* \pi_{234}^* \widetilde{g})^*(\alpha^*-, \dots, -, \dots, \alpha^*-) \cong \Box_!  (\pi_{234}^*\widetilde{g})^*(-, \dots, \alpha_!^{(S_i \hookrightarrow \overline{S}_i)}  -, \dots, -). \]
using that $(\twwc \alpha)_! \cong \alpha_!^{(S_i \hookrightarrow \overline{S}_i)}$ on the subcategory (\ref{eqsub2}).

Now let $\widetilde{g}$ be 1-ary and $i \in J$ be an object. 
We have to show that 
\begin{equation}\label{toshowkanext}
 i_!^{(A \hookrightarrow \overline{A})} \widetilde{g}_i^* \rightarrow \Box_! (\pi_{234}^*\widetilde{g})^* i_!^{(S \hookrightarrow \overline{S})}  \end{equation}
is an isomorphism. Here  $\widetilde{g}_i$ is $\widetilde{g}$ evaluated at $\tww i$. 
This can be checked point-wise at $j \in J$:
\[ j^* i_!^{(A \hookrightarrow \overline{A})} \widetilde{g}_i^* \rightarrow (\twwc j)^* \pi_{4567,!} f_* \pi_{1267}^*  (\pi_{234}^*\widetilde{g})^* i_!^{(S \hookrightarrow \overline{S})}  \] 
Using that $\pi_{4567}$ is an opfibration with fiber (over $\twwc j$) equal to ${}^{\downarrow \uparrow \uparrow} J \times_{/J} j \hookrightarrow {}^{\downarrow \uparrow \uparrow \downarrow \uparrow \uparrow \downarrow} J$
(subcategory of elements of the form $j_1\rightarrow j_2 \rightarrow j_3 \rightarrow j=j=j=j$), we get
\[ j^* i_!^{(A \hookrightarrow \overline{A})} \widetilde{g}_i^* \rightarrow \hocolim_{{}^{\downarrow \uparrow \uparrow} J \times_{/J} j} f_* \pi_{1267}^*  (\pi_{234}^*\widetilde{g})^* i_!^{(S \hookrightarrow \overline{S})}  \] 
Then, inserting the precise calculation from Example~\ref{EXRELKAN}, we get for the right hand side
\[  \hocolim_{{}^{\downarrow \uparrow \uparrow} J \times_{/J} j} f_* \pi_{1267}^* (\pi_{234}^*\widetilde{g})^* (\twwc \pi)_! f_{1,*} \iota_{1,!} g_1^* p^*  \] 
for 
\[ \xymatrix{
i \times_{/J,\pi_1}  {}^{\downarrow \uparrow \uparrow}  (J \times_{/J} j )\ar[r] \ar[d]  & i \times_{/J,\pi_1} \twwc J \ar[d]^{\twwc \pi} \\
 {}^{\downarrow \uparrow \uparrow} (J \times_{/J} j)  \ar[r]^{\pi_{1267}} & \twwc J
} \]
which is homotopy exact. 

Using that $(\pi_{234}\widetilde{g})^*$ and $f_*$ commute with homotopy colimits, we can write this as
\[ \hocolim_{i \times_{/J,\pi_1}  {}^{\downarrow \uparrow \uparrow}  (J \times_{/J} j )} f_{2,*}  g_2^* f_{1,*} \iota_{1,!} g_1^* P^* \]
for $P: i \times_{/J,\pi_1}  {}^{\downarrow \uparrow \uparrow}  (J \times_{/J} j ) \rightarrow \cdot$,
where, denoting an object in the fiber by $i \rightarrow j_1 \rightarrow j_2 \rightarrow j_3 \rightarrow j = j = j = j$ the various morphisms are given point-wise by
\begin{eqnarray*}
g_1: \widetilde{S}(i = i \rightarrow j) &\rightarrow&  \widetilde{S}(i=i=i)  \\
\iota_1:  \widetilde{S}(i = i \rightarrow j) &\rightarrow&   \widetilde{S}(i \rightarrow j = j)   \\
f_1: \widetilde{S}(i \rightarrow j = j)& \rightarrow&  \widetilde{S}( j_2 \rightarrow j = j)   \\
g_2:  \widetilde{A}( j_2 \rightarrow j = j) &\rightarrow &\widetilde{S}( j_2 \rightarrow j = j)   \\
f_2:  \widetilde{A}( j_2 \rightarrow j = j) &\rightarrow& \widetilde{A}( j = j = j)  
\end{eqnarray*}
The argument of $\hocolim_{i \times_{/J,\pi_1}  {}^{\downarrow \uparrow \uparrow}  (J \times_{/J} j )}$ does not depend on $j_1$ and $j_3$. We therefore factor
\[ \xymatrix{ i \times_{/J,\pi_1}  {}^{\downarrow \uparrow \uparrow}  (J \times_{/J} j ) \ar[r]^-{\pi_2} & (i \times_{/J} J \times_{/J} j )^{\op} \ar[r]^-{P_1} &  \cdot } \]
Like in \ref{REMKANEXT} one shows that for $\pi_2$ we have $\pi_{2,!} \pi_2^* \cong \id$ (it is the composition of an opfibration and a fibration each with contractible fibers). 
Hence we are left with a homotopy colimit over $(i \times_{/J} J \times_{/J} j )^{\op}$,
which splits up into a union over $\Hom_J(i,j)$ and on each component is evaluation at the final object $i=i  =i \rightarrow  j = j = j = j = j$. Hence we can set $j_2:=i$ in the formulas above and replace $P$ with 
$P_2: i \times_{/J} j \rightarrow \cdot$ (Note:  $i \times_{/J} j$ is the discrete category with objects $\Hom_J(i,j)$). 
Now we have a commutative diagram:
\[ \xymatrix{
\widetilde{A}(i = i =i ) \ar[d]^{\widetilde{g}_i} & \ar[l]_{G} \widetilde{A}(i = i \rightarrow j) \ar[r]^{\iota_3} \ar[d]^{g_3} & \widetilde{A}( i \rightarrow j = j) \ar[d]^{g_2} \ar@{=}[r]  & \widetilde{A}(i \rightarrow j = j) \ar[d]^{g_2} \ar[r]^{f_2}  & \widetilde{A}(j = j = j) \ar[d]^{\widetilde{g}_j} \\
\widetilde{S}(i = i =i ) &  \ar[l]^{g_1} \widetilde{S}(i = i \rightarrow j)  \ar[r]_{\iota_1} & \widetilde{S}( i \rightarrow j = j) \ar@{=}[r]_{f_1} &  \widetilde{S}(i \rightarrow j = j) \ar[r]_{f_2} &  \widetilde{S}(j = j = j)
} \]
The left hand square is Cartesian by Lemma~\ref{LEMMACARTDIA}, hence we have
\[ g_2^* \iota_{1,!} \cong \iota_{3,!} g_3^* \]
Hence we arrive at
\[ \hocolim_{i \times_{/J} j} f_{2,*}  \iota_{3,!} G^* P_2^* \widetilde{g}_i^*.  \]
This is the same as the left hand side of (\ref{toshowkanext}) (use Kan's formula (FDer4 left)). A tedious check shows that the natural transformations match. 
\end{proof}

\begin{DEF}\label{DEFDER6FU2OPLAX}With the notation as in \ref{BEGINSECTIONDER6FU},
we construct morphisms of (symmetric) pre-2-(multi)derivators with domain $\Catlf$
\[
 \EE^{\mathrm{(op)lax}} \rightarrow \SSS^{\cor,\comp,0,\mathrm{(op)lax}} 
\]
The values have been constructed in \ref{DEFDER6FU1OPLAX}.  For a functor $\alpha: I \rightarrow J$ in $\Catlf$ we define the pullback $\alpha^*=\EE(\alpha)$
to be $\DD(\twwc \alpha \times \id \times \id)$. Note that $\alpha$ induces a functor $\twwc \alpha: \twwc I \rightarrow \twwc J$ and that $\DD(\twwc \alpha \times \id \times \id)$ preserves the relevant conditions of being (co)Cartesian. 
The pre-2-multiderivator $\EE$ is defined on natural transformations as follows. A natural transformation $\mu: \alpha \Rightarrow \beta$ can be seen as a functor $\mu: \Delta_1 \times I \rightarrow J$.
The pullback of a diagram in $\mathcal{E} \in \EE(\twwc J \times \Delta^{\op} \times \Delta)$ and taking partial underlying diagram gives a functor in   $\Fun(\twwc \Delta_1, \EE(\twwc I \times \Delta^{\op} \times \Delta))$ which has the right
strong (co)Cartesianity conditions. It is, by definition, a morphism \[ \alpha^* \mathcal{E} = e_0^* \mu^* \mathcal{E} \rightarrow  \beta^* \mathcal{E} = e_1^*  \mu^* \mathcal{E} \]  in $\EE(I)$ which we 
define to be the pseudo-natural transformation $\EE(\mu)$ at $\mathcal{E}$. 
We have to see that the contraints of the analogous pseudo-naturality for $\SSS^{\cor,\comp,0,\mathrm{(op)lax}}$ lift (in fact, they lift uniquely, because we have 1-categorical fibers). 
We will sketch the oplax case for a 1-ary morphism. 

 A 1-morphism given by an element in $E_I^{\oplax}(\Delta_1)$ (note that a general morphism in $\EE(I)$ is freely generated by those) consists of a compactification $X \hookrightarrow \overline{X}$ in $\mathcal{S}^{\tw (\Delta_1 \times I)}$ plus an object
  \[ \mathcal{F} \in \Fun(\twwc \Delta_1,  \DD( \twwc   I))_{\pi_{234}^*\widetilde{X}^{\op}}^{4-\cocart, 3-\cocart^*, 2-\oplax-{\cart}}.  \]
  It is a morphism from $\mathcal{E}_1 := e_0^*\mathcal{F}$ to $\mathcal{E}_2 := e_1^*\mathcal{F}$.
  In $\mathcal{S}^{\cor,0,\oplax}$ we have the following oplax square  
  We have a 2-morphism
    \begin{equation} \label{eqoplaxmor} \vcenter{ \xymatrix{ \alpha^* X_1  \ar[r] \ar[d] \ar@{}[rd]|{\Downarrow} &  \beta^* X_1  \ar[d]  \\
  \beta^* X_2 \ar[r]  & \beta^* X_2 
    } } \end{equation}
    obtained as follows: The pullback $\mu^* X$ and taking the partial underlying diagram  gives rise to a diagram  in $\mathcal{S}^{\tw (\Delta_1^2 \times I)}$: 
 \begin{equation}\label{eqsquare} \vcenter{ \xymatrix{ \alpha^*X_1   & \ar[l]_-{G_1} \ar[r]^-{F_1}   \ar@{}[rd]|\Box &  \beta^*X_1  \\
 \ar[u]^{\alpha^*f } \ar[d]_{\alpha^* g} \ar@{}[rd]|\blacksquare A_1  & \ar[r]^-{F_2} \ar[d]^{g_\mu} W \ar[u]_{f_\mu} \ar[l]_-{G_2} & \ar[u]_{\beta^*f} \ar[d]^{\beta^* g} \\
   \alpha^*X_2   & \ar[l]^-{G_3} \ar[r]_-{F_3  } A_2  & \beta^*X_2   }  
 } \end{equation}
 in which the square denoted $\Box$ is Cartesian and the square denoted $\blacksquare$ is weakly Cartesian. 
 Denoting
 \[ \xymatrix{  W \ar[rrd]^{G_2} \ar[ddr]_{f_\mu}  \ar[rd]^h \\ 
 &  C \ar[r]^G \ar[d]^F \ar@{}[rd]|\Box & \ar[d] A_1 \\
 &  A_2  \ar[r] &   \alpha^*X_2 } \]
 it is the proper morphism $h$ that induces the 2-morphism in (\ref{eqoplaxmor}). We have to show that the square lifts to a lax commutative square:
    \begin{equation*}  \vcenter{ \xymatrix{ \alpha^* \mathcal{E}_1  \ar[r] \ar[d] \ar@{}[rd]|{\Downarrow} &  \beta^* \mathcal{E}_1  \ar[d]  \\
  \beta^* \mathcal{E}_2 \ar[r]  & \beta^* \mathcal{E}_2 
    } } \end{equation*}
    
    We have already seen that $\EE(I) \rightarrow \SSS^{\cor,\comp,0,\oplax}(I)$ can, as obfibration, be transported to an obfibration $\EE'(I) \rightarrow \SSS^{\cor,0,\oplax}(I)$ and hence we can forget about the compactifications to construct
    the 2-isomorphism. 
    
    The morphism given by $\mathcal{F}$ is determined by  
 \[ f_! g^* \mathcal{E}_1 \rightarrow \mathcal{E}_2  \]
 where $f_!$ is the oplax push-forward (defined as $\Box_* \overline{f}_* \iota_!$ w.r.t.\@ any compactification of $f$).

In view of Lemma~\ref{LEMMACOMPSIXFUOPLAX}, 4.\@ we have to show that the existence of $\mu^* \mathcal{F}$ gives rise to the commutativity of the diagram:  
\[ \footnotesize \xymatrix{ (\beta^* f)_! F_{2,!} g^*_\mu G_1^* \alpha^* \mathcal{E} \ar[d] \ar@{=}[r] & F_{3,!} f_{\mu,!} G_2^* (\alpha^*g)^* \alpha^* \mathcal{E}  \ar@{=}[r] & F_{3,!} F_! \underbrace{h_*}_{=h_!} h^* G^* (\alpha^*g)^* \alpha^* \mathcal{E} & \ar[l]^-{\mathrm{unit}} F_{3,!} F_!  G^* (\alpha^*g)^* \alpha^* \mathcal{E}  \ar[d]  \\
  (\beta^* f)_! (\beta^* g)^* F_{1,!} G_1^* \alpha^* \mathcal{E} \ar[d]  & & & F_{3,!} G_3^*  (\alpha^* f)_!  (\alpha^*g)^* \alpha^* \mathcal{E}  \ar[d]\\
 (\beta^* f)_! (\beta^* g)^*  \beta^* \mathcal{E} = \beta^* (f_! g^* \mathcal{E}) \ar[d] & & & F_{3,!} G_3^*  \alpha^* \mathcal{F}  \ar[d]  \\
\beta^*  \mathcal{F} \ar@{=}[rrr]& & & \beta^*  \mathcal{F} 
      } \]
 This is omitted. 
 
This endows $\EE(\mu): \alpha^* \Rightarrow \beta^*$ with the structure of pseudo-natural transformation and one checks that this construction yields a pseudo-functor:
\[ \Fun(I, J) \rightarrow \Fun^{\mathrm{strict}}(\EE^{\oplax}(J), \EE^{\oplax}(I)). \]

The lax case is done using the adjoint functors which are computed point-wise (i.e.\@ commute with $\alpha^*$ and $\beta^*$). For the right adjoint of $\Box_! \widetilde{g}^*$ this follows from Proposition~\ref{PROPCARTPROJCOLIMITS}.
\end{DEF}

\begin{HAUPTSATZ}\label{HAUPTSATZOPLAX}Let $\mathcal{S}$ be a category with compactifications, and let $\SSS^{\op}$ be the symmetric pre-multiderivator represented by $\mathcal{S}^{\op}$ with the symmetric multicategory structure \ref{PAROPMULTCAT}. Let $\DD \rightarrow \SSS^{\op}$ be a left and right (symmetric) fibered (multi)derivator with domain $\Invlf$ satisfying axioms (F1--F6) and (F4m--F5m) of \ref{PARAXIOMS}. 
Assume that $\DD$ is infinite (i.e.\@ satisfies (Der1${}^\infty$)). 
\begin{enumerate}
\item

The morphism of oplax pre-2-multiderivators
\[ \EE^{\oplax} \rightarrow \SSS^{\cor, 0, \comp, \oplax}  \]
constructed in Definition~\ref{DEFDER6FUOPLAX1} is an oplax left (symmetric) fibered (multi)derivator with domain $\Catlf$.
\item 
If $\DD \rightarrow \SSS^{\op}$
 is infinite and has stable and perfectly generated fibers then the morphism of lax pre-2-multiderivators 
 \[ \EE^{\lax} \rightarrow \SSS^{\cor,0,\comp,\lax} \] 
constructed in Definition~\ref{DEFDER6FUOPLAX1} is a lax right (symmetric) fibered (multi)derivator with domain $\Catlf$. 
\end{enumerate}
\end{HAUPTSATZ}
\begin{proof}
1.\@ As in the plain case, with the following modifications:

The second statement of (FDer0 left) follows from Proposition~\ref{PROPCOCARTPROJ} because on objects of the form $\twwc i$ the functor $\Box_*$ does not do anything.

(FDer5 left)  Because by the strong form of (FDer0 left) the push-forward along oplax morphisms is computed point-wise it suffices to see this for projections $p: I \rightarrow \cdot$, i.e.\@ for homotopy colimits. Over a point though the condition ``oplax'' is vacuous.

2.\@
We have to show that the resulting morphism of 2-pre-multiderivators
\[ \EE^{\lax} \rightarrow \SSS^{\cor, 0,\comp, \lax} \]
is a lax right fibered multiderivator: The 1-fiberedness of $\EE^{\lax}(I) \rightarrow \SSS^{\cor,0,\comp,\lax}(I)$ was obtained in Proposition~\ref{PROPFIBRATIONLAX}. 
The second statement of (FDer0 right) --- which involves only opfibrations --- follows from Proposition~\ref{PROPCARTPROJCOLIMITS}. For 1-ary morphisms 
the pull-back is even computed point-wise which follows from the additional statement of Proposition~\ref{PROPCARTPROJCOLIMITS}. 
(FDer5 right) --- which (in the lax case) involves only relative right Kan extensions along fibrations --- follows from Lemma~\ref{COMMCARTPROJFIB}.
\end{proof}

As in the plain case, one can construct equivalent oplax left (resp.\@ lax right) (symmetric) fibered (multi)derivators
\[ \EE^{\mathrm{(op)lax}} \rightarrow \SSS^{\cor,0,\mathrm{(op)lax}}.   \]
In other words (in the multi-case), we actually get a (symmetric) proper derivator six-functor-formalism as in Definition~\ref{DEF6FUDER}. Furthermore, as in the plain case, 
these fibered (multi)derivators extend to $\Cat$. Indeed, the fibers $\EE^{\mathrm{(op)lax}}(I)_{X}$ are the same as $\EE(I)_X$ (plain case) so it is just a matter of extending
the oplax push-forward, resp.\@ lax pull-back which is straight-forward using the definition of the extension in \cite{Hor17b}, cf.\@ also the construction in Proposition~\ref{PROPIOTA}.

\section{Example I: The classical six-functor-formalism on topological spaces}\label{SECTIONTOP}

In this section, a derivator enhancement of the classical six-functor-formalism on topological spaces will be constructed. 
The input symmetric fibered multiderivator exists in very great generality on arbitrary ringed sites (using work of Hovey, Gillespie, and Recktenwald).
To dispose of a category with compactifications for which the axioms of \ref{PARAXIOMS} hold one has to restrict the situation appropriately. 

\begin{DEF}
Let $(X, \OO_X)$ be a ringed site. Denote by
\[ \mathrm{Ch}(\mathrm{Mod}(X, \OO_X)) \]
the category of {\em unbounded} complexes of $\OO_X$-module sheaves on $X$. Denote by
$\mathcal{W}$ the class of quasi-isomorphisms. For a morphisms of ringed sites 
\[f: (X, \OO_X) \rightarrow (Y, \OO_Y) \]
denote by $f_*$ and $f^*$ the pull-back and push-forward functors and by $\otimes, \mathcal{HOM}$ the tensor product and internal Hom 
extended in the usual way to complexes. Actually, those are better encoded (together with their coherence isomorphisms) by saying that there is a 
bifibration of symmetric multicategories
\[ \mathrm{Ch}(\mathrm{Mod}) \rightarrow \mathrm{RSite}^{\op} \]
whose fibers are the above categories of complexes and where $\mathrm{RSite}^{\op}$ carries the natural symmetric multicategory structure \ref{PAROPMULTCAT}.
\end{DEF}

\begin{SATZ}[Hovey, Gillespie, Recktenwald]\label{SATZRECKTENWALD}
There are model category structures \[ (\mathrm{Ch}(\mathrm{Mod}(X, \OO_X)), \Cof, \Fib, \mathcal{W}) \] (for varying $(X, \OO_X)$ and with $\mathcal{W}$ being the quasi-isomorphisms) such that 
\[ \mathrm{Ch}(\mathrm{Mod}) \rightarrow \mathrm{RSite}^{\op} \]
is a bifibration of multi-model categories in the sense of \cite[Definition~5.1.3]{Hor15}. In this case this statement boils down to 
$\otimes, \mathcal{HOM}$ being a Quillen adjunction of two variables (plus a condition on units, i.e.\@ the model categories have to be monoidal) and for a
morphism $f: (X, \OO_X) \rightarrow (Y, \OO_Y)$ of ringed sites 
to $f^*, f_*$ being a Quillen adjunction.
\end{SATZ}
\begin{proof}
The model category structures in question are the ones obtained from the (flat, cotorsion) {\em cotorsion pair} on the Abelian category of sheaves of $\OO_X$-modules on $X$. 
See \cite[Corollary~3.3.5]{Rec19} which builds on ideas of Hovey \cite{Hov07} and Gillespie \cite{Gil06}. 
\end{proof}

\begin{KOR}\label{KORRECKTENWALD}
Let $\mathcal{S} = \mathrm{RSite}$ be the category of ringed sites. There is a symmetric fibered multiderivator with domain $\Cat$
\[ \mathbb{M} \rightarrow \mathbb{S}^{\op} \]
such that for a ringed space $(X, \OO_X)$ the fiber $I \mapsto \mathbb{M}(I)_{p^*(X, \OO_X)}$ is the usual derivator associated with the category $\mathrm{Ch}(\mathrm{Mod}(X, \OO_X))$ of unbounded complexes of sheaves of $\OO_X$-modules considered above. The fibered multiderivator is infinite and has stable, perfectly generated fibers. 
\end{KOR}
\begin{proof}\cite[Theorem 6.2]{Hor17b} using Theorem~\ref{SATZRECKTENWALD}. The fact that the fibers are perfectly generated is \cite[Theorem 4.1.12]{Rec19}.
\end{proof}

We proceed to discuss the verification of the axioms (F1--F6), (F4m--F5m) for certain ringed topological spaces.

\begin{PAR}\label{PARCOMPTOP}
Let $\mathcal{S}$ be the category of locally compact Hausdorff spaces.
Consider the following subclasses of morphisms in $\mathcal{S}$:
\begin{eqnarray*} 
\mathcal{S}_0 &:=& \{ \text{ open embeddings } \} \\
 \mathcal{S}_1 &:=& \{ \text{ dense open embeddings } \} \\
 \mathcal{S}_2 &:=& \{ \text{ proper morphisms } \}  
\end{eqnarray*} 
\end{PAR}

\begin{PROP}
The structure in \ref{PARCOMPTOP} is a category with compactificatiions in the sense of \ref{DEFCATCOMP}.
\end{PROP}
\begin{proof}
\cite[Example~4.1.3]{Rec19}.
\end{proof}

\begin{PAR}
Consider the following axioms for a topological space $X$ (cf.\@ \cite[4.5.1--2]{Rec19}):
\begin{itemize}
\item[(\dag)]  For every exact complex $\mathcal{E}_\bullet$ of soft sheaves $\Gamma(X;S_\bullet)$ is exact as well.
\item[(c\dag)]  For every exact complex $\mathcal{E}_\bullet$ of $c$-soft\footnote{i.e.\@ sheaves $\mathcal{E}$ such that $\Gamma(X; \mathcal{E}) \rightarrow \Gamma(K; \mathcal{E})$ is surjective for any compact $K \subseteq X$} sheaves $\Gamma_c(X;S_\bullet)$ is exact as well.
\end{itemize}
For compact spaces (\dag) and (c\dag) are equivalent. 
The property (c\dag) is stable under restriction to locally closed subspaces and under fiber products. Furthermore, the full subcategory of locally compact Hausdorff spaces satisfying (c\dag) is also a category with compactifications (with the restriction of the structure in \ref{PARCOMPTOP}), cf.\@ \cite[4.5.3--5]{Rec19}.
\end{PAR}

\begin{PROP}
Every topological space which is locally of finite cohomological dimension has {\em (c\dag)}. 
\end{PROP}
\begin{proof}
\cite[Theorem~4.5.9]{Rec19}.
\end{proof}

\begin{PAR}\label{PARPROPERTIESTOP}
Let $\mathcal{S} \subset \mathrm{RTop}$ be a subcategory (closed under fiber products) of the category of ringed spaces satisfying:
\begin{enumerate}
\item For all proper morphisms $f: X \rightarrow Y$ in $\mathcal{S}$ the fibers have property (\dag). (This is for example the case, if $X$ has  (c\dag) )
\item Every Cartesian square 
\[ \xymatrix{ (W, \OO_W) \ar[r]  \ar[d] & (Z, \OO_Z) \ar[d] &  \\
(Y, \OO_Y) \ar[r] & (X, \OO_X) & }\]
in $\mathcal{S}$ is Tor-independent. 
\item The (dense) open embeddings and proper morphisms as in \ref{PARCOMPTOP} form a compactification structure.
\end{enumerate}
cf.\@ \cite[Corollary~4.6.2]{Rec19}.
\end{PAR}
If 2.\@ is not satisfied it would be, of course, appropriate to work with cosimplicial rings instead...

\begin{BEISPIEL}\label{EXTOP}
Let $\mathcal{S}$ be the category of locally compact Hausdorff topological spaces satisfying {\em (c\dag)} and let $R$ be any commutative ring.
Endow every object in $\mathcal{S}$ with the constant sheaf defined by $R$. Then $\mathcal{S} \subset \mathrm{RTop}$ satisfies the properties of \ref{PARPROPERTIESTOP}.
\end{BEISPIEL}
cf.\@ \cite[Corollary~4.6.3]{Rec19}.

\begin{SATZ}[Spaltenstein, Recktenwald] Let $\mathcal{S}$ be a category of ringed spaces satisfying the axioms of \ref{PARPROPERTIESTOP} (e.g.\@ the category of Example~\ref{EXTOP}).
Then the restriction of the symmetric fibered multiderivator
\[ \mathbb{M} \rightarrow \mathbb{S}^{\op} \]
of Corollary~\ref{KORRECKTENWALD} satisfies the axioms (F1--F6), and (F4m--F5m) of \ref{PARAXIOMS} w.r.t.\@ the natural compactification structure \ref{PARCOMPTOP}.
\end{SATZ}
\begin{proof}
\cite[Corollary~4.6.2]{Rec19} building on ideas of Spaltenstein \cite{Spa88}.
\end{proof}

\begin{KOR}
 Let $\mathcal{S}$ be a category of ringed spaces satisfying the axioms of \ref{PARPROPERTIESTOP} (e.g.\@ the category of Example~\ref{EXTOP}).
 There is a symmetric derivator six-functor-formalism (i.e.\@ symmetric fibered multiderivator)
 \[ \mathbb{M} \rightarrow \SSS^{\cor}   \]
 such that for $(X, \OO_X) \in \mathcal{S}$ the fiber $I \mapsto \mathbb{M}(I)_{p^*(X, \OO_X)}$ is the usual derivator associated with the category $\mathrm{Ch}(\mathrm{Mod}(X, \OO_X))$ of unbounded complexes of sheaves of $\OO_X$-modules, and such that the push-forward along a multicorrespondence  
\[ \xymatrix{ & & & (A, \OO_A)  \ar[llld]_{g_1}\ar[ld]^{g_n}\ar[rd]^f \\
(S_1, \OO_{S_1}) & \cdots & (S_n, \OO_{S_n}) & ; & (T, \OO_T) 
 }\]
 in $\SSS^{\cor}(\cdot)$ is given up to unique isomorphism by
\[ f_! ( L g^*_1 - \overset{L}{\otimes} \cdots \overset{L}{\otimes} Lg_n^* -) \]
with the usual (derived) push-forward $f_!$ with proper support. Furthermore this extends to a proper six-functor-formalism in the sense of Definition~\ref{DEF6FUDER}, i.e.\@ we have extensions as oplax left fibered multiderivator and lax right fibered multiderivator
 \[ \mathbb{M} \rightarrow \SSS^{\cor,0,\oplax}  \quad  \mathbb{M} \rightarrow \SSS^{\cor,0,\lax} \]
\end{KOR}
\begin{proof}Apply Corollary~\ref{KORDER6FU} and Theorem~\ref{HAUPTSATZOPLAX}.
\end{proof}

\section{Example II: The stable homotopy categories and categories of motives}\label{SECTIONAYOUB}

In this section, it will be shown that the algebraic derivators $\mathbb{SH}$ of Ayoub which comprise
the stable homotopy categories of Morel-Voevodsky and various kinds of Voevodsky motives give rise to a symmetric fibered multiderivator satisfying the axioms of \ref{PARAXIOMS} and thus yield a symmetric derivator six-functor-formalism.

\begin{PAR} \label{PARCOMPSCH}
To discuss the examples of Ayoub, we take $\mathcal{S} := \mathcal{SCH}_S$ the category of quasi-projective schemes over a base scheme $S$.
$\mathcal{SCH}_S$ is equipped with a natural compactification structure in which
\begin{eqnarray*} 
\mathcal{S}_0 &=& \{ \text{ open immersions } \} \\
 \mathcal{S}_1 &=& \{ \text{ dense open immersions } \} \\
 \mathcal{S}_2 &=& \{ \text{ projective morphisms } \}  
 \end{eqnarray*}
 We leave it to the reader to check the axioms of \ref{DEFCATCOMP}. For (S5) see \cite[Lemme~1.3.9]{Ayo07I}.

The setting of Ayoub has been generalized to more general schemes over $S$. We will show in a subsequent article \cite{Hor22} that
even its restriction to affine schemes  of finite type over $S$ extends automatically to arbitrary schemes locally of finite type over $S$ (in fact, to certain higher geometric stacks). Hence the 
restriction does not matter so much. 
\end{PAR}

Recall \cite[Definition 4.4.23]{Ayo07I}:
\begin{DEF}\label{DEFCATCOEFF}
A {\bf category of coefficients} is a model category $\mathcal{M}$ with the following properties
\begin{enumerate}
\item $\mathcal{M}$ is left proper, cofibrantly generated, and stable;
\item the weak equivalences are stable under finite coproducts;
\item there is a subset $\mathcal{E} \subset \mathcal{M}$ of homotopically compact\footnote{\cite[Definition 4.4.22]{Ayo07II}} objects which generate $h(\mathcal{M})$ under arbitrary coproducts. 
\end{enumerate}
\end{DEF}

\begin{PAR}\label{PARSETTINGAYOUB}
Consider a triple $(\tau, \mathcal{M}, T)$ as in Ayoub \cite[Section~4.5]{Ayo07II} in which
\begin{itemize}
\item $\tau$ is either the etale or Nisnevich topology on $\mathcal{SCH}_S$.
\item $\mathcal{M}$ is a category of coefficients in the sense of Definition~\ref{DEFCATCOEFF}.
\item $T$ is a projectively cofibrant object of $\mathrm{PreShv}(\mathrm{Sm}/S, \mathcal{M})$ with the condition in \cite[4.5.18]{Ayo07II}.
\end{itemize}
\end{PAR}

\begin{SATZ}[Ayoub]\label{SATZAYOUB1}
Let $\SSS^{\op}$ be the symmetric pre-multiderivator represented by $\mathcal{S}^{\op} = \mathcal{SCH}_S^{\op}$ with the  symmetric multicategory structure \ref{PAROPMULTCAT}.
There is a symmetric fibered multiderivator with domain $\Cat$
\[ \mathbb{SH}^T_{\mathcal{M}} \rightarrow \SSS^{\op} \]
such that for a diagram $F: I \rightarrow \mathcal{S}^{\op}$ of schemes, we have
\[ \mathbb{SH}^T_{\mathcal{M}}(I)_F = \mathbb{SH}^{T}_{\mathcal{M}}(F^{\op}, I^{\op}) \]
where the right hand side is the ``algebraic derivator'' defined by Ayoub \cite[D\'efinition 4.5.21]{Ayo07I} and such that the pull-back along a multimorphism $g=(g_1, \dots, g_n)$ is given up to unique isomorphism by
\[ (L g^*_1 - ) \overset{L}{\otimes} \cdots \overset{L}{\otimes} (L g_n^*-) \]
with the functors $g_i^*$ as in \cite[Th\'eor\`eme~4.5.23]{Ayo07II}.
The fibered multiderivator is infinite (i.e.\@ satisfies (Der1${}^\infty$) and has stable, well-generated fibers. 
\end{SATZ}
\begin{proof}By \cite[Th\'eor\`eme~4.5.24]{Ayo07II}, the association 
\[ (I, F) \mapsto \mathbb{SH}^{T}_{\mathcal{M}}(F^{\op}, I^{\op}) \quad
 (\alpha, f) \mapsto L (f^{\op}, \alpha^{\op})^* \]
defines a pseudo-functor
\[ \Cat(\mathcal{S}^{\op}) \rightarrow \mathcal{CAT} \]
 noting that we have an isomorphism of strict 2-categories $\Cat(\mathcal{S}^{\op}) \cong (\mathrm{DiaSch}/S)^{1-\op}$ given by $(I, F) \mapsto (F^{\op}, I^{\op})$. 
Fixing $I$, the Grothendieck construction (applied to the obvious composition $\SSS^{\op}(I) \rightarrow \Cat(\mathcal{S}^{\op}) \rightarrow \mathcal{CAT})$ yields an opfibration
\begin{equation}\label{eqmpd}  \mathbb{SH}^T_{\mathcal{M}}(I) \rightarrow \SSS^{\op}(I). \end{equation}
The symmetric monoidal structure on $\mathbb{SH}^T_{\mathcal{M}}(F^{\op}, I^{\op})$ turns this into an opfibration of symmetric multicategories because of the monoidality of the $(g^{\op})_\bullet = L g^*$ (cf.\@ \cite[Th\'eor\`eme~4.5.24]{Ayo07II}). This defines the values of the pre-multiderivator $\mathbb{SH}^T_{\mathcal{M}}$.
One checks that the functoriality in $I$ turns $\mathbb{SH}^T_{\mathcal{M}}$ into a pre-multiderivator. We need to show the axioms of a fibered multiderivator:

(FDer0 left) holds by construction.

The first part of (FDer0 right), i.e.\@ the fiberedness of (\ref{eqmpd}), follows from Ayoub's axiom (DerAlg 2d) and the fact that the categories $\mathbb{SH}^{T}_{\mathcal{M}}(F^{\op}, I^{\op})$ are closed monoidal (the closedness is part of Ayoub's definition of monoidal model category). The second part of (FDer0 right) follows from (FDer5 left) by adjunction (cf.\@ \cite[Lemma 2.3.9]{Hor15}) and the validity of the latter axiom  will be shown below. 
 
(Der1) is Ayoub's axiom (DerAlg 0) and is clear in this case. Since the values of $\mathbb{SH}^{T}_{\mathcal{M}}$ are obtained as homotopy categories of model categories also
(Der1${}^{\infty}$) holds, i.e.\@ $\mathbb{SH}^{T}_{\mathcal{M}}$ is infinite. 

(Der2) follows from Ayoub's axiom (DerAlg 1).

(FDer3 left) follows from Ayoub's axiom (DerAlg 2g).

(FDer3 right) follows from Ayoub's axiom (DerAlg 2d).

(FDer4 left) follows from Ayoub's axiom (DerAlg 4'g) \cite[Remark~2.4.16]{Ayo07I}, cf.\@ 
\cite[Lemme~4.5.5]{Ayo07II}.

(FDer4 right) follows from (FDer4 left) and the other axioms\footnote{Indeed, it follows (FDer4 left), (FDer0 left), and (Der2) that for every diagram
\[ \xymatrix{  I \times_{/J} K \ar[r]^-B  \ar[d]_A \ar@{}[rd]|{\Swarrow^\mu}  & I \ar[d]^{\alpha} \\
K \ar[r]_-\beta  & J }  \]
the exchange 
$A_*^{(\beta^*S)} \SSS(\mu)(S)_\bullet B^* \rightarrow \beta^* \alpha_*^{(S)}$
is an isomorphism and hence also its adjoint
$ \alpha^* \beta_!^{(S)}  \rightarrow B_!^{(\alpha^*S)} \SSS(\mu)(S)^\bullet  A^*$.
The axiom (FDer4 right) is the special case $I=\{i\}$.}.

(FDer5 left) In the presence of the other axioms, the statement can be shown point-wise, i.e.\@ for a tuple of morphisms $g = (g_i)$ in $\mathcal{S}^{\op}$ and a diagram $I$ with projection $\pi: I \rightarrow \cdot$ the morphism
\[ \pi_! (\pi^* g)_\bullet(\pi^*-, \dots, \pi^*-, -, \pi^*-, \dots, \pi^*-) \rightarrow f_\bullet(-, \dots, -, \pi_!-, -, \dots, -)  \]
has to be an isomorphism. In other words $g_\bullet$ has to commute with homotopy colimits in all variables. For composites this may be shown for each factor individually. 
For the $g_{i,\bullet} = L (g_i^{\op})^*$ this follows because those have an adjoint which is computed point-wise, or also by the other axioms of a left fibered multiderivator, cf.\@ \cite[Proposition~2.3.26]{Hor15}. Hence it boils down to the statement that $\overset{L}{\otimes}$ commutes with homotopy colimits in both variables. This follows because $\otimes$ is left Quillen by definition of a monoidal model category. Note that $\mathbb{SH}^{T}_{\mathcal{M}}(S, \cdot)$ is, by definition, the homotopy category of such (cf.\@ \cite[Section~4.5.2]{Ayo07I}). 

Ayoub's axioms (DerAlg 0--4) are shown for $\mathbb{SH}$ in \cite[Th\'eor\`eme~4.5.30]{Ayo07II}.
The fibers of $\mathbb{SH}^{T}_{\mathcal{M}}$ are stable because those are the (usual) derivators associated with a stable model category \cite[Corollaire~4.4.21]{Ayo07II}, \cite[Lemme~4.4.35]{Ayo07II} and \cite[Corollaire~4.3.77]{Ayo07II}. In fact, this is the same reasoning as in \cite[Th\'eor\`eme~4.5.24]{Ayo07II} showing that the values $\mathbb{SH}^{T}_{\mathcal{M}}(F^{\op}, I^{\op})$ are triangulated. 
The fact that the triangulated categories $\mathbb{SH}^{T}_{\mathcal{M}}(X, \cdot)$ are well-generated in the sense of Neeman is stated in \cite{CD19}.
\end{proof}

\begin{BEISPIEL}
For example, let $\mathcal{M}$ be (unbounded) complexes of $\Lambda$-modules for a commutative ring $\Lambda$, let $T = (\PP^1_S, \infty_S) \otimes \Lambda$, and let $\tau$ be the etale topology. Then \cite[\S 3]{Ayo14}: 
\[ \mathbb{SH}_\mathcal{M}^T(\cdot)_X \cong \mathbb{DA}^{et}(S, \Lambda)  \]
where the right hand side is the category of etale Voevodsky motives without transfers (which is often equivalent to those with transfers). 
\end{BEISPIEL}

\begin{PAR}
To obtain a derivator six-functor-formalism using the construction in this article the  axioms of \ref{PARAXIOMS} have to be checked.
The constructions in this article comprise the construction of the $!$-functors together with the proofs of the validity of base-change, projection formula etc.\@
which is already done in Ayoub's work \cite{Ayo07I}. Nevertheless, a derivator six-functor-formalism is a substantially richer object than
the combination of crossed functors and algebraic derivators of Ayoub, for it comprises categories of ``coherent diagrams'' over diagrams {\em of correspondences} of schemes. Hence, in any case, it is necessary to redo the constructions. 
\end{PAR}

\begin{SATZ}[Ayoub]\label{SATZAYOUB2}
The symmetric fibered multiderivator
\[ \mathbb{SH}^T_{\mathcal{M}} \rightarrow \SSS^{\op} \]
of Theorem~\ref{SATZAYOUB1} satisfies the axioms (F1--F6), and (F4m--F5m) w.r.t.\@ the compactification structure of \ref{PARCOMPSCH}.
\end{SATZ}
\begin{proof}
Recall that the 
association 
\[ X \mapsto \mathbb{SH}_{\mathcal{M}}^T(\cdot)_X \]
is a stable homotopic pseudo-functor (in the sense of \cite[D\'efinition~1.4.1]{Ayo07I}) with values in triangulated categories, according to the results of \cite[Section~4.5.3]{Ayo07II}.

(F1) The existence of the left adjoint $\iota_!$ of $\iota^*$ for a morphism $\iota: (I, U) \rightarrow (I, S)$ which is a point-wise embedding follows from axiom (DerAlg 2g). 
$\iota^*$ (as functor on underlying categories) is furthermore triangulated. Hence, in an infinite stable derivator, $\iota^*$ (as morphism of derivators) commutes will all homotopy limits.
It follows also that $\iota_!$, as morphism between fibers, i.e.\@ for $\iota$ between constant diagrams, is a morphism of derivators, i.e.\@ is computed point-wise. This is not true for non-constant diagrams in general.

(F3) follows from $f_*=f_!$ for projective morphisms \cite[Scholie 1.4.2, 4]{Ayo07I} because $f_!$ (as functor on underlying categories) is triangulated and has a right adjoint. Hence, in an infinite stable derivator, $f_* = f_!$ (as morphism of derivators) commutes with all homotopy colimits.

All the other axioms only involve the functors between the underlying categories, and this will not be  mentioned explicitly anymore. 

(F2) is part of axiom (2) of a homotopy stable functor because $\iota_!$ is the left adjoint of $\iota^*$ in this case. Hence the unit $1 \rightarrow \iota^* \iota_!$ is the adjoint of the counit $\iota^* \iota_* \rightarrow 1$. Thus $\iota_*$ and $\iota_!$ being fully-faithful are equivalent statements.

(F4) is \cite[Scholie 1.4.2, 5]{Ayo07I}.

(F5) is part of axiom (3) of a homotopy stable functor.

(F6) follows from $f_!=f_*$ for projective morphisms \cite[Scholie 1.4.2, 4]{Ayo07I}. 

(F4m) follow from the projection formula \cite[Th\'eor\`eme~2.3.40]{Ayo07I} together with $f_!=f_*$ for projective morphisms \cite[Scholie 1.4.2, 4]{Ayo07I}. 

(F5m), using \cite[Remark 6.3]{Hor17}, also follows from the projection formula \cite[Th\'eor\`eme~2.3.40]{Ayo07I} together with the fact that $\iota_!$ (the functor appearing in \cite[Th\'eor\`eme~2.3.40]{Ayo07I}) is the left adjoint of $\iota^*$ in this case. 
\end{proof}

\begin{KOR}\label{KORAYOUB}For any choice of objects as in \ref{PARSETTINGAYOUB}
there is a symmetric derivator six-functor-formalism (i.e.\@ symmetric fibered multiderivator) with domain $\Cat$
\[ \mathbb{SH}_{\mathcal{M}}^T \rightarrow \SSS^{\cor} \]
such that 
\begin{enumerate}
\item for a diagram $F: I \rightarrow \mathcal{S}^{\op}$ of quasi-projective schemes over $S$ (embedded via the inclusion $\SSS^{\op} \rightarrow \SSS^{\cor}$), we have 
\[ \mathbb{SH}_{\mathcal{M}}^T(I)_F \cong \mathbb{SH}_{\mathcal{M}}^T(F^{\op}, I^{\op}) \]
(equivalence of monoidal categories)
where the right hand side is the ``algebraic derivator'' defined by Ayoub \cite[D\'efinition 4.2.21]{Ayo07II};
\item the push-forward along a multicorrespondence 
\[ \xymatrix{ & & & A  \ar[llld]_{g_1}\ar[ld]^{g_n}\ar[rd]^f \\
S_1 & \cdots & S_n & ; & T 
 }\]
 in $\SSS^{\cor}(\cdot)$ is given up to unique isomorphism by
\[ f_! ( L g^*_1 - \overset{L}{\otimes} \cdots \overset{L}{\otimes} Lg_n^* -) \]
with the functors $f_!$ of \cite[Proposition~1.6.46]{Ayo07I} and the $g_i^*$ as in \cite[Th\'eor\`eme~4.5.23]{Ayo07II}, cf.\@ also \cite[Scholie~1.4.2]{Ayo07I};
\item $\mathbb{SH}_{\mathcal{M}}^T$ is infinite (i.e.\@ satisfies (Der1${}^\infty$)) and has stable, well-generated fibers. 
\item $\mathbb{SH}_{\mathcal{M}}^T$ extends to a proper symmetric derivator six-functor-formalism, i.e.\@ to an oplax left symmetric fibered multiderivator, resp.\@ to a lax right symmetric fibered multi-derivator 
\[  \mathbb{SH}_{\mathcal{M}}^T \rightarrow \SSS^{\cor,0,\oplax} \qquad \mathbb{SH}_{\mathcal{M}}^T \rightarrow \SSS^{\cor,0,\lax}.  \]
\end{enumerate}
\end{KOR}
\begin{proof}
Apply Corollary~\ref{KORDER6FU}, and Theorem~\ref{HAUPTSATZOPLAX}, respectively, to the fibered multiderivator 
\[ \mathbb{SH}^T_{\mathcal{M}} \rightarrow \SSS^{\op} \]
of Theorem~\ref{SATZAYOUB1} and the compactification structure of \ref{PARCOMPSCH}. 
The fact that the functors $f_!$ constructed in the main construction of this article coincide with Ayoub's up to unique isomorphism can be checked for open embeddings $\iota$ and projective morphisms $\overline{f}$. And $\iota_!$ is canonically isomorphic to a left adjoint of $\iota^*$, and $f_!$ is canonically isomorphic to $f_*$, respectively, in both constructions.
\end{proof}

\appendix

\section{Construction of 2-multicategories}\label{APPENDIX2MULTICAT}

In this article (and the subsequent article \cite{Hor22}) the following construction of (strict) symmetric 2-multicategories will be used at several places. It is similar to the usual strictification of bicategories, however, without 
the pain of explicitly defining a bicategory (or here even bimulticategory). The construction assumes that the (weakly) associative composition is encoded in a strict contravariant functor 
\[ \Delta^{\op} \rightarrow \mathcal{CATW} \quad (\text{resp.} \   \Delta^{\op}_T \rightarrow \mathcal{CATW}, \quad \text{resp.} \  \Delta^{\op}_S \rightarrow \mathcal{CATW})  \]
from the simplex category (resp.\@ category of trees, resp.\@ category of symmetric trees) to categories with weak equivalences satisfying obvious axioms. 
We state and prove a non-multi variant first:

\begin{PROP}\label{PROPCONSTR2CAT}
Let 
\[ C: \Delta^{\op} \rightarrow \mathcal{CATW} \]
be a strict functor with values in categories with weak equivalences such that
\begin{enumerate}
\item The induced functor
\begin{equation}\label{eqproj0} \prod e_{i}^* : C(\Delta_n) \rightarrow \prod_{0 \dots n} C(\Delta_0) \end{equation}
is surjective on objects\footnote{One could weaken this, of course, to being essentially surjective, which would be the correct notion. However, this gadget is used only for constructional
purposes and the definition as stated saves a bit of pain.}.
\item For $X_0, \dots, X_n \in C(\Delta_0)$, the induced functor
\begin{equation}\label{eqproj1} \prod e_{i-1,i}^*: C(\Delta_n)_{(X_i)}[\mathcal{W}^{-1}_{(X_i)}]  \rightarrow \prod_{1 \dots n} C(\Delta_1)_{(X_{i-1}, X_{i})}[\mathcal{W}^{-1}_{(X_{i-1}, X_{i})}]   \end{equation}
is an equivalence
where $C(\Delta_n)_{(X_i)}$ denotes the fiber over $(X_0, \dots, X_n)$ under the functor (\ref{eqproj0})
and $\mathcal{W}_{(X_i)}$ is the restriction of the class of weak equivalences to it. 
\end{enumerate}
Then there exists a 2-category $\mathcal{C}$ defined as follows
\begin{enumerate}
\item Objects are the objects of $C(\Delta_0)$.
\item 1-morphisms in $\Hom(X, Y)$ for objects in $X, Y \in C(\Delta_0)$ are chains $\xi_1, \dots, \xi_n$ (possibly empty) of objects of $C(\Delta_1)$ with matching start and endpoints connecting $X$ and $Y$.
\item 2-morphisms from  $\xi_1, \dots, \xi_n$  to $\nu_1, \dots, \nu_m$ in $\Hom(X, Y)$,  are the morphisms from $\xi_1 \circ \cdots \circ \xi_n$ to $\xi_1 \circ \dots \circ \xi_n$ in the homotopy category
\[ C(\Delta_1)_{(X, Y)}[\mathcal{W}^{-1}_{(X, Y)}].   \]
Here, for $n \ge 1$, $\xi_1 \circ \cdots \circ \xi_n := e_{0n}^* \xi_{1, \dots, n}$ where $\xi_{1, \dots, n}$ is in the essential preimage of $\xi_1, \dots, \xi_n$ under the functor (\ref{eqproj1}) and for $n=0$ it is set to $\delta^* X$ where $\delta: \Delta_1 \rightarrow \Delta_0$ is the degeneracy morphism. 
\end{enumerate}
The composition of 1-morphisms is by concatenation. The vertical composition of 2-morphisms is constructed in the proof. 
\end{PROP}

\begin{BEM}
The objects $\xi_{1, \dots, n}$ together with isomorphisms $\alpha_{i-1,i}: e_{i-1,i}^* \xi_{1, \dots, n} \cong \xi_i$ for all $i$ have to be chosen once and for all for any chain $\xi_1, \dots, \xi_n$ a priori. Obviously the resulting sets of 2-morphisms do not depend on this choice up to unique isomorphism and thus the resulting 2-category is well determined up to unique isomorphism.

Alternatively, defining the categories of morphisms simply as $C(\Delta_1)_{(X, Y)}[\mathcal{W}^{-1}_{(X, Y)}]$ one would obtain a bicategory. However, we decided to work with (strict) 2-categories everywhere in this project, and will not even give the definition of bimulticategory. 
\end{BEM}

\begin{proof}
Using the equivalence of categories (\ref{eqproj1}), the isomorphisms  $\alpha_{i,i-1}$ {\em uniquely} determine an isomorphism
\[ \alpha_e: e^*(\xi_{1,\dots,n}) \cong \xi_{e(1),\dots,e(m)}   \]
for each inclusion $e: \Delta_{m} \hookrightarrow \Delta_{n}$ of the form $i \mapsto i+k$ in such a way that 
\[ \xymatrix{  e_{e(i)-1,e(i)}^*e^*(\xi_{1,\dots,n}) \ar[r]^{\alpha_e}\ar[rd]_{\alpha_{i-1,i}}   &  e_{e(i)-1,e(i)}^*\xi_{e(1),\dots,e(n)}  \ar[d]^{\alpha_{e(i-1),e(i)}}  \\
& \xi_i } \]
commutes for all $i=0, \dots, m$. 

The only unclear point  is the vertical composition of 2-morphisms and its associativity:
 
Consider two composable and parallel pairs of 1-morphisms $(\xi_1,\dots, \xi_n) \in \Hom(X,Y)$ and $(\xi_{n+1}, \dots, \xi_{m}) \in \Hom(Y,Z)$,  as well as, $(\xi_1',\dots, \xi_{n'}') \in \Hom(X,Y)$ and $(\xi_{n'+1}', \dots, \xi'_{m'}) \in \Hom(Y,Z)$ and 2-morphisms induced by
 \[  e_{0,n}(\xi_{1,\dots,n}) \rightarrow e_{0,n}(\xi_{1,\dots,n'}') \text{ and } e_{n,m}(\xi_{n+1,\dots,m}) \rightarrow e_{n',m'}(\xi'_{n'+1,\dots,m'}) \]
 
 For an empty chain we define the `composition' in $C(\Delta_1)_{(X, X)}$ to be the pull-back $\delta^*X$ under the functor $\delta: \Delta_0 \rightarrow \Delta_1$. 
 We leave it to the reader
 to verify the statements involving empty chains and assume $n\ge1$, $m \ge 1$, $n' \ge 1$, and $m' \ge 1$ from now on. 
 
 As mentioned there are uniquely determined isomorphisms 
 \begin{eqnarray*}
  \alpha_{0,\dots,n}: e_{0,\dots,n}(\xi_{1, \dots, m}) &\rightarrow& \xi_{1,\dots,n}   \\
  \alpha_{n,\dots,m}: e_{n,\dots,m}(\xi_{1, \dots, m}) &\rightarrow& \xi_{n+1,\dots,m}   \\
  \alpha_{0,\dots,n'}: e_{0,\dots,n}(\xi'_{1, \dots, m'}) &\rightarrow& \xi_{1,\dots,n'}'   \\
  \alpha_{n',\dots,m'}: e_{n',\dots,m'}(\xi_{1, \dots, m'}') &\rightarrow& \xi_{n'+1,\dots,m'}'
 \end{eqnarray*}
 compatible with the chosen isomorphisms $\alpha_{i-1,i}$ (for the respective composition). 
 We get the compositions
 \[ \xymatrix{ e_{0,n}^*e_{0,\dots,n}^*(\xi_{1, \dots, m}) \ar[r] &  e_{0,n}^*\xi_{1,\dots,n} \ar[r]  &  e_{0,n'}^*\xi_{1,\dots,n'} \ar[r] &  e_{0,n'}^*e_{1,\dots,n'}^*(\xi'_{1, \dots, m'})    } \]
 \[ \xymatrix{ e_{n,m}^*e_{n,\dots,m}^*(\xi_{1, \dots, m}) \ar[r] &  e_{n,m}^*\xi_{n+1,\dots,m} \ar[r]  &  e_{n',m'}^*\xi_{n'+1,\dots,m'} \ar[r] &  e_{n',m'}^*e_{n'+1,\dots,m'}^*(\xi'_{1, \dots, m'})    } \]
 This defines (using the equivalence of categories (\ref{eqproj1})) a unique morphism 
 \[ \xymatrix{ e_{0,n,m}^*(\xi_{1, \dots, m}) \ar[r] &  e_{0,n',m'}^*(\xi'_{1, \dots, m'})    } \]
 and the composition is defined to be its image in 
 \[ \xymatrix{ e_{0,m'}^*(\xi_{1, \dots, m}) \ar[r] &  e_{0,m'}^*(\xi'_{1, \dots, m'})    } \]
 
 Actually, the same procedure defines a well-defined $n$-ary composition for arbitrary $n$. A simple calculation shows that this
 $n$-ary composition is compatible with iterated compositions thus showing that the (vertical) composition of 2-morphisms is associative. 
\end{proof}

We will need also a multicategorical and symmetric multicategorical version of this construction with essentially the same proof:
\begin{PROP}\label{PROPCONSTRSYMMULTI}
Denote $\Delta_T$ the category of trees and $\Delta_S$ the category of symmetric trees. 
Let 
\[ C: \Delta_T^{\op} \rightarrow \mathcal{CATW}  \text{ resp.\@ } C: \Delta_S^{\op} \rightarrow \mathcal{CATW} \]
be a strict functor with values in categories with weak equivalences
such that
\begin{enumerate}
\item For each $\tau \in \Delta_T$ the induced functor
\begin{equation}\label{eqproj0m} \prod_{o \in \tau} e_{o}^*:  C(\tau) \rightarrow \prod_{o \in \tau} C(\Delta_0) \qquad \text{resp.\@} \qquad  \prod_{o \in \tau} e_{o}^*:  C(\tau^S) \rightarrow \prod_{o \in \tau} C(\Delta_0)   \end{equation}
is surjective on objects. 
\item For each $\tau \in \Delta_T$ and collection
$(X_o)_{o \in \tau}$ with $X_o \in C(\Delta_0)$, the induced functor
\begin{equation}\label{eqproj1m} \prod_m e_{m}^*: C(\tau)_{(X_o)}[\mathcal{W}^{-1}_{(X_o)}]  \rightarrow \prod_{m} C(\Delta_{1,k_m})_{(X_o)}[\mathcal{W}^{-1}_{(X_o)}]
\end{equation}
resp.\@
\begin{equation*} \prod_m e_{m}^*: C(\tau^S)_{(X_o)}[\mathcal{W}^{-1}_{(X_o)}]  \rightarrow \prod_{m} C(\Delta_{1,k_m}^S)_{(X_o)}[\mathcal{W}^{-1}_{(X_o)}]
\end{equation*}
is an equivalence where $C(\tau)_{(X_o)}$ denotes the fiber over $(X_o)_{o \in \tau}$ under the functor (\ref{eqproj0m}) 
and $\mathcal{W}_{\{X_i\}}$ is the restriction of the class of weak equivalences to it.
Here $m$ runs over {\em the generating} multimorphisms of $\tau$. By definition of tree each such $m$ defines a unique embedding $\Delta_{1,k_m} \hookrightarrow \tau$ for a uniquely determined $k_m \in \N_0$. 
\end{enumerate}
Then there exists a 2-multicategory (resp.\@ symmetric 2-multicategory) $\mathcal{C}$ defined as follows
\begin{enumerate}
\item Objects are the objects of $C(\Delta_0)$ 
\item 1-morphisms in $\Hom(X_1, \dots, X_k; Y)$ for objects in $X_1, \dots, X_n, Y \in C(\Delta_0)$ are given by a tree $\tau$, possibly empty, an inclusion $e: \Delta_{1,k} \rightarrow \tau$  (resp.\@ $e: \Delta_{1,k}^S \rightarrow \tau^S$) with image of maximal length, collections $\xi_m$ of objects of $C(\Delta_{1,k_m})$ (resp.\@ of $C(\Delta_{1,k_m}^S)$) for each morphism $m$ of length $1$ in $\tau$ with matching start and endpoints connecting $X_1, \dots, X_k$ and $Y$ (matched via $e$). Note that in the non-symmetric case the morphism $e$ is unique. 
\item 2-morphisms from  $(\tau, e, (\xi_m))$  to $(\tau', e', (\xi_m'))$ in $\Hom(X_1, \dots, X_k; Y)$, are the morphisms from $e^*(\circ (\xi_m))$ to $(e')^*(\circ (\xi_m'))$ in the homotopy category
\[ C(\Delta_{1,k})_{(X_1, \dots, X_k; Y)}[\mathcal{W}^{-1}_{(X_1, \dots, X_k; Y)}]  \quad (\text{resp. } C(\Delta_{1,k}^S)_{(X_1, \dots, X_k; Y)}[\mathcal{W}^{-1}_{(X_1, \dots, X_k; Y)}])  \]
Here, for $n \ge 1$, $\circ (\xi_m) := e^* \xi$ where $\xi$ is in the essential preimage of $(\xi_m)$ under the functor (\ref{eqproj1m}) and for $n=0$ it is set to $\delta^* X$ where $\delta: \Delta_1 \rightarrow \Delta_0$ is the degeneracy morphism. 
\end{enumerate}
In the symmetric case, the respective symmetric groups act on the inclusions $e$ yielding a strict action
\[ \sigma: \Hom(X_1, \dots, X_n; Y) \rightarrow \Hom(X_{\sigma(1)}, \dots, X_{\sigma(n)}; Y) \]
compatible with composition. 
The composition of 1-morphisms is by concatenation and the composition of 2-morphisms is done completely analogously to  \ref{PROPCONSTR2CAT}. 
\end{PROP}

\section{(co)Cartesian projectors}\label{COCARTPROJ}

In this appendix, the (co)Cartesian projectors needed in the construction of {\em proper} derivator six-functor-formalisms in Section~\ref{SECTCONSTPROPER} are constructed.
It would be possible, in principle, to construct them using the theory of well-generated triangulated categories (cf.\@ \cite[\S 4.3]{Hor15}). However, in this case, an explicit construction is available with the aid of which many properties become more clearly visible. 
Recall the notation from \ref{COMPONENTS} and \ref{COMPONENTSOPLAX}. Let $\widetilde{S}$ be any object in $\SSS(\tww{I})$ that is the specialization of an interior compactification of a $S \hookrightarrow \overline{S}$ in $\Cor^{\comp,\mathrm{(op)lax}}_I(\tau)$ (cf.\@ Definition~\ref{DEFCORCOMP})
to any object of $\tww{\tau}$. (Note that this does not need to be an interior compactification itself).

\begin{PAR}\label{LEFTCARTPROJ}
We will show that the fully-faithful inclusion
\[ \DD( \twwc I)^{2-{\cart}}_{\pi_{234}^* (\widetilde{S}^{\op})}\ \hookrightarrow \DD(\twwc I)^{}_{\pi_{234}^* 
(\widetilde{S}^{\op})} \]
has a left adjoint $\Box_!$, which we will call a {\bf left Cartesian projector} (cf.\@ also \cite[Section~2.4]{Hor15}).
It induces a left adjoint also of the restriction: 
\[ \DD( \twwc I)^{4-\cocart, \ws, 2-{\cart}}_{\pi_{234}^* (\widetilde{S}^{\op})}\ \hookrightarrow \DD(\twwc I)^{ 4-{\cart}, \ws}_{\pi_{234}^* 
(\widetilde{S}^{\op})} \]
that is to say: $\Box_!$ preserves the conditions of being simultaneously $4$-coCartesian and well-supported. 
Such a left Cartesian projector (or rather its composition with the fully-faithful inclusion) can be specified by an endofunctor of $\DD(\twwc I)^{}_{\pi_{234}^* 
\widetilde{S}^{\op}}$ together with a natural transformation
\[ \nu: \id  \Rightarrow \Box_! \]
such that
\begin{enumerate}
\item $\Box_! \mathcal{E}$ is $2$-Cartesian for all objects $\mathcal{E}$,
\item  $\nu_{\mathcal{E}}$ is an isomorphism on $2$-Cartesian objects $\mathcal{E}$,
\item $ \nu_{\Box_!\mathcal{E}} = \Box_! \nu_{\mathcal{E}}$ holds true. 
\end{enumerate}
\end{PAR}

\begin{PAR}\label{RIGHTCOCARTPROJ}
Furthermore, we will show that the fully-faithful inclusion
\[ \DD( \twwc I)^{4-\cocart, \ws, 2-{\cart}}_{\pi_{234}^* (\widetilde{S}^{\op})}\ \hookrightarrow \DD(\twwc I)^{\ws, 2-{\cart}}_{\pi_{234}^* (\widetilde{S}^{\op})} \]
has a right adjoint $\Box_*$ which we will call a {\bf right coCartesian projector} (cf.\@ also \cite[Section~2.4]{Hor15}).

A right coCartesian projector (or rather its composition with the fully-faithful inclusion) can be specified by an endofunctor $\Box_*$ of $\DD(\twwc I)^{\ws, 2-{\cart}}_{\pi_{234}^* \widetilde{S}^{\op}}$ together with a natural transformation
\[ \nu: \Box_* \Rightarrow \id  \]
such that 
\begin{enumerate}
\item $\Box_* \mathcal{E}$ is $4$-coCartesian for all objects $\mathcal{E}$ (in the source category),
\item  $\nu_{\mathcal{E}}$ is an isomorphism on $4$-coCartesian objects  $\mathcal{E}$,
\item $ \nu_{\Box_*\mathcal{E}} = \Box_* \nu_{\mathcal{E}}$ holds true. 
\end{enumerate}

This, in particular, gives a push-forward functor
\[  \Box_* (\pi_{234}^*\widetilde{f})_*: \DD(\twwc I)^{4-\cocart, \ws, 2-{\cart}}_{\pi_{234}^* (\widetilde{A}')^{\op}}\ \rightarrow \DD(\twwc I)^{4-\cocart, \ws, 2-{\cart}}_{\pi_{234}^* \widetilde{T}^{\op}}\]
for a point-wise proper morphism
\[ \widetilde{f}: \widetilde{A}' \rightarrow \widetilde{T} \]
arising from an oplax morphism as in \ref{COMPONENTSOPLAX}.

Note that $(\pi_{234}^*\widetilde{f})_*$ preserves automatically the condition of being $2$-Cartesian and well-supported by Lemma~\ref{LEMMAEXISTENCE3FUNCTORSOPLAX}, 3. Proposition~\ref{PROPCOCARTPROJ} below shows that this is still computed point-wise (in the sense of $\EE' \rightarrow \SSS^{\cor, 0, \oplax}$), i.e. that we have for any $\alpha: I \rightarrow J$
\[ \alpha^* \Box_* (\pi_{234}^*\widetilde{f})_* \cong  \Box_* (\pi_{234}^*(\tww\alpha)^*\widetilde{f})_* \alpha^*.  \]
\end{PAR}

\begin{PAR}\label{COCARTPROJPREP}
We need some technical preparation. Consider the projections:
\[ \pi_{1234}, \pi_{1237},\pi_{1267}, \pi_{1567}, \pi_{4567}: {}^{\downarrow\uparrow\uparrow\downarrow\uparrow\uparrow\downarrow}I \rightarrow \twwc I.  \]
We have obvious natural transformations
\[ \pi_{1234} \Rightarrow \pi_{1237} \Leftarrow \pi_{1267} \Leftarrow \pi_{1567} \Rightarrow \pi_{4567} \]
and therefore
\[ \pi_{1234}^* \Rightarrow \pi_{1237}^* \Leftarrow \pi_{1267}^* \Leftarrow \pi_{1567}^* \Rightarrow \pi_{4567}^* \]
If we plug in $\pi_{234}^*\widetilde{S}$ for $\widetilde{S}$ an object in $\SSS(\tww{I})$ that is the specialization of an interior compactification of a morphism $\tau \rightarrow \Fun^{\mathrm{(op)lax}}(I, \mathcal{S}^{\cor})$
to an object of $\tww{\tau}$ (cf.\@ \ref{COMPONENTS} and \ref{COMPONENTSOPLAX}), we get morphisms of diagrams in $\mathcal{S}$:
\[ \xymatrix{  \pi_{234}^*\widetilde{S}^{} \ar@{<-}[r]^-{g} & \pi_{237}^*\widetilde{S}^{} \ar@{^{(}->}[r]^-{\iota} & \pi_{267}^*\widetilde{S}^{} \ar@{->>}[r]^-{\overline{f}} & \pi_{567}^*\widetilde{S}^{} \ar@{=}[r] &  \pi_{567}^*\widetilde{S}^{}  } \]
and therefore natural transformations
\begin{eqnarray*}
 g^* \pi_{1234}^* &\Rightarrow&  \pi_{1237}^*  \\
 \pi_{1237}^* &\Leftarrow& \iota^* \pi_{1267}^*  \\
 \overline{f}_* \pi_{1567}^* &\Leftarrow&  \pi_{1267}^*  
\end{eqnarray*}
of functors between fibers.
\end{PAR}

\begin{LEMMA}
The natural transformation
\[ \pi_{4567,!}^{(\pi_{234}^*\widetilde{S}^{\op})} \pi_{1567}^* \Rightarrow \id \]
induced by the natural transformation
\[ \pi_{1567}^* \Rightarrow \pi_{4567}^* \]
of functors 
\[  \pi_{1567}^* ,  \pi_{4567}^*:  \DD({}^{\downarrow\uparrow\uparrow\downarrow\uparrow\uparrow\downarrow}I)_{\pi_{567}^*S^{\op}} \rightarrow \DD(\twwc{I})_{\pi_{234}^*S^{\op}} \]
is an isomorphism.
\end{LEMMA}
\begin{proof}
$\pi_{4567}$ is an opfibration (cf.\@ \ref{PARTW2}). Denote
by $e_i:  {}^{\downarrow\uparrow\uparrow} (I \times_{/I} i_1) \hookrightarrow {}^{\downarrow\uparrow\uparrow\downarrow\uparrow\uparrow\downarrow} I$ the inclusion of
the fiber over an object 
\[ i = \{i_1 \rightarrow i_2 \rightarrow i_3 \rightarrow i_4 \}. \]
We have
\[ i^* \pi_{4567,!} \pi_{1567}^* = \hocolim_{{}^{\downarrow\uparrow\uparrow} (I \times_{/I} i_1)} e_i^* \pi_{1567}^*  \]
We can factor 
$\pi_{1567} \circ  e_i$ in the following way
\[ \xymatrix{ {}^{\downarrow\uparrow\uparrow} (I \times_{/I} i) \ar[r]^-{\pi_1} & I \times_{/I} i_1 \ar[r]^\rho \ar[r] & \twwc{I} }\]
where $\rho$ maps $i' \rightarrow i_1$ to $i' \rightarrow i_2 \rightarrow i_3 \rightarrow i_4$. 
The functor $\pi_1$ is a fibration with fibers of the form $\beta \times_{/{}^{\uparrow\uparrow} (I \times_{/I} i)} {}^{\uparrow \uparrow} (I \times_{/I} i)$. 
Since these fibers have an initial object, the unit $\id \rightarrow \pi_{1,*}\pi_1^*$ is actually an isomorphism. Therefore also the counit $\pi_{1,!}\pi_1^* \rightarrow \id$ (which is its adjoint) is an isomorphism. Hence we have 
\[ i^* \pi_{4567,!} \pi_{1567}^* \cong \hocolim_{I \times_{/I} i_1} \rho^* \]
Since $I \times_{/I} i_1$ has a final object, the homotopy colimit is actually evaluation at the latter, therefore we get 
\[ i^* \pi_{4567,!} \pi_{1567}^* \cong i^*. \]
\end{proof}

If $\mathcal{E}$ is an object in $\DD(\twwc I)^{\ws, 2-{\cart}}_{\pi_{234}^* \widetilde{S}^{\op}}$ we have that the morphisms
\begin{eqnarray}
 \pi_{1237}^* \mathcal{E} &\leftarrow& \iota^* \pi_{1267}^* \mathcal{E} \label{eqiota} \\
 \overline{f}_* \pi_{1567}^* \mathcal{E} &\leftarrow&  \pi_{1267}^* \mathcal{E}  
\end{eqnarray}
are isomorphisms.

\begin{LEMMA}\label{LEMMAPOINTWISEOPLAX}
Assume $\DD$ is infinite (i.e.\@ satisfies (Der1${}^\infty$)) or that $I$ has finite Hom sets.

If $\mathcal{E}$ is well-supported, then
the inverse of (\ref{eqiota}) induces an isomorphism
\[ \iota_! \pi_{1237}^* \mathcal{E} \cong \pi_{1267}^* \mathcal{E} . \]
\end{LEMMA}
\begin{proof}
The assertion follows if we can show that $\iota_!$ is computed point-wise on $\pi_{1237}^* \mathcal{E}$.
Consider the projection 
$\pi_{12}:  {}^{\downarrow\uparrow\uparrow\downarrow\uparrow\uparrow\downarrow} I\rightarrow \tw I$.  It is an opfibration.
Every coCartesian morphism in ${}^{\downarrow\uparrow\uparrow\downarrow\uparrow\uparrow\downarrow} I$ w.r.t.\@ this opfibration is mapped by $\pi_{237}^*\widetilde{S}^{\op}$
and $\pi_{267}^*\widetilde{S}^{\op}$ to a proper morphism. Therefore by Lemma~\ref{LEMMACART1} and Lemma~\ref{LEMMAPOINTWISEEXBYZERO}, 1. $\iota_!$ commutes with the inclusion of the fiber. 
On the fiber we will check the condition of Lemma~\ref{LEMMAPOINTWISEEXBYZERO}, 3. Let $\widetilde{\iota}: \widetilde{S}' \rightarrow \widetilde{S}$ be the morphism (point-wise dense embedding) as in the definition of ``well-supported'' (\ref{COMPONENTS}). 
Let $\alpha: i'' \rightarrow i'$ and $\mu: i' \rightarrow i$ be morphisms in the fiber. 
Applying $\pi_{237}$ and $\pi_{267}$ we get the following situation: 

\[ \xymatrix{
  & \widetilde{S}(i_2 \rightarrow i_3 \rightarrow i_7)  \ar@{^{(}->}[d]^{\iota_1} \ar@{^{(}->}[r]^{\iota_2} \ar@/_40pt/[dd]_(.7){\pi_{237}(\mu^{\op})} &  \widetilde{S}(i_2 \rightarrow i_6 \rightarrow i_7) \ar@{^{(}->}[d]^{\iota_3} \ar@/^40pt/[dd]^{\pi_{267}(\mu^{\op})} \\
\widetilde{S}'(i_2 = i_2 \rightarrow i_7) \ar[d]  \ar@{^{(}->}[r]^{}  \ar@{^{(}->}[ru]^{}& \widetilde{S}(i_2 \rightarrow i_3' \rightarrow i_7)  \ar[d]^{g_1} \ar@{^{(}->}[r]^{\iota_4} &  \widetilde{S}(i_2 \rightarrow i_6' \rightarrow i_7)  \ar[d]^{g_2} \\
\widetilde{S}'(i_2 = i_2 \rightarrow i_7') \ar[d] \ar@{^{(}->}[r]  & \widetilde{S}(i_2 \rightarrow i_3' \rightarrow i_7') \ar@/_40pt/[dd]_(.3){\pi_{237}(\alpha^{\op})}  \ar[d]^{g_3} \ar@{^{(}->}[r]^{\iota_5} &  \widetilde{S}(i_2 \rightarrow i_6' \rightarrow i_7')  \ar[d]^{g_4} \ar@/^40pt/[dd]^{\pi_{267}(\alpha^{\op})} \\
\widetilde{S}'(i_2 = i_2 \rightarrow i_7'') \ar@{^{(}->}[r] \ar@{^{(}->}[rd] & \widetilde{S}(i_2 \rightarrow i_3' \rightarrow i_7'')  \ar@{^{(}->}[d]^{\iota_6} \ar@{^{(}->}[r]^{\iota_7} &  \widetilde{S}(i_2 \rightarrow i_6' \rightarrow i_7'')  \ar@{^{(}->}[d]^{\iota_8} \\
  & \widetilde{S}(i_2 \rightarrow i_3'' \rightarrow i_7'') \ar@{^{(}->}[r]^{\iota_9} &  \widetilde{S}(i_2 \rightarrow i_6'' \rightarrow i_7'')  \\
} \]

We have to show
\[   \iota_{2,!} \ \iota_1^*  g_1^* \  g_3^*  \iota_6^* \  \mathcal{E}(i_1 \rightarrow i_2 \rightarrow i_3'' \rightarrow i_7'') \cong    \iota_3^*  g_2^* \ \iota_{5,!}  \ g_3^*  \iota_6^*\  \mathcal{E}(i_1 \rightarrow i_2 \rightarrow i_3'' \rightarrow i_7'').  \]

Because $\mathcal{E}$ is 3-coCartesian this is the same as

\[   \iota_{2,!} \ \iota_1^*  g_1^* \  g_3^* \mathcal{E}(i_1 \rightarrow i_2 \rightarrow i_3' \rightarrow i_7'') \cong   \iota_3^*  g_2^*\  \iota_{5,!}\   g_3^*  \mathcal{E}(i_1 \rightarrow i_2 \rightarrow i_3' \rightarrow i_7'').  \]

By assumption $\mathcal{E}(i_1 \rightarrow i_2 \rightarrow i_3' \rightarrow i_7'') $ has support in  $\widetilde{S}(i_2 = i_2 \rightarrow i_7'')$, therefore
 $g_3^* \mathcal{E}(i_1 \rightarrow i_2 \rightarrow i_3' \rightarrow i_7'')$ has support in $\widetilde{S}(i_2 = i_2 \rightarrow i_7')$, and
 $g_1^* g_3^* \mathcal{E}(i_1 \rightarrow i_2 \rightarrow i_3' \rightarrow i_7'')$  has support in $\widetilde{S}(i_2 = i_2 \rightarrow i_7)$,  by (the adjoint of) axiom (F5). Note that the relevant squares are Cartesian (Lemma~\ref{LEMMACARTDIAM}).
Therefore on such an object, we have $\iota_{2,!}  \iota_1^* \cong   \iota_{3}^* \iota_{3,!} \iota_{2,!}  \iota_1^* \cong  \iota_{3}^* \iota_{4,!} \iota_{1,!} \iota_1^* \cong \iota_{3}^* \iota_{4,!}$.
 Inserting this, we get the morphism:
 \[ \iota_3^*   \iota_{4,!}   g_1^* g_3^* \mathcal{E}(i_1 \rightarrow i_2 \rightarrow i_3' \rightarrow i_7'') \cong   \iota_3^*  g_2^* \iota_{5,!}  g_3^*  \mathcal{E}(i_1 \rightarrow i_2 \rightarrow i_3' \rightarrow i_7'').  \]
That this is an isomorphism follows from Proposition~\ref{PROPPROPERTIESCORCOMPM}, 2.\@ because of the support condition. 
Hence $\iota_!$ is also computed point-wise on the fiber. 
\end{proof}

Warning: Even if $\mathcal{E}$ is $3$-coCartesian (but not well-supported) $\iota_!$ will not be computed point-wise on $\pi_{1237}^* \mathcal{E}$ in general.

\begin{PROP}\label{PROPCARTPROJ}Using the notation of \ref{COCARTPROJPREP}, 
denote $\Box_! :=   \pi_{4567;!} \overline{f}_* \pi_{1267}^*$. This functor, together with 
\[ \xymatrix{  \mathcal{E} & \ar[l]_-\sim  \pi_{4567;!} \pi_{1567}^* \mathcal{E} \ar[r] & \pi_{4567;!} \overline{f}_* \pi_{1267}^* \mathcal{E}    },  \]
defines a left Cartesian projector:
\[ \Box_!: \DD(\twwc I)^{}_{\pi_{234}^* \widetilde{S}^{\op}} \ \rightarrow \DD( \twwc I)^{2-{\cart}}_{\pi_{234}^* \widetilde{S}^{\op}} \]
$\Box_!$ preserves the conditions of being (simultaneously) $4$-coCartesian and well-supported. 
\end{PROP}

\begin{proof}We have to check the properties 1.--3.\@ of \ref{LEFTCARTPROJ}.

1. Consider a morphism $\mu$ of type 2 in $\twwc I$:
\[ \xymatrix{
i =( i_1 \ar[r] \ar@<10pt>@{=}[d] & i_2  \ar[r] \ar@{<-}[d] & i_3  \ar[r] \ar@{=}[d] & i_4)  \ar@{=}[d] \\
i' = (i_1 \ar[r] & i_2'  \ar[r] & i_3  \ar[r] & i_4)
} \]
We have to see that 
\[ \widetilde{S}(\pi_{234}\mu)_* i^* \pi_{4567;!} \overline{f}_* \pi_{1267}^* \rightarrow  (i')^* \pi_{4567;!} \overline{f}_* \pi_{1267}^*  \] 
is an isomorphism. Using that $\pi_{4567;!}$ is an obfibration, we get
\[ \widetilde{S}(\pi_{234}\mu)_* \hocolim_{{}^{\downarrow\uparrow\uparrow}( I \times_{/I} i_1)} e_i^* \overline{f}_* \pi_{1267}^* \rightarrow \hocolim_{{}^{\downarrow\uparrow\uparrow}( I \times_{/I} i_1)} e_{i'}^* \overline{f}_* \pi_{1267}^*  \] 
However, we already have an isomorphism
\[ (\widetilde{S}(\pi_{234}\mu))_* e_i^* \overline{f}_*  \rightarrow e_{i'}^* \overline{f}_*. \] 

2. follows because for a $2$-Cartesian object $\mathcal{E}$ the morphism
\[ \overline{f}_* \pi_{1567}^* \mathcal{E} \leftarrow  \pi_{1267}^* \mathcal{E}   \]
is already an isomorphism. 

3. Is proven as for \cite[Proposition 8.5]{Hor16}.

We now show that $\Box_!$ preserves the conditions of being (simultaneously) $4$-coCartesian and well-supported.
Let $\mathcal{E}$ be an objects with these properties. 

Consider a morphism $\mu$ of type 4 in $\twwc I$:
\[ \xymatrix{
i = (i_1 \ar[r] \ar@<10pt>@{=}[d] & i_2  \ar[r] \ar@{=}[d] & i_3  \ar[r] \ar@{=}[d] & i_4)  \ar[d] \\
i' = (i_1 \ar[r] & i_2  \ar[r] & i_3  \ar[r] & i_4')
} \]
and let $G:=\widetilde{S}(\pi_{234}(\mu))$. 
We have to see that 
\[ G^* \hocolim_{{}^{\downarrow\uparrow\uparrow}( I \times_{/I} i_1)} e_i^* \overline{f}_* \pi_{1267}^* \mathcal{E}  \rightarrow  \hocolim_{{}^{\downarrow\uparrow\uparrow}( I \times_{/I} i_1)} e_{i'}^* \overline{f}_* \pi_{1267}^* \mathcal{E} \]
is an isomorphism. Since $G^*$ commutes with homotopy colimits (being a left adjoint) it suffices to show that point-wise
\[ G^* e_i^* \overline{f}_* \pi_{1267}^* \mathcal{E}  \rightarrow  e_{i'}^* \overline{f}_* \pi_{1267}^* \mathcal{E} \]
is an isomorphism. 
Pick an object $j_1 \rightarrow j_2 \rightarrow j_3 \rightarrow i_1$ in $({}^{\uparrow \uparrow \downarrow} I \times_{/I} i_1)$
and consider the diagram
\[ \xymatrix{
\widetilde{S}(j_2 \rightarrow i_3 \rightarrow i_4') \ar[r]^{G} \ar@{->>}[d]_{F} & \widetilde{S}(j_2 \rightarrow i_3 \rightarrow i_4) \ar@{->>}[d]^{f} \\
\widetilde{S}(i_2 \rightarrow i_3 \rightarrow i_4') \ar[r]^{g} & \widetilde{S}(i_2 \rightarrow i_3 \rightarrow i_4)
} \]
We are left to show that
\[ g^* f_* \mathcal{E}_{j_1 \rightarrow j_2 \rightarrow i_3 \rightarrow i_4} \rightarrow  F_* \mathcal{E}_{j_1 \rightarrow j_2 \rightarrow i_3 \rightarrow i_4'}   \]
is an isomorphism. However, $\mathcal{E}$ is $4$-coCartesian, hence this is the same as 
\[ g^* f_* \mathcal{E}_{j_1 \rightarrow j_2 \rightarrow i_3 \rightarrow i_4} \rightarrow  F_* G^* \mathcal{E}_{j_1 \rightarrow j_2 \rightarrow i_3 \rightarrow i_4}   \]
which is an isomorphism by Proposition~\ref{PROPPROPERTIESCORCOMPM}, 1.\@ because $\mathcal{E}$ is also well-supported and thus has support in $\widetilde{S}'(j_2 = j_2 \rightarrow i_4)$ (with $\widetilde{S}'$ as in the definition of well-supported, cf.\@ \ref{COMPONENTS}).

Consider now a morphism $\mu$ of type 3
\[ \xymatrix{
i=(i_1 \ar[r] \ar@<10pt>@{=}[d] & i_2  \ar[r] \ar@{=}[d] & i_3  \ar[r] \ar@{<-}[d] & i_4)  \ar@{=}[d] \\
i'=(i_1 \ar[r] & i_2  \ar[r] & i_3'  \ar[r] & i_4)
} \]
and let $\iota:=\widetilde{S}(\pi_{234}(\mu))$. 
To show that $\Box_! \mathcal{E}$ is well supported, we have to see that 
\[ \iota_! \hocolim_{{}^{\downarrow\uparrow\uparrow}( I \times_{/I} i_1)} e_{i}^* \overline{f}_* \pi_{1267}^* \mathcal{E} \rightarrow \hocolim_{{}^{\downarrow\uparrow\uparrow}( I \times_{/I} i_1)} e_{i'}^* \overline{f}_* \pi_{1267}^* \mathcal{E} \]
is an isomorphism and that the right hand side has support in $\widetilde{S}'(i_2 = i_2 \rightarrow i_4)$ (with $\widetilde{S}'$ as in the definition of well-supported, cf.\@ \ref{COMPONENTS}). 
Since $\iota_!$ commutes with homotopy colimits (being a left adjoint) and is computed point-wise on constant diagrams it suffices to show that point-wise
\[ \iota_! e_i^* \overline{f}_* \pi_{1267}^* \mathcal{E}  \rightarrow  e_{i'}^* \overline{f}_* \pi_{1267}^* \mathcal{E} \]
is an isomorphism. 
Pick an object $j_1 \rightarrow j_2 \rightarrow j_3 \rightarrow i_1$ in $({}^{\uparrow \uparrow \downarrow} I \times_{/I} i_1)$
and consider the diagram
\[ \xymatrix{
\widetilde{S}(j_2 \rightarrow i_3 \rightarrow i_4) \ar[r]^{I} \ar@{->>}[d]_{F} & \widetilde{S}(j_2 \rightarrow i_3' \rightarrow i_4) \ar@{->>}[d]^{f} \\
\widetilde{S}(i_2 \rightarrow i_3 \rightarrow i_4) \ar[r]^{\iota} & \widetilde{S}(i_2 \rightarrow i_3' \rightarrow i_4)
} \]
We are left to show that
\[ \iota_! F_* \mathcal{E}_{j_1 \rightarrow j_2 \rightarrow i_3 \rightarrow i_4} \rightarrow  f_* \mathcal{E}_{j_1 \rightarrow j_2 \rightarrow i_3' \rightarrow i_4}   \]
is an isomorphism. However, $\mathcal{E}$ is stongly $3$-coCartesian, hence this is the same as 
\[ \iota_! F_* \mathcal{E}_{j_1 \rightarrow j_2 \rightarrow i_3 \rightarrow i_4} \rightarrow  f_* I_! \mathcal{E}_{j_1 \rightarrow j_2 \rightarrow i_3 \rightarrow i_4}   \]
which is an isomorphism by Proposition~\ref{PROPPROPERTIESCORCOMPM}, 3.\@
For the condition of being well-supported, it suffices to see that $e_{i'}^* \overline{f}_* \pi_{1267}^* \mathcal{E}$ point-wise has support in $\widetilde{S}'(i_2 = i_2 \rightarrow i_4)$. This is shown in the same way. 
\end{proof}

\begin{LEMMA}\label{COMMCARTPROJIOTA}
For a morphism $\widetilde{\iota}: \widetilde{A} \hookrightarrow \widetilde{A}'$ as in \ref{COMPONENTSOPLAX} the following diagram 2-commutes
\[ \xymatrix{
\DD(\twwc I)^{4-\cocart, \mathrm{ws}}_{\pi_{234}^* \widetilde{A}^{\op}} \ar[r]^-{\Box_!} \ar[d]_{(\pi_{234}^*\widetilde{\iota})_!} \ar@{}[rd]|{\Swarrow^\sim} & \DD(\twwc I)^{4-\cocart, \mathrm{ws}, 2-\cart}_{\pi_{234}^* \widetilde{A}^{\op}}  \ar[d]^{(\pi_{234}^*\widetilde{\iota})_!} \\
\DD(\twwc I)^{4-\cocart, \mathrm{ws}}_{\pi_{234}^* (\widetilde{A}')^{\op}}  \ar[r]_-{\Box_!} & \DD(\twwc I)^{4-\cocart, \mathrm{ws}, 2-\cart}_{\pi_{234}^* (\widetilde{A}')^{\op}} 
} \]
(The natural transformation is induced by the fact that $(\pi_{234}^*\widetilde{\iota})^*$ preserves the condition of being $2$-Cartesian, all relevant squares being Cartesian.)
\end{LEMMA}
\begin{proof}
We will show that it $(\pi_{234}^*\widetilde{\iota})_!$ ``commutes'' with the three functors in $\pi_{4567;!} \overline{f}_{\widetilde{A},*} \pi_{1267}^*$, resp.\@ $\pi_{4567;!} \overline{f}_{\widetilde{A}',*} \pi_{1267}^*$ in the obvious sense. 
We have 
\[ (\pi_{234}^*\widetilde{\iota})_! \pi_{4567;!} \cong  \pi_{4567;!} (\pi_{567}^*\widetilde{\iota})_!   \]
because the adjoint relation follows from (FDer0 left) for $\DD \rightarrow \SSS^{\op}$. 

As in the proof of Lemma~\ref{LEMMAEXISTENCE3FUNCTORSOPLAX}, one shows that
$(\pi_{267}^*\widetilde{\iota})_!$ and $(\pi_{567}^*\widetilde{\iota})_!$ are both computed point-wise on objects in the image of the subcategory in question under the functors $\pi_{1267}^*$, resp.\@ $\overline{f}_{\widetilde{A},*} \pi_{1267}^*$.
We have therefore
\[ (\pi_{567}^*\widetilde{\iota})_!  \overline{f}_{\widetilde{A},*}  \cong \overline{f}_{\widetilde{A}',*}  (\pi_{267}^*\widetilde{\iota})_!   \]
on the subcategory in question because of axiom (F6).

Finally 
\[ (\pi_{267}^*\widetilde{\iota})_!  \pi_{1267}^*  \cong \pi_{1267}^* (\pi_{234}^*\widetilde{\iota})_!  \]
on the subcategory in question, again because both $(-)_!$ functors are computed point-wise when restricted to the subcategories in question (cf.\@ Lemma~\ref{LEMMAEXISTENCE3FUNCTORSOPLAX}).
\end{proof}

\begin{LEMMA}\label{COMMCARTPROJF}
For a morphism $\widetilde{f}: \widetilde{A}' \hookrightarrow \widetilde{T}$ as in \ref{COMPONENTSOPLAX} the following diagram 2-commutes
\[ \xymatrix{
\DD(\twwc I)^{4-\cocart, \mathrm{ws}}_{\pi_{234}^* (\widetilde{A}')^{\op}} \ar[r]^-{\Box_!} \ar[d]_{(\pi_{234}^*\widetilde{f})_*} \ar@{}[rd]|{\Swarrow^\sim} & \DD(\twwc I)^{4-\cocart, \mathrm{ws}, 2-\cart}_{\pi_{234}^* (\widetilde{A}')^{\op}}  \ar[d]^{(\pi_{234}^*\widetilde{f})_*} \\
\DD(\twwc I)^{4-\cocart, \mathrm{ws}}_{\pi_{234}^* \widetilde{T}^{\op}}  \ar[r]_-{\Box_!} & \DD(\twwc I)^{4-\cocart, \mathrm{ws}, 2-\cart}_{\pi_{234}^* \widetilde{T}^{\op}} 
} \]
(The natural transformation is the exchange induced by the fact that $(\pi_{234}^*\widetilde{f})_*$ preserves the condition of being $2$-Cartesian.)
\end{LEMMA}
\begin{proof}
This is proven as the previous Lemma. Note that $(\pi_{234}^*\widetilde{f})_*$ is computed point-wise (FDer0 right) and commutes with homotopy colimits by Axiom (F3).
\end{proof}

\begin{LEMMA}\label{LEMMAFIB}
Let $\alpha: I \rightarrow J$ be a fibration in $\Catlf$, and consider the sequence of functors:
\[ \xymatrix{ {}^{\downarrow\uparrow\uparrow\downarrow\uparrow \uparrow \downarrow }{I} \ar[rrr]^-{q_1=(\twwc{\alpha},\pi_{234567})} &&& {}^{\downarrow\uparrow\uparrow\downarrow\uparrow \uparrow \downarrow } J \times_{({}^{\uparrow\uparrow\downarrow\uparrow \uparrow \downarrow}{J})} {}^{\uparrow\uparrow\downarrow\uparrow \uparrow \downarrow}{I} \ar[rr]^-{q_2=\id \times \pi_{3456}} && {}^{\downarrow\uparrow\uparrow\downarrow\uparrow \uparrow \downarrow } J \times_{\twwc J} \twwc{I}. } \]
\begin{enumerate}
\item The functor $q_1$ is a fibration. The fiber of $q_1$ over a pair $j_1 \rightarrow j_2 \rightarrow j_3 \rightarrow j_4\rightarrow j_5 \rightarrow j_6 \rightarrow j_7$ and $i_2 \rightarrow i_3\rightarrow i_4\rightarrow i_5\rightarrow i_6\rightarrow i_7$ is
\[  I_{j_1}  \times_{/I_{j_1}} i_1\]
where $i_1$ is the source of a Cartesian arrow over $j_1 \rightarrow j_2$ with destination $i_2$. 
\item The functor $q_2$ is an opfibration. The fiber of $q_2$ over a pair $j_1 \rightarrow j_2 \rightarrow j_3 \rightarrow j_4\rightarrow j_5 \rightarrow j_6 \rightarrow j_7$ and $i_4\rightarrow i_5\rightarrow i_6\rightarrow i_7$ (lying over $j_4\rightarrow j_5 \rightarrow j_6 \rightarrow j_7$) is 
\[ (I_{j_2} \times_{/I_{j_2}} I_{j_3} \times_{/I_{j_3}} i_3 )^{\op} \]
where $i_3$ is the source of a Cartesian arrow over $j_3 \rightarrow j_4$ with destination $i_4$ and the first comma category is 
constructed via the functor $I_{j_3} \rightarrow I_{j_2}$ being the Cartesian pull-back along $j_2 \rightarrow j_3$.  
\end{enumerate}
\end{LEMMA}
\begin{proof}Straightforward. \end{proof}

\begin{LEMMA}\label{LEMMAKAN3}
For the composition $\kappa = q_2 \circ q_1$ we have that the counit
\[ \kappa_! \kappa^* \rightarrow \id \] 
is an isomorphism. 
\end{LEMMA}
\begin{proof}As in the proof of Lemma~\ref{LEMMAKAN2}, the previous Lemma implies that 
the unit $\id \rightarrow q_{1,*} q_{1}^*$ and counit $ q_{2,!} q_{2}^* \rightarrow \id$  are isomorphisms.
Therefore also the counit $q_{1,!} q_{1}^* \rightarrow \id$ is an isomorphism (it is the adjoint of the first unit) and finally
also $ \kappa_! \kappa^* \rightarrow \id$. 
\end{proof}

\begin{LEMMA}\label{COMMCARTPROJFIB}
Let $\widetilde{S} \in \Fun(\tww J, \mathcal{S})$ be as above and let $\alpha: I  \rightarrow J$ be a fibration. The following diagram 2-commutes: 
\[ \xymatrix{
\DD(\twwc J)^{4-\cocart, \mathrm{ws}}_{\pi_{234}^* \widetilde{S}^{\op}} \ar[r]^-{\Box_!} \ar[d]_{(\twwc \alpha)^*} \ar@{}[rd]|{\Swarrow^\sim} & \DD(\twwc J)^{4-\cocart, \mathrm{ws}, 2-\cart}_{\pi_{234}^* (\widetilde{S})^{\op}}  \ar[d]^{(\twwc \alpha)^*} \\
\DD(\twwc I)^{4-\cocart, \mathrm{ws}}_{\pi_{234}^* (\tww \alpha)^* \widetilde{S}^{\op}}  \ar[r]_-{\Box_!} & \DD(\twwc I)^{4-\cocart, \mathrm{ws}, 2-\cart}_{\pi_{234}^* (\tww \alpha)^* \widetilde{S}^{\op}} 
} \]
(The natural transformation is the exchange induced by the fact that $(\twwc \alpha)^*$ preserves the condition of being $2$-Cartesian.)
\end{LEMMA}
\begin{proof}
We have $\Box_! = \pi_{4567;!} \overline{f}_* \pi_{1267}^* \mathcal{E}$ by definition. $\pi_{1267}^*$ and $\overline{f}_*$ clearly commute with arbitrary pullbacks in the obvious sense. 
Consider now the diagram with Cartesian square
\[ \xymatrix{
{}^{\downarrow\uparrow\uparrow\downarrow\uparrow \uparrow \downarrow} I \ar[rrrd]^{\pi_{4567}} \ar[rd]^\kappa \ar[rddd]_{{}^{\downarrow\uparrow\uparrow\downarrow\uparrow \uparrow \downarrow} \alpha} \\
& {}^{\downarrow\uparrow\uparrow\downarrow\uparrow \uparrow \downarrow} J  \times_{\twwc J} \twwc I \ar[rr]^{\pr_2} \ar[dd]^{\pr_1} & & \twwc I \ar[dd]^{\twwc \alpha} \\
\\
& {}^{\downarrow\uparrow\uparrow\downarrow\uparrow \uparrow \downarrow} J  \ar[rr]_{\pi_{4567}}  & & \twwc J
} \]
It shows (using Lemma~\ref{LEMMAKAN3}) that
\begin{eqnarray*} \pi_{4567,!} ({}^{\downarrow\uparrow\uparrow\downarrow\uparrow \uparrow \downarrow} \alpha)^* 
 &\cong& \pr_{2,!}  \kappa_! \kappa^* \pr_1^* \\ 
 &\cong& \pr_{2,!}  \pr_1^* \\
 &\cong&  (\twwc \alpha)^*  \pi_{4567,!}     
 \end{eqnarray*}
\end{proof}

We now turn to the case of the coCartesian projector. 

\begin{PROP}\label{PROPCOCARTPROJ}Using the notation of \ref{COCARTPROJPREP}, 
denote $\Box_* :=  \pi_{4567;!} \overline{f}_* \iota_! g^* \pi_{1234}^*$. This functor, together
with the composition
\[ \xymatrix{  \mathcal{E} & \ar[l]_-\sim  \pi_{4567;!} \pi_{1567}^* \mathcal{E} \ar[r]^-\sim & \pi_{4567;!} \overline{f}_* \pi_{1267}^* \mathcal{E} &  \pi_{4567;!} \overline{f}_* \iota_! \pi_{1237}^* \mathcal{E}  \ar[l]_{\sim}   & \ar[l]  \pi_{4567;!} \overline{f}_* \iota_! g^* \pi_{1234}^* \mathcal{E}  =  \Box_* \mathcal{E} \ar@/^20pt/[lll]^{\nu_{\mathcal{E}}}   },  \]
defines a right coCartesian projector:
\[\Box_* :  \DD(\twwc I)^{\ws, 2-{\cart}}_{\pi_{234}^* \widetilde{S}^{\op}} \rightarrow \DD(\twwc I)^{4-\cocart, \ws, 2-{\cart}}_{\pi_{234}^* \widetilde{S}^{\op}}. \]

This projector has the following property: 
\begin{itemize}
\item For each $i \in I$ the natural transformation
\begin{equation} \label{pointwisecocartproj} (\twwc i)^* \Box_*  \rightarrow  (\twwc i)^*   
\end{equation}
is an isomorphism. 
\end{itemize}
\end{PROP}

\begin{proof}
First note that the statement of the Proposition makes sense, because 
\[  \pi_{4567;!} \pi_{1567}^* \mathcal{E} \rightarrow \pi_{4567;!} \overline{f}_* \pi_{1267}^* \mathcal{E}  \]
is an isomorphism on $2$-Cartesian objects by Proposition~\ref{PROPCARTPROJ}.
We have to show the assertions 1--3. of \ref{RIGHTCOCARTPROJ}. 

1. We have to show that 
\[ \Box_* \mathcal{E} \in \DD(\twwc I)^{4-\cocart, \ws, 2-{\cart}}_{\pi_{234}^* \widetilde{S}^{\op}} \]
for a $2$-Cartesian and well-supported object $\mathcal{E}$. 

Consider a morphism $\mu$ of type 4 in $\twwc I$
\[ \xymatrix{
i=( i_1 \ar[r] \ar@<10pt>@{=}[d] & i_2  \ar[r] \ar@{=}[d] & i_3  \ar[r] \ar@{=}[d] & i_4  ) \ar[d] \\
i'=( i_1 \ar[r] & i_2  \ar[r] & i_3  \ar[r] & i_4' )
} \]
and let $G_0:=\widetilde{S}(\pi_{234}(\mu))$. 
We have to see that 
\[ G_0^* \hocolim_{{}^{\downarrow\uparrow\uparrow}( I \times_{/I} i_1)} e_i^* \overline{f}_*  \iota_! g^*  \pi_{1234}^* \mathcal{E}  \rightarrow  \hocolim_{{}^{\downarrow\uparrow\uparrow}( I \times_{/I} i_1)} e_{i'}^*  \overline{f}_* \iota_!  g^*  \pi_{1234}^* \mathcal{E} \]
is an isomorphism. Since $G_0^*$ commutes with homotopy colimits (being a left adjoint) it suffices to show that point-wise
\[ G_0^* e_i^*  \overline{f}_* \iota_! g^*  \pi_{1234}^* \mathcal{E}  \rightarrow  g^*   e_{i'}^*  \overline{f}_* \iota_!  g^*  \pi_{1234}^* \mathcal{E} \]
is an isomorphism. Note that by the proof of Lemma~\ref{LEMMAPOINTWISEOPLAX} $\iota_!$ is computed point-wise at the given input. 
Pick an object $j_1 \rightarrow j_2 \rightarrow j_3 \rightarrow i_1$ in $({}^{\uparrow \uparrow \downarrow} I \times_{/I} i_1)$
and consider the diagram
\[ \xymatrix{
 \widetilde{S}(j_2 \rightarrow j_3 \rightarrow i_1) \ar@{=}[r] \ar@{<-}[d]_{G} & \widetilde{S}(j_2 \rightarrow j_3 \rightarrow i_1) \ar@{<-}[d]^{g} \\
\widetilde{S}(j_2 \rightarrow j_3 \rightarrow i_4') \ar[r]^{G_2} \ar@{^{(}->}[d]_{I} & \widetilde{S}(j_2 \rightarrow j_3 \rightarrow i_4) \ar@{^{(}->}[d]^{\iota} \\
\widetilde{S}(j_2 \rightarrow i_3 \rightarrow i_4') \ar[r]^{G_1} \ar@{->>}[d]_{F} & \widetilde{S}(j_2 \rightarrow i_3 \rightarrow i_4) \ar@{->>}[d]^{f} \\
\widetilde{S}(i_2 \rightarrow i_3 \rightarrow i_4') \ar[r]^{G_0} & \widetilde{S}(i_2 \rightarrow i_3 \rightarrow i_4)
} \]
We are left to show that
\[  G_0^* f_*  \iota_! g^* \mathcal{E}_{j_1 \rightarrow j_2 \rightarrow j_3 \rightarrow i_1} \rightarrow F_* I_!  G^* \mathcal{E}_{j_1 \rightarrow  j_2 \rightarrow j_3 \rightarrow i_1}   \]
is an isomorphism. This is the composition
\begin{eqnarray*}
G_0^* f_*  \iota_! g^* \mathcal{E}_{j_1 \rightarrow j_2 \rightarrow j_3 \rightarrow i_1} & \rightarrow & F_*  G_1^*  \iota_! g^* \mathcal{E}_{j_1 \rightarrow j_2 \rightarrow j_3 \rightarrow i_1} \\
&\rightarrow&  F_*  I_! G_2^*  g^* \mathcal{E}_{j_1 \rightarrow j_2 \rightarrow j_3 \rightarrow i_1} \\
&\rightarrow&  F_* I_!  G^* \mathcal{E}_{j_1 \rightarrow  j_2 \rightarrow j_3 \rightarrow i_1}  
\end{eqnarray*}
That these morphisms are isomorphisms follows from Proposition~\ref{PROPPROPERTIESCORCOMPM}, because $\mathcal{E}_{j_1 \rightarrow j_2 \rightarrow j_3 \rightarrow i_1}$ has support in $\widetilde{S}'(j_2 = j_2 \rightarrow i_1)$ and hence
$g^* \mathcal{E}_{j_1 \rightarrow j_2 \rightarrow j_3 \rightarrow i_1}$ has support in $\widetilde{S}'(j_2 = j_2 \rightarrow i_4)$.

Consider a morphism $\mu$ of type 3 in $\twwc I$: 
\[ \xymatrix{
i=(i_1 \ar[r] \ar@<10pt>@{=}[d] & i_2  \ar[r] \ar@{=}[d] & i_3  \ar[r] \ar@{<-}[d] & i_4)  \ar@{=}[d] \\
i'=(i_1 \ar[r] & i_2  \ar[r] & i_3'  \ar[r] & i_4)
} \]
and let $I:=\widetilde{S}(\pi_{234}(\mu))$. 
We have to see that 
\[ I_! \hocolim_{{}^{\downarrow\uparrow\uparrow}( I \times_{/I} i_1)} e_i^* \overline{f}_* \iota_! g^* \pi_{1234}^* \mathcal{E}  \rightarrow \hocolim_{{}^{\downarrow\uparrow\uparrow}( I \times_{/I} i_1)} e_{i'}^*  \overline{f}_* \iota_! g^* \pi_{1234}^* \mathcal{E} \]
is an isomorphism, and the the right hand side has support in $\widetilde{S}'(i_2 = i_2 \rightarrow i_4)$ (with $\widetilde{S}'$ as in the definition of well-supported, cf.\@ \ref{COMPONENTS}).

 Since $I_!$ commutes with homotopy colimits (being a left adjoint) and is computed point-wise on constant diagrams it suffices to show that point-wise
\[ I_! e_i^* \overline{f}_* \iota_! g^* \pi_{1234}^* \mathcal{E}  \rightarrow  e_{i'}^* \overline{f}_* \iota_! g^* \pi_{1234}^* \mathcal{E} \]
is an isomorphism. Note that by the proof of Lemma~\ref{LEMMAPOINTWISEOPLAX} $\iota_!$ is computed point-wise at the given input. 
Pick an object $j_1 \rightarrow j_2 \rightarrow j_3 \rightarrow i_1$ in $({}^{\uparrow \uparrow \downarrow} I \times_{/I} i_1)$
and consider the diagram
\[ \xymatrix{
\widetilde{S}(j_2 \rightarrow j_3 \rightarrow i_1) \ar@{=}[r] \ar@{<-}[d]_{g} & \widetilde{S}(j_2 \rightarrow j_3 \rightarrow i_1) \ar@{<-}[d]^{g} \\
\widetilde{S}(j_2 \rightarrow j_3 \rightarrow i_4) \ar@{=}[r] \ar@{^{(}->}[d]_{I} & \widetilde{S}(j_2 \rightarrow j_3 \rightarrow i_4) \ar@{^{(}->}[d]^{\iota} \\
\widetilde{S}(j_2 \rightarrow i_3' \rightarrow i_4) \ar@{^{(}->}[r]^{I_1} \ar@{->>}[d]_{F} & \widetilde{S}(j_2 \rightarrow i_3 \rightarrow i_4) \ar@{->>}[d]^{f} \\
\widetilde{S}(i_2 \rightarrow i_3' \rightarrow i_4) \ar@{^{(}->}[r]^{I_0} & \widetilde{S}(i_2 \rightarrow i_3 \rightarrow i_4)
} \]
We are left to show that
\[  I_{0,!} F_* I_! g^* \mathcal{E}_{j_1 \rightarrow j_2 \rightarrow j_3 \rightarrow i_1} \rightarrow  f_* \iota_! g^* \mathcal{E}_{j_1 \rightarrow j_2 \rightarrow j_3 \rightarrow i_1} \cong f_* I_{1,!} I_! g^* \mathcal{E}_{j_1 \rightarrow j_2 \rightarrow j_3 \rightarrow i_1}      \]
is an isomorphism. This is true by axiom (F6). Similarly one sees that if $\mathcal{E}_{j_1 \rightarrow j_2 \rightarrow j_3 \rightarrow i_1}$ has support in $\widetilde{S}'(j_2=j_2 \rightarrow i_1)$ (with $\widetilde{S}'$ as in the definition of well-supported, cf.\@ \ref{COMPONENTS}) then 
$f_* \iota_! g^* \mathcal{E}_{j_1 \rightarrow j_2 \rightarrow j_3 \rightarrow i_1}$ has support in $\widetilde{S}'(i_2=i_2 \rightarrow i_4)$.

Consider now a morphism $\mu$ of type 2 in $\twwc I$: 
\[ \xymatrix{
i=(i_1 \ar[r] \ar@<10pt>@{=}[d] & i_2  \ar[r] \ar@{<-}[d] & i_3  \ar[r] \ar@{=}[d] & i_4)  \ar@{=}[d] \\
i'=(i_1 \ar[r] & i_2'  \ar[r] & i_3  \ar[r] & i_4)
} \]
and let $F_0:=\widetilde{S}(\pi_{234}(\mu))$. 
We have to see that 
\[ F_{0,*} \hocolim_{{}^{\downarrow\uparrow\uparrow}( I \times_{/I} i_1)} e_i^* \overline{f}_* \iota_! g^* \pi_{1234}^* \mathcal{E}  \rightarrow  \hocolim_{{}^{\downarrow\uparrow\uparrow}( I \times_{/I} i_1)} e_{i'}^*  \overline{f}_* \iota_! g^* \pi_{1234}^* \mathcal{E} \]
is an isomorphism.

 By axiom (F3) $F_{0,*}$ commutes with homotopy colimits and is computed point-wise. Therefore it suffices to show that point-wise
\[ F_{0,*} e_i^* \overline{f}_* \iota_! g^* \pi_{1234}^* \mathcal{E}  \rightarrow  e_{i'}^* \overline{f}_* \iota_! g^* \pi_{1234}^* \mathcal{E} \]
is an isomorphism. Note that by the proof of Lemma~\ref{LEMMAPOINTWISEOPLAX} $\iota_!$ is computed point-wise at the given input. 

Pick an object $j_1 \rightarrow j_2 \rightarrow j_3 \rightarrow i_1$ in $({}^{\uparrow \uparrow \downarrow} I \times_{/I} i_1)$
and consider the diagram
\[ \xymatrix{
\widetilde{S}(j_2 \rightarrow j_3 \rightarrow i_1) \ar@{=}[r] \ar@{<-}[d]_{g} & \widetilde{S}(j_2 \rightarrow j_3 \rightarrow i_1) \ar@{<-}[d]^{g} \\
\widetilde{S}(j_2 \rightarrow j_3 \rightarrow i_4) \ar@{=}[r] \ar@{^{(}->}[d]_{\iota} & \widetilde{S}(j_2 \rightarrow j_3 \rightarrow i_4) \ar@{^{(}->}[d]^{\iota} \\
\widetilde{S}(j_2 \rightarrow i_3 \rightarrow i_4) \ar@{=}[r] \ar@{->>}[d]_{F} & \widetilde{S}(j_2 \rightarrow i_3 \rightarrow i_4) \ar@{->>}[d]^{f} \\
\widetilde{S}(i_2' \rightarrow i_3 \rightarrow i_4) \ar@{^{(}->}[r]^{F_0} & \widetilde{S}(i_2 \rightarrow i_3 \rightarrow i_4)
} \]
We have to show that
\[ F_{0,*} F_* \iota_! g^*\mathcal{E}_{j_1 \rightarrow j_2 \rightarrow j_3 \rightarrow i_1}  \rightarrow f_* \iota_! g^*\mathcal{E}_{j_1 \rightarrow j_2 \rightarrow j_3 \rightarrow i_1}   \]
is an isomorphism which is clear.

2.\@ follows beause for a 4-coCartesian object $\mathcal{E}$ the morphism 
\[ g^* \pi_{1234}^* \mathcal{E} \rightarrow \pi_{1237}^* \mathcal{E} \] 
is an isomorphism. 

3.\@ is proven as for \cite[Proposition 8.5]{Hor16}. 

To see that (\ref{pointwisecocartproj}) is an isomorphism, it suffices to see that {\em on the fiber (of $\pi_{4567}$) over an object of the form $\twwc i$}, 
the morphism 
\[ g^* \pi_{1234}^* \mathcal{E} \rightarrow \pi_{1237}^* \mathcal{E} \]
is always an isomorphism. However, the natural transformation $\pi_{1234}^* \Rightarrow  \pi_{1237}^*$ restricts to an identity on this fiber. 
\end{proof}

\newpage
\bibliographystyle{abbrvnat}
\bibliography{6fu}

\end{document}